\documentclass[11pt]{article}
\usepackage{bbm}
 \usepackage{amssymb}
\usepackage{amssymb, amsthm, amsmath, amscd}
\usepackage[usenames,dvipsnames]{color}

\newtheorem{thm}{Theorem}[section]
\newtheorem{lem}[thm]{Lemma}
\newtheorem{prop}[thm]{Proposition}
\newtheorem{cor}[thm]{Corollary}
\theoremstyle{definition}
\newtheorem{NN}[thm]{}
\theoremstyle{definition}
\newtheorem{df}[thm]{Definition}
\theoremstyle{definition}
\newtheorem{rem}[thm]{Remark}

\theoremstyle{definition}

\renewcommand{\phi}{\varphi}

\newcommand{\N}{\mathbb{N}}
\newcommand{\Z}{\mathbb{Z}}

\newcommand{\R}{\mathbb{R}}
\newcommand{\C}{\mathbb{C}}

\numberwithin{equation}{section}

\newcommand{\Aff}{\operatorname{Aff}}

\newcommand{\id}{\operatorname{id}}

\newcommand{\aff}{\rm aff}

\newcommand{\cpc}{completely positive contractive map}
\newcommand{\morp}{contractive completely positive linear map}

\newcommand{\hm}{homomorphism}
\newcommand{\dt}{\delta}
\newcommand{\ep}{\varepsilon}

\newcommand{\ld}{\lambda}

\newcommand{\la}{\langle}
\newcommand{\ra}{\rangle}
\newcommand{\andeqn}{\,\,\,{\rm and}\,\,\,}
\newcommand{\rforal}{\,\,\,{\rm for\,\,\,all}\,\,\,}
\newcommand{\CA}{C*-algebra}
\newcommand{\SCA}{sub-C*-algebra}

\newcommand{\af}{{\alpha}}
\newcommand{\bt}{{\beta}}
\newcommand{\dist}{{\rm dist}}

\newcommand{\one}{{\bf 1}}

\newcommand{\diag}{{\rm diag}}

\newcommand{\wilog}{without loss of generality}
\newcommand{\Wlog}{Without loss of generality}

\newcommand{\beq}{\begin{eqnarray}}
\newcommand{\eneq}{\end{eqnarray}}
\newcommand{\tforal}{\,\,\,\text{for\,\,\,all}\,\,\,}
\newcommand{\tand}{\,\,\,\text{and}\,\,\,}

\usepackage{amsfonts}
\usepackage{mathrsfs}
\usepackage{textcomp}
\usepackage[all]{xy}

\title{Simple stably projectionless C*-algebras with generalized tracial rank one}
\author{George A. Elliott, Guihua Gong,  Huaxin Lin, and Zhuang Niu
 }
\date{}
\begin{document}

\maketitle

\begin{abstract}
We study a class of stably projectionless simple \CA s which may be viewed as having
generalized tracial rank one in analogy with the unital case.  {{A number of}} structural questions concerning these simple \CA s are studied,
%The paper also deals with technical questions related to
%serves as a technical support for 
pertinent to
the classification
of separable stably projectionless simple amenable  Jiang-Su stable \CA s.
%The underlying question is whether these algebras exhaust the axiomatically determined class of stably projectionless simple Jiang-Su stable separable amenable \CA s.
\end{abstract}

\section{Introduction}

Recent developments in the program of  classification of simple amenable \CA s
led to the classification of unital simple separable \CA s with finite nuclear dimension
in the UCT class (see, for example, \cite{GLN}, \cite{EGLN}, and \cite{TWW}).
The isomorphism theorem was established  first for those unital simple separable amenable
\CA s with generalized tracial rank at most one (see \cite{GLN}).  {{Unital}} simple \CA s
with
% finite 
generalized tracial rank at most one have good regularity properties. For example,
these simple \CA s have strict comparison for positive elements, have stable rank one,
and are quasidiagonal. All quasitraces are traces. When they are separable  and amenable, they
are ${\cal Z}$-stable.  These \CA s have many other good properties which lead to the classification
by the Elliott invariant when they satisfy the UCT.  In \cite{EGLN}, using \cite{TWW},  we showed that unital finite simple (separable) \CA s with finite nuclear dimension which satisfy the UCT  have generalized tracial rank at most one
after tensoring with a UHF-algebra of infinite type. This completed the classification
of unital simple separable \CA s with finite nuclear dimension in the UCT class.
In the present paper,  we study the non-unital version of the notion of generalized tracial rank one.

{\bf Acknowledgements.}   The greater  part of this  research
was done {{while}} the second and third named  authors were at the Research Center for Operator Algebras at East China Normal University
in the summer of 2016. Much of the revision from the initial work was done when the
last three authors were there in the summer of 2017.
All  authors acknowledge the support of the Research Center
which is partially supported by Shanghai Key Laboratory of PMMP, Science and Technology Commission of Shanghai Municipality (STCSM), grant \#13dz2260400 and a NNSF grant (11531003).
 The first named author was partially supported by NSERC of Canada.
The second named author was partially supported by the NNSF of China (Grant \#11531003),
the third named author was partially supported by NSF grants (DMS \#1361431 and  \#1665183), and the fourth named author was partially supported by a Simons Collaboration Grant (Grant \#317222).

\section{Notation}

%\begin{df}[3.10.4 of \cite{Pbk}]\label{Dtr>0}
%Let $A$ be a \CA\, and $a\in A_+$ be a positive element.
% Recall that $a$ is called strictly positive if $f(a)\not=0$ for all states $f$ of $A.$
%Note that if $a$ is strictly positive then $A=\overline{aAa}$ and
%it is $\sigma$-unital.
%\end{df}

\begin{df}\label{DTtilde}
Let $A$ be a \CA. Denote by ${\mathrm{Ped}(A)}$ the Pedersen ideal (see Section 5.6 of \cite{Pbook}). 

Denote by ${\widetilde{\mathrm{T}}}(A)$ the topological convex cone of all  densely defined  lower semicontinuous  positive traces
equipped   with  the topology of   point-wise convergence on elements of  ${\mathrm{Ped}(A)}$. Recall that, if $\tau\in {\widetilde{\mathrm{T}}}(A)$ and $b\in {\mathrm{Ped}(A)}_+$, then
$\tau$ is a finite (equivalently, bounded) trace on $\overline{bAb}.$
{{One checks easily that ${\widetilde{\mathrm{T}}}(A)$ is a Hausdorff space.}}

 Let ${\mathrm{T}(A)}$ denote the
 %\% set of 
 {{tracial}} state  space of $A.$ 
 Set  
 $$
 {\mathrm{T}_1(A)}=\{\tau\in {\widetilde{\mathrm{T}}}(A): |\tau(a)|\le 1,\,\, a\in {\mathrm{Ped}}(A)\andeqn \|a\|\le 1\}.
 $$ 
 If $\tau\in  {\mathrm{T}_1(A)},$ then it can be extended to a 
  positive linear functional {{on}} $A$ with norm  at most one.
  %no more than one.  
  Therefore, we may view ${\mathrm{T}_{{1}}(A)}$
  as a subset of the unit ball of $A^*.$ 
  In this way, we may write that  ${\mathrm{T}}(A)\subset {\mathrm{T}_1(A)}.$
 ${\rm T}_1(A)$  is also a closed set in ${\widetilde{\mathrm{T}}}(A).$
{{If}} $A={\mathrm{Ped}(A)}, $ {{then}} ${\mathrm{T}(A)}$ generates ${\widetilde{\mathrm{T}}}(A)$
 as a cone.
%We shall consider also the set of finite traces of norm at most one, ${\mathrm{T}_0(A)}=\{\af \cdot \tau: \tau\in {\mathrm{T}(A)},\,\, \, 0\le \af\le 1\}.$

Suppose that $A$ is $\sigma$-unital.
In the case that ${\mathrm{Ped}(A)}_+$ contains a full element $a$ of $A$ (in particular when
$A$ is simple),  let us clarify the structure of ${\widetilde{\mathrm{T}}}(A).$
Put $A_1=\overline{aAa}.$   Then we may identify $A$ with a $\sigma$-unital hereditary
\SCA\, of $A_1\otimes {\cal K}$ by Brown's theorem (\cite{Br1}).
By 2.1 of \cite{aTz}, $A_1={\rm Ped}(A_1).$  Therefore, ${\mathrm{T}}_{{1}}(A_1)$ generates 
${\widetilde{\mathrm{T}}}(A)={\widetilde{\rm{T}}}(A_1)$ 
as a cone.  
On the other hand, again, since $A_1={\rm Ped}(A_1),$
${\mathrm{T}_1}(A_1)$ is the weak*-compact convex subset of all tracial positive linear functionals {{on}} $A_1$ with norm at most
{{one}}, %.  Thus  ${\mathrm{T}_0}(A_1)$  has the usual  structure of a topological  convex set and it
and is a Choquet simplex.

{{Consider the closure $S$ of ${\mathrm{T}}(A)$ in ${\widetilde{\mathrm{T}}}(A)$ in the above mentioned topology. 
Let ${\overline{{\mathrm{T}}(A)}}^{\mathrm{w}}$ denote  the weak* closure of 
${\mathrm{T}}(A)$  in the dual space of $A.$}} {{Clearly, as a set, ${\overline{{\mathrm{T}}(A)}}^{\mathrm{w}}
\subset S.$}}  {{Note that each element in ${\overline{{{\mathrm{T}}}(A)}}^{\mathrm{w}}$
is a trace of $A$ with norm at most one.  Let $\imath: {\overline{{\mathrm{T}}(A)}}^{\mathrm{w}}\to S$ be the embedding as a subset.
So the map $\imath$ is one-to-one.
Suppose $t\in S.$ Then, $t$ is a densely defined linear functional 
and $|t(b)|\le 1$ for all $b\in {\rm{Ped}}(A)$ with $\|b\|\le 1.$
Thus, it uniquely extends to an element of $A^*$ with norm at most one. 
Choose a net $(\tau_\af)$ in  ${\mathrm{T}}(A)$
such that $\tau_\af(b)\to t(b)$ for all $b\in {\rm{Ped}}(A).$ 
 Since ${\rm{Ped}}(A)$ is dense in $A,$ 
and $\|\tau_\af\|=1$ and $\|t\|\le 1,$ one concludes that $\tau_\af(a)\to t(a)$ for all $a\in A.$ 
This shows that $S={\overline{{\mathrm{T}}(A)}}^{\mathrm{w}}$ as subsets.
In other words, $\imath$ is a bijection. On the other hand, if $\tau_\bt(a)\to \tau(a)$ for all $a\in A,$ 
where $\tau_\bt, \tau\in {\overline{{\mathrm{T}}(A)}}^{\mathrm{w}},$ 
then $\tau_\bt(b)\to \tau(b)$ for all $b\in {\rm{Ped}}(A).$ %This means that 
In other words, $\imath$ is continuous.
Moreover,  ${\overline{{\mathrm{T}(A)}}}^{\mathrm{w}}\subset A^*$ is compact 
and $S\subset {\widetilde{T}}(A)$ is Hausdorff.  It follows that  $\imath$ is a homeomorphism. In what follows, 
 we will identify 
$S$ with ${\overline{{\mathrm{T}(A)}}}^{\mathrm{w}},$ In particular, $S$ is compact.}}

%mentioned topology. As a set, one may view 
%Since $a\in \mathrm{Ped}(A)$, ${\widetilde{\mathrm{T}}}(A)=\{\af \tau: \tau\in {\mathrm{T}_0}(A_1)\andeqn \af\in \R_+\}.$
%It follows that we may view ${\widetilde{\mathrm{T}}}(A)$ as the cone generated by ${\mathrm{T}_0}(A_1)$ and equipped with the topology
%induced by that of ${\mathrm{T}_0}(A_1).$

%Consider the weak*-compact convex subset $\mathrm{T}_a=\{\tau\in {\widetilde{\mathrm{T}}}(A): \tau(a)=1\}.$
%$\mathrm{T}_a$ is a Choquet simplex, and is a base for the topological cone ${\widetilde{{\mathrm{T}}}}(A)$.

\end{df}

\begin{df}\label{Dfep}
Let $1>\ep>0.$  Define
\beq
f_{\ep}(t)=\begin{cases} 0, &\text{if}\,\, t\in [0,\ep/2],\\
                                      \displaystyle{t-\ep/2\over{\ep/2}}, &\text{if}\,\, t\in (\ep/2, \ep],\\
                                       1 & \text{if}\,\, t\in (\ep, \infty).\end{cases}
                                       \eneq
\end{df}

\begin{df}\label{Ddimf}
Let $A$ be a \CA\, and let $a\in A_+.$  Suppose that ${\widetilde{\mathrm{T}}}(A)\not=\O.$
Define
$$
\mathrm{d}_\tau(a)=\lim_{\ep\to 0} \tau(f_\ep(a)),\,\,\,\,\,\, \tau\in {\widetilde{\mathrm{T}}}(A),
$$
with possibly infinite values.  Note that $f_\ep(a)\in {\mathrm{Ped}(A)}_+,$ and by definition
$\tau \mapsto \tau(f_\ep(a))$ is a continuous affine function  on ${\widetilde{\mathrm{T}}}(A)$
(to $[0,+\infty)$). It follows that $\tau\mapsto\mathrm{d}_\tau(a)$ is a lower semicontinuous affine function on $\widetilde{\mathrm{T}}(A)$ (to $[0, +\infty]$). Note
that
$$
\mathrm{d}_\tau(a)=\lim_{n\to\infty}\tau(a^{1/n}),\quad  \tau\in {\widetilde{\mathrm{T}}}(A).
$$

%Suppose that $A$ is non-unital.
Let $a\in A_+$ be a strictly positive element.
Define
$$
\Sigma_A(\tau)=\mathrm{d}_\tau(a),\quad  \tau\in {\widetilde{\mathrm{T}}}(A).
$$
The lower semicontinuous affine function $\Sigma_A: {\widetilde{\mathrm{T}}}(A) \to [0, +\infty]$ is independent of the choice
of $a,$ and will be called the scale function (or just scale) of $A.$
\end{df}

\begin{df}\label{Dcu}
Let $A$ be a \CA\, and let $a, b\in A_+.$
We write $a\lesssim b$ if there exists a sequence $(x_n)$ in $A$
such that $x_n^*bx_n\to a$ in norm. If $a\lesssim b$ and $b\lesssim a,$ we
write $a\sim b$ and say that $a$ and $b$ are Cuntz equivalent.
 It is known that $\sim$ is an equivalence relation.
 %Let $W(A)$ be the equivalence class of positive elements in $M_n(A)$ for all $n$ with the usual embedding from $M_n$ into $M_{n+1}.$
 Denote by
 ${\mathrm{Cu}}(A)$ the set of Cuntz equivalence classes of positive elements of $A\otimes {\cal K}.$
 It is an ordered  abelian semigroup (\cite{CEI}).
 Denote by $\mathrm{Cu}(A)_+$ the subset of those elements which cannot be represented by
 projections. 
 We shall write  $\la a \ra$ for the equivalence class represented by $a.$
 Thus, $a\lesssim b$ will be also written as $\la a\ra \le \la b\ra.$
 Recall that we write $\la a\ra \ll \la b\ra$ if  the following holds: for any increasing sequence $(\la y_n\ra ),$
  if $\la b\ra \le  \sup\{\la y_n \ra \}$ then  there exists $n_0\ge 1$ such that ${{\la a\ra \le \la y_{n_0}\ra}} $
In what follows we will also use the objects
%notation 
${\mathrm{Cu}}^{\sim}(A)$ and ${\mathrm{Cu}}^{\sim}(\phi)$  introduced
in \cite{Rl}.
\end{df}

\begin{df}\label{Dstrictcom}

{{If}} ${{B}}$ is a \CA,
we will use $\mathrm{QT}({{B}})$ for the set of quasitraces $\tau$ with $\|\tau\|=1$
{{(see \cite{BH}).}}
{{Let $A$ be a $\sigma$-unital  \CA.}}
%Suppose that ${\widetilde{\mathrm{T}}}(A)\not=\{0\}.$ 
Suppose  that every quasitrace  {{of every hereditary \SCA\, $B$ of $A$}} is a trace.

If $\tau\in {\widetilde{\mathrm{T}}}(A),$ we will extend it to $(A\otimes {\cal K})_+$ by the rule
$\tau(a\otimes b)=\tau(a)\mathrm{Tr}(b),$  for all $a\in A$ and $b\in {\cal K},$ where Tr is the canonical densely defined trace on ${\cal K}.$

Recall that  $A$ has the (Black{{a}}dar) property of  strict comparison for positive elements, if for any two elements $a, \, b\in (A\otimes\mathcal K)_+$
%(for any integer $n\ge 1$)
with the property that $\mathrm{d}_\tau(a)<\mathrm{d}_\tau(b)<+\infty$ for all $\tau\in {\widetilde{\mathrm{T}}}(A)\setminus \{0\},$
necessarily $a\lesssim b.$  In general (without knowing that quasitraces are traces), this property  will be called {\it strict comparison for positive elements using 
traces}.
% even without knowing quasitraces 

Let $S$ be a topological convex set. Denote by $\Aff(S)$ the set of all real continuous affine functions, 
and by $\Aff_+(S)$ the set of all real continuous affine functions $f$ with $f(s)>0$ for all $s,$ together with zero function.

%$A.$   We view $A$ as a $\sigma$-unital hereditary \SCA\, of $A_1\otimes {\cal K}$ as in \ref{DTtilde} (with $A_1=\overline{aAa}$).
%View ${\mathrm{T}_0}(A_1)$ as a convex subset of ${\widetilde{\mathrm{T}}}(A)$, generating this cone.

%Then $A$ has strictly comparison for positive elements if, and only if,
%for any two positive elements $a, \, b\in (A\otimes\mathcal K)_+$ with
%with property that $\mathrm{d}_\tau(a)<\mathrm{d}_\tau(b)<\infty$ for all $\tau\in {\mathrm{T}_0}(A_1)\setminus \{0\},$
%necessarily $a\lesssim b.$

%Suppose that ${\mathrm{T}(A)}\not=\O.$
Recall 
$\overline{{\mathrm{T}(A)}}^\mathrm{w}$ denotes the closure of ${\mathrm{T}(A)}$ in
${\widetilde{\mathrm{T}}}(A)$
{{with respect to pointwise  convergence on ${\rm Ped(A)}$}} {{(see the end of \ref{DTtilde})}}.
Suppose that $0\not\in \overline{{\mathrm{T}(A)}}^\mathrm{w}$
%%(see \ref{compactrace} below)
%and is weak*-compact (in $A^*$),
and {{that}} ${\mathrm{T}}(A)$ generates ${\widetilde{\mathrm{T}}}(A),$
in particular. {{(By \ref{compactrace} below,}} these properties hold,  in the case that $A=\mathrm{Ped}(A)$.)  Then $A$ has strict comparison  for positive elements using traces if
and only if
$d_\tau(a)<d_\tau(b)$ for all $\tau\in \overline{{\mathrm{T}(A)}}^\mathrm{w}$ implies 
$a\lesssim b,$  for any $a, \, b\in (A\otimes {\cal K})_+.$  
%where   $n$ is any positive integer.
%whenever $d_\tau(a)<d_\tau(b)$
%for all
%We shall say that $A$ has  the  property of {\it strong} strict comparison for positive elements, if for any two elements $a, \, b\in (A\otimes\mathcal K)_+$
%with the property that $\mathrm{d}_\tau(a)<\mathrm{d}_\tau(b)$ for all $\tau\in \overline{{\mathrm{T}(A)}}^\mathrm{w},$
%necessarily $a\lesssim b.$   In the case that ${\mathrm{T}}(A)$ generates ${\widetilde{\mathrm{T}}}(A),$
%in particular, in the case that $A=\mathrm{Ped}(A),$ strong strict comparison is the same as strict comparison.
%(we know of no examples that they are different).

%Then it is easy to check
\end{df}

\begin{df}\label{DAq}
Let $A$ be a \CA\, %with ${\mathrm{T}(A)}\not=\O$
such that $0\not\in \overline{{\mathrm{T}(A)}}^\mathrm{w}.$
There is  
%an affine 
{{a linear}}  map
$r_{\aff}: A_{\mathrm{s.a.}}\to \Aff(\overline{{\mathrm{T}(A)}}^\mathrm{w}),$ {{from $A_{s.a.}$ to}} the set of all
real affine {{continuous}} functions on $\overline{{\mathrm{T}(A)}}^\mathrm{w},$  defined by
$$
r_{\aff}(a)(\tau)=\hat{a}(\tau)=\tau(a)\tforal \tau\in \overline{{\mathrm{T}(A)}}^\mathrm{w}
$$
and for all $a\in A_{\mathrm{s.a.}}.$ Denote by $A_{\mathrm{s.a.}}^q$ the space  $r_{\aff}(A_{\mathrm{s.a.}})$ and by 
$A_+^q$ the cone $r_{\aff}(A_+)$ (see \cite{CP}).

Denote by ${\Aff}_0({\rm T}_1(A))$ the set of all 
real continuous affine functions which vanish at zero, and denote by ${\Aff}_{0+}({\rm T}_1(A))$ 
the subset of those $f\in \Aff_0({\rm T}_1(A))$ such that $f(t)>0$ for all $t\in {\rm T}_1(A)\setminus \{0\}$ 
and  the zero function.
Denote by ${\rm LAff}_{0+}({\overline{\rm{T}(A)}^{\rm w}})$ the set of those functions $f$ on $\overline{{\mathrm{T}(A)}}^\mathrm{w}$
%$\overline{{\mathrm{T}(A)}}^\mathrm{w}$ 
(with values  in $[0, {{+}}\infty]$) such that there exists an increasing  sequence 
of  continuous affine functions $f_n\in \Aff_{0+}({\mathrm{T}}_1(A))$
 such
that $f_n|_{\overline{{\mathrm{T}(A)}}^\mathrm{w}}\nearrow f$ (as $n\to\infty$) {{and the}} zero function. 
%%Denote by ${\rm LAff}_{b,+}(\overline{{\mathrm{T}(A)}}^\mathrm{w})$
%%the subset of bounded functions in ${\rm LAff}_{+}(\overline{{\mathrm{T}(A)}}^{\mathrm w}).$
In particular, if $f\in \Aff_{0+}({\rm T}_1(A)),$ then $f|_{\overline{{\mathrm{T}(A)}}^\mathrm{w}}\in {\rm LAff}_{0+}({\overline{\rm{T}(A)}^{\rm w}}).$
%%Denote by ${\rm LAff}_{+}(\overline{{\mathrm{T}(A)}}^\mathrm{w})$ the set of  those
%$\overline{{\mathrm{T}(A)}}^\mathrm{w}$ (with values  in $[0, {{+}}\infty]$) such that there exists a sequence
%%of strictly positive continuous affine functions $f_n\in \Aff(\overline{{\mathrm{T}(A)}}^\mathrm{w})$ such
%that $f_n\nearrow f$ (as $n\to\infty$), {{together with the}} zero function. 
Denote by ${\rm LAff}_{b,0+}({\overline{{\mathrm{T}(A)}}^\mathrm{w}})$
the subset of bounded functions in ${\rm LAff}_{0+}({\overline{{\mathrm{T}(A)}}^{\mathrm w}}).$

Note if $T(A)=\overline{{\rm T}(A)}^{\rm w},$ Then ${\rm LAff}_{0+}({\overline{\rm{T}(A)}^{\rm w}})$ 
is the set of all those functions $f$ on $T(A)$ which is the limit of an increasing sequence of strictly 
positive continuous affine functions $f_n\in \Aff(T(A))$ and zero function.
In this case the set will also be denoted by ${\rm LAff}_+({\rm{T}(A)}).$ 
We will also use  ${\rm LAff}_{b,+}({\mathrm{T}}(A))$
for ${\rm LAff}_{b,0+}({\overline{{\mathrm{T}(A)}}^\mathrm{w}})$ in this case. 
%strictly positive
%lower-semi-continuous affine functions on $\overline{T(A)}^\mathrm{w}.$
\end{df}

\pagebreak

%\begin{df}\label{Dcpcamen}
%Let $A$ and $B$ be \CA s. Let $\phi: A\to B$ be a \cpc.
%We say that $\phi$ is amenable (or nuclear), if, for any $\ep>0,$ and any finite subset ${\cal F}\subset A,$
%there exists an integer $k\ge 1$ and there exist  \cpc s $\psi_0: A\to \mathrm{M}_k$ and
%$\psi_1: \mathrm{M}_k\to B$ such that
%$$
%\|\psi_1\circ\psi_0(a)-\phi(a)\|<\ep\rforal a\in {\cal F}.
%$$
%\end{df}

\begin{df}[cf.~\cite{Rlz}]\label{Dalst1}
Let $A$ be a non-unital \CA. We shall say that $A$ almost has stable rank one
if 
for any integer $m\ge 1$ and 
any hereditary \SCA\, $B$ of 
%$A,$
%\subset 
$\mathrm{M}_m(A),$
$B\subset \overline{\mathrm{GL}({\widetilde B})}, $ where  $\mathrm{GL}({\widetilde B})$ is the group of invertible elements of
${\widetilde B}.$  
This definition is slightly different from that in \cite{Rlz}.

%Note that if $A\otimes {\cal K}$  is $\sigma$-unital 
Suppose that $A\otimes {\cal K}$ is $\sigma$-unital, almost has stable rank one, 
and contains a full projection $e,$ then $B=e(A\otimes {\cal K})e$
is unital. Since $B$  almost has stable rank one, it follows that $B$ has stable rank one.
{{By Theorem 6.4 of \cite{Rff1}, $B\otimes {\cal K}$ has stable rank one. By Brown's stable isomorphism
theorem (\cite{Br1}), $A\otimes {\cal K}$ has stable rank one.}}
%{\blue{Therefore $B\otimes {\cal K}$ has 
This implies that $A$ has stable rank one
(see, for example, 3.6 of \cite{BP}).
Therefore, a $\sigma$-unital simple 
\CA\,   $A$  which almost has stable rank one but does not have stable rank one  must be stably
projectionless. (We know of no such example.)

\end{df}

\begin{df}\label{Dappcpc}
Let $A$ and $B$ be \CA s and let $\phi_n: A\to B$ be
\cpc s. We shall say  that $(\phi_n)_{n=1}^\infty$ is a sequence of approximately multiplicative \cpc s
if
$$
\lim_{n\to\infty}\|\phi_n(a)\phi_n(b)-\phi_n(ab)\|=0\rforal a,  b\in A.
$$
\end{df}

\begin{df}
Let $A$ be a \CA.  Denote by $A^{\bf 1}$ the {{closed}} unit ball of $A$, and by
$A_+^{q, {\bf 1}}$ the image of the intersection $A_+\cap A^{{\bf 1}}$ in $A_+^q.$
\end{df}

{{We would like to end this section with the following proposition which is probably known.

\begin{prop}\label{180915sec2}
Let $A$ be a $\sigma$-unital \CA\, and $B$ be another \CA.
Suppose that $\phi: A\to B$ is a
% \morp\,  
contractive positive linear map and suppose 
that $e\in A_+$ is a strictly positive element. 
Then 
%$b=\phi(e)$
%\sum_{n=1}^{\infty}\phi(e^{1/n})/2^n$ 
%is a strictly positive element of 
$\overline{\phi(e)B\phi(e)}=\overline{\phi(A)B\phi(A)}.$ 
%and 
%$b=\phi(e)$ is a strictly positive element of 
\end{prop}

\begin{proof}

Let $C=\overline{\phi(A)B\phi(A)}.$ 
%and 
%let $g$ be a state of $C.$  
Consider an approximate identity $(b_\af)$ of $B.$
Then $\phi(e)b_\af\phi(e)\to \phi(e)^2.$ It follows that $\phi(e)^2\in C.$
Consequently $\phi(e)\in C.$ It follows that $\overline{\phi(e)B\phi(e)}\subset C.$%

Let $g$ be a state of $C.$ Suppose that $g(\phi(e))=0.$ We will show 
that $g=0.$ 

For any $n\ge 1,$ 
%$e^{1/n}\le 2^n b.$ 
there exists  $\lambda_n>0$ such that $f_{1/n}(e)\le \lambda_n  e.$
It follows that $g(\phi(f_{1/n}(e)))\le \lambda_ng(\phi(e))=0.$
Fix $a\in A_+.$ 
%and let $\ep>0.$ There exists $n\ge 1$ such that
%$\|a-e^{1/n}ae^{1/n}\|<\ep.$ 
%$\|a-f_{1/n}(e)af_{1/n}(e)\|<\ep.$
%Therefore 
%\beq
%\|\phi(a)-\phi(f_{1/n}(e)af_{1/n}(e))\|<\ep.
%\eneq
%But 
Then $$g(\phi(f_{1/n}(e)af_{1/n}(e)))\le \|a\| g(\phi(f_{1/n}(e)^2))\le \|a\|g(\phi(f_{1/n}(e)))=0.$$
Since $\lim_{n\to\infty}\|\phi(a)-\phi(f_{1/n}(e)af_{1/n}(e)\|=0,$ one concludes that $g(\phi(a))=0.$
%It follows 
%that
%\beq
%|g(\phi(a))|<\ep.
%\eneq
%Since this holds for all $\ep>0,$ $g(\phi(a))=0.$ 
This implies that $g(\phi(x))=0$ for all $x\in A.$ 

We claim that, for any $a\in A_+\setminus\{0\},$ $g(\phi(a)b\phi(y))=0$ for any $b\in B$ and $y\in A.$
In fact, 
\beq
|g(\phi(a)b\phi(y))|^2\le g(\phi(a)\phi(a))g(\phi(y)^*b^*b\phi(y))=g(\phi(a)^2)g(\phi(y)^*b^*b\phi(y))\\
=g(\|a\|^2\phi({a\over{\|a\|}})^2)g(\phi(y)^*b^*b\phi(y))=\|a\|^2g(\phi({a\over{\|a\|}})^2)g(\phi(y)^*b^*b\phi(y))\\
\le \|a\|^2g(\phi({a\over{\|a\|}}))g(\phi(y)^*b^*b\phi(y))=0.
\eneq
%In general,
Recall that, 
 for any $x\in A,$ we may write $x=(x_1-x_2)+i(x_3-x_4),$ where $x_i\in A_+,$ $i=1,2,3,4.$
Therefore, for any $b\in B$ and $y\in A,$ 
\beq
g(\phi(x)b\phi(y))=g(\sum_{i=1}^4\phi(x_i)b\phi(y))=\sum_{i=1}^4g(\phi(x_i)b\phi(y))=0.
\eneq
It follows that $g(z)=0$ if $z=\sum_{j=1}^m\phi(a_j)b_j\phi(c_j),$ where $b_j\in B$ and $a_j, c_j\in A,$ $j=1,2,...,m.$
Since the set $\{\sum_{j=1}^m \phi(a_j)b_j\phi(c_j): b_j\in B,\, a_j, c_j\in A\}$ is dense in $C,$ 
one concludes that  $g(c)=0$ for all $c\in C.$  In other words, $g=0.
$ 

This shows that $\phi(e)$ is a strictly positive element of $C.$
Therefore  $\overline{\phi(e)C\phi(e)}=C.$ On the other hand, 
$C\supset \overline{\phi(e)B\phi(e)}\supset \overline{\phi(e)C\phi(e)}=C.$ So $C=\overline{\phi(e)B\phi(e)}.$
%\fi
\iffalse
Now let $c\in C_+.$ For any $\ep>0,$ there exist $x, y\in A$  and $b\in B$ such that
\beq
\|c-\phi(x)b\phi(y)\|<\ep.
\eneq
This implies that $|f(c)|<\ep.$ 

One notes
that (the second inequality follows from the Stinespring theorem)
%since $\phi$ is a \morp,
\beq
|g(\phi(x)b_1\phi(y))|^2 &\le&  g(\phi(x)\phi(x)^*)g(\phi(y)^*b_1^*b_1\phi(y))\\
 &\le&  g(\phi(xx^*))g(\phi(y)^*b_1^*b_1\phi(y))=0.
\eneq
 It follows that $|g(c)|<\ep.$ 
 
 Since this holds for all $\ep>0,$ one concludes that $g(c)=0.$
Therefore $g=0.$
\fi
\end{proof}
}}

\section{ Some results of R\o rdam}
%Preliminaries}

For convenience, we  would like to have the following version
of  a lemma of R\o rdam:

\begin{lem}[R\o rdam, Lemma 2.2 of \cite{Rr11}]\label{Lrorm}
Let $a, \, b\in A$ with $0\le a,\, b\le 1$ be such that
$\|a-b\|<\dt/2.$
Then there exists $z\in A$ with $\|z\|\le 1$ such that
$$
(a-\dt)_+=z^*bz.
$$
\end{lem}

\begin{proof}
{{It follows  {{from the hypothesis}} that there exists $2>\af>1$ 
%and $0<\dt_0<\dt/2$ 
such that
\beq
\dt_0:=\|a-b^{\af}\|<\dt/2.
\eneq
Put $c=b^{\af}.$}} 
{{By Lemma 2.2 of \cite{Rr11}, 
%\%with $\dt_0=\|a-b\|,$
$$
f_\dt(a)^{1/2} (a-\dt_0\cdot 1)f_\dt(a)^{1/2}\le f_\dt(a)^{1/2}c f_\dt(a)^{1/2},
$$}}
{{where $f_\dt$ is as defined in \ref{Dfep}.}}
{{Therefore,
$$
f_\dt(a)^{1/2} (a-\dt_0\cdot 1)_+f_\dt(a)^{1/2}\le  f_\dt(a)^{1/2}c f_\dt(a)^{1/2}.
$$
Thus,
$$
(a-\dt)_+\le f_\dt(a)^{1/2} (a-\dt_0\cdot 1)_+f_\dt(a)^{1/2}\le  f_\dt(a)^{1/2}c f_\dt(a)^{1/2}.
$$}}
{{Choose $0<\bt<1$ such that $\bt\af>1.$ 
Put $a_1=(a-\dt)_+$ and $b_1=f_\dt(a)^{1/2}c f_\dt(a)^{1/2}.$
Then, as in the proof of Lemma 2.3 of \cite{Rr11}, by 1.4.5 of \cite{Pbook}, 
there is $r_1\in A$ such that $\|r_1\|\le \|b_1^{1/2-\bt/2}\|\le 1$ and
$a_1^{1/2}=r_1b_1^{\bt/2}.$ Therefore, 
$a_1=r_1b_1^{\bt}r_1^*.$
Note that $b_1^{\bt}=(f_\dt(a)^{1/2}c f_\dt(a)^{1/2})^{\bt}.$ Write 
$y=f_\dt(a)^{1/2}c^{1/2}.$ Then $yy^*=b_1.$ 
Let $y=v(y^*y)^{1/2}$ be the polar decomposition of $y$ in $A^{**}$ (so that $v\in A^{**}$). 
It follows from 1.4 of \cite{Cu1} 
that  $vx\in A$ for all $x\in \overline{(y^*y)A(y^*y)}$ and 
$v(y^*y)^{\bt}v^*=b_1^{\bt}.$ Note that 
\beq
(y^*y)^{\bt}=(c^{1/2}f_\dt(a)c^{1/2})^{\bt}\le c^{\bt}=b^{\af\bt}.
\eneq
Put $\gamma=1/(\af \bt).$ Then $0<\gamma<1.$ 
Let $x=(c^{1/2}f_\dt(a)c^{1/2})^{\bt/2}.$ 
Then $x^2\le c^{\bt}=b^{\af\bt}.$  Put $u_n=x((1/n)+(b^{\af\bt})^{1/2}{{)}} (b^{\af\bt})^{1/2-\gamma/2},$ $n=1,2,....$ 
Then, as in the proof of 1.4.5 of \cite{Pbook}, 
$\|u_n\|\le \|(b^{\af\bt})^{1/2-\gamma/2}\|\le 1$  and 
$(u_n)_{n\ge 1}$ converges to $u\in A$ in norm. Moreover,  
$x=u(b^{\af\bt})^{\gamma/2}.$ It follows that 
that $(y^*y)^{\bt}=xx=xx^*=u(b^{\af\bt})^{\gamma}u^*=ubu^*.$
Note that, $x= (y^*y)^{\bt/2},$\, $vx, v^*x\in A.$ 
Therefore $vu_n\in A$ for all $n.$ It follows that $vu\in A.$ Note also {{that}} $\|vu\|\le 1.$
Now 
$$
(a-\dt)_+=a_1=r_1b_1^{\bt}r_1=r_1(v(y^*y)^{\bt}v^*)r_1^*=(r_1vu)b(u^*v^*r_1)=z^*bz,
$$
where $z=u^*v^*r_1=(vu)^*r_1\in A$ and $\|z\|\le 1.$}}
%%%%%%%%%
\iffalse
The conclusion of the lemma follows {{now}} from Lemma 2.3 of \cite{Rr11}.
{{In fact, let $x=f_\dt(a)^{1/2} (a-\dt_0\cdot 1)_+f_\dt(a)^{1/2}.$ In the commutative \CA\, 
generated by $a$ and $1$ which is isomorphic to $C({\rm sp}(a)),$
$(a-\dt)_+(t)+\eta/2<x(t)$ for all $t\in {\rm sp}(a)\subset [0,1].$ It follows that there is 
$\af>1$ such that $x^{\af}\ge (a-\dt)_+.$}}
Then, as in the proof 2.3 of \cite{Rr11}, by 1.4.5 of \cite{Pbook}, $((a-\dt)_+)^\af=r(f_\dt(a)^{1/2}b f_\dt(a)^{1/2}),$
where $\|r\|\le \|(f_\dt(a)^{1/2}b f_\dt(a)^{1/2})^{1/2-1/4}\|\le 1.$
Then 
$$
(a-\dt)_+=r(y)^{1/2}r
$$
\fi
%%%%%%%%
\end{proof}

\begin{lem}[{{cf.\,(v) of Proposition 2.4 of \cite{RorUHF2} and 
%Proposition 1  
Theorem 3 of \cite{CEI}}}]\label{Lalmstr1}
Suppose that $A$ is a (non-unital) \CA\, which almost has stable rank one.
Suppose that $a, b\in A_+$ are
% two elements
such that $a\lesssim b.$
Then, for any $0< \dt,$  there exists a unitary $u\in {\widetilde A}$ such that
$$
u^*f_\dt(a)u\in \overline{bAb}.
$$
Moreover, there exists $x\in A$ such that
$$
x^*x=a \tand  xx^*\in \overline{bAb}.
$$
Furthermore, if $0\le a_1, a_2, b\le 1$ are in $A,$ and 
$a_1a_2=a_1,$  and $a_2\lesssim b,$ then there exists 
a unitary $u\in {\widetilde A}$ such that
\beq
u^*a_1u\in \overline{bAb}.
%\tforal x\in \overline{a_1Aa_1}.
\eneq

{{In addition, if $d\in (A\otimes {\cal K})_+,$ then, for any $\ep>0,$  there exists a unitary $u\in {\widetilde{A\otimes {\cal K}}}$ such that 
$uf_\ep(d)u^*\in {\rm M}_n(A)$ for some large $n.$ This last statement also holds when $A$ is unital.}}
\end{lem}

\begin{proof}
{{The first   statement follows from the proof of part (v) of \cite{RorUHF2}. 
The second  statement (also the first one) 
follows from Proposition 3.3 of \cite{Rlz} (see also  Corollary 6 of \cite{Pedjot87} and Lemma 1.4 of \cite{Lncuntz}).}}

To see {{the third statement, note that, by the first  statement,   for any $\dt>0,$}} 
%From the first part of the lemma,
 there exists a unitary $u\in {\widetilde A}$ such that
\beq
u^*f_\dt(a_2)u\in \overline{bAb}.
\eneq
{{Since}} $a_1a_2=a_1=a_2a_1,$ {{one has}} $f_\dt(a_2)^{1/2}a_1=a_1.$ 
Therefore, 
\beq
u^*a_1u=u^*f_\dt(a_2)^{1/2}a_1f_\dt(a_2)^{1/2}u\le u^*f_\dt(a_2)u\in \overline{bAb}.
\eneq

{{To see the last statement, let $(e_n)_{n\ge 1}$ be an approximate identity of $A\otimes {\cal K}$ such that
$e_n\in M_n(A),$ $n=1,2,....$  \Wlog, we may assume that $0\le d\le 1.$ 
Then, for any $\ep>0,$ there exists 
$n\ge 1$ such that $\|d-e_nde_n\|<\ep/4.$  By \ref{Lrorm}, $f_{\ep/2}(d)\lesssim e_nde_n.$ 
Thus the last conclusion follows from the first statement. One also notes that we do not use the condition that $A$ is not unital 
in the last few lines.}}
\end{proof}

We shall also need the following variant of \ref{Lrorm}.

%We shall also need following variant of \ref{Lrorm}.

\begin{lem}\label{LRL}
Let $1>\ep>0$ and $1>\sigma>0$ {{be given}}. There exists $\dt>0$ satisfying the following condition:
If $A$ is a \CA, and {{if}} $x, y\in A_+$ are such that $0\le x,\, y\le 1$ and
$$
\|x-y\|<\dt,
$$
then there exists a partial isometry $w\in A^{**}$ with
\beq\label{LRL-1}
ww^*f_{\sigma}(x)=f_{\sigma}(x)ww^*=f_{\sigma}(x), \tand\\
w^*cw\in \overline{yAy}, \,\,\,
%\tforal a\in \overline{f_\sigma(x)Af_\sigma(x)},\tand\\
\|w^*cw-c\|<\ep\|c\|\tforal  c\in \overline{f_{\sigma}(x)Af_{\sigma}(x)}.
\eneq

If $A$ almost has stable rank one, then $w$ may be chosen to be a unitary in ${\widetilde A}.$
%there exists a unitary $u\in {\widetilde A}$
%to replace $w$ above.
%such that
%If $A$ and a finite subset ${\cal F}\subset \overline{xAx}$ is given,
%one can choose $\dt$ so that in addition,
%$$
%\|w^*cw-c\|<\ep\rforal c\in {\cal F}.
%$$
\end{lem}

\begin{proof}
%{\red{I didn't insert George's comment on this proof. There are too many need to verify.}}
Let $\ep/4>\dt_1>0$ be such that, for any \CA\, $B,$ and any pair of positive elements
$x',\, y'\in B$ with $0\le x',\, y'\le 1$ such that
$$
\|x'-y'\|<\dt_1,
$$
then
\beq\label{LRL-n1}
\|f_{\sigma/2}(x')-f_{\sigma/2}(y')\|<\sigma\cdot \ep/8.
%\andeqn \|(x')^{1/2}-(y')^{1/2}\|<\sigma \cdot \ep/8
\eneq

Put $\eta=(\sigma\dt_1/16)^2.$
%\min\{(\sigma\cdot \dt_1/16)^2, (\sigma/8)^2\}.$
Define
$g(t)=f_{\sigma/2}(t)/t$ for all $0<t\le 1$ and $g(0)=0.$ Then ${{g}}\in C_0((0,1]).$
Note that $\|g\|\le 2/\sigma.$
Set  $\dt_2=\eta\dt_1/16$ and choose $0<\dt<\eta/3$ such that, for any \CA\, $B,$ and any pair of positive elements
$x'',\, y''\in B$ with $0\le x'',\, y''\le 1$ such that
$$
\|x''-y''\|<2\dt,
$$
one has
\beq\label{LRL-nn1}
%\|f_{\sigma/2}(x')-f_{\sigma/2}(y')\|<\sigma\cdot \ep/64\andeqn
\|(x'')^{1/2}-(y'')^{1/2}\|<\dt_2.
\eneq

Now let $A$ be a \CA\, and let $x,\, y\in A$ be such that
$0\le x,\, y\le 1$ and $\|x-y\|<\dt.$

Then
\beq\label{LRL-n2}
\|x^2-y^2\|=\|x^2-xy+xy -y^2\|<2\dt.
\eneq
Set $z=yf_{\eta}(x^2)^{1/2}.$
Then, by  \eqref{LRL-nn1},
\beq\label{LRL-n3}
\|(z^*z)^{1/2}-x\| &=&\|(f_{\eta}(x^2)^{1/2}y^2f_{\eta}(x^2)^{1/2})^{1/2}-x\|\\
&<& \dt_2+\|(f_{\eta}(x^2)^{1/2}x^2f_{\eta}(x^2)^{1/2})^{1/2}-x\|\\\label{LRL-n3+}
&<& \dt_2+\sqrt{\eta}<\sigma\cdot \dt_1/8.
\eneq
Also,
\beq\label{LRL-n4}
\|(z^*z)^{1/2}-z\|&<&\sigma\cdot \dt_1/8+\|x-yf_{\eta}(x^2)^{1/2}\|\\
&<& \sigma\cdot \dt_1/8+\dt+\|x-xf_{\eta}(x^2)^{1/2}\|\\\label{LRL-n4+}
&<&  \sigma \cdot \dt_1/8+\dt+\sqrt{\eta}<\sigma \cdot \dt_1/4.
\eneq
Consider the polar decomposition $z=v(z^*z)^{1/2}$  of $z$  in $A^{**}.$
Then
\beq\nonumber
%\label{LRL-3}
\hspace{-0.05in}\|vf_{\sigma/2}(x)-f_{\sigma/2}(x)\| &\le & \|vf_{\sigma/2}(x)-vf_{\sigma/2}((z^*z)^{1/2})\|+\|vf_{\sigma/2}((z^*z)^{1/2})-f_{\sigma/2}(x)\|\\\nonumber
&<&\sigma\cdot \ep/8+\|v(z^*z)^{1/2} g((z^*z)^{1/2})-f_{\sigma/2}(x)\|\hspace{0.8in}\, ({\rm using\, \eqref{LRL-n1}})\\\nonumber
&=&\sigma\cdot \ep/8+\|zg((z^*z)^{1/2})-f_{\sigma/2}(x)\|\\\nonumber
&\le & \sigma\cdot \ep/8+\|zg((z^*z)^{1/2})-(z^*z)^{1/2}g((z^*z)^{1/2})\|\\\nonumber
&&\hspace{0.2in}+\|(z^*z)^{1/2}g((z^*z)^{1/2})-f_{\sigma/2}(x)\|\\\nonumber
&< &\ep/8 +\dt_1/2+\|(z^*z)^{1/2}g((z^*z)^{1/2})-f_{\sigma/2}(x)\|\hspace{0.5in} ({\rm using\, \eqref{LRL-n4+}}) \\\nonumber
&<&\ep/4+\|f_{\sigma/2}((z^*z)^{1/2})-f_{\sigma/2}(x)\|\\\nonumber
&<&\ep/4+\sigma\cdot \ep/8<\ep/2.\hspace{1.5in} ({\rm using\, \eqref{LRL-n1}\, \, and\,\, \eqref{LRL-n3+}})
\eneq
Hence, 
%by  the last  inequality, 
for any  $c\in \overline{f_{\sigma}(x)Af_{\sigma}(x)}$ with $\|c\|\le 1,$
\beq\label{LRL-12}
\|vcv^*-c\| &=&\|vf_{\sigma/2}(x)cf_{\sigma/2}(x)v^*-c\|\\
&<& \ep/2+\|f_{\sigma/2}(x)cf_{\sigma/2}(x)v^*-c\|\\
&=& \ep/2+\|vf_{\sigma/2}(x)c^*f_{\sigma/2}(x)-c^*\|\\
&<& \ep/2+\ep/2=\ep.
\eneq

%Note that  $v^*v$ is  the open projection arising as the range projection of
%the positive element $f_\eta(x^2)^{1/2}y^2f_\eta(x^2)^{1/2}.$ 
It follows from {{(the proof of)}} 2.2 of \cite{Rr11} that
$$
(\eta-\|x^2-y^2\|)f_\eta(x^2)\le f_\eta(x^2)^{1/2} y^2 f_\eta(x^2)^{1/2}\le f_\eta(x^2).
$$
%Therefore $v^*v$ is also the range projection of
%$f_\eta(x^2).$ 
So $\overline{f_\eta(x^2)Af_\eta(x^2)}$ is the same as the hereditary \SCA\, generated by $z^*z=f_\eta(x^2)^{1/2}y^2f_\eta(x^2)^{1/2}.$  Note also that the hereditary \SCA\, generated by $zz^*$ is contained  in $\overline{yAy}.$
It follows
that
\beq\label{GE-1}
vcv^*\in 
%\overline{y^{1/2}Ay^{1/2}}=
\overline{yAy} \rforal c\in  \overline{f_\eta(x^2)Af_\eta(x^2)}.
\eneq
%for all $c\in \overline{f_\eta(x^2)^{1/2}y^2f_\eta(x^2)^{1/2}Af_\eta(x^2)^{1/2}y^2f_\eta(x^2)^{1/2}}.$
%ITherefore,
%for any $c\in \overline{f_\eta(x^2)Af_\eta(x^2)}.$
%$$
%vcv^*\in \overline{y^{1/2}Ay^{1/2}}.
%$$
%Moreover, $v^*v$ is the open projection associated with
%the positive element $f_\eta(x^2)^{1/2}y^2f_\eta(x^2)^{1/2},$ whence, by the above,
%is also the open projection associated with $f_\eta(x^2).$

Choose $w=v^*.$  Then, since $\sqrt{\eta}< \sigma/4,$
\beq\label{LRL-14}
f_{\sigma}(x)f_\eta(x^2)=f_{\sigma}(x)\andeqn {\rm hence}\,\,\,
ww^*f_\sigma(x)=f_\sigma(x)ww^*=f_\sigma(x).
\eneq
Thus \eqref{GE-1} holds for all $c\in \overline{f_\sigma(x)Af_\sigma(x)}.$
If $A$ almost has stable rank one, we can choose $\dt$ for $\ep/2$ and $\sigma/4$ first.
Then, for $b=vf_{\sigma/4}(x),$ by Theorem 5 of \cite{Pedjot87}, there is a unitary
$u\in {\widetilde A}$ such that $b=u^*f_{\sigma/4}(x).$
Then, for any $c\in  \overline{f_{\sigma}(x)Af_{\sigma}(x)},$
$u^*cu=vcv^*$ and so $w$ can be replaced by $u.$

\end{proof}

\begin{lem}[\cite{Rr11}]\label{Lfullep}
Let $A$ be a \CA\, and $a\in A_+$ be a full element.
Then, for any $b\in A_+,$ any $\ep>0$ and
any $g\in C_0((0,{{+}}\infty))$ whose support is
in $[\ep,{{+}}\infty),$  there are
$x_1, x_2,...,x_m\in A$ such that
$$
g(b)=\sum_{i=1}^m x_i^*ax_i.
$$
\end{lem}

\begin{proof}
Fix $\ep>0.$ {{Since $a$ is full,  and, $a$ and $b$ are positive,}} there are $z_1,z_2,...,z_m\in A$
such that 
$$
\|\sum_{i=1}^m z_i^*az_i-b\|<\ep.
$$
Therefore, by 2.2 and 2.3 of \cite{Rr11},   there is $y\in B$
%with $\|y\|\le 1/\ep$
such that
$$
f_{\ep}(b)=y^*(\sum_{i=1}^m z_i^*az_i)y.
$$
Therefore, since $f_\ep g=g,$
$$
g(b)=g(b)^{1/2}y^*(\sum_{i=1}^m z_i^*az_i)yg(b)^{1/2}.
$$

\end{proof}

We  shall also need
%would also like to include
the following {{slight variant}} of a result of R{\o}rdam:

\begin{thm}[cf.\,4.6 of \cite{Rrzstable}]\label{TRozs}
Let $A$ be an exact   simple \CA\, which is ${\cal Z}$-stable.
Then $A$ has  the  strict comparison property  for positive elements:
If $a, \, b\in (A\otimes {\cal K})_+$ 
%(for some $n\ge 1$) 
are two elements
such that
\beq\label{Trozs-1}
\mathrm{d}_\tau(a)<\mathrm{d}_\tau(b)<+\infty\tforal  \tau\in {\widetilde{\mathrm{T}}}(A)\setminus \{0\},
% \overline{T(B)}^\mathrm{w},
\eneq
then $a\lesssim b.$
%where
%$B=\overline{cAc}$ for some $c\in {\mathrm{Ped}(A)}_+\setminus \{0\}.$
\end{thm}

\begin{proof}

Let $a, b\in (A\otimes\mathcal K)_+$ be as in \eqref{Trozs-1}, and set
$$
\{\tau\in {\widetilde{\mathrm{T}}}(A): \mathrm{d}_\tau(b)=1\}=S.
$$
%Then $S\subset \overline{T(B_1)}^\mathrm{w}.$
The assumption \eqref{Trozs-1} implies that
\beq\label{Trozs-2}
\mathrm{d}_\tau(a)<\mathrm{d}_\tau(b)\rforal \tau\in S.
\eneq
%We may view $a\in B_1\otimes {\cal K}.$
Since $A$ is simple and $b\not=0,$  for every $\ep>0,$ $\la (a-\ep)_+\ra \le K \la b\ra $ for some integer $K\ge 1.$
{{Hence,}}
% the above implies that
$$
f(a)<f(b)\rforal f\in S({{{\rm{Cu}}(A)}}, {{\la b\ra}})
$$
(see \cite{Rrzstable} 
for the notation)
% and other details). 
Since, by Theorem 4.5 of \cite{Rrzstable}, $W(A\otimes {\cal K})$ is almost unperforated, and by
3.2 of \cite{Rrzstable},
$a\lesssim b.$

\end{proof}

\begin{cor}\label{{CRozs}}
Let $A$ be an exact  simple separable  \CA\, which is ${\cal Z}$-stable.
Then $A$ has  the  following strict comparison property  for positive elements:
If $a, \, b\in (B\otimes {\cal K})_+$  are two elements
such that
\beq\nonumber
\label{Trozs-3}
\mathrm{d}_\tau(a)<\mathrm{d}_\tau(b)<+\infty\tforal \tau\in
\overline{{\mathrm{T}}(B)}^{\mathrm{w}},
\eneq
where
$B=\overline{cAc}$ for some $c\in {\mathrm{Ped}(A)}_+\setminus \{0\},$ then $a\lesssim b.$
%%(Here we also identify $A\otimes {\cal K}$ with $B\otimes {\cal K}$ (see also \ref{Dstrictcom}).
\end{cor}

\begin{proof}
The condition \eqref{Trozs-1} of \ref{TRozs} holds  since ${{\R_+\overline{{\mathrm{T}}(B)}^{\mathrm{w}}}}=
%\overline{\mathrm{T}(B)^{\mathrm{w}}}}}=
{\widetilde{\mathrm{T}}}(A).$
\end{proof}

We 
would like to include 
%{{shall need}} 
the following statement.

\begin{lem}\label{103L}
Let $B$ be a separable semiprojective \CA\, and $A$ be another 
\CA\, such that there is an isomorphism $j: A\otimes {\cal K}\to B\otimes {\cal K}.$
Then, for any $\ep>0$ and any finite subset ${\cal F}\subset A,$ 
there exist $\dt>0$ and a finite subset ${\cal G}\subset A$  {{with the following property:}}
%satisfying the following:
If $D$ is a \CA\, and $L: A\to D$ is a ${\cal G}$-$\dt$-multiplicative \cpc, then 
there exists a \hm\, $h: A\to D\otimes {\cal K}$ such that 
$$
\|h(a)-L(a)\otimes e_{1,1}\|<\ep\tforal a\in {\cal F},
$$
where $\{e_{i,j}\}$ is a {{system of}} matrix unit{{s}} for ${\cal K}.$
\end{lem}

\begin{proof}
Let us write $B\otimes {\rm M}_n\subset B\otimes {\rm M}_{n+1}$ for all $n$ and 
 $\bigcup_{n=1}^{\infty}B\otimes {\rm M}_n$ is dense in $B\otimes {\cal K}.$ 
 Let $e_n=\sum_{i=1}^n e_{i,i}.$  Define $d_n: B\otimes {\cal K}\to B\otimes {\rm M}_n$
 by sending $b\otimes c$ to $b\otimes (e_nce_n)$ for all $b\in B$ and $c\in {\cal K}.$ 
 %Find a finite subset ${\cal F}_1\subset B\otimes M_N$ for some large $N$ such that
% ${\rm dist}(j(a), {\cal F}_1)<\ep/8$ for all $a\in {\cal F}.$
 %In particular, 
 There is an integer $N$
 %for each $a\in {\cal F},$  there is ${\widetilde{j(a)}}\in B\otimes M_N$
 % there are $b_{a,i}\in B$ and  $c_{a,i}\in M_N,$ $1\le  i\le m(a)$ 
 such that 
 \beq\label{semiP1}
 \|j(a)-d_N( j(a))\|<\ep/8\rforal a\in {\cal F}.
 \eneq
  Write $d_N(j(a))=\sum_{i=1}^{m(a)} b_{a,i}\otimes c_{a,i},$ where $b_{a,i}\in B$ and  $c_{a,i}\in {\rm M}_N,$ $1\le  i\le m(a).$
 %${\cal F}_b=\{b_{a,i}: 1\le i\le m(a),\,\,a\in {\cal F}\},$ 
 Put ${\cal F}_M=\{c_{a,i}: 1\le i\le m(a),\,\,\, a\in {\cal F}\}$ and 
 ${\cal F}_1=\{\sum_{i=1}^{m(a)} b_{a,i}\otimes c_{a,i}: a\in {\cal F}\}.$ 
 Set $\Lambda=\max\{(\|b_{a,i}\|+1)(\|c_{a,i}\|+1)m(a): 1\le i\le m(a): a\in {\cal F}\}.$
 %We also identify $A$ with the first corner of $A\otimes {\cal K},$ whenever it becomes convenient.
 
 Since $B\otimes {\rm M}_N$ is semiprojective,  there are a finite subset  ${\cal G}_1\subset B\otimes {\rm M}_N$  
 and $\dt_0>0$ satisfying the following condition: if $L': B\otimes {\rm M}_N\to D'$ (for any \CA\, $D'$) is a
 ${\cal G}_1$-$\dt_0$-multiplicative \cpc, there exists a \hm\, $H: B\otimes {\rm M}_N\to D'$
 such that
 \beq
 \|H(b)-L'(b)\|<\ep/8\rforal b\in {\cal F}_1.
 \eneq
Since $j$ is an isomorphism, there exist a finite subset ${\cal G}\subset A$ and $\dt>0$ satisfying the following condition:
if $L: A\to D$ (for any \CA\, $D$) is a ${\cal G}$-$\dt$-multiplicative \cpc,  then
$(L\otimes {\rm id}_{\cal K})(j^{-1})$ is {{a}} ${\cal G}_1$-$\dt$-multiplicative \cpc. 

Let $\iota: {\rm M}_N\otimes {\cal K}\to {\cal K}\otimes {\cal K}$ {{denote}}  the inclusion and {{let}}
$\phi: {\cal K}\otimes {\cal K}\to {\rm M}_N\otimes {\cal K}$ be an isomorphism. There is  a   unitary  $U\in {\widetilde{{\rm M}_N\otimes {\cal K}}}$ such that
$$
{\rm Ad}\, U\circ \phi\circ \iota\approx_{\ep/4\Lambda}  {\id}_{{\rm M}_N\otimes {\cal K}}\,\,\,{\rm on}\,\,\, {\cal F}_M\otimes 1_{\cal K}.
$$
Put $\phi_1={\rm Ad}\, U\circ \phi.$  Consider $\Psi=({\rm id}_B\otimes \phi_1)\circ  {{(j\otimes {\rm id}_{\cal K})}}: A\otimes {\cal K}\otimes {\cal K}\to B\otimes {\rm M}_N\otimes {\cal K}.$
Thus, for all $a\in {\cal F}$ (we identify $A$ with the first corner of $A\otimes {\cal K}$),
\beq\nonumber
\hspace{-0.3in}\Psi(a\otimes e_{1,1}\otimes e_{1,1})&=&({\rm id}_B\otimes \phi_1)(j(a)\otimes e_{1,1})\approx_{\ep/8} 
({\rm id}_B\otimes \phi_1))(d_N(j(a))\otimes e_{1,1})
%(\sum_{i=1}^{m(a)}b_{a,i}\otimes c_{a,i}\otimes e_{1,1})
\\\nonumber
&=&({\rm id}_B\otimes \phi_1)(\sum_{i=1}^{m(a)}b_{a,i}\otimes \iota(c_{a,i}\otimes e_{1,1}))
\approx_{\ep/4} \sum_{i=1}^{m(a)}b_{a,i}\otimes c_{a,i}\otimes e_{1,1}\\\label{semiP3}
&\approx_{\ep/8}& j(a)\otimes e_{1,1}.
\eneq

Now assume $L: A\to D$ is a ${\cal G}$-$\dt$-multiplicative \cpc.
Consider {{the maps}} $L\otimes {\rm id}_{\cal K}: A\otimes {\cal K}\to D\otimes {\cal K}$
and $(L\otimes {\rm id}_{\cal K})(j^{-1}): B\otimes {\cal K}\to D\otimes {\cal K},$ 
%Define 
{{and the restriction}} $\Phi:=(L\otimes {\rm id}_{\cal K})(j^{-1})|_{B\otimes {\rm M}_N}: B\otimes {\rm M}_N\to D\otimes {\cal K}.$

Now $\Phi$ is a ${\cal G}_1$-$\dt_0$-multiplicative \cpc. Therefore there 
is a \hm\, $h_0: B\otimes {\rm M}_N\to D\otimes {\cal K}$ such that
\beq\label{semiP2}
\|h_0(b)-\Phi(b)\|<\ep/8\rforal b\in {\cal F}_1. 
\eneq
Then, for all $a\in {\cal F},$ by \eqref{semiP3}, \eqref{semiP1},    and \eqref{semiP2},
%(we identify $A$ with the first corner of $A\otimes {\cal K}$),
\beq\nonumber
&&\hspace{-0.2in}(h_0\otimes  {\rm id}_{\cal K})\circ \Psi(a\otimes e_{1,1}\otimes e_{1,1})\approx_{\ep/2} 
(h_0\otimes {\rm id}_{\cal K})(j(a)\otimes e_{1,1})\approx_{\ep/4}
(\Phi\otimes {\rm id}_{\cal K})(d_N(j(a))\otimes e_{1,1})\\\nonumber
&&{{\approx_{\ep/8}}}(L\otimes {\rm id}_{\cal K})(a\otimes e_{1,1}\otimes e_{1,1})=L(a)\otimes e_{1,1}.
\eneq

Define $h: A\to D\otimes {\cal K}$ {{as the composed map}}
$({\id}_D\otimes \psi)\circ h_0\circ \Psi\circ \iota_A,$
where $\iota_A(a)=a\otimes e_{11}\otimes e_{1,1}$ for all $a\in A$ and 
$\psi: {\cal K}\otimes {\cal K}\to {\cal K}$ is any isomorphism.  From the last estimate
the lemma follows.
The above proof may be summarized by the following non-commutative diagram with the upper triangle approximately commutative 
on ${\cal F},$ 
%the upper right  commutative 
and the lower 
%left commutative and lower
right   one  approximately commutative on ${\cal F}_1\otimes e_{1,1}:$
 \begin{displaymath}
\xymatrix{
%&&_{\small{h=({\rm id}_D\otimes \psi)\circ (h_0\otimes {\rm id}_{\cal K})\circ({\rm id}_B\otimes \phi_1)\circ (j\otimes {\rm id}_{\cal K})\circ \iota_A }}\\
&&A 
\ar[d]^{d_N\circ j\circ \iota_A} 
%\ar[rrd]^{h} 
%\ar[d]_{\Psi\circ \iota_A} 
\ar[dll]_{(j\otimes {\rm id}_{\cal K})\circ \iota_A}
&\\
%_{\small{h=({\rm id}_D\otimes \psi)\circ (h_0\otimes {\rm id}_{\cal K})\circ({\rm id}_B\otimes \phi_1)\circ (j\otimes {\rm id}_{\cal K})\circ \iota_A }}\\
%\ar[d]_{L_1} & A \ar[d]^{L_2}\\
B\otimes {\cal K}\otimes {\cal K}
\ar@{.>}[rrd]_{j^{-1}\otimes {\id}_{\cal K}}
\ar@{-->}[rr]^{{\rm id}_B\otimes \phi_1}
&&B\otimes {\rm M}_N\otimes {\cal K}\ar[d]_{j^{-1}\otimes {\rm id}_{\cal K}} 
\ar@{-->}[rr]^{h_0\otimes {\rm id}_{\cal K}}
%\ar@{_{(}->}[ll]
&& D\otimes {\cal K}\otimes {\cal K}\ar[rr]^{{\rm id}_D\otimes \psi} && D\otimes {\cal K}.\\
&&A\otimes {\cal K}\otimes {\cal K} \ar[urr]_{L\otimes {\rm id}_{\cal K\otimes {\cal K}}}
%\ar[r]_\id & C,
}
\end{displaymath}
$$_{\small{h=({\rm id}_D\otimes \psi)\circ (h_0\otimes {\rm id}_{\cal K})\circ({\rm id}_B\otimes \phi_1)\circ (j\otimes {\rm id}_{\cal K})\circ \iota_A }}
$$

\end{proof}

\section{Compact \CA s}

\begin{df}\label{Dcompact}
A $\sigma$-unital  \CA\, ${{A}}$ is said to be {\it {{compact}}}, if there is $e\in (A\otimes {\cal K})_+$
with $0\le e\le 1$
and a partial isometry $w\in (A\otimes {\cal K})^{**}$ such
that
\beq\nonumber
\hspace{0.1 in}w^*a,\, w^*aw\in A\otimes {\cal K}, ww^*a=aww^*=a, \andeqn ew^*aw=w^*awe=w^*aw\rforal a\in A,
\eneq
where $A$ is identified with $A\otimes e_{11}.$

\end{df}

\begin{prop}\label{Pcompact1}
Let $C$ be a $\sigma$-unital  \CA\, and let $c\in C_+\setminus \{0\}$ with $0\le c \le 1$ be a full element of $C.$
Suppose that there is $e_1\in C$  with $0\le e_1\le 1$ such that $e_1c=ce_1=c.$
Then $\overline{cCc}$ is {{compact}}.
\end{prop}

\begin{proof}
%Set $\overline{cCc}=B.$  By \cite{Br1}, we may view that $e_1\in B\otimes {\cal K}.$ 
 % Let $p\in C^{**}$ be the open projection corresponding to the hereditary \SCA\, $B.$
 % Put $w=p.$ Then $w^*=w,$ $w^*b=b=bw,$ $ww^*b=b=bww^*$ and $e_1w^*bw=e_1b=b=w^*bw$
%  for all $b\in B.$

Set $\overline{cCc}=B.$
%w^*b\in C\otimes {\cal C}
{{Consider  the
 \SCA\,
$$
E=\{(a_{ij})_{2\times 2}: a_{11}\in B, a_{12}\in \overline{BC}, a_{21}\in \overline{CB}, a_{22}\in C\}\subset M_2(C),
$$
containing $B_1=B\otimes e_{11}$ and $C_2=C\otimes e_{22}$ as full corners, where $\{e_{i,j}: 1\le i,j\le 2\}$ 
is  a system of 
%forms 
matrix units for $M_2.$}}
\iffalse
{{Consider  the
 \SCA\,
$$
E=\{(a_{ij})_{2\times 2}: a_{11}\in B, a_{12}\in \overline{BC}, a_{21}\in \overline{CB}, a_{22}\in C\}\subset M_2(C),
$$
containing $B=B\otimes e_{11}$ and $C=C\otimes e_{22}$ as full corner.}}
%%%Let $E\subset {\mathrm{M}}_2(C)$ be a \SCA\,
%%%of the form
%%$$
%%E=\{(a_{ij})_{2\times 2}: a_{11}\in B, a_{12}\in \overline{BC}, a_{21}\in \overline{CB}, a_{22}\in C\}.
%%$$
\fi
We {{may then in a natural way view}}
$B_1\otimes {\cal K}$ and $C_2\otimes {\cal K}$ as full corners of $E\otimes {\cal K}.$
{{We may write 
$$
E\otimes {\cal K}=\{(a_{ij})_{2\times 2}: a_{11}\in B_1\otimes {\cal K}, a_{12}\in \overline{BC}\otimes {\cal K}, a_{21}\in \overline{CB}\otimes {\cal K}, a_{22}\in C_1\otimes {\cal K}\}\subset M_2(C\otimes {\cal K}).
$$
We also write $B_2=B\otimes e_{22}.$
Moreover, let $p_1$ be the range projection of $B_1\otimes {\cal K}$ and let $p_2$
be the range projection of $C_1\otimes {\cal K};$ then
$p_1, p_2\in \mathrm{M}(E\otimes {\cal K}).$
Put
$U=(a_{ij})_{2\times 2},$ where $a_{11}=a_{12}=a_{22}=0$ and $a_{21}=p_1.$ Then $Uc\in E\otimes {\cal K}$
for all $c\in E\otimes {\cal K}$ and $UxU^*\in B_2\otimes {\cal K}\subset C_2\otimes {\cal K}$ for
all $x\in B_1\otimes {\cal K}.$}}

By  2.8 of \cite{Br1}, there  is a partial  isometry $W\in \mathrm{M}(E\otimes {\cal K})$
such that $W^*(B_1\otimes {\cal K})W=C_2\otimes {\cal K},$
$WW^*=p_1$, and $W^{{*}}W=p_2.$ 
Since $W\in \mathrm{M}(E\otimes {\cal K}),$ $Wb\in E\otimes {\cal K}$ for all $b\in B_2\otimes {\cal K}.$
It follows that 
$$
Up_1Wb=UWb\in B_2\otimes {\cal K}\rforal b\in B_2\otimes {\cal K}.
$$

In what follows we identify $B$ with $B_2.$
Put  $w=W^*p_1U^*=W^*U^*.$
Then, for any $b\in B(=B_2),$ 
$$
(w)^*b=Up_1Wb\in B_2\otimes {\cal K},\, w^*bw=UWbW^*U^*\in B_2\otimes {\cal K}(=B\otimes {\cal K}).
$$
This also implies that $w^*$is  a left multiplier of $B_2\otimes {\cal K}.$ So we may write 
$w\in (B_2\otimes {\cal K})^{**}$ (which we identify with $(B\otimes {\cal K})^{**}$).
Moreover, with $e=Up_1We_1W^*p_1U^*\in B_2\otimes {\cal K}(=B\otimes {\cal K})$ (where $e_1:=e_1\otimes e_{22}$),
$$
ew^*bw=UWe_1W^*U^*UWbW^*U^*=UWe_1p_2bW^*U^*=UWe_1bW^*U=UWbW^*U^*=w^*bw
$$
for all $b\in B_2.$
We also have $w^*bwe=ew^*bw$, $ww^*b={{W^*U^*UWb=p_2b}}=b,$
%\%p_1W^*Wp_1b=p_1b=b$, 
and
$bww^*=b$ for all $b\in B_2(=B).$

 Thus,
 $B$ is {{compact}}.
\end{proof}

\begin{cor}\label{CherePA=A}
A $\sigma$-unital full hereditary \SCA\, of a $\sigma$-unital {{compact}} \CA\, is {{compact}}.
\end{cor}

\begin{proof}
Let $A$ be a {{compact}} \CA\, and let $b\in A_+$ be a full element.
% with $0\le b\le 1.$
Put $B=\overline{bAb}.$  Let $e$ and $w$ be  {{as}} in Definition \ref{Dcompact}.
Put $c=w^*bw,$  $B_1=w^*Bw,$ and $C=A\otimes {\cal K}.$  Then $B\cong B_1.$ Moreover,
$B_1=\overline{cCc}$ and $ec=ce=c.$ So, by \ref{Pcompact1},  $B_1$ is {{compact}}. 

\end{proof}

\begin{lem}\label{subcommatrix}
Let $A$ be a $\sigma$-unital  \CA\, which is {{compact}}.
Then, there exists an integer $N\ge 1,$  a partial isometry $w\in {\rm M}_N(A)^{**}$ and $e\in {\rm M}_N(A)$
with $0\le e\le 1$
such that
$$
\,\,{{w^*a,}}\, w^*aw\in {\rm M}_N(A), \,\, ww^*a=aww^*=a,\ \mathrm{and}\ w^*awe=ew^*aw=w^*aw\rforal a\in A.
$$
\end{lem}

\begin{proof}
{{If $A$ is unital, then we may choose $N=1$ and $e=1_A.$}}

Let $b\in  A$ be a strictly positive element with $0\le b\le 1.$
We may assume, by  Definition \ref{Dcompact}, that   $A$ is a full hereditary \SCA\, of a
$\sigma$-unital \CA\, $C$  such that  
%{{by Definition \ref{subcommatrix}}}
%$b\in C$ and
there is  $e_1\in C$ with $0\le e_1\le 1$ such that
$be_1=b={{e_1}}b.$ Moreover, $\overline{bCb}=A.$
Since  $b$ is full in $C,$  by  Lemma \ref{Lfullep},
there exist  $x_1, x_2, ...,x_m\in C$ such that
$$
\sum_{i=1}^m x_i^*bx_i=f_{1/4}(e_1).
$$
Note that $f_{1/4}(e_1)b=b={{b}}f_{1/4}(e_1).$
Consider $X^*:=(x_1^*b^{1/2}, x_2^*b^{1/2},...,x_m^*b^{1/2})$ as {{a}} 1-row element {{of}} $\mathrm{M}_m(C).$
Then
$$
X^*X=f_{1/4}(e_1)\andeqn XX^*\in \mathrm{M}_m(\overline{b^{1/2}Cb^{1/2}})=\mathrm{M}_m(A).
$$
{{Consider}}  the polar decomposition $X=v|X^*X|^{1/2}$ of $X$ in $\mathrm{M}_m(C)^{**}.$
Then
$$
vav^*=v|X^*X|^{1/2}a|X^*X|^{1/2}v^*=XaX^*\in \mathrm{M}_m(A)\rforal a\in A.
$$
Note {{that}} $Xb^{1/n}\in {\rm M}_m(A)$ {{for all $n.$}}  Denote by  $p$ the range  projection of $b.$
Then $Xp\in \mathrm{M}_m(A)^{**}.$ Note also
that $Xp=v|X^*X|^{1/2}p=vp.$ Set $w=(Xp)^*.$ Then $w^*=vp$ and $ww^*=pX^*Xp=pf_{1/4}(e_1)p=p.$
So $w$ is a partial isometry.   {{Note that, for any $a\in A,$ $w^*b^{1/n}a=vpb^{1/n}a=Xb^{1/n}a\in A.$
It follows that $w^*a\in A.$}}
%Set $e=XX^*.$ 
Then
$$
 w^*aw=XaX^*\in {\rm M}_m(A)\rforal a\in A.
$$
{{Set $e=XX^*.$ Then,}}
\beq\nonumber
 w^*awe=XaX^*XX^*=Xaf_{1/4}(e_1)X^*=XaX^*=w^*aw\andeqn\\
ew^*aw=XX^*XaX^*=Xf_{1/4}(e_1)aX^*=XaX^*=w^*aw.
\eneq
\end{proof}

\begin{lem}\label{compactrace}
Let $A$ be a $\sigma$-unital  {{compact}}  \CA.
%\, with ${\mathrm{T}(A)}\not=\O.$
%{\red{The weak* closure}}  $\overline{{\mathrm{T}(A)}}^\mathrm{w}$ of ${\mathrm{T}(A)}$ in ${\widetilde{\mathrm{T}}}(A)$ is compact and
{{Then}} $0\not\in \overline{{\mathrm{T}(A)}}^\mathrm{w}.$
%  Moreover,
{{Hence}}
if $a\in A$ with $0\le a\le 1$ is
%strictly positive, or is
full,
then
there is $d>0$ such
that
$$
\mathrm{d}_\tau(a)\ge d\tforal \tau\in \overline{{\mathrm{T}(A)}}^\mathrm{w}.
$$

\end{lem}

\begin{proof}
%\%It is clear that we only need to consider the case that 
{{We may assume that}} ${\mathrm{T}}(A)\not=\O.$
As in the proof of Lemma
%{\red{George proposed to change notation here, I choose not to---just don't have enough energy to do so}}
%By
\ref{subcommatrix}, \wilog, we  may assume that $A$ is a full hereditary \SCA\, of $B\otimes {\cal K}$ for some
$\sigma$-unital
\CA\, $B$ such that  there is $e_1\in (B\otimes {\cal K})_+,$  $0\le e_1\le 1,$ and
$e_1x=xe_1=x$ for all $x\in A.$  Put $C=B\otimes {\cal K}.$
%Let ${\widetilde{\mathrm{T}}}(C)$ be the convex set of all traces defined on the Pedersen ideal $\mathrm{Ped}(C).$
Note that $A\subset \mathrm{Ped}(C).$
Since $A$ is full in $C,$ we may also assume that each $\tau\in {\mathrm{T}(A)}$ has been extended  to
{{a unique}}  element of  ${\widetilde{\mathrm{T}}}(C).$ Let $e_0\in A_+$ be a
strictly positive
element of $A.$
 %%%Since
 %a strictly element of $A$ is full, we can always assume that
%%% $a$ is full in $A,$  it  is also full in $C.$
%%%By \ref{Lfullep}, there are $x_1, x_2,....,x_m\in C$ such that
%%%\beq\label{tcompact-2}
%%%\sum_{i=1}^mx_i^*ax_i=f_{1/4}(e_1).
%%%%\eneq
%Note that $f_{1/4}(e_1)x=xf_{1/4}(e_1)=x$ for all $x\in A,$ and 
%$f_{1/4}(e_1)\in \mathrm{Ped}(C).$
%Put
%$$T_0=\{\tau\in {\widetilde{\mathrm{T}}}(C): \tau(b)\le 1\tforal b\in A\andeqn \|b\|\le 1\}.$$
%%Since $A\subset {\rm Ped(C)},$ it is clear that $\overline{{\mathrm{T}(A)}}^\mathrm{w}\subset  {\mathrm{T}}_{\red{1}}(A).$
{{Recall from \ref{DTtilde} that $\overline{{\mathrm{T}(A)}}^\mathrm{w}$ is weak*-compact subset of ${\rm T}_1(A).$}}
%Since $
%Since $A\subset {\rm Ped(C)},$  is clear that $\overline{{\mathrm{T}(A)}}^\mathrm{w}\subset  {\mathrm{T}}_{\red{1}}(A).$
%On the other hand,  since  ${\mathrm{T}}_{\red{1}}(A)$ has been  identified with a subset  in the unit ball of $A^*,$ it is weak*- compact.
  Consider the set
$$
S=\{\tau\in  {\mathrm{T}}_{{1}}(A): \tau(f_{1/4}(e_1))\ge 1\}.
%{\widetilde{\mathrm{T}}}(C): \tau(f_{1/4}(e_1))\ge 1 \andenq  \tau(a)\le 1\}.
$$
Then 
% {{$0\not\in S.$}}
$S$ is {{closed and $0\not\in S.$}}
% compact.
Note that
$$
{\mathrm{T}(A)}=\{\tau\in {\widetilde{\mathrm{T}}}(C): \mathrm{d}_\tau(e_0)=1\}.
$$
{{Then,}}  for any $\tau\in {\mathrm{T}(A)},$
\beq\label{tcompact-4}
\tau(f_{1/4}(e_1))\ge \mathrm{d}_\tau(e_0)=1.
\eneq
It follows that ${\mathrm{T}(A)}\subset S,$ and so $\overline{{\mathrm{T}(A)}}^\mathrm{w}\subset S.$
Therefore, 
%$\overline{{\mathrm{T}(A)}}^\mathrm{w}$
% the weak*-closure of ${\mathrm{T}(A)}$
%is compact and
$0\not\in \overline{{\mathrm{T}(A)}}^\mathrm{w}.$

Since $a$ is full in $A,$ $\tau(a)>0$ for all nonzero traces. In 
%other words, 
particular, $\tau(a)>0$  for all 
$\tau\in \overline{{\mathrm{T}(A)}}^\mathrm{w}.$ 
% {{It follows  that
{{Thus the lower semicontinuous function 
$\tau\mapsto {\rm d}_\tau(a)$  si strictly positive 
%or all $\tau$ 
on the compact set $\overline{{\mathrm{T}(A)}}^\mathrm{w}.$ Therefore 
$d:=\inf\{{\rm d}_\tau(a): \tau\in \overline{{\mathrm{T}(A)}}^\mathrm{w}\}>0.$}}

\iffalse
By \eqref{tcompact-2},
\beq\label{tcompact-3}
\la f_{1/4}(e_1)\ra\le  m\la a\ra \,\,\, {\rm in}\,\,\, {\rm Cu}(C).
\eneq
Put $d=1/m.$
It follows that, for any $\tau\in \overline{{\mathrm{T}(A)}}^\mathrm{w},$
$$
\mathrm{d}_\tau(a)\ge \tau(f_{1/4}(e_1))/m\ge 1/m=d.
$$
\fi
\end{proof}

\begin{cor}\label{Csemicon}
Let $A$ be a $\sigma$-unital \CA\, which is {{compact}} 
% with  ${\mathrm{QT}}(A)={\mathrm{T}(A)}\neq\O$ 
and
let $a\in A$ be a strictly positive element with $0\le a\le 1.$
Suppose
that
$$d{{\,:=}}\inf \{\mathrm{d}_\tau(a): \tau\in \overline{{\mathrm{T}(A)}}^\mathrm{w}\}>0.$$
Then, for any $d/3<d_0<d,$  there exists an integer $n\ge 1$ such that, for all $m\ge n,$
$$
\tau(f_{1/m}(a))\ge d_0\tand \tau(a^{1/m})\ge d_0\tforal \tau\in \overline{{\mathrm{T}(A)}}^\mathrm{w}.
$$
\end{cor}

\begin{proof}
{{This holds as both increasing sequences $(\tau(f_{1/m}(a)))_{m=1}^{\infty}$ and 
$(\tau(a^{1/m}))_{m=1}^{\infty}$ converge pointwise to $d_\tau(a),$  and 
$\overline{{\mathrm{T}(A)}^{\mathrm{w}}}$ is comapct.}}
%is a standard compactness argument, using Lemma \ref{compactrace}.
\end{proof}

\begin{thm}\label{Tpedersen}
Let $A$ be a $\sigma$-unital  \CA.  Then $A$ is {{compact}} if and only if
${\mathrm{Ped}(A)}=A.$
%$P(A\otimes {\cal K})\cap A=A.$

\end{thm}

\begin{proof}
Let $a\in A_+$ be a strictly positive element. 

First assume that $A$ is {{compact}}.
{{We will identify $A$ with $A\otimes e_{11}\subset A\otimes {\cal K}.$}} Then there exists
$e\in (A\otimes {\cal K})_+$ and a partial isometry ${{w}}\in (A\otimes {\cal K})^{**}$ such that
\beq\nonumber
w^*xw\in A\otimes {\cal K},\,\,\, ww^*x=xww^*=x, \andeqn w^*xwe=ew^*xw=w^*xw\rforal x\in A.
\eneq
{{S}}et $z=w^*a^{1/2}.$ Then $zz^*\in \mathrm{Ped}(A\otimes {\cal K})_+.$
{{Hence, by}} {{5.6.2 of \cite{Pbook},}}  $a=z^*z\in \mathrm{Ped}(A\otimes {\cal K})_+.$ 
%{{(see 5.6.2 of \cite{Pbook})}}.  
Therefore
the hereditary \SCA\, generated by $a$ is contained in $\mathrm{Ped}(A\otimes {\cal K}).$  
%%Consequently
{{In other words,}}
$A\subset \mathrm{Ped}(A\otimes {\cal K}).$
%To see  that $A={\mathrm{Ped}(A)}, $   one applies 
{{By  Theorem 2.1 of \cite{aTz}, $A={\mathrm{Ped}(A)}.$}}

%Take $b\in A_+\setminus \{0\}$ and $c\in {\mathrm{Ped}(A)}_+\setminus \{0\}.$
%Since $A\otimes {\cal K}$ is simple,
%there are $x_1, x_2,...,x_k\in A_+$ with $0\le x_i\le 1$ such that
%$$
%b^{1/2}\le \sum_{i=1}^k g_i(x_i),
%$$
%where $g_i\in C_0((0, \infty))$ such that the support of $g_i$ contained
%in $[\ep, N]$ for some $1>\ep>0$ and $N\ge 1,$ $i=1,2,...,k.$

%There are $z_{1,i}, x_{2,i},...,x_{m,i}\in A\otimes {\cal K}$  such that
%$$
%\sum_{j=1}^m z_{j,i}^*cz_{j,i}=g_i(x_i).
%$$
%Therefore
%\beq\label{Pedersen-1}
%\sum_{j=1}^m b^{1/4}z_{j,i}^*cz_{j,i}b^{1/4}=b^{1/4}g_i(x_i)b^{1/4},\,\,\,i=1,2,...,k.
%\eneq

%Note that $b^{1/4}z_{j,i}^*c^{1/4}, c^{1/4}z_{j,i}b^{1/4}\in A$
%and $c^{1/2}\in {\mathrm{Ped}(A)}$ (see 5.6.2 of \cite{Pbook}).
%It follows that
%$$
%b^{1/4}g_i(x_i)b^{1/4}\in {\mathrm{Ped}(A)}.
%$$
%On the other hand,
%$$
%b=b^{1/4}b^{1/2}b^{1/4}\le (\sum_{i=1}^k b^{1/4}g_i(x_i)b^{1/4}.
%$$
%Since ${\mathrm{Ped}(A)}$ is hereditary, this implies that $b\in {\mathrm{Ped}(A)}.$
%Thus $A={\mathrm{Ped}(A)}.$

Conversely,  assume that ${\rm Ped}(A)=A.$ Then 
%suppose that
there are $b_i
%1, b_2,...,b_m
\in A_+$ with $\|b_i\|\le 1,$ $i=1,2,...,m,$  such that
$$
a^{1/2}\le \sum_{i=1}^mg_i(b_i),
%\andeqn 0\le b_1, b_2,...,b_m\le 1,
$$
where $g_i\in C_0((0,\infty))$
% for some  $N\ge 1$ 
and the support of $g_i$ is in $ [\sigma, \infty)$
for some $1/2>\sigma>0.$ Note  that for each $i,$ $g_i=\sum_{j=1}^K g_{i,j}$ for some 
%$\|g_{i,j}\|+1\ge K\ge 1,$
%where  
$0\le g_{i,j}\le 1$ with the  support of $g_{i,j}$  still in $[\sigma, \infty).$ Thus, \wilog, we may assume
that $0\le g_i\le 1.$

Let $c_i=(g_i(b_i))^{1/2},$ $i=1,2,...,m.$
Define
$$
Z=(c_1, c_2,...,c_m),
$$
which we view as a $m\times m$ matrix with zero rows other than the first row.
Define
$$E=\diag(f_{\sigma/2}(b_1), f_{\sigma/2}(b_2),...,f_{\sigma/2}(b_m))\in {\mathrm{M}}_m(A).$$
Note that
\vspace{-0.12in} $$
ZZ^*=\sum_{i=1}^mc_i^2\ge a\andeqn Z^*Z=(d_{i,j})_{m\times m},
$$
where
\vspace{-0.1in} $$
d_{i,j}=c_ic_j,\,\,\,i,j=1,2,...,m.
$$
It follows that
$$
E(Z^*Z)=E(c_ic_j)_{m\times m}=(c_ic_j)_{m\times m}=(Z^*Z)E.
$$
{{Consider the polar decomposition}} $Z^*=V|Z^*|.$ Then
$$
Vx\in {\mathrm{M}}_m(A), VV^*|Z|=|Z|VV^*=|Z|,\andeqn
(VxV^*)E=E(VxV^*)=(VxV^*)
$$
for all $x\in \overline{(ZZ^*){\mathrm{M}}_m(A)(ZZ^*)}.$
Note that $A\subset \overline{(ZZ^*){\mathrm{M}}_m(A)(ZZ^*)}.$ Therefore
$A$ is {{compact}}.

\end{proof}

The number $n$ below will be used later.

\begin{lem}\label{Lbkqc}
Let $A$ be a $\sigma$-unital   \CA\,  with
% with ${\mathrm{QT}}(A)={\mathrm{T}(A)}\not={\O},$ and  
%$\overline{{\mathrm{T}(A)}}^\mathrm{w}$
%is compact  and
$0\not\in \overline{{\mathrm{T}(A)}}^\mathrm{w}.$
Suppose that $A$ almost has stable rank one and has the following property:
If $a, b\in {\mathrm{M}}_{m}(A)_+$ for some integer $m\ge 1,$ and
${{{\rm d}_\tau}}(a)<{{{\rm d}_\tau}}(b)$ for all $\tau\in \overline{{\mathrm{T}(A)}}^\mathrm{w},$ then
$a\lesssim b.$
%the strong strict comparison for positive elements and almost has stable rank one.
Then $A$ is {{compact}}.
Moreover, let $a\in A$ with $0\le a\le 1$ be a strictly positive element,  {{set}}
$$
d=\inf\{\mathrm{d}_\tau(a): \tau\in \overline{{\mathrm{T}(A)}}^\mathrm{w}\},
$$
and  let $n$ be an integer such that $nd>1$.
There exist elements $e_1, e_2\in {\mathrm{M}}_n(A)$  with $0\le e_1, e_2\le 1,$
$e_1e_2=e_2e_1=e_1$, and $w\in {\mathrm{M}}_n(A)^{**}$
such that
\beq\label{Lbkqc-1}
w^*c, cw\in {\mathrm{M}}_n(A), ww^*c=cww^*=c\tforal c\in A,\tand\\\label{Lbkqc-1n}
w^*cwe_1=e_1w^*cw=w^*cw\tforal c\in A.
\eneq
Furthermore,
 there exist a full element $b_0\in {\mathrm{Ped}(A)}$ with $0\le b_0\le 1$ and $e_0\in {\mathrm{Ped}(A)}_+$
such that $b_0e_{{0}}=e_{{0}}b_0=b_0.$
\end{lem}

\begin{proof}
Let $a\in A_+$ with $0\le a\le 1$ be a strictly positive element.
Since $0\not\in \overline{{\mathrm{T}(A)}}^\mathrm{w}$ and $\overline{{\mathrm{T}(A)}}^\mathrm{w}$  is compact
{{(see the end of \ref{DTtilde}),}} and $\mathrm{d}_\tau(a)$ is lower semicontinuous, {{as stated in Lemma \ref{compactrace},}}
$$
\inf\{\mathrm{d}_\tau(a): \tau\in \overline{{\mathrm{T}(A)}}^\mathrm{w}\}=d>0.
$$
%Suppose
{{Let  $n$ be an integer such}} that $nd>1.$

By \ref{Csemicon}
there exists  $\ep >0$  such that
$$
\inf\{\tau(f_\ep(a)): \tau\in \overline{{\mathrm{T}(A)}}^\mathrm{w}\}=d_0>2d/3
$$
and $nd_0>1.$
{{So, for any $\tau\in \overline{{\mathrm{T}(A)}}^\mathrm{w},$
$$
d_\tau(a)\le 1<nd_0\le n\tau(f_{\ep}(a))\le d_\tau(f_{\ep}(a)).
$$}}
Therefore, 
\vspace{-0.1in} $$
a\lesssim \diag(\overbrace{f_\ep(a), f_\ep(a),...,f_\ep(a)}^n)\,\,\, {\rm in}\,\, \mathrm{M}_n(A).
$$
Put $b=\diag(\overbrace{f_\ep(a), f_\ep(a),...,f_\ep(a)}^n).$
Since $A$ is assumed {{almost to}} have stable rank one, by \ref{Lalmstr1}, there exists $x\in \mathrm{M}_n(A)$
such that
$$
x^*x=a^{1/2}\andeqn xx^*\in \overline{b\mathrm{M}_n(A)b}.
$$
By considering the polar decomposition of $x$ in  $\mathrm{M}_n(A)^{**},$
one obtains  a partial isometry $w\in \mathrm{M}_n(A)^{**}$ such that
$wA, Aw^*\subset  {\mathrm{M}}_n(A),$ $w^*wc=cw^*w=c$ for all $c\in A,$ and
$wAw^*\subset   \overline{b{\mathrm{M}}_n(A)b}.$ Put $e=\diag(\overbrace{f_{\ep/2}(a), f_{\ep/2}(a),...,f_{\ep/2}(a)}^n).$
Then $0\le e\le 1$ and
$$
ewcw^*=wcw^*e=wcw^*\rforal c\in A.
$$
%Thus 
{{This shows that}} $A$ is {{compact}}.
The second part of the statement {{with \eqref{Lbkqc-1} and}} \eqref{Lbkqc-1n} also holds.

For the last part of the statement, choose $b_0=f_{\ep}(a)$ and
$e_0=f_{\ep/2}(a).$

\end{proof}

\begin{rem}\label{Dboundedscale}
{{Let $A$ be a $\sigma$-unital \CA\, and let $e\in A$ be a strictly positive element.
Set
\beq\label{Dlambdas}
\lambda_s=\inf\{{\rm d}_\tau(e): \tau\in {\rm T}(A)\}.
\eneq
By \ref{Ddimf}, $0\le \lambda_s\le 1.$
Note that $0\not\in \overline{{\rm T}(A)}^{\rm w}$ if and only if $\lambda_s>0.$ In particular, by \ref{compactrace}, 
if $A$ is compact, $\lambda_s>0.$  Note also $\lambda_s=\inf\{{\rm d}_\tau(e): \tau\in \overline{{\rm T}(A)}^{\rm w}\}.$ }}

Now let $A$ be a $\sigma$-unital  exact simple \CA. {{Let}} $e\in {\mathrm{Ped}(A)}_+\setminus\{0\}.$ {{Consider 
the  set of traces normalized  on $e,$}} 
${\mathrm{T}}_e(A)=\{\tau\in {\widetilde{\mathrm{T}}}(A): \tau(e)=1\}.$
It is a compact convex set (see 2.6 of \cite{aTz} and  2.6 of \cite{Lncs1}).
%\ref{compactrace}). Let $a\in A$ be a strictly positive element.
%Recall that $A$ has continuous scale if and only if $\mathrm{d}_\tau(a)$ is a continuous function
%under the assumption that 
If $A$ has strict comparison,
% for positive elements, 
$A$ is said to have bounded scale if $\mathrm{d}_\tau(a)$ is a bounded function
on ${\mathrm{T}}_e(A)$ (see \cite{Btrace}).  In the absence of strict
comparison, 
%one defines
let us say  that $A$ has bounded scale if there exists an integer $n\ge 1$ such
that $n\la e\ra \ge \la a\ra$  for any $a\in A_+.$ {{As}} first noted  in \cite{Btrace}, this is equivalent
to {{saying that}} $A$ is algebraically simple,  
{{and this}} in turn (in  view of \ref{Tpedersen} above) is equivalent to saying that $A$ is {{compact}}.

\end{rem}

\begin{prop}\label{Pheretc}
Let $A$ be a $\sigma$-unital  \CA\, with $0\not\in \overline{{\mathrm{T}(A)}}^\mathrm{w}$ such that
every trace in ${\widetilde{\mathrm{T}}}(A)$ is finite (equivalently, bounded) on $A.$
Let $B{{\subset}} A$ be a $\sigma$-unital full hereditary \SCA.
Then $0\not\in \overline{{\mathrm{T}}(B)}^\mathrm{w}.$
\end{prop}

\begin{proof}
Let $b\in B_+$ with $\|b\|=1$ be  {{such that}} 
%a strictly positive element of $B$ and 
$B=\overline{bBb}.$ Let $e\in A_+$ with $\|e\|=1$ such that
$A=\overline{eAe}.$

Since $b$ is full,  $\tau(b)>0$ for all $\tau\in \overline{{\mathrm{T}(A)}}^\mathrm{w}.$
Then, {{by continuity and compactness,}}
$$
1>r_0=\inf\{\tau(b): \tau\in \overline{{\mathrm{T}(A)}}^\mathrm{w}\}>0.
$$

For any $t\in {\mathrm{T}}(B),$ 
%there is an
the unique  extension $\tau\in {\widetilde{\mathrm{T}}}(A)$ is finite, i.e., bounded, 
by hypothesis.  %Since $\tau$ is a positive linear functional, it is continuous.
%Therefore,
{{Set}}
$\tau_0=\tau/\|\tau\|\in {\mathrm{T}(A)}$ and 
$t=\|\tau\|\cdot \tau_0|_{B}.$
It follows (since $\tau_0(b) \geq r_0$ and $\|{{\tau}}\| \geq 1$) that
$$
t(b)\ge \|\tau\|\cdot r_0\ge r_0.
$$
This shows that $0\not\in \overline{\mathrm{T}(B)}^\mathrm{w}.$

\end{proof}

\begin{prop}\label{Subset}
{{Let $A$ be a $\sigma$-unital simple \CA, let $c\in {\rm{Ped}}(A)$ be a positive element
and {{set}} $C=\overline{cAc}.$
Then each $\tau\in {\overline{{\rm{T}}(C)}}^{\rm w}$ extends  to {{a unique}}  element $\imath(\tau)\in 
{\widetilde{T}}(A).$ Moreover $\imath(\overline{{\rm{T}}(C)}^{\rm w})$ is compact in ${\widetilde{T}}(A)$
and $0\not\in \imath(\overline{{\rm{T}}(C)}^{\rm w}).$}}
\end{prop}

\begin{proof}
{{Note that the extension  is unique. So $\imath$ is well defined. Moreover  map $\imath$ is one-to-one. 
Note,  by \ref{CherePA=A} and \ref{Tpedersen}, $C={\rm Ped}(C).$
Put $K={\overline{\rm{T}(C)}}^{\rm w}.$ Then,  by \ref{compactrace}, $0\not\in K.$

Consider $\tau_\af, \tau\in K$ such that $\tau_\af(b)\to \tau(b)$ for all $b\in C.$ 
Let us show that $\imath(\tau_\af)(a)\to \imath(\tau)(a)$ for all $a\in {\rm{Ped}}(A)_+.$}}

{{By the definition of the Pedersen ideal, $a\le \sum_{i=1}^n x_i$ for some 
$x_i=g(y_i),$ where $y_i\in A_+,$ and $g\in C_0((0, +\infty))_+$ with compact support, 
$i=1,2,...,n.$ It follows from \ref{Lfullep} that $\la g(y_i)\ra\le  m\la c\ra$ for some 
integer $m\ge 1.$ Therefore, $\la a \ra \le nm\la c\ra.$  It follows 
that 
\beq
{\rm d}_{\imath(\tau_\af)}(a), {\rm d}_{\imath(\tau)}(a)\le nm.
\eneq
Then, for any $b\in \overline{aAa}$ with $\|b\|\le 1,$ $|\imath(\tau_\af)(b)|\le nm$ and $|\imath(\tau)(b)|\le nm.$
In other words, 
\beq\label{180911-1}
\|\imath(\tau_\af)|_{\overline{aAa}}\|\le nm \andeqn \|\imath(\tau)|_{\overline{aAa}}\|\le nm.
\eneq
%$$
%|\imath(
%$$
For any $\ep>0,$ by \ref{Lfullep}, there are $z_1,z_2,...,z_N\in A$ such that
\beq
\sum_{i=1}^Nz_i^*cz_i=f_\ep(a).
\eneq
It follows that
\beq\label{180912-1}
\sum_{i=1}^Na^{1/2}z_i^*cz_i a^{1/2}=a^{1/2}f_\ep(a)a^{1/2}.
\eneq
Consider {{the}} element $b=\sum_{i=1}^Na^{1/2}z_i^*cz_i a^{1/2}
%%$ with the form  $b=\sum_{i=1}^{N'}z_i'^*cz_i'\in \overline{aAa},$  where $z_i', z_i\in 
\in \overline{aAa}.$ {{Set}} $z_i'=z_ia^{1/2},$
$i=1,2,...,N.$
Then,  since $c^{1/2}z_i'z_i'^*c^{1/2}\in C,$ 
\beq\nonumber
\imath(\tau_\af)(b)&=&\sum_{i=1}^{N}\imath(\tau_\af)(z_i'^*cz_i)=
\sum_{i=1}^{N}\imath(\tau_\af)(c^{1/2}z_i'z_i'^*c^{1/2})=\sum_{i=1}^{N}\tau_\af(c^{1/2}z_i'z_i'^*c^{1/2})\\\nonumber
&\to& \sum_{i=1}^{N}\tau(c^{1/2}z_i'z_i'^*c^{1/2})= \sum_{i=1}^{N}\imath(\tau)(c^{1/2}z_i'z_i'^*c^{1/2})\\
&=&\sum_{i=1}^{N}\imath(\tau)(z_i'^*cz_i)=\imath(\tau)(b).
\eneq
Since $a^{1/2}f_\ep(a)a^{1/2}\to a$ in norm, by \eqref{180911-1} and by \eqref{180912-1},
$$
\imath(\tau_\af)(a)\to \imath(\tau)(a).
$$
This shows that $\imath$ is continuous. 
Since $K$ is compact,
%and $\imath(K)$ is Hausdorff, 
$\imath(K)$ is compact in ${\widetilde{T}}(A).$   Since $0\not\in K,$ $0\not\in\imath(K)=\imath(\overline{{\rm{T}}(C)}^{\rm w}).$}}
%(and 
%$0\not\in \imath(K).$)}}
\end{proof}

\begin{df}\label{Dkerrho}
Let $A$ be a $\sigma$-unital \CA\, with ${\widetilde{\mathrm{T}}}(A)\not=\{0\}.$
Suppose that there  is  {{ a nonzero}}  element $e\in {\mathrm{Ped}(A)}_+$ which is full in $A.$

{{Set}} $A_e=\overline{eAe}.$ Then, {{by Lemma \ref{Tpedersen},}} $A_e$ is {{compact}}.
Consequently, by {{Lemma}} \ref{compactrace}, $0\not\in \overline{\mathrm{T}(A_e)}^\mathrm{w}.$
Assume that $A$ is not unital. 
%One extends 
{{Each}}  $\tau\in \overline{\mathrm{T}(A_e)}^\mathrm{w}$ {{extends uniquely}}
to a tracial state on ${\widetilde A}_e.$
There is {{a canonical}}  order-preserving \hm\, $\rho_{{\widetilde A}_e}: \mathrm{K}_0({\widetilde A}_e)\to \Aff(\overline{\mathrm{T}(A_e)}^\mathrm{w}).$
By \cite{Br1}, one may identify $\mathrm{K}_0(A)$ with $\mathrm{K}_0(A_e).$
The composition of maps from $\mathrm{K}_0(A)$ to $\mathrm{K}_0(A_e),$ then from
$\mathrm{K}_0(A_e)$ to $\mathrm{K}_0({\widetilde A}_e)$ and then to $\Aff(\overline{\mathrm{T}(A_e)}^\mathrm{w})$
 is a %%order preserving
\hm\, which will be denoted  by $\rho_A.$
Denote by ${\rm ker}\rho_{A}$ the subgroup
of  $\mathrm{K}_0(A)$  consisting of those $x\in \mathrm{K}_0(A)$ such that
$\rho_{A}(x)=0.$ Elements in ${\rm ker}\rho_A$  are called infinitesimal elements.
%This subgroup does not depend on the choice of $e.$
\end{df}

\section{Continuous scale and fullness}

\begin{df}\label{Dconscale}
The previous section discussed \CA s with bounded scale. Let us recall the definition of continuous scale
(\cite{Lncs1} and \cite{Lncs2}).

Let $A$ be a $\sigma$-unital 
%and non-unital 
\CA.
Fix an increasing  approximate unit $(e_n)$ for $A$
with the property that
$$
e_{n+1}e_n=e_ne_{n+1}=e_n\rforal n\ge 1.
$$
The C*-algebra $A$ is said to have continuous scale if, for any $b\in A_+\setminus \{0\},$
there exists $N\ge 1$ such that
$$
e_m-e_n\lesssim b,\quad m>n\geq N.
$$
%for all $m>n\ge N.$
This definition does not depend on the choice of $(e_n)$ above.
{{By \ref{Pconscale} below, if $A$ has continuous scale and $T(A)\not=\O,$ then $\lambda_s=1.$}}

\end{df}

\begin{rem}\label{Rcontext}
%It should be noted that for any non-elementary separable simple \CA\, $A$
%there is a non-zero hereditary \SCA\, $B\subset A$ such that $B$ has continuous scale (2.3 of \cite{Lncs2}).
%%%%
\iffalse
Let $A$ be a separable finite exact simple \CA\, such that $A\otimes {\cal Z}\cong A.$
Then, by 6.6 of \cite{ERS11} (see also 6.22 of \cite{Rl}), the following result (together with  \ref{Subcontsc}) implies that there exists a non-zero hereditary \SCA\, $B$ of $A$
such that $B$ has continuous scale (see \ref{Ccontsc} below).
\fi
%%%%%
Let $A$ be a 
%separable finite 
exact simple \CA\, such that $A\otimes {\cal Z}\cong A.$
Then, by 6.6 of \cite{ERS11} (see also 6.2.3 of \cite{Rl}), 
{{the map $\la a\ra \mapsto \widehat{\la a\ra}$ is an isomorphism of 
the ordered semigroup of purely non-compact elements of 
${\rm Cu}(A)$ {{with}} ${\rm LAff}_+({\tilde{T}}(A)).$ 
Hence  \ref{Subcontsc}
%Theorem \ref{Pconscale} 
below implies that there exist a non-zero hereditary \SCA\,  $B$ of 
$A\otimes {\cal K}$
such that $B$ has continuous scale.}}

%{\red{Thus}} Theorem  \ref{Pconscale} together with \ref{Subcontsc} imply that there exists a non-zero hereditary \SCA\, $B$ of $A$
%such that $B$ has continuous scale (see \ref{Ccontsc} below).

{{In fact slightly more can be said. 
Let $a\in A_+$ be a strictly positive element. 
One finds a non-zero  element $b\in (A\otimes {\cal K})_+$ such that $\widehat{\la b\ra}$ is continuous 
and $\widehat{\la b\ra}< \widehat{\la a \ra}.$  
Write $C=A\otimes {\cal K}$ and  view $A$ as a hereditary \SCA\, of $C.$
Note that $C$ is also ${\cal Z}$-stable.
 It follows from part (ii) of Theorem 1.2 of \cite{Rlz}
that there exists a {{nonzero}} positive element $b_1\in C$ such that $\overline{b_1C}\subset \overline{aC}$
such that $\la b_1^2\ra=\la b_1\ra=\la b\ra.$ Note that $b_1^2=b_1b_1^*\in \overline{aCa}=A.$
In other words, $A$ contains a positive element $b_1$ such that $\widehat{\la b_1\ra}$ is
continuous.}}

{{Thus}} Theorem  \ref{Pconscale} together with \ref{Subcontsc} imply that there exists a  non-zero hereditary
 \SCA\,  $B$ of $A$
such that $B$ has continuous scale (see \ref{Ccontsc} below).

%{{Moreover, for a general separable simple \CA\, $A,$ one can always find a nonzero hereditary \SCA\, $B$ of $A$ 
%such that $B$ has a continuous scale (see Proposition 2.3 of \cite{Lncs2}).}}

%{\red{Thus the}} following result  {Pconscale} together implies that there exists a non-zero hereditary \SCA\, $B$ of $A$
%such that $B$ has continuous scale (see \ref{Ccontsc} below).

\end{rem}

\begin{thm}[cf.~\cite{Lncs2}]\label{Pconscale}
Let $A$ be a $\sigma$-unital  
%infinite dimensional 
non-elementary 
simple \CA\, with continuous scale.
% such that ${\mathrm{T}(A)}\not=\O.$
Then

{\rm (1)} ${\mathrm{T}(A)}$ is compact;

{\rm (2)} $\mathrm{d}_\tau(a)$ is continuous on ${\widetilde{\mathrm{T}}}(A)$ for any strictly positive element $a$ of $A;$

{\rm (3)} $\mathrm{d}_\tau(a)$ is continuous on $\overline{{\mathrm{T}(A)}}^\mathrm{w}$  for any strictly positive element $a$ of $A.$

Conversely, if 
%${\mathrm{QT}}(A)={\mathrm{T}}(A),$ 
$A$   has strict comparison for
positive elements using tracial states (see \ref{Dstrictcom}),  and $A$ is algebraically simple, then {\rm (1)}, {\rm (2)}, and {\rm (3)} are  equivalent  and also
equivalent to each of the following conditions:

{\rm (4)} $A$ has continuous scale;

{\rm (5)} $\mathrm{d}_\tau(a)$ is continuous on $\overline{{\mathrm{T}(A)}}^\mathrm{w}$  for some  strictly positive element $a$ of $A;$

{\rm (6)} $\mathrm{d}_\tau(a)$ is continuous  on ${\widetilde{\mathrm{T}}}(A)$ for some strictly positive element $a$ of $A.$

\end{thm}

\begin{proof}
Most parts of the theorem  are well known.  That (1) holds is  perhaps less well known.

%Suppose that
Since  $A$ has continuous scale, $A$ is algebraically simple  (3.3 of \cite{Lncs1}).
In particular, $A={\mathrm{Ped}(A)}.$ As  noted in Definition \ref{DTtilde},
% in {\red{\ref{compactrace} and \ref{Tpedersen}}}, 
% \ref{Csemicon},
$K=\overline{{\mathrm{T}(A)}}^\mathrm{w}$ is compact.
Let $a\in A$ be a  strictly positive element.
Fix an element $b\in  A_+\setminus \{0\}$ {{with $\|b\|=1$.}}
Put $B=\overline{f_{1/2}(b)Af_{1/2}(b)}.$ 
%Note that $
Since $A$ is not  
%{{finite dimensional.}}
%
elementary,
%{{Then}}
$B_+$  contains infinitely many mutually orthogonal  non-zero elements  $\{x_n\}$
with $0\le x_n\le 1,$ $n=1,2,....$ 
{{By repeatedly applying Lemma 3.5.4 of \cite{Lnbk},  one then finds, for each $n,$ 
$n$ nonzero mutually  orthogonal positive elements 
$\{x_{n,1}, x_{n,2},...,x_{n,n}\}$ in $A$ with $0 \le x_{n,j}\le 1$ 
such that $x_{n,1}\lesssim x_{n,2}\lesssim\cdots \lesssim x_{n,n}$ (see also 2.3 of \cite{Lncs1}). }}

Note that $\tau(f_{1/8}(b))$ is bounded on $K$ and 
\beq
d_\tau(f_{1/4}(b))\le \tau(f_{1/8}(b))\rforal \tau\in K.
\eneq
Since $f_{1/4}(b)x_{n,j}=x_{n,j}$ for all $1\le j\le n$ and all $n,$  it follows that,  
for any $\ep>0,$ there exists $x_{n(\ep), 1}$ such that $\mathrm{d}_\tau(x_{(\ep),1})<\ep$ for all
$\tau\in K.$
Note that
$$
\mathrm{d}_\tau(a)=\lim_{n\to \infty}\tau(f_{1/2^n}(a))\rforal \tau\in K.
$$
%Denote 
{{Note that, with}} $e_n=f_{1/2^n}(a),$ $n=1,2,...,$ 
{{$(e_n)_{n=1}^{\infty}$ forms an approximate identity 
with $e_{n+1}e_n=e_n$ for all $n.$}} Since $A$ has continuous scale,  for any $\ep>0,$
there exists $N\ge 1$ such that
$$
e_m-e_n\lesssim x_{k(\ep),1}\rforal m>n\ge N.
$$
In particular,
\vspace{-0.12in} \beq\label{Pconscale-2}
\tau(e_m)-\tau(e_n)<\ep\rforal \tau\in K.
\eneq
It follows that $\mathrm{d}_\tau(a)$ is continuous on $K.$
Since $\mathrm{d}_\tau(a)=1$ on ${\mathrm{T}(A)}$ and ${\mathrm{T}(A)}$ is dense in $K,$
$\mathrm{d}_\tau(a)=1$ for all $\tau\in K.$ This implies that
${\mathrm{T}(A)}=K.$ This proves (1)  and (3).
%It is clear that 
{{Note that (2) is}}  equivalent to (3), as $A={\rm Ped}(A),$ 
{{and so ${\widetilde{\mathrm{T}}(A)}=\R_+{\mathrm{T}}(A).$}}

Conversely,  suppose that  $A$ is as stated, suppose that $\mathrm{d}_\tau(a)$ is continuous,
and suppose that $e_n=f_{1/2^n}(a),$ $n=1,2,....$  Then, $\tau(e_n)$ converges to 
${{{\rm d}_\tau(a)}}$ uniformly on $K.$ 
{{For any nonzero $b\in A_+,$ there exists 
$N\ge 1$ such that
${\rm d}_\tau(e_m-e_n)<{\rm d}_\tau(b)$ for all $\tau\in T(A)$ and for all 
$m>n\ge N.$}} 
%on $K$ for some strictly positive element,
 %then,
%\eqref{Pconscale-2} holds for every $\ep.$
Since $A$ has strict comparison for positive elements, 
 it follows {{that}}, 
 %for any {\red{nonzero}} $b\in A_+,$  there exists $N\ge 1$ such that
$$
e_m-e_n\lesssim b\rforal m>n\ge N.
$$
%%This implies that 
%{{In other words,}} 
Thus, $A$ has continuous scale.

In other words, in this case, if $A$ does not have continuous scale,
$\mathrm{d}_\tau(a)$ is not continuous on $K.$ In particular, $\mathrm{d}_\tau(a)$ is not identically
1. This implies $K\not={\mathrm{T}(A)}.$
The above shows that, under the assumption that $A$ {{is}} as stated in the second part of the theorem,
(1), (4) and (5) are equivalent.  Since (5) and (6) are equivalent, these are also equivalent to (6).
Since the notion of continuous scale is independent of {{the}} choice of $a,$ these
{{conditions}}  are also equivalent
to (2) and (3).

\end{proof}

\begin{prop}\label{Subcontsc}
Let $A$ be a $\sigma$-unital exact simple \CA\, with 
strict comparison {{for positive elements.}}
% and with continuous scale.   
Suppose that ${\mathrm{T}(A)}\not=\O.$
Let $a\in A_+$ be such that $\mathrm{d}_\tau(a)$ is continuous on ${\widetilde{\mathrm{T}}}(A).$
Then $\overline{aAa}$ has continuous scale.
%${\mathrm{T}(A)}.$
If, in addition,  $A$ is algebraically simple and $\mathrm{d}_\tau(a)$ is just assume to be continuous 
on $\overline{{\mathrm{T}(A)}}^\mathrm{w},$
then $\overline{aAa}$ has continuous scale.
\end{prop}

\begin{proof}
{{Put $B=\overline{aAa}.$ We may assume that $a\not=0.$
Choose  a nonzero element $c\in {\rm{Ped}}(A)$ with $0\le c\le 1.$
Put $C=\overline{cAc}.$ By \ref{CherePA=A} and \ref{Tpedersen}, $C={\rm Ped}(C).$
Put $K= \overline{{\mathrm{T}}(C)}^{\mathrm{w}}.$ 
Then, by \ref{compactrace},  $0\not\in K.$ 

Note that each $\tau\in K$ extends uniquely to an element {{of}} ${\widetilde{\mathrm{T}}}(A).$ 
Let $\imath: K\to {\widetilde{\mathrm{T}}}(A)$ {{denote this map  as in \ref{Subset}.}}
% be given by the extension. 
By \ref{Subset}, $0\not\in \imath(K)$ and $\imath(K)$ is compact. 
Therefore ${\rm{d}}_{\tau}(a)$ is continuous on $\imath(K).$ 
Let $e_n=f_{1/2^n}(a),$ $n=1,2,...$
Then $(e_n)_{n=1}^{\infty}$ is an approximate identity for $B$ such that $e_{n+1}e_n=e_n$ for all $n.$
Then ${\rm d}_{\tau}(e_n)\nearrow {\rm d}_\tau(a)$ uniformly on the compact set $K.$ 
For any $b_0\in B_+\setminus \{0\},$ there exists $N\ge 1$ such that, for all $m>n\ge N,$ 
\beq\label{189sec5-1}
{\rm d}_\tau(e_{m}-e_n)<{\rm d}_\tau(b_0)\tforal \tau\in K.
\eneq
Since $K$ generates ${\widetilde{T}}(A)$ as a cone,   \eqref{189sec5-1} holds for all 
$\tau\in {\widetilde{T}}(A)\setminus \{0\}.$
It follows that, for all $m>n\ge N,$ 
\beq
e_m-e_n\lesssim b_0.
\eneq
Therefore $B$ has continuous scale.

%We omit the proof of t
The last part of the statement  follows  since ${\tilde T}(A)=\R_+ T(A).$}}
%%%%%%%%%%%%%
\iffalse
{{Note  by \ref{CherePA=A} and \ref{Tpedersen}, $C={\rm Ped}(C).$}} 
Then, {{as pointed out in the proof of of \ref{Pconscale}, $B={\rm Ped}(B),$ and so 
${\widetilde{\mathrm{T}}}(B)=\R_+{\mathrm{T}}(B).$}} 
Each $\tau\in {\mathrm{T}}(B)$ extends uniquely to an element in ${\widetilde{T}}(A).$ 
Consider $K=\{\tau\in {\widetilde{\mathrm{T}}}(A): {\mathrm{d}}_\tau(a)=1\}.$
Since ${\mathrm{d}}_\tau(a)$ is a continuous, $K$ is weak* closed. Thus, as a weak* closed 
subset of the unit ball of $B^*.$ It is compact. 
{{As pointed out in the proof 
of \ref{Pconscale}, $A={\rm Ped}(A),$ and so ${\widetilde{\mathrm{T}}}(A)=\R_+{\mathrm{T}}(A).$}}
{{Note also, by \ref{CherePA=A} and \ref{Tpedersen}, $B={\rm Ped}(B).$ 
We also have ${\widetilde{\mathrm{T}}}(B)=\R_+{\mathrm{T}}(B).$ On the other hand, every 
tracial state of $B$ extended uniquely to a tracial state
}}
Since $\mathrm{d}_\tau(a)$ is continuous on ${\mathrm{T}(A)},$ {{as by \ref{Pconscale},
${\mathrm{T}}(A)$ is compact,}}
{{one sees that}} $\mathrm{d}_\tau(a)$ is a continuous function on ${\widetilde{\mathrm{T}}}(A)$
which, in this case, is  naturally identified with  ${\widetilde{\mathrm{T}}}(B).$  By {{\ref{Pconscale},}}
%{Rcontext},
$B$ has continuous scale.
%We omit the proof of the last part of the statement.
\fi
%%%%%%%%%%%%%%%%%%%%%%%%%%%
%Therefore $\{\tau\in {\widetilde{\mathrm{T}}}(A): \mathrm{d}_\tau(a)=1\}$ is compact.
\end{proof}

Now we turn to the important concept of local uniform fullness.

\begin{df}\label{fullunif}
Let $A$ be a \SCA\, of a 
%$\sigma$-unital 
\CA\, $B.$
% Fix an approximate identity $\{e_n\}$ for
%$B.$
An element $a\in A_+\setminus \{0\}$ is said to be {\it uniformly full} in $B,$ if
 there are a positive number $M(a)>0$ and an integer $N(a)\ge 1$ such
that, for any $b\in B_+$ with $\|b\|\le 1$ and any $\ep>0,$  there are $x_i(a), x_2(a),..., x_{n(a)}(a)\in B$ such
that $\|x_{i}(a)\|\le M(a),$ $n(a)\le N(a),$ and
$$
\|\sum_{i=1}^{n(a)}x_{i}(a)^*ax_{i}(a)-b\|<\ep.
$$

In this case, we shall also say that $a$ is $(N({{a}}), M(a))$ full.

We shall say that $a$ is {\it strongly uniformly full} in $B,$ if the above property holds
with  $\ep=0$  and replacing ``$<\ep$" by $=0.$

We shall say that $A$ is {\it locally uniformly full},
if every element $ a\in A_+\setminus \{0\}$ is uniformly full;
and we say $A$ is {\it strongly locally uniformly full} if every $a\in A_+\setminus \{0\}$ is
strongly uniformly full.
%if, for any $a\in A\setminus\{0\},$ there are positive number $M(a)>0$ and an integer $N(a)\ge 1$ such
%that, for any $b\in B_+$ with $\|b\|\le 1$ and any $\ep>0,$  there are $x_i(a), x_2(a),...x_{m(a)}(a)\in B$ such
%that $\|x_{i}(a)\|\le M(a),$ $m(a)\le N(a)$ and
%$$
%\|\sum_{i=1}^{m(n)}(x_{i}(a))^*ax_{i}(a)-b\|<\ep.
%$$

%We say that $A$ is {\it strongly locally uniformly full} in $B$ if, the above holds with $\ep=0.$

If $B$ is unital and $A$ is full in $B,$ then $A$ is always  strongly locally uniformly full.
In fact, for each $a\in A\setminus \{0\},$ there are
%$M(a)>0,$  an integer $N(a)\ge 1$
$x_1, x_2,...,x_m\in B$ such that
$$
\sum_{i=1}^m x_i^*ax_i=1_B.
$$
Choose $M(a)=\max\{\|x_i\|: 1\le i\le m\}$ and $N(a)=m.$

Let $A$ be a 
%separable
 \CA, let $B$ be non-unital \CA, and let $L: A\to B$ be a positive linear map.
 Let $F: A_+\setminus \{0\}\to \N\times \R_+\setminus\{0\}.$
Suppose that ${\cal H}\subset A_+\setminus \{0\}$ is  a subset.
We shall say that $L$ is $F$-${\cal H}$-full if, {{for any $a\in {\cal H},$}} for any $b\in B_+$ with $\|b\|\le 1,$ 
{{and}} any $\ep>0,$
there are $x_1, x_2,...,x_m\in B$ such that
$m\le N(a)$ and $\|x_i\|\le M(a),$ where
$(N(a), M(a))=F(a),$  and
\beq\label{localfull-1}
\|\sum_{i=1}^mx_i^*L(a)x_i-b\|\le \ep.
\eneq

We {{shall}} say $L$ is exactly $F$-${\cal H}$-full, if \eqref{localfull-1} holds for $\ep=0.$
\end{df}

\begin{prop}\label{localunit}
Let $A$ be a  {{nonzero}} $\sigma$-unital 
%\%full 
\SCA\, of a \CA\, $B.$
Suppose  that $B$ is 
{{$\sigma$-unital and}}
 algebraically simple.
%that there are $b,e\in (B_1)_+$ with $0\le b\le e\le 1$ and with
%$eb=be=b.$
%Suppose that $B=\overline{bB_1b}$ and $A\subset B.$
Then $A$ is strongly locally uniformly full in $B.$
\end{prop}

\begin{proof}
Let ${{b}}\in A$ be a strictly positive element.
Then $\overline{{{b}}B{{b}}}$ is {{a}} full hereditary \SCA\, $B.$
It suffices to show that $\overline{{{b}}B{{b}}}$ is {{strongly}} locally uniformly full in $B.$
Put $B_1=\overline{{{b}}B{{b}}}.$ 
In what follows we will identify $B$ with $B\otimes e_{11}$ in ${\rm M}_n(B).$

{{Since $B$ is algebraically simple,
 $B={\rm Ped}(B).$  
By \ref{Tpedersen}, $B$ is compact.}}
{{Applying \ref{subcommatrix},}} {{l}}et $e\in \mathrm{M}_n(B)$ {{for some 
$n\ge 1$}} with $0\le e\le 1$ and $w\in \mathrm{M}_n(B)^{**}$ {{be}} such that
$$
w^*a,\,w^*aw\in \mathrm{M}_n(B), \,\,\, ww^*a=aww^*=a,\andeqn w^*awe=ew^*aw=w^*aw\rforal a\in B.
$$
Note that also $aw\in {\mathrm{M}}_n(B)$ for all $a\in B.$

 By Lemma \ref{Lfullep}, for any $1/4>\ep>0$ and any $a\in (B_1)_+\setminus \{0\},$
there are $x_1,x_2,...,x_m\in {\mathrm{M}}_n(B)$ such that
%$$
%\|\sum_{i=1}^mx_i^*ax_i-e\|<\ep.
%$$
%Therefore, by 2.2 and 2.3 of \cite{Rjfa2},   there are $y\in B$ such that
$$
f_{\ep}(e)=\sum_{i=1}^m x_i^*ax_i.
$$
Let $p$ {{denote}} the range projection of $B.$  Then $px_i\in \mathrm{M}_n(B)$ for {{$i=1,2,...,m.$}}
%\in \mathrm{M}_n(B).$ 
We may assume that $px_i=x_i,$ $i=1,2,...,m.$

%Note that $f_{\ep}(e)w^*bw=w^*bwf_{\ep}(e)=w^*bw.$   It follows 
%${{Note}} that, for any  {{fix}} 
Fix $x\in B_+$ with $\|x\|\le 1.$ Then
$$
f_\ep(e) w^*xw=w^*xwf_\ep(e) =w^*xw.
$$
Let $M(a)=\max\{\|x_i\|: 1\le i\le m\}$ and $N(a)=m.$
Then, 
%for any $x\in B_+$ with $\|x\|\le 1,$ 
$w^*x^{1/2}wf_\ep(e)w^*x^{1/2}w=w^*xw.$
Therefore,
$$
x^{1/2}wf_\ep(e)w^*x^{1/2}=w(w^*xw)w^*=x.
$$
Put $z_i=x_iw^*x^{1/2},$ $i=1,2,...,m.$ 
%%Since $w^*{\red{z_i}}^{1/2}\in \mathrm{M}_n(B)$  and $px_i=x_i,$ 
{{Then}} $z_i\in B,$ ${{i=1,2,...,m.}}$
 Then $\|z_i\|\le M(a)$ and
\beq\nonumber
\sum_{i=1}^mz_i^*az_i&=&x^{1/2}w(\sum_{i=1}^mx_i^*ax_i)w^*x^{1/2}\\
&=&x^{1/2}wf_\ep(e)w^*x^{1/2}=x.
\eneq

\end{proof}

\begin{thm}\label{Tqcfull}
Let $A$ be a non-unital separable simple \CA\, with $A={\mathrm{Ped}(A)}$ and with
${\mathrm{T}(A)}\not=\O.$
Fix an element $e\in A_+\setminus \{0\}$ with $\|e\|=1$ and
$$0<d<\min\{\inf\{\tau(e): \tau\in {\mathrm{T}(A)}\}, \inf\{\tau(f_{1/2}(e)): \tau\in {\mathrm{T}(A)}\}\}.$$
Then there exists a map
$T: A_+\setminus \{0\}\to \N\times \R_+\setminus \{0\}$ 
%satisfying 
{{with the following property}}:
For any finite subset ${\cal H}_1\subset A_+^{\bf 1}\setminus \{0\},$
there {{are}} a finite subset ${\cal G}\subset A$ and $\dt>0$ satisfying the following conditions:
For any  \CA\, $B$ with  ${\rm{QT}}(C)={\rm T}(C)$
for all hereditary \SCA s $C$ of $B,$  and ${\mathrm{T}}(B)\not=\O,$ and
$0\not\in \overline{{\mathrm{T}}(B)}^\mathrm{w}$  which has  strict comparison  and almost 
has stable rank one, and
for any
${\cal G}$-$\dt$-multiplicative \cpc\, $\phi: A\to B$
such that
$$
\tau(f_{1/2}(\phi(e)))>d/2\tforal \tau\in {\mathrm{T}}(B),
$$
necessarily $\phi$ is exactly $T$-${\cal H}_1$-full.
Moreover, for any $c\in {\cal H}_1,$
$$
\tau(f_{1/2}(\phi(c)))\ge {d\over{8\min\{M(c)^2\cdot N(c): c\in {\cal H}_1\}}}\tforal \tau\in {\mathrm{T}}(B).
$$

\end{thm}

\begin{proof}
Since $A$ is a 
%$\sigma$-unital 
simple \CA\, with $A={\mathrm{Ped}(A)},$
 there is  a map ${\rm T}_1: A_+\setminus \{0\}\to \N\times \R_+\setminus\{0\}$ such that
the identity map ${\rm id}_A$ is exactly $T_1$-$A_+\setminus\{0\}$-full.

Write $T_1=(N_1,M_1),$ where $N_1: A_+\setminus\{0\}\to \N$ and
$M_1:  A_+\setminus\{0\}\to \R_+\setminus \{0\}.$

Let $n\ge 2$ be  an integer such that $nd/2>1.$
Set  $N=2nN_1$ and $M=2M_1$ and $T=(N, M).$
%To simplify notation, \wilog, we may assume that there is $e\in

Let ${\cal H}_1\subset A_+\setminus \{0\}$ be a fixed finite subset.

Suppose that $x_{i,h},...,x_{N_1(h), h}\in A$ with
$\|x_{i,h}\|\le M_1(h)$ are  such that
\beq\label{Tqcfull-2}
\sum_{i=1}^{N_1(h)} x_{i,h}^* h^2 x_{i,h}=f_{1/64}(e)\rforal h\in {\cal H}_1.
\eneq
Choose a large enough ${\cal G}$ and small enough $\dt>0$ that, for any ${\cal G}$-$\dt$-multiplicative
\cpc\, $\phi$ from $A,$
%$  with $\|\phi\|\not=0$ has
%the property
%that
\beq\label{Tqcfull-3}
&&\|\phi(f_{1/64}(e))-f_{1/64}(\phi(e))\|<1/64\andeqn\\\label{Tqcfull-3+}
&&\|\sum_{i=1}^{N_1(h)} \phi(x_{i,h})^* \phi(h)^2\phi(x_{i,h})-f_{1/64}(\phi(e))\|<1/64\rforal h\in {\cal H}_1.
\eneq
Now let  that $\phi: A\to B$ (for any $B$ that fits the description in the theorem)
%which is 
be a ${\cal G}$-$\dt$-multiplicative \cpc\,
such that
\beq\label{Tqcfull-4}
\tau(f_{1/2}(\phi(e)))\ge d/2\rforal \tau\in \overline{{\mathrm{T}}(B)}^\mathrm{w}.
\eneq

Applying {{Lemma}} \ref{Lrorm} (using \eqref{Tqcfull-3+}),  one finds $y_{i,h}\in B$ with
$\|y_{i,h}\|\le 2\|x_{i,h}\|,$ $i=1,2,...,N_1(h)$ such that
\beq\label{Tqcfull-5}
\sum_{i=1}^{N_1(h)}y_{i,h}^*\phi(h)^2y_{i,h}=f_{1/16}(\phi(e))\rforal h\in {\cal H}_1.
\eneq

By the hypotheses on $B,$ and since $A$ is $\sigma$-unital, applying \ref{Lbkqc},  we may choose $e_1, e_2\in \mathrm{M}_n(B)_+$
and $w\in \mathrm{M}_n(B)^{**}$ as described there.
Put
\vspace{-0.14in} \beq\label{Tqcfull-6}
E_0&=&\diag(\overbrace{f_{1/8}(\phi(e)), f_{1/8}(\phi(e)), ..., f_{1/8}(\phi(e))}^{2n})\andeqn\\
E_1&=&\diag(\overbrace{f_{1/16}(\phi(e)), f_{1/16}(\phi(e)), ..., f_{1/16}(\phi(e))}^{2n})
\in {\mathrm{M}}_{2n}(B)_+.
\eneq
Then, by the strict comparison,
$$
e_2\lesssim E_0\in {\mathrm{M}}_{2n}(B).
$$
Since $B$ almost has stable rank one, there exists a unitary $u\in \widetilde{{\mathrm{M}}_{2n}(B)}$ such
that
$$
u^*f_{1/16}(e_2)u \in \overline{E_0({\mathrm{M}}_{2n}(B))E_0}.
$$
Then
\vspace{-0.11in} \beq\label{Tqcfull-8}
u^*f_{1/16}(e_2)uE_1=E_1u^*f_{1/16}(e_2)u=u^*f_{1/16}(e_2)u.
\eneq
We then may write
$$
\sum_{i=1}^{2nN_1(h)} {{(y_{i,h}')^*}}\phi(h)^{{2}} y_{i,h}'=E_1  \rforal h\in {\cal H}_1,
$$
where $y_{i,h}'\in M_{2n}(B)$ and $\|y_{i,h}'\|=\|y_{j,h}\|$ for some $j\in \{1,2,...,N_1(h)\},$ $i=1,2,...,2nN_1(h).$ 
%for all $h\in {\cal H}_1.$
Then
$$
\sum_{i=1}^{2nN_1(h)} (f_{1/16}(e_2)^{1/2}u {y_{i,h}'}^* )\phi(h)^2 (y_{i,h}'u^*f_{1/16}(e_2)^{1/2})= f_{1/16}(e_2).
$$
Therefore, for any $b\in B_+$ with $\|b\|\le 1,$
$$
\sum_{i=1}^{2nN_1(h)} (w^*b^{1/2}w)(f_{1/16}(e_2)^{1/2}u {y_{i,h}'}^*)\phi(h)^{1/2}\phi(h)\phi(h)^{1/2} (y_{i,h}'u^*f_{1/16}(e_2)^{1/2})(w^*b^{1/2}w)=w^*bw.
$$
Then
$$
\sum_{i=1}^{2nN_1(h)}(b^{1/2}w)(f_{1/16}(e_2)^{{1/2}}u {y_{i,h}'}^*\phi(h)^{1/2} )\phi(h) (\phi(h)^{1/2}(y_{i,h}'u^*f_{1/16}(e_2)^{{1/2}})w^*b^{1/2}=b.
$$
Note that $b^{1/4}w\in \mathrm{M}_n(B)$ and $f_{1/16}(e_2)\in \mathrm{M}_n(B).$
Therefore
$$
(b^{1/4}w)f_{1/16}(e_2)\in \mathrm{M}_n(B).
$$
It follows that
\vspace{-0.1in} \beq\label{Tqcfull-16}
&&(b^{1/2}w)(f_{1/16}(e_2)^{1/2}u {y_{i,h}'}^*\phi(h)^{1/2}) \in B
\andeqn\\
&&\|(b^{1/2}w)(f_{1/16}(e_2)^{1/2}u {y_{i,h}'}^*\phi(h)^{1/2}\|\le 2M(h)\rforal h\in {\cal H}_1.
\eneq
This implies that $\phi$ is exactly $T$-${\cal H}_1$-full.
\end{proof}

\begin{rem}\label{Rqcfull}
In the light of \ref{Pcom4C} below, Theorem \ref{Tqcfull}  can be applied with
% to the case 
%that 
%works for 
\CA s $B$ {{in the class}} ${\cal C}'$ {{defined just before \ref{Runitz}.}}
\end{rem}

\section{Non-unital and non-commutative  one dimensional complexes }

\begin{df}\label{Dbuild1}
%\begin{df}[See \cite{ET-PL} and  \cite{point-line}]\label{DfC1}
{\rm
Let $F_1$ and $F_2$ be two finite dimensional \CA s.
Suppose that there are  \hm s
$\phi_0, \phi_1: F_1\to F_2.$
Consider the mapping torus $M_{\phi_1, \phi_2}$:
$$
A=A(F_1, F_2,\phi_0, \phi_1)
:=\{(f,g)\in  C([0,1], F_2)\oplus F_1: f(0)=\phi_0(g)\andeqn f(1)=\phi_1(g)\}.
$$

For $t\in (0,1),$ define $\pi_t: A\to F_2$ by $\pi_t((f,g))=f(t)$ for all $(f,g)\in A.$
For  $t=0,$ define $\pi_0: A\to \phi_0(F_1)\subset F_2$ by $\pi_0((f, g))=\phi_0(g)$ for all $(f,g)\in A.$
For $t=1,$ define $\pi_1: A\to \phi_1(F_1)\subset F_2$ by $\pi_1((f,g))=\phi_1(g)$ for all $(f,g)\in A.$
In what follows, we will call $\pi_t$  a point evaluation of $A$ at $t.$
There is a canonical map $\pi_e: A \to F_1$ defined by $\pi_e(f,g)=g$ 
every $(f, g)\in A.$  It is a surjective map.
%{\it The notation $\pi_e$ will be used for this map throughout  this paper.}

The class of all \CA s described above will be denoted by ${\cal C}.$

If $A\in {\cal C}$, then $A$ is the pull-back of
\begin{equation}\label{pull-back}
\xymatrix{
A \ar@{-->}[rr] \ar@{-->}[d]^-{\pi_e}  && C([0,1], F_2) \ar[d]^-{(\pi_0, \pi_1)} \\
F_1 \ar[rr]^-{(\phi_0, \phi_1)} & & F_2 \oplus F_2.
}
\end{equation}
Every such pull-back is an algebra in ${\cal C}.$
Infinite dimensional  C*-algebras in ${\cal C}$ are sometime called one-dimensional non-commutative finite CW complexes (NCCW)
(see \cite{ELP1} and \cite{ELP2}) and Elliott-Thomsen building blocks (see \cite{point-line}).

{{Suppose that $F_1=M_{R_1 }(\C)\oplus M_{R_2 }(\C)\oplus \cdots \oplus M_{R_l}(\C)$ and 
 $F_2=M_{r_1}(\C)\oplus M_{r_2 }(\C)\oplus \cdots \oplus M_{r_k }(\C).$
In what follows
we may write $C([0,1], F_2)=\bigoplus_{j=1}^k C([0,1]_j, M_{r_j}{{)}},$ where $[0,1]_j$ denotes
 the $j$-th interval.}}

Denote by ${\cal C}_{0}$ the class of all \CA s $A$  in ${\cal C}$ which
satisfy the following conditions:

(1) ${\rm K}_1(A)=\{0\},$

(2) ${\rm K}_0(A)_+=\{0\},$

(3) $0\not\in \overline{{\mathrm{T}(A)}}^\mathrm{w}.$

%Here $\overline{{\mathrm{T}(A)}}^\mathrm{w}$ is the weak*-closure of ${\mathrm{T}(A)}$ in ${\mathrm{T}_0(A)}.$

\CA s  in  ${\cal C}_0$ are stably projectionless. Condition (3) is equivalent to compact spectrum.

Examples of \CA s in ${\cal C}_0$  can be found in \cite{Raz}.
Let $F_1={\mathrm{M}}_k$ for some $k\ge 1$ and $F_2={\mathrm{M}}_{(m+1)k}$
for some $m\ge 1.$
Define $\psi_0, \psi_1: F_1\to F_2$ by
$$
\psi_0(a)=\diag(\overbrace{a,a,...,a}^m, 0)\andeqn
\psi_1(a)=\diag(\overbrace{a,a,..., a,a}^{m+1})
$$
for all $a\in F_2.$
Let {{us write}} 
\beq\label{ddraz}
A=A(F_1, F_2, \psi_0, \psi_1):=R(k, m, m+1).
\eneq
Then,  as  shown in \cite{Raz},  ${\rm K}_0(A)=\{0\}={\rm K}_1(A)$ and it is easy to check that $0\not\in \overline{{\mathrm{T}(A)}}^\mathrm{w}.$
Let $e\in R(k,m,m+1)$ be a strictly positive element.
Then {{(see \ref{Dboundedscale})}}
$$
\lambda_s(R(k,m, m+1))=\inf\{d_\tau(e): \tau\in \mathrm{T}(R(k,m,m+1))\}=m/(m+1).
$$
Denote by ${\cal R}_{\mathrm{az}}$ the class of \CA s which are finite direct sums of \CA s  {{as}}
% in the form 
in \eqref{ddraz}.
%Denote by ${\cal C}_0$ the subclass of \CA s in ${\cal C}_\omega$ which satisfies
%the condition
%(2)' $K_0(A)={\rm ker}\rho_A.$
Denote by ${\cal C}_0^0$ the subclass of \CA s in ${\cal C}_0$ which
also satisfy the {{stronger}} condition
(2)' ${\rm K}_0(A)=\{0\}.$

{{ Let $F_1=\C\oplus \C, F_2={\mathrm{M}}_{2n}(\C)$. For $(a,b)\in \C\oplus \C=F_1$, define
$$\psi_0(a, b)=\diag(\underbrace{a,a...a}_{n-1},\underbrace{b,b...b}_{n-1}, 0, 0)
~~~\mbox{and}~~~~\psi_1(a, b)=\diag(\underbrace{a,a...a}_n,\underbrace{b,b...b}_n).$$
Then $A(F_1, F_2, \psi_0,\psi_1)=A$ has the property that
$\mathrm{K}_0(A)=\{(k,-k )\in \Z\oplus \Z)\}$ (which is isomorphic to $\Z$) but $\mathrm{K}_0(A)_+=\{0\}$.
Also, $\mathrm{K}_1(A)=\{0\}.$
Thus $A\in {\cal C}_0$ but $A\notin {\cal C}_0^0.$}}

Let  ${\cal C}'$ denote the class of all full hereditary \SCA s of \CA s in ${\cal C},$ 
let ${\cal C}_0'$ denote the class of all full hereditary \SCA s of \CA s in ${\cal C}_0,$
%let ${\cal C}_{0}'$ denote the class of all full hereditary \SCA s of \CA s in ${\cal C}_{0}$ 
and
let ${{\cal C}_0^0}'$ denote the class of all full hereditary \SCA s of \CA s in ${\cal C}_0^0.$

}

\end{df}

%\begin{df}
%${\cal C}$ denote the class of \CA s which are hereditary \SCA\, of non-unital projectionless
%non-commutative one dimensional compleces.
%
%${\cal C}_0^0$ denote the class of \CA s  $D$ in ${\cal C}$ with $K_i(D)=\{0\},$ $i=0,1,$
%such that $0\not\in \overline{T(D)}^\mathrm{w}.$
%
%Let ${\cal C}_{00}'$ denote the class of all full hereditary \SCA s of \CA s in ${\cal C}_{00}.$
%
%Note that $0\not\in \overline{T(D)}^\mathrm{w}.$
%
%${\cal C}_0$ denote the class of \CA s  $D$ in ${\cal C}$ with $K_1(D)=\{0\}$
%and  ${\cal C}_0'$ denote the class of all full hereditary \SCA\, of \CA s in ${\cal C}_0.$
%
%${\cal C}_1$ denote ?

%\end{df}

%\begin{prop}\label{Punitiz}
\begin{rem}\label{Runitz}
Let $A=A(F_1, F_2, \psi_0, \psi_1)\in {\cal C}_0.$
Then ${\widetilde A}\in {\cal C}.$ Moreover, 
${\widetilde A}=A(F_1', F_2, \psi_0',\psi_1')$ with both
$\psi_0'$ and $\psi_1'$  unital, defined as follows:

Let $F_1'=F_1\oplus \C$ and let
$p=\psi_0(1_{F_1})\in F_2$ and $q=\psi_1(1_{F_1})\in F_2$.  Define $\psi_0',~{\psi_1'}: F_1'\to F_2$ by
$$
\psi_0'((a,\lambda))=\psi_0(a)\oplus \lambda\cdot (1_{F_2}-p) ~~\mbox{and}~~\psi_1'((a,\lambda))=\psi_{ 1}(a)\oplus \lambda\cdot (1_{F_2}-q)
$$
 for all $a\in F_1\andeqn \lambda\in \C.$
 
{{One checks that ${\rm K}_0({\widetilde A})$ is finitely generated (see Proposition 3.4 of \cite{GLN}). 
In fact, ${\rm K}_0({\widetilde{A}})_+$ is finitely generated (see Theorem 3.15 of \cite{GLN}).
%Suppose that $x=[p_i]-[q_i]\in K_0(A),$ where $p_i, q_i\in M_n({\widetilde{A}})$ are projections.
Let $\pi: {\widetilde{A}}\to \C$ {{denote}} the quotient map. 
%Then $\pi_{*0}(x)=0.$ 
%It follows that $\pi(p)-\pi(q)=0.$ Let $n$ be the rank of $\pi(p).$ 
Suppose that $\{[p_i]: 1\le i\le k\}$ generates {{the semigroup}} ${\rm K}_0({\widetilde{A}})_+.$ Let $x\in {\rm K}_0(A)\subset {\rm K}_0({\widetilde{A}}).$
Then $x=\sum_{i=1}^k(m_i[p_i]-n_i[p_i])=[p]-[q],$ where $m_i\ge 0, n_i\ge 0$ and $p, q\in {\rm M}_N({\widetilde{A}})$
(for some integer $N\ge 1$)
are projections such that $[p]=\sum_{i=1}^k m_i[p_i]$ and $[q]=\sum_{i=1}^k n_i[p_i].$
One also has, since $x\in {\rm K}_0(A),$ $\pi(p)$ and $\pi(q)$ are equivalent in ${\rm M}_N.$ Let $n$ denote  the rank of $\pi(p)$ and  $r_i$ the rank of 
$\pi(p_i),$
%and $R_i$ 
%denote the rank of $\pi(q_i),$ 
$1\le i\le k.$ 
Then $\sum_{i=1}^km_ir_i=n=\sum_{i=1}^kn_i r_i.$
Consequently, 
\beq\nonumber
&&(\sum_{i=1}^k (m_i(([p_i]-r_i[1_{\tilde A}]))-(n_i[p_i]-r_i[1_{\tilde A}])))=
(\sum_{i=1}^k m_i[p_i]-n[1_{\tilde A})-(n_i[(p_i]-n[1_{\tilde A}])\\
&&=\sum_{i=1}^k(m_i[p_i]-n_i[p_i])=x.
\eneq
It follows that ${\rm K}_0(A)$ is generated by $\{([p_i]-r_i[1_{\tilde{A}}]):1\le i\le k\}.$ 
In other words, ${\rm K}_0(A)$ is finitely generated.}}

Since $A\in {\cal C}_0,$ either $\psi_0$ or $\psi_1$ is not unital. Hence at least one of $\psi_0'$ and $\psi_1'$ is nonzero on the second direct summand $\C$ in $F_1'=F_1\oplus \C$.

%Since $A\in {\cal C}_0,$ then either $\psi_0$ or $\psi_1$ are not unital.
%\Wlog, we may assume that $\psi_0$ is not unital.
%Let $F_1'=F_1\oplus \C$ and let
%$p=\psi_0(1_{F_1})\in F_2$ be a projection.
%%Define $\psi_0'': \C\to F_2$ by
%%$$
%%\psi_0''(\lambda)=\lambda\cdot (1_{F_2}-p)\rforal a\in F_1.
%%$$
%Define $\psi_0': F_1'\to F_2$
%by
%$$
%\psi_0'((a,\lambda))=\psi_0(a)\oplus (\lambda\cdot (1_{F_2}-p)\rforal a\in F_1\andeqn \lambda\in \C.
%$$
\end{rem}

%\end{prop}

\begin{prop}\label{Pcom4C}
{\rm (1)} Let $A\in {\cal C}'$.
 Then, for any $a_1, a_2\in A_+,$  $a_1\lesssim a_2$ if and only if
 $d_{tr\circ \pi}(\pi(a_1))\le d_{tr\circ \pi}(\pi(a_2))$ for every irreducible representation $\pi$
 of $A,$ where we use $tr$ for the tracial state on matrix algebras. 

{\rm (2)} Let $A\in {\cal C}',$ and
let $c\in A_+\setminus \{0\}.$ Then $c$ is full if and only if, for any
$\tau\in {\mathrm{T}(A)},$ $\tau(c)>0.$

\end{prop}

\begin{proof}
 For {\rm (1)}, we first  consider the case that $A\in {\cal C}.$ By considering ${\widetilde A},$ one sees that
this case follows from 3.18 of  \cite{GLN}.

Since a \CA\,  $A\in {\cal C}'$ is a hereditary \SCA\, of {{some}} $B$ in ${\cal C},$ 
it is easy to see that $A$ also has the above-mentioned comparison property.

For {\rm (2)}, let us  first assume again that $A\in {\cal C}.$ It is clear that if $c\in A_+$ and $\tau(c)=0,$ for some $\tau\in {\rm T}(A)$ then $c$ has zero value
somewhere in $\mathrm{Sp}(A)=\bigsqcup_j(0,1)_j\cup\mathrm{Sp}(F_1).$ Therefore $c$ is in a proper {{closed two-sided}} ideal of $A.$

Now assume that $\tau(c)>0$ for all $\tau\in {\mathrm{T}(A)}.$
It follows that $\pi(c)>0$ for every finite dimensional irreducible representation of $A.$ 
Therefore $c$ is full in $A.$
\iffalse
Write $c=(a,b),$ where $a\in C([0,1], F_2)$ and $b\in F_1.$ Then $b>0$ and $a(t)>0$
for all $t\in [0,1].$  It follows that
$$
d_{tr\circ \pi}(\pi(c))>0
$$
for all 
%$c\in A.$ 
{{$\pi\circ {\rm{Sp}}(A).$}} Since we assume that $0\not\in \overline{{\mathrm{T}(A)}}^\mathrm{w},$
this implies that
$$
\inf\{d_{tr\circ \pi}(\pi(c)): \pi \}>0.
$$
There is an integer $n\ge 1$ such that
$$
d_{tr\circ \pi}(\pi({\bar c}))>2
$$
for all irreducible representations $\pi$ of $A$ (or ${\mathrm{M}}_n(A)$),
where
$$
{\bar c}=\diag(\overbrace{c,c,...,c}^n).
$$
By (1), this implies $a\lesssim {\bar c},$ where $a$ is a strictly positive element.
This implies that $c$ is a full element in $A.$
\fi
In general, let $A$ be a full hereditary \SCA\, of $B\in {\cal C}.$
Let $c\in A_+.$ Then  $c\in A_+$ is full if and only if it is full in $B.$
Therefore the general case follows from the case that $A\in {\cal C}.$

\end{proof}

\begin{prop}\label{str=1}
%\linebreak

{\rm (1)} Every \CA\, in ${\cal C}'$  has stable rank one;

{\rm (2)}  If $A\in {\cal C}$ and $A$ is unital, {{then}}
 the exponential rank of $A$   is at most $2+\ep.$
 If $A\in {\cal C}$ and $A$ is not unital, then ${\widetilde A}$ has exponential rank at most $2+\ep.$

{\rm (3)} Every \CA\, in ${\cal C}$ is semiprojective.

{\rm (4)} Let $A\in {\cal C}$ and let $k\ge 1$ be an integer.  Suppose that every irreducible representation
of $C$  has dimension at least $k.$
 Then, for any $f\in {\rm
LAff}_{b,0+}({\overline{{\mathrm{T}}(A)}^{\rm w}})$ with $0\le f \le 1,$ there exists a positive
element $a\in {\mathrm{M}}_2(A)$ such that
$$
\max_{\tau\in \mathrm{T}(A)}|\mathrm{d}_{\tau}(a)-f(\tau)|\le 2/k.
$$

\end{prop}

\begin{proof}
(1) follows from 3.3 of \cite{GLN}. (2) follows from 3.16 of \cite{GLN} (see also 5.19 of \cite{Lncbms}).

It was shown in \cite{ELP1} that every \CA\, in ${\cal C}$ is semiprojective.
%Now let $B\subset A$ be a full hereditary \SCA\, of $A.$
%We may write $B=\overline{bAb}.$
%Then, it is known and standard that, for any $\ep>0,$ there is $b_0\le b$ and
%$\|b-b_0\|<\ep$ such that $\overline{b_0Bb_0}=\overline{b_0Ab_0}\in {\cal C}.$
(4) follows exactly the {{same}} proof as 10.4 of \cite{GLN}.
%By the proof of  (2) of \ref{Pcom4C},  one may view that $A$ is a hereditary \SCA\, of $M_n(B).$
%Thus $B$ is also semiprojective.

\end{proof}

%\begin{prop}\label{P1-dmq}
%Let $A\in {\cal C}_\omega'.$
% (or $A\in {\cal C}_0'$ or $A\in {\cal C}_0^{0'}$).
%Suppose that $B=A/I$ for some proper ideal of  $I$ of $A$  such that
%$B=A/I$ is non-unital.
%Then, for any $\ep>0,$ and any finite subset ${\cal F}\subset B/I,$
%there exists  $B_0\subset A/I$ which is in ${\cal  C}'$
%(or in ${\cal C}_0'$ or in ${\cal C}_0^{0'}$)
%such that
%$$
%{\rm dist}(f, B_0)<\ep\rforal f\in {\cal F}.
%$$
%Moreover, if $B$ is projectionless, so is $B_0.$
%\end{prop}

%\begin{proof}
%First consider the case that $A\in {\cal C}_0$ (or in ${\cal C}_0'$).
%Note that we have assumed that $K_1(A)=\{0\}.$
%Let $\ep>0$ and a finite subset ${\cal F}\subset B$ be given.
%By \ref{Runitz} and by Lemma 3.20 of \cite{GLN}, there is \SCA\, $B_0'\subset {\widetilde A}/I$
%such that
%$$
%{\rm dist}(f, B_0')<\ep/2\rforal f\in {\cal F},
%$$
%where $B_0'\in {\cal C}$ and $K_1(B_0')=\{0\}.$
%Set $B_0=B_0'\cap A.$ If $B_0\cap A/I=\{0\},$ then $B_0=\C\cdot 1.$ This is not possible
%since we assume that $A$ is not unital and $I$ is proper.
%Since $A/I$ is an ideal of ${\widetilde A}/I,$ $B_0'={\tilde B}_0.$

%\end{proof}

\section{Maps from 1-dimensional non-commutative complexes}

\begin{lem}[Lemma 2.1 of \cite{BT}]\label{Lbt1}
Let $A$ be a simple  \CA\,
% with strict comparison of positive elements and
with $A={\mathrm{Ped}(A)}$ and $n\ge 1$ be an integer.
%Assume  that
%$\iota: W(A)_+\to {\rm LAff}_{b+}(\overline{{\mathrm{T}(A)}}^\mathrm{w})$ is  surjective.

Let $a\in \mathrm{M}_n({\widetilde A})_+\setminus \{0\}$ be such that $0$ is  a limit point of the spectrum of $a.$
%which is not Cuntz equivalent to a projection.
Then, for any $\ep>0,$  there exist $\dt>0$ and a continuous affine
function  $f:  {\rm T}_1(A)\to \R+$  with $f(0)=0$ such that
%\overline{{\mathrm{T}(A)}}^\mathrm{w}\to \R^+$ such that
$$
\mathrm{d}_\tau((a-\ep)_+)<f(\tau)<\mathrm{d}_\tau((a-\dt)_+)\tforal \tau\in \overline{{\mathrm{T}(A)}}^\mathrm{w}.
$$
\end{lem}

\begin{proof}
This  is essentially proved in the proof of 
Lemma 2.1 of \cite{BT}.
Note that $f$ is  {{just}} a function.  The proof of Lemma 2.1 of \cite{BT} does not involve
 comparison since no elements in $A$ need to be produced. It does require
that $\mathrm{d}_\tau(b)>0$ for any $b\in \mathrm{M}_n({\widetilde A})_+\setminus \{0\}$ and  for any
nonzero trace $\tau$ of $A.$ The rest of the proof is a compactness argument and an application
of the Portmanteau theorem. It should be noted that, as exactly in the proof of Lemma 2.1 of \cite{BT}, 
since the function $d_\tau((a-\dt)_+)$ is in ${\rm LAff}_{b,0+}({\overline{\rm{T}(A)}^{\rm w}}),$
the sequence $f_n$ can be chosen as $g_n|_{\overline{\rm{T}(A)}^{\rm w}},$ where each $g_n$ are in $\Aff_{0+}({\rm T}_1(A)).$
\iffalse
 assumes  that $A$ has stable rank one.
It is used in the first sentence of the proof, namely, one can assume that zero is a limit point
of $sp(a).$  If $0$ is an isolated point of $sp(a),$ then there is $0<\dt<1/2$ such that
$f_\dt(a)$ is a projection. It is easy to see that $\la a \ra =\la f_\dt(a) \ra $ in this case.
The rest of the proof works exactly the same as that of Lemma 2.1 of \cite{BT}
which does require $A$ has an identity.
\fi
\end{proof}

%\begin{df}\label{Dalst1}
%Let $A$ be a non-unital \CA. We say $A$ almost has stable rank one
%if for any integer $m\ge 1$ and any hereditary \SCA\, $B\subset A,$
%$B=\overline{GL({\tilde B})}.$
%\end{df}

%The following will not be used until later sections.

\begin{lem}\label{Lbt1ACT}
Let $A$ be a non-unital simple \CA\, with strict comparison for positive element
which almost has stable rank one.
Suppose that ${\rm{QT}}(A)={\rm{T}}(A),$ $A={\mathrm{Ped}(A)}$ and 
{{the canonical map}} $\imath: W(A)_+\to {\rm LAff}_{b,0+}({\overline{{\mathrm{T}(A)}}^\mathrm{w}})$
is surjective.

Let $0\le a\le 1$ be a non-zero element in $A$ which is not Cuntz equivalent to
a projection.
Then, for any $\ep>0$   there exist $\dt>0$ and  an element  $e\in A$
with
\beq\label{Lbt1act-1}
0\le f_\ep(a)\le e\le f_{\dt}(a)
\eneq
such that the function $\tau\mapsto \mathrm{d}_\tau(e)$ is continuous on $\overline{{\mathrm{T}(A)}}^\mathrm{w}.$
\end{lem}

\begin{proof}
{{Fix $\ep>0.$}}
By \ref{Lbt1}, there are  continuous affine  functions
$g_1, g_2\in \Aff_0({\rm T}_1(A))$ such that
%(\overline{{\mathrm{T}(A)}}^\mathrm{w})$ such that
\beq\label{Lbt1act-2}
\mathrm{d}_\tau(f_{\ep/8}(a))<g_1(\tau)<\mathrm{d}_\tau(f_{\dt_1}(a))<g_2(\tau)<\mathrm{d}_\tau(f_{\dt_2}(a)) \rforal \tau\in \overline{{\mathrm{T}(A)}}^\mathrm{w},
\eneq
where $0<\dt_2<\dt_1<1.$
Since $\iota$ is surjective, there is $c\in \mathrm{M}_m(A)$ for some integer $m\ge 1$
such that $0\le c\le 1$ and
$\mathrm{d}_\tau(c)=g_2(\tau)$ for all $\tau\in \overline{{\mathrm{T}(A)}}^\mathrm{w}.$
It follows from \ref{Lalmstr1}  and \eqref{Lbt1act-2} that there exists $x\in \mathrm{M}_m(A)$ such that
\vspace{-0.08in} $$
x^*x=c\andeqn xx^*\in \overline{f_{\dt_2}(a)Af_{\dt_2}(a)}.
$$
Put $c_0=xx^*.$ Then $0\le c_0\le 1.$
Note that
\beq\label{Lbt1act-2+}
\mathrm{d}_\tau(c_0)=\mathrm{d}_\tau(c)\rforal \tau\in \overline{{\mathrm{T}(A)}}^\mathrm{w}.
\eneq
Since $g_1$  
%and ${\red{g}}_2$ are 
is continuous, there is $m\ge 2$ such that
\beq\label{Lbt1act-2++}
%\mathrm{d}_\tau(g_1)
{{g_1(\tau)}}<\tau(f_{1/m}(c_0))\rforal \tau\in \overline{{\mathrm{T}(A)}}^\mathrm{w}.
\eneq
By \eqref{Lbt1act-2} and 
\ref{Lalmstr1} again,  there is a unitary $u$ in the unitization of $\overline{f_{\dt_2}(a)Af_{\dt_2}(a)}$
 such that
\vspace{-0.12in} \beq\label{Lbt1act-3}
u^*f_{\ep/8}(f_{\ep/8}(a))u\in \overline{f_{1/m}(c_0)Af_{1/m}(c_0)}.
\eneq
%Define 
{{Set}} $c_1=uc_0u^*.$
Then
\beq\label{Lbt1act-4}
f_{\ep/8}(f_{\ep/8}(a))\in \overline{f_{1/m}(c_1)Af_{1/m}(c_1)}\subset  
\overline{c_1Ac_1}.
%%{\red{\subset}} \overline{f_{\dt/2}(a)Af_{\dt/2}(a)}.
\eneq
There is a $g\in C_0((0,1])$ with $0\le g\le 1$ such that $g(t)\not=0$ for all $t\in (0,1],$
$g(t)f_{1/m}=f_{1/m}.$
Put $e=g(c_1).$  Then $\la e\ra=\la c_1\ra=\la c_0\ra =\la c\ra.$ 
Moreover,
\vspace{-0.1in} \beq\label{Lbt1act-5}
&&\mathrm{d}_\tau(e)=\mathrm{d}_\tau(c_1)=g_2(\tau) \rforal \tau\in \overline{{\mathrm{T}(A)}}^\mathrm{w}\andeqn\\
&&f_\ep(a)\le f_{\ep/8}(f_{\ep/8}(a))\le e\le f_{\dt_2/2}(a).
\eneq
Choose $\dt=\dt_2/2.$
\end{proof}

The following  {{theorem}} is a restatement of a result of Robert.

\begin{thm}[Theorem 6.2.3 of \cite{Rl}]\label{TRobert1}
Let $A$ be a stably projectionless simple \CA\, with strict comparison for
positive elements which has stable rank one and $\mathrm{QT}(A)=\mathrm{T}(A).$
Suppose that $A={\mathrm{Ped}(A)}$ and $\imath: {\mathrm{Cu}}(A)\to {\rm LAff}_{0+}({\overline{\rm{T}(A)}^{\rm w}})$ 
%${\rm LAff}_+|_{\overline{\rm{T}(A)}^{\rm w}}$
is an ordered semigroup isomorphism in {\bf Cu}. Then the map defined in (6.6) of \cite{Rl} is an  isomorphism.
of ordered semigroups.

Moreover, if $a, b\in ({\widetilde A}\otimes {\cal K})_+$ with
$\la \pi(a)\ra =k<\infty,$ $\la \pi(b)\ra =m<\infty,$
where $\pi: {\widetilde A}\to \C$ is the quotient map,  are such that
\beq
\mathrm{d}_\tau(a)+m< \mathrm{d}_\tau(b)+k\tforal \tau\in \overline{\mathrm{T}(A)}^\mathrm{w},
\eneq
then
\beq
\la a \ra +m\la 1_{\widetilde A}\ra \le \la b\ra +k \la 1_{\widetilde A}\ra.
\eneq
Furthermore, if either $\la a \ra,$ or $\la b\ra $ is not represented by a projection, and
\beq
\mathrm{d}_\tau(a)+m\le  \mathrm{d}_\tau(b)+k\tforal \tau\in \overline{\mathrm{T}(A)}^\mathrm{w},
\eneq
then  $\la a\ra \le \la b\ra.$
\end{thm}

\begin{proof}
The proof of 6.2.3 of \cite{Rl} applies since we assume that $A$ has stable rank one and the conclusion
of 6.2.1 of \cite{Rl} holds for $A\otimes {\cal K}.$
Denote the map defined in (6.6) of  \cite{Rl} by $\Gamma.$
For the
reader's convenience we include a detailed proof that the inverse {{of $\Gamma$}}  is
order preserving, since we will use this in an important
way.
 Let us first check that the inverse of $\Gamma$ restricted
 to the elements  $\mathrm{LAff}_+^\sim ({\widetilde{\mathrm{T}}}(A))$ is {{order preserving.}}
% $\Gamma(\la a\ra-k\la 1\ra)$  with $\la a\ra\not=\la p\ra$ for any projection $p,$ is also
%an order
%isomorphism.

We will use some notation from \cite{Rl} (but recall that {{our}} ${\widetilde{\mathrm{T}}}(A)$ is 
${\mathrm T}_0(A)$ in \cite{Rl}).
Let $a_1\in {\mathrm{Cu}}({\widetilde A}),$ $\la a_1\ra\not=\la p\ra$ for any projection and   $\la \pi(a_1)\ra=k.$
%and $[a_1\ra{\hat{}}-k$ is not positive.
Suppose also that $\la a_2\ra\in {\mathrm{Cu}}({\widetilde A})$ such that
% $\la a_2\ra-m\la 1_{\widetilde A}\ra\in Cu^\sim(A),$
$\la \pi(a_2)\ra =m,$ where $k$ and $m$ are integers,  and
 %$m\ge k$
\beq
{\widehat{\la a_1\ra}}-k{\widehat{\la 1_{\widetilde A} \ra}}\le \widehat{\la a_2\ra}-m{\widehat{\la 1_{\widetilde A}\ra}} .
\eneq
There are $\bt_1, \bt_2\in \Aff_+({\widetilde{\mathrm{T}}}(A))$ and
$\gamma_1, \gamma_2\in \mathrm{LAff}_+({\widetilde{\mathrm{T}}}(A))$ such that
\beq
\la a_1\ra +\bt_1&=&k\la 1_{\widetilde A}\ra  +\gamma_1\\
\la a_2\ra +\bt_2&=&m\la 1_{\widetilde A}\ra +\gamma_2\andeqn\\
\gamma_1-\bt_1&\le & \gamma_2-\bt_2.
\eneq
{{Note that we have used the notation in the proof of 6.2.3,  and in particular, we identify 
$\bt_1,\bt_2, \gamma_1, \gamma_2$ with elements of ${\rm Cu}(A).$}}
Thus,
\beq
\gamma_1+\bt_2 &\le &   \gamma_2+\bt_1\andeqn\\
\la a_1\ra +\bt_1+\bt_2+m\la 1_{\widetilde A}\ra &=&(k+m)\la 1_{\widetilde A}\ra  +\gamma_1+\bt_2\\
 && \le (k+m)\la 1_{\widetilde A}\ra + \gamma_2+\bt_1\\
 && =k\la 1_{\widetilde A}\ra +\la a_2\ra +\bt_2+\bt_1.
\eneq
Put $\bt=\bt_1+\bt_2.$
We have
\beq
\la a_1\ra +\bt+m\la 1_{\widetilde A}\ra \le k\la 1_{\widetilde A}\ra +\la a_2\ra +\bt.
\eneq

Exactly as proved in 6.2.3 of \cite{Rl},
one has
\beq
\la (a_1-\ep)_+\ra +\bt\ll \la a_1\ra +\bt
\eneq
which implies that, also,
\beq
 \la (a_1-\ep)_+\ra +\bt+m\la 1_{\widetilde A}\ra \ll \la a_1\ra +\bt +m\la 1_{\widetilde A}\ra.
\eneq
Therefore,
\beq
 \la (a_1-\ep)_+\ra +\bt+m\la 1_{\widetilde A}\ra \ll k\la 1_{\widetilde A}\ra +\la a_2\ra +\bt.
 \eneq
Since $A$ has stable rank one, by  weak cancellation (4.3 of \cite{RW}),
 \beq
\la (a_1-\ep)_+\ra +m \la 1_{\widetilde A}\ra \le \la a_2\ra+k\la 1_{\widetilde A}\ra.
\eneq
It follows that
\beq
\la a_1\ra +m \la 1_{\widetilde A}\ra \le \la a_2\ra+k\la 1_{\widetilde A}\ra.
\eneq
In particular, 
%Note that the above argument also implies that
this shows that  $\Gamma$ is injective.

\iffalse
%for some $\bt\in \Aff_+(T_0(A)).$
%as we identify $\bt$ as an element in $Cu(A).$
%$Cu(A)$ with $\mathrm{LAff}_+(\mathrm{T}_0(A)).$
If  ${\widehat{\la a_2\ra}}-m{\widehat{\la 1_{\widetilde A}}} \,\hat{}-(\la a_1\ra\,\hat{}-k\la 1_{\widetilde A}\ra)$ is
a strictly positive lower semicontinuous function,  then
$$
\la a_1\ra\, \hat{}+g'+m\la 1_{\widetilde A} \ra \,\hat{}=\la a_2\ra\, \hat{}+k \la 1_{\widetilde A}\ra\,\hat{}
$$
for some $g'\in \Aff_+(\mathrm{T_0}(A)).$
It follows that
$$
\la a_1\ra\, \hat{}+g'+m\la 1_{\widetilde A} \ra \,\hat{}=\la a_2\ra\, \hat{}+k \la 1_{\widetilde A}\ra\,\hat{}.
$$
By the injectivity of $\Gamma,$
\beq
\la a_1\ra+\la c'\ra +m\la 1_{\widetilde A}\ra=\la a_2\ra +k\la 1_{\widetilde A}\ra
%\le \la a_2\ra +m[1_{\widetilde A}],
\eneq
for some $c\in {\mathrm Cu}(A).$
It follows that
$$
\la a_1\ra\, +m\la 1_{\widetilde A} \ra \le \la a_2\ra +k\la 1_{\widetilde A}\ra.
$$
%Since ${\widetilde A}$ has stable rank one, by \cite{RW},
% $\la a_1\ra -k\la 1_{\widetilde A}\ra \le \la a_2\ra -m\la 1_{\widetilde A}\ra $  in $C^\sim (A).$
 %$ and  $\la a_1\ra\le \la a_2\ra$ in
%$Cu({\widetilde A}).$
\fi

Note that, above, we do not assume that $\la a_2\ra $ is not  represented by a projection. Therefore it remains to show the
following:

If $(\la a\ra -k\la 1_{\widetilde A}\ra)\,\widehat{}<(\la b\ra -m\la 1_{\widetilde A}\ra)\,\widehat{}$\, on $\overline{\mathrm{T}(A)}^{\mathrm w},$
then
$$\la a\ra -k\la 1_{\widetilde A}\ra\le \la b\ra -m\la 1_{\widetilde A}\ra$$
 for all
$\la a\ra -k\la 1_{\widetilde A}\ra, \la b\ra -m\la 1_{\widetilde A}\ra\in {\mathrm Cu}^\sim(A).$

We only need to consider the case that
$\la a\ra$  is represented by a projection.  Then ${\widehat{\la a\ra}}$ is continuous.
It follows that there are non-zero elements  $\bt_0, \bt \in \mathrm{Aff}_+({\widetilde{{\mathrm{T}}}}(A))$ such that
\beq
(\la a\ra +\bt +\bt_0+ m\la 1_{\widetilde A}\ra)\, \widehat{}\,< (\la b\ra +\bt +k\la 1_{\widetilde A}\ra)\,\widehat{}\,\,\,\, {\rm on}\,\,\,\overline{\mathrm{T}(A)}^{\mathrm w}.
\eneq
Since $A$ is stably projectionless, from what has been proved,
\beq
\la a\ra +\bt +\bt_0+ m\la 1_{\widetilde A}\ra \le  \la b\ra +\bt +k\la 1_{\widetilde A}\ra.
\eneq
Then,
since $\la a\ra $ is represented by a projection,
\beq
\la a\ra +\bt +(1/2)\bt_0+ m\la 1_{\widetilde A}\ra \ll \la a\ra +\bt +\bt_0+ m\la 1_{\widetilde A} \ra\le
\la b\ra +\bt +k\la 1_{\widetilde A}\ra.
\eneq
By the weak cancellation,
\beq
\la a\ra +(1/2)\bt_0+ m\la 1_{\widetilde A}\ra \le  \la b\ra  +k\la 1_{\widetilde A}\ra.
\eneq
It follows that
\beq
\la a\ra + m\la 1_{\widetilde A}\ra \le  \la b\ra  +k\la 1_{\widetilde A}\ra.
\eneq

\iffalse

Otherwise, as proved above,
\beq
\la (a-\ep)_+\ra +\bt  +m\la 1_{\widetilde A}\ra\ll \la a\ra +\bt +m\la 1_{\widetilde A}\ra \le \la b\ra +\bt +k\la 1_{\widetilde A}\ra.
\eneq
By the weak cancellation again,
\beq
\la (a-\ep)_+ \ra +m\la 1_{\widetilde A}\ra \le \la b\ra  +k\la 1_{\widetilde A}\ra.
\eneq
Thus
\beq
\la a \ra +m\la 1_{\widetilde A}\ra \le \la b\ra  +k\la 1_{\widetilde A}\ra.
\eneq
\fi
\iffalse
Lastly, we will show that

 If ${\widehat{\la a\ra}-k\la 1_{\widetilde A}\ra \le {\widehat{\la b\ra}-m\la 1_{\widetilde A}\ra\,\hat{},$ where $\la b\ra $ is represented by
  projection $p$ and $\la a \ra$ is not represented
 by a projection, then
 \beq
 \la a\ra +m \la 1_{\widetilde A}\ra \le \la p\ra +k\la 1_{\widetilde A}\ra.
 \eneq
Since $A$ is a stably projectionless simple \CA, for any $\ep>0,$
\beq
\la (a-\ep)_+\ra\, \hat{} +m \la 1_{\widetilde A}\ra \,\hat{}\le \la p\ra\,\hat{} +k\la 1_{\widetilde A}\ra\, \hat{}.
\eneq
It follows from above that
\beq
\la (a-\ep)_+\ra+ m \la 1_{\widetilde A}\ra \le \la p\ra +k\la 1_{\widetilde A}\ra.
\eneq
This implies that
\beq
 \la a\ra +m \la 1_{\widetilde A}\ra \le \la p\ra +k\la 1_{\widetilde A}\ra.
\eneq
\fi
\end{proof}

%The proof above also shows (using $\la c_{n_0}\ra$) the following:

\begin{lem}\label{Lbt1ACT-1}
%Let $A$ be a projectionless simple  exact  simple \CA\, with strict comparison for positive element
%which  has stable rank one.
%Suppose that $A={\mathrm{Ped}(A)}$ and $\imath: W(A)_+\to {\rm LAff}_{b+}(\overline{{\mathrm{T}(A)}}^\mathrm{w})$
%is surjective.
With {{the}} same assumption on $A$ as in \ref{TRobert1}, we have the following statement:
Let $0\le a\le 1$ be a non-zero element of  ${\widetilde A}\otimes {\cal K}$
with $\pi(a)$ a projection of rank $m$ for some integer $m\ge 1$  and $a$   not Cuntz equivalent to
a projection.
%such that $\mathrm{d}_\tau(a)-m$ is not positive in $\overline{\mathrm{T(A)}}^\mathrm{w}.$
Then, for any $1/2>\ep>0$   there exist $1>\dt>0$ and  an element  $e\in {\widetilde A}\otimes {\cal K}$
with
\beq\label{Lbt1act-1}
0\le f_\ep(a)\le e\le f_{\dt/2}(a)
\eneq
such that the function $\tau\mapsto \mathrm{d}_\tau(e)$ is continuous on $\overline{{\mathrm{T}(A)}}^\mathrm{w}.$
\end{lem}

\begin{proof}
Note that this statement is similar to that of \ref{Lbt1ACT} (the case $m=0$).
{{By the last statement of \ref{Lalmstr1}, we may assume that, for any $\ep>0,$ there exists 
$n\ge 1,$  $f_\ep(a)\in M_n({\widetilde{A}}).$ }}
By \ref{Lbt1}, there exists $f\in \Aff_{0+}({\rm T}_1(A))$
%({\overline{{\mathrm{T}}(A)}}^{\mathrm{w}})$ 
such that, for some $\dt>0.$
$$
\mathrm{d}_\tau(f_\ep(a))<f(\tau)<\mathrm{d}_\tau(f_\dt(a))\rforal \tau\in \overline{{\mathrm{T}(A)}}^\mathrm{w}.
$$
Since we assume that $\pi(a)$ is a projection, $\pi(f_\ep(a))=\pi(f_\dt(a))=\pi(a),$ where $\pi: {\widetilde A}\to \C$ is the quotient map.
The surjectivity of $\Gamma$ in \ref{TRobert1} implies that
there is $c\in ({\widetilde A}\otimes {\cal K})_+$ such that $\pi(c )=\pi(a),$  
\beq
d_\tau(c)=f(\tau) \andeqn \la f_{\ep/8}(a)\ra\ll \la c\ra\ll \la f_{\dt}(a)\ra.
\eneq
It remains to show that we can find $e$ with $\la e\ra =\la c\ra$ but also satisfies \eqref{Lbt1act-1}.
For this we will use the same argument used in the proof of \ref{Lbt1ACT}.
Since ${\widetilde A}$ has stable rank one, the  {{proof may be completed as in}}
%argument of 
\ref{Lbt1ACT}.
% applies.

\end{proof}

We shall need  the following two lemmas.

\begin{lem}[Lemma 2.2 of \cite{BT}]\label{Lbt2}
Let $A$ and  $a\in ({\widetilde A}\otimes {\cal K})_+$  be as in \ref{Lbt1ACT-1}.  Then there exists
a sequence $(a_n)_{n=1}^\infty$ of elements  in $({\widetilde A}\otimes {\cal K})_+$  which satisfies the following:

{\rm (1)} $\la a\ra =\sup_n \la a_n \ra ;$

{\rm (2)} $a_n\in {\mathrm{M}}_{n(k)}({\widetilde A})$  for some $n(k)\in \N$ and $\la \pi(a_n)\ra =\la \pi(a)\ra,$
where $\pi: {\widetilde A}\to \C$ is the quotient map;

{\rm (3)} the function $\tau\mapsto \mathrm{d}_\tau(a_n)$ is continuous  on $\overline{{\mathrm{T}(A)}}^\mathrm{w}$ for each $n\in \N;$ and

{\rm (4)}  $\mathrm{d}_\tau(a_n)<\mathrm{d}_\tau(a_{n+1})$ for all $\tau\in \overline{{\mathrm{T}(A)}}^\mathrm{w}$ and $n\in \N.$

\end{lem}

\begin{lem}\label{LLbt}
Let $A$ be as in \ref{TRobert1}. Suppose that $a, b\in \mathrm{Ped}({\widetilde A}\otimes {\cal K})_+$
(with $0\le a \le 1$ and $0\le b\le 1$)  such that neither are Cuntz equivalent
to a projection.  Suppose that ${{\la a\ra}}\ll  {{\la b\ra }}.$
Then there exist $\dt>0$ and $c \in \mathrm{Ped}({\widetilde A}\otimes {\cal K})_+$ with $0\le c\le 1$ such that
\beq\label{LLbt-0}
\la a \ra \le \la f_{\dt}(c)\ra, \,\,\, f_{\dt/2}(c)\le f_{\dt/4}(b)\tand
\inf \{\tau(f_{\dt}(c))-\mathrm{d}_\tau(a): \tau\in \overline{{\mathrm{T}(A)}}^\mathrm{w}\}>0.
\eneq
\end{lem}

\begin{proof}
By \ref{Lbt2}, choose $b_n\in ({\widetilde A}\otimes {\cal K})_+$ such that $(b_n)$ satisfies {\rm (1)},
{\rm (2)},  {\rm (3)},  and {\rm (4)}
in \ref{Lbt2}.  Since ${{\la}} a{{\ra}}\ll {{\la}} b{{\ra}},$ there is $n_0\ge 1$ such that
$\la a\ra \le {{\la}} b_n{{\ra}}$ for all $n\ge n_0.$
Therefore we have
\beq\label{LLbt-1}
\mathrm{d}_\tau(a)\le \mathrm{d}_{\tau}(b_{n_0})<\mathrm{d}_\tau(b_{n_0+1})<\mathrm{d}_\tau(b_{n_0+2})<\mathrm{d}_\tau(b_{n_0+3})\le \mathrm{d}_\tau(b).
\eneq
Note that
$$
\tau(f_{1/n}(b))\nearrow \mathrm{d}_\tau(b)\andeqn \tau(f_{1/n}(b_{n_0+1}))\nearrow \mathrm{d}_\tau(b_{n_0+1}) .
$$
It follows, for example,  from 5.4 of \cite{Lnloc} that there exists $n_1\ge 1$ such that,
for all $n\ge n_1,$
$$
\tau(f_{1/n}(b))>\mathrm{d}_\tau(b_{n_0+2})\andeqn
\tau(f_{1/n}(b_{n_0+1}))>\mathrm{d}_{\tau}(b_{n_0}))\rforal \tau\in \overline{{\mathrm{T}(A)}}^\mathrm{w}.
$$
Note that $\pi( f_{1/2n}(b))=\pi(b)$ and
$\la \pi(b_n)\ra =\la \pi(b)\ra.$ By \ref{TRobert1},
%the strict comparison of positive element,
we conclude that
$$
\la f_{1/2n}(b)\ra \ge \la b_{n_0+2} \ra \andeqn  \la f_{1/2n}(b_{n_0+1})\ra \ge \la b_{n_0} \ra .
$$
Put $c=b_{0+1}.$
Since $A$ has stable rank one, one may assume  that
$f_{1/2n}(c)\le f_{1/4n}(b).$
Thus we may choose $0<\dt<1/2n_1.$

Since $\overline{{\mathrm{T}(A)}}^\mathrm{w}$ is compact and both
functions in the above inequality are continuous, together with \eqref{LLbt-1}, we obtain
\vspace{-0.12in} $$
\inf\{\tau(f_\dt(b))-\mathrm{d}_\tau(a): \tau\in \overline{{\mathrm{T}(A)}}^\mathrm{w}\}>0.
$$

\end{proof}

In what follows, ${\cal C}^{(1)}$ is the collection of
all  \CA s which are inductive limits of full hereditary \SCA s of  1-dimensional  non-commutative CW complexes with
trivial ${\rm K}_1$ groups whose connecting maps are injective.

%\iffalse
\begin{df}\label{DPropertyR}
 
Fix a \CA\, $C\in {\cal C}^{(1)}.$ A \CA\, $A$ is said to have the property (R) associated with 
$C,$ if the following condition holds:
For any finite subset ${\cal F}\subset C$ and $\ep>0$ there exists a finite subset 
$G\subset {\rm Cu}^\sim (C)$ such that for any two \hm s $\phi, \psi: C\to A,$  if 
\beq
{\rm Cu}^\sim(\phi)(g')\le {\rm Cu}^{\sim}(\psi)(g)\tand 
{\rm Cu}^\sim(\psi)(g')\le  {\rm Cu}^\sim(\phi)(g)
\eneq
for all $g',g\in G$ with $g' \ll g,$ 
then  there exists a  unitary $u\in {\widetilde A}$  such that
\beq
\|u^*\phi(f)u-\psi(f)\|<\ep\tforal f\in {\cal F}.
\eneq
This definition is taken from  3.3.1 of \cite{Rl} and we adapt  the notation from there. 
Note, by 3.3.1 of \cite{Rl}, every \CA\, with stable rank one  has the property (R)  associated with $C.$
\end{df}

\begin{thm}{\rm (See Theorem 3.3.1 of \cite{Rl}, Theorem 5.2.7 of \cite{Lncbms}, and Theorem 8.4 of \cite{GLN})}\label{Lrluniq}
Let $C$ be in ${\cal C}^{(1)}$ and assume
%a full hereditary \SCA\,  of
%an inductive limit of
%1-dimensional non-commutative complices with $K_1(C)=\{0\}.$
%or inductive limit of these \CA s.
%Assume
that $\mathrm{Ped}(C)=C$
%$P(C)=C$ and
%$0\not\in \overline{T(C)}^\mathrm{w}$
%a 1-dimensional non-commutative complex with $K_1(C)=\{0\},$
and  let
$\Delta: C^{q,{\bf 1}}\setminus \{0\}\to (0,1)$ be an order preserving map.
%and
%$A$ be as in \ref{Lbt1}.
Then, for any $\ep>0$ and any finite subset
${\cal F}\subset C,$ there exist a finite subset ${\cal G}\subset C,$ a finite subset ${\cal P}\subset {\rm K}_0(C),$
a finite subset ${\cal H}_1\subset C_+^{\bf 1}\setminus \{0\},$ a finite subset ${\cal H}_2\subset C_{\rm{s.a.}},$
$\dt>0,$ and $\gamma>0$  satisfying the following condition:
for any two ${\cal G}$-$\dt$-multiplicative \morp s $\phi_1, \phi_2: C\to A$ for some $A$
which is $\sigma$-unital, simple,  stably projectionless, 
%has the strict comparison for positive elements, 
has stable rank one,
 %is quasi-compact, and which satisfies the property that
%every hereditary \SCA\, of $A\otimes {\cal K}$
$\mathrm{QT}(A)=\mathrm{T}(A),$ and the property that the map $\mathrm{Cu}_+(A)\to {\rm LAff}_{0+}({\overline{\rm{T}(A)}^{\rm w}})$ is an ordered semigroup isomorphism,
and $\mathrm{Ped}(A)=A$ 
%as well as $A$ has the property (R) associated with $C,$   
such that
%as in \ref{Lbt1} such that
\beq\label{Lrluniq-1}
&&[\phi_1]|_{\cal P}=[\phi_2]|_{\cal P},\\
&&\tau(\phi_i)(a)\ge \Delta(\hat{a})\tforal a\in {\cal H}_1\tand \hspace{-0.1in}\tforal \tau\in \overline{{\mathrm{T}(A)}}^\mathrm{w},\tand\\
&&|\tau(\phi_1(b))-\tau(\phi_2(b))|<\gamma\tforal b\in {\cal H}_2\tand \tforal \tau\in \overline{{\mathrm{T}(A)}}^\mathrm{w},
\eneq
there exists a unitary $u\in {\widetilde A}$ such that
$$
\|u^*\phi_2(f)u-\phi_1(f)\|<\ep\tforal f\in {\cal F}.
$$

\end{thm}

\begin{proof}
%In the terminology of \cite{Rl},  as proved in
%\cite{Rl}, ${\mathrm{Cu}}^{\sim}$ classifies \hm s from $C.$
%Upon examining the proofs of \cite{Rl}, one sees that Lemma 3.3.1
%and Theorem 1.0.1 of \cite{Rl} hold for $B$ almost having stable rank one
%instead of having stable rank one as in \ref{Lalmstr1}.
We will use 3.3.1 of \cite{Rl}.

%\ref{LRobert}.
%{Lamostr1}.
Let  $\ep>0$ be  given. There exists $e_0$ with
 $\ep/16>\ep_0>0$  satisfying the following condition: 
 In any \CA,  if $0\le x\le 1$ is an element in the \CA\,  and 
$\|xg-g\|<\ep_0$ and
$\|gx-g\|<\ep_0$ for any $\|g\|\le 1$ in the \CA,   then
\beq
\|x^{1/2}g-gx^{1/2}\|<\ep/64.
\eneq

Let us first assume that $C$ is a single  full hereditary \SCA\, of a 1-dimensional non-commutative CW complex.
Fix $\ep_0>0$  as above and ${\cal F}\subset C.$
Let $G\subset {\mathrm{Cu}}^{\sim}(C)$ be as required by
% Theorem 3.3.1 of \cite{Rl} 
Property (R) associated with $C$  for
$\ep_0/16$ (in place of $\ep$) and ${\cal F}.$  \Wlog, we may assume
that ${\cal F}$ is contained in the unit ball of $A.$

Recalling that $C$ has stable rank one, as  shown in \cite{Rl}, we may assume
that $G$ consists of a finite subset ${\cal P}\subset {\rm K}_0(C)$
and a finite subset $\{[a_1]-k_1[1_{\widetilde A}], [a_2]-k_2[1_{\widetilde C}],...,k_m[1_{\widetilde C}]\}$ {{of the}}
Cuntz semigroup of ${\rm{Cu}}^\sim (C)$
such that $[a_i]$  can be represented by positive elements
$0\le a_i\le 1$ in ${\widetilde C}\otimes {\cal K}$ which are  not Cuntz equivalent to a projection, and
$k_i$ are non-negative integers,
$i=1,2,...,m.$
%(see 3.1.2 of \cite{Rl}).
Write  ${\cal P}=\{z_1-k'_1[1_{\widetilde C}], z_2-k'_2[1_{\widetilde C}],..., z_{m_0}-k_{m_0}'[1_{\widetilde C}]\},$
where  {{the}} elements $z_i$ are represented by projections in ${\widetilde C}\otimes {\cal K}.$
Note here we assume that $\la \pi(a_i)\ra =k_i\la 1\ra $ and $[\pi_{*}(z_l)]=k_l'[1],$
where $\pi: {\widetilde C}\to \C$ is the quotient map, $i=1,2,...,m$ and $j=1,2,...,m_0.$
%Note if $a_i\le z_l$ (for some $1\le l\le m_0$), then $z_l\ge 0$ and
%is therefore represented by a projection.

Suppose that
%Suppose that
$\la a_i\ra  +k_j[1_{\widetilde C}]\ll \la a_j \ra +k_i[1_{\widetilde C}].$
For each of these pairs $i, j,$ put $a_{ij}=a_i\oplus 1_{{\mathrm{M}}_{k_j}}$ and $a_{ji}=a_j\oplus 1_{{\mathrm{M}}_{k_i}}.$
Then,  since ${\widetilde C}\otimes {\cal K}$ also has stable rank one,
%by \ref{LLbt},
there {{are}}  a number  $1/4 > \eta(i,j)>0$   and  an element
$0\le c_{i,j}\le 1$ in $({\widetilde C}\otimes {\cal K})_+$ such that
\beq\label{Lrluniq-3}
\la a_{ij} \ra \le  \la  f_{\eta_{i,j}}(c_{i,j})\ra\andeqn  f_{\eta_{i,j}/2}(c_{i,j})\le f_{\eta_{i,j}/4}(a_{ji}).
\eneq
%Similarly, if $\la a_i\ra \le z_l,$
%there is $1/4> \eta_{i,l}_0>0$  and $c_{i,l}_0\in A_+$ with
%$0\le c_{i,l}\le 1$ such
%that
%$$
%\la a_i\ra \le f_{\eta_{i,l}}(c_{i,l}_0)\andeqn \la _{\eta_{i,l}/2}(c_{i,l}_0)\le f_{\eta_{i,l}_0/2
%$$
Since $a_{ij}$ 
%($i=1,2,...,m$) 
is not Cuntz equivalent to a projection,
%(and assume that $x_i\not=x_j$ if $i\not=j$),
we may choose $\eta(i,j)$ so that
$$
f_{\eta_{i,j}/4}(a_{ji})-f_{\eta_{i,j}/2}(c_{i,j})\not=0.
$$
Choose  a finite subset ${\cal H}_1\subset C_+$
which contains  non-zero positive elements $b_{i,j}$ such that
$$
b_{i,j}\lesssim f_{\eta_{i,j}/4}(a_{ji})-f_{\eta_{i,j}/2}(c_{i,j})
$$
for all possible pairs of $i$ {{and}} $j$ such that $\la a_{ij}\ra \ll  \la a_{ji} \ra.$
%Moreover, if $a_i\ll  z_l,$

Let
\vspace{-0.12in} \beq\label{Lrluniq-4}
\dt_0=\inf \{ \Delta(\hat{g}): g\in {\cal H}_1\}.
\eneq
%Choose $\gamma=\dt_0/16.$

Choose  a finite subset ${\cal H}_2'$ of  ${{(}}{\widetilde C}\otimes {\cal K})_+$ which contains
$f_{\eta_{i,j}}(c_{i,j}), f_{\eta_{i,j}/2}(c_{i,j}),
f_{\eta_{i,j}/4}(a_{ji}) $ for all possible $i,j$ as described above.

Let the finite subset ${\cal H}_2\subset C_{s.a.}$ containing ${\cal H}_1$ and
$\dt_1>0$ be such that
%property:
\beq\label{Lrluniq-4+}
|\tau(h_1^\sim(g))-\tau(h_2^\sim(g))|<\dt_0/16\tforal g\in {\cal H}_1\cup {\cal H}_2'
\eneq
and for all $\tau\in \overline{\mathrm{T}(B)}^\mathrm{w},$
whenever   $h_1, h_2: C\to B$ are  \hm s {{with}} $B$ any \CA\,
with $\mathrm{T}(B)\not=\O$ and $0\not\in \overline{\mathrm{T}(B)}^\mathrm{w}$ such that
\beq\label{Lrluniq-4++}
|\tau\circ h_1(f)-\tau\circ h_2(f)|<\dt_1\rforal f\in {\cal H}_2\andeqn \tau\in \overline{\mathrm{T}(B)}^\mathrm{w},
\eneq
where $h_i^\sim: {\widetilde C}\to {\widetilde B}$ is the unital extension of $h_i,$ $i=1,2.$

Put $\gamma=\min\{\dt_0/16, \dt_1/4\}.$
%$a_i, f_{\eta_{i,j}}(c_{i,j}), f_{\eta_{i,j}/2}(c_{i,j}),
%f_{\eta_{i,j}/4}(a_i) $ for all possible $i,j.$
Since $1$-dimensional NCCW complexes are
%$C$ is
%weakly 
semiprojective (\cite{ELP1}),
by choosing a large ${\cal G}$ and small $\dt,$ by applying \ref{103L}, we may assume
that  there are \hm s $\psi_i: C\to A\otimes {\cal K}$ such that
\beq\label{Lrluniq-5}
(\psi_i)_{*0}|_{\cal P}=[\phi_i]|_{\cal P}\andeqn \|\psi_i(g)-\phi_i(g)\|<\min\{\ep_0/16, \gamma\},\,\,\,i=1,2,
\eneq
for all $g\in {\cal F}\cup{\cal H}_1\cup {\cal H}_2,$
where   $\phi_1$ and $\phi_2$ are ${\cal G}$-$\dt$-multiplicative \cpc s from $C$ to a \CA\, $A$
satisfying  the assumptions of the theorem.

Since $\mathrm{Ped}(C)=C,$ $\psi_i(C)\subset \mathrm{Ped}(A\otimes {\cal K}).$
%By replacing $\psi_i$ by its unitar
%are as in \ref{Lbt1}.

Assume that $\phi_1, \phi_2: C\to A$ have  the described properties for the above
defined ${\cal G},$ $\dt,$  ${\cal P},$ ${\cal H}_1, $ ${\cal H}_2,$ $\gamma.$

Let $\psi_i: C\to A$ be as provided in \eqref{Lrluniq-5},\,\,\, $i=1,2.$
Then
\beq\label{Lrluniq-6}
&&(\psi_1)_{*0}|_{\cal P}=(\psi_2)_{*0}|_{\cal P},\\
&&\tau\circ \psi_i(g)\ge \dt_0/2\rforal g\in {\cal H}_1, \andeqn\\
&&|\tau\circ \psi_1(b)-\tau\circ \psi_2(b)|<\dt_0/2\rforal b\in {\cal H}_2
\eneq
for all $\tau\in \overline{{\mathrm{T}(A)}}^\mathrm{w}.$
In particular,  if $\la a_{ij}\ra \ll \la a_{ji}\ra ,$  then, by the choice of ${\cal H}_1,$ ${\cal H}_2$, and $\gamma$ above,
\beq\label{Lrluniq-7}
\mathrm{d}_\tau(\psi_1(a_{ij})) &\le& \tau(\psi_1(f_{\eta_{i,j}}(c_{i,j})))<\dt_0/2+\tau(\psi_2(f_{\eta_{i,j}}(c_{i,j})))\\
&\le  &  (\tau(\psi_2(f_{\eta_{i,j}/4}(a_j))-\tau(\psi_2(f_{\eta_{i,j}}(c_{i,j}))))+\tau(\psi_2(f_{\eta_{i,j}}(c_{i,j})))\\
&\le & \mathrm{d}_{\tau}(a_{ji})
\eneq
for all $\tau \in \overline{{\mathrm{T}(A)}}^\mathrm{w}.$
Therefore, if $\la a_{ij}\ra \ll \la a_{ji}\ra ,$
\beq\label{Lruniq-8}
\la \psi_1(a_{ij})\ra \le \la \psi_2(a_{ji}) \ra.
\eneq
Note also, if $\la a_i\ra +k_l'\la 1_{\widetilde C}\ra  \ll z_l+k_i'\la 1_{\widetilde C}\ra ,$ then
$\la \psi_1(a_i)\ra +k_l[1_{\widetilde A}]\ll {\mathrm Cu}^{\sim}(\psi_1)(z_l+k_i'\la 1_{\widetilde C}\ra )={\mathrm{Cu}}^{\sim}(\psi_2)(z_l+k_i'\la 1_{\widetilde C}\ra).$
Combining these with \eqref{Lrluniq-6}, we conclude that, using the terminology of \cite{Rl},
\beq\label{Lruniq-9}
\mathrm{Cu}^{\sim}(\psi_1(g))\le \mathrm{Cu}^{\sim}(\psi_2(g')\andeqn \mathrm{Cu}^{\sim}(\psi_2(g))\le \mathrm{Cu}^{\sim}(\psi_1(g'))
\eneq
for all $g,\, g'\in G$ and $g\ll g'.$  Since $A$ has   the property (R) associated with $C,$ by 
the choice of $G,$ 
%3.3.1 of \cite{Rl},
%\ref{LRobert},
there exists a unitary
$v\in (A\otimes {\cal K})^\sim$ such that
$$
\|v \psi_2(f)v^*-\psi_1(f)\|<\ep_0/16\rforal f\in {\cal F}.
$$
From this and \eqref{Lrluniq-5}, we obtain that
$$
\|v\phi_2(f)v^*-\phi_1(f)\|<\ep_0/16+\ep_0/16\rforal f\in {\cal F}.
$$
Choose $0\le e_1, e_2\le 1$ in $A$ such that
\beq
\|\phi_i(f)e_i-\phi_i(f)\|<\ep_0/32\andeqn \|e_i\phi_i(f)-\phi_i(f)\|<\ep_0/32\rforal f\in {\cal F}.
\eneq
%Now, when $\ep>0$ is   given, there exists $e_0$ with
% $\ep/16>\ep_0>0$  satisfying the following condition: If $0\le x\le 1,$ and 
%$\|xg-g\|<\ep_0$ and
%$\|gx-g\|<\ep_0$ for any $\|g\|\le 1,$  then
%\beq
%\|x^{1/2}g-gx^{1/2}\|<\ep/64\andeqn
%\eneq
%Choose this $\ep_0$ at the beginning.

Put $y=e_1v^*e_2$ and $x=y^*y=e_2ve_1e_1v^*e_2.$
Then
\beq
&&\|\phi_2(f)x-\phi_2(f)\|=\|v^*v(\phi_2(f)x-\phi_2(f))\|\\
&&<\|v^*v\phi_2(f)y^*v-v^*v\phi_2(f)v^*v^*\|+\ep_0/32\\
&&<\ep_0/32+\|v^*(v\phi_2(f)v^*)e_2-v^*v\phi_2(f)v^*v\|+\ep_0/32<\ep_0/2
\eneq
for all $f\in {\cal F}.$ Similarly,
\beq
\|x\phi_2(f)-\phi_2(f)\|<\ep_0/2\rforal f\in {\cal F}.
\eneq
Consider  $y=W|y|=Wx^{1/2}$  the polar decomposition of $y$ in $A^{**}.$
Since $A$ almost has stable rank one, by Theorem 5 of \cite{Pedjot87}, there exists a unitary
$u\in {\widetilde A}$ such that $uf_{\ep/16}(x^{1/2})=Wf_{\ep/16}(x^{1/2}).$
By the choice of $\ep_0,$ we have
\beq
\|u^*\phi_2(f)u-\phi_1(f)\|<\ep\rforal f\in {\cal F}.
\eneq
(Note that if $C$ is a 1-dimensional non-commutative CW complex, then $\phi_i$ can be chosen
to map $C$ into $A$ so that $v$ can be chosen in ${\widetilde A}.$)

For the general case, given a finite subset ${\cal F}\subset C,$ we may assume
that ${\cal F}\subset C_n$ for some $C_n$ which is a full hereditary \SCA\, of a 1-dimensional non-commutative
CW complex.  Then the above argument applies.
\end{proof}

\iffalse
\begin{cor}\label{Cruniq}
The exactly the same statement holds for $C$ being replaced by
full hereditary \SCA s of inductive limits of \CA s in ${\cal C}.$
\end{cor}
\fi

\begin{prop}\label{CDdiag}
Let $C\in {\cal C}_0'$ and let $A$ be a {{$\sigma$-unital stably projectionless}} simple exact \CA\,  
%with strict comparison for positive elements,
with ${\rm K}_0(A)=\{0\},$   with stable rank one and with continuous scale.
Suppose that ${\mathrm{Cu}}(A)={\rm LAff}_{0+}({\overline{{\mathrm{T}(A)}}^{\rm w}}).$
%Suppose also that 
{{Let}} $\phi: C\to A$ {{be}}  a \hm. Then, for any $\ep>0,$  any finite subset
${\cal F}\subset C,$  and any integer $n\ge 1,$ there is another \hm\, $\phi_0: C\to B=B\otimes e_{11}\subset  {\mathrm{M}}_n(B)\subset A,$
where $B$ is a hereditary \SCA\, of $A,$
such that
\beq
\|\phi(x)-\diag(\overbrace{\phi_0(x), \phi_0(x),...,\phi_0(x)}^n)\|<\ep\tforal x\in {\cal F}.
\eneq
\end{prop}

\begin{proof}
Fix a strictly positive element $e\in A_+$ with $\|e\|=1.$
{{We may assume that $A$ is an infinite dimensional.}}
There are  mutually orthogonal elements nonzero $e_1, e_2,...,e_n\in A_+$
such that $\la e_i\ra =\la e_1\ra$ in $\mathrm{Cu}(A)$ and
$\la \sum_{i=1}^n e_i \ra =\la e\ra$ {{(see also the proof of \ref{Pconscale}).}}
Let $B=\overline{e_1Ae_1}\subset A.$ 
Then, with $D:=\overline{(\sum_{i=1}^ne_i)A(\sum_{i=1}^n e_i{{)}}},$ we have 
$D\cong {\rm M}_n(B)\subset A.$
%There exist a \SCA\, $D\subset A$
%such  that $\mathrm{M}_n(B)\cong D.$ We write $D=\mathrm{M}_n(B).$ 
Note that $\mathrm{K}_0(A)=\{0\}.$
So $\mathrm{Cu}^\sim (A)=\mathrm{LAff}_+^\sim ({\widetilde{\mathrm{T}}}(A))$ (see \ref{TRobert1} above 
and 6.2.3 of \cite{Rl}).
Let $j: \mathrm{LAff}_+^\sim ({\widetilde{\mathrm{T}}}(A))\to \mathrm{LAff}_+^\sim ({\widetilde{\mathrm{T}}}(A))$
be defined by $j(f)=(1/n)f.$ Define
$\lambda: \mathrm{Cu}^\sim (C)\to \mathrm{Cu}^\sim (B)$  by
$\lambda=j\circ (\mathrm{Cu}^\sim(\phi).$
{{By Theorem 1.0.1 of}} \cite{Rl},  there exists a \hm\, $\phi_0': C\to B$ such that
$\mathrm{Cu}^{\sim}(\phi_0')=\lambda.$
Define $\psi: C\to \mathrm{M}_n(B)$ by
$\psi(a)=\diag(\overbrace{\phi_0'(a), \phi_0'(a),...,\phi_0'(a)}^n)$ for all $a\in C.$
Then $\mathrm{Cu}^{\sim}(\psi)=\mathrm{Cu}^{\sim}(\phi).$
It follows from {{Theorem 1.0.1 of}} \cite{Rl} that $\phi$ and $\psi$ are approximately unitarily equivalent, as desired.
%Lemma then follows.
\end{proof}

%%%%

%%%%%

%%%%%%%%

\section{Tracially one-dimensional complexes}

\begin{df}\label{DNtr1div}
Let $A$ be a simple \CA\,  with a strictly positive element $a\in A$
with $\|a\|=1.$  Suppose that there exists
$1> \mathfrak{f}_a>0,$ for any $\ep>0,$  any
finite subset ${\cal F}\subset A$ and any $b\in A_+\setminus \{0\},$  there are ${\cal F}$-$\ep$-multiplicative \cpc s $\phi: A\to A$ and  $\psi: A\to D,$  with $\phi(A)\perp D,$ i.e., $\phi(A)D=\{0\},$ for some
\SCA\, $D\subset A,$ such that
\beq\label{DNtr1div-1}
&&\|x-(\phi{{+}} \psi)(x)\|<\ep\rforal x\in {\cal F}\cup \{a\},\\\label{DNtrdiv-2}
&& D\in {{\cal C}_0^{0}}' ({\rm or}\,\, D\in {\cal C}_0'),\\\label{DNtrdiv-3}
&&\phi(a)\lesssim b,{{\tand}}\\\label{DNtrdiv-4}
&&t(f_{1/4}(\psi(a)))\ge \mathfrak{f}_a\rforal t\in {\mathrm{T}}(D).
%&& \|\psi\|=1,\\\label{Dtrdiv-4+}
%&&f_{1/4}(\psi(a))\,\,\,
%\psi(f_{1/4}(a_1)) \,\,\,
%{\rm is\,\,\, full\,\,\, in}\,\, D \andeqn
%\psi_D={\rm id}_D\andeqn
\eneq
%$\psi(a)$ is strictly positive in $D.$
  Then
we shall say $A\in {\cal D}_{0}$ (or $A\in {\cal D}$).

\end{df}

\begin{prop}\label{PD0=tad}
Let $A$ be a $\sigma$-unital simple \CA\, in ${\cal D}$ (${\cal D}_{0}$).
Then, in Definition \ref{DNtr1div}, we may further require 
that
$\|\psi(x)\|\ge (1-\ep)\|x\|$  for all $x\in {\cal F}$ and
that $\psi(a)$ {{be}}  strictly positive in $D$ (and so full 
%and {\red{hence}} is full 
in $D$).
Moreover, \eqref{DNtrdiv-3} may be replaced by
$c\lesssim b$ for some strictly positive element {{$c$}} of $\overline{\phi(A)A\phi(A)}.$
%$A$ is {\rm TA}${\cal S}$ (${\cal S}={\cal C}_0',$ ${\cal C}_0^{0'}$).
\end{prop}

\begin{proof}
Fix a strictly positive element $a\in A$ with $\|a\|=1.$
%%\%{\blue{For the first part of the statement,}}
%%it suffices to show  that in  the definition, we may further assume that $\|\psi(x)\|\ge (1-\ep)\|x\|$ for all $x\in {\cal F}$ and
%%%%%that $\psi(a)$ is  a strictly positive element  of $D.$
%strictly positive in $D.$

%Let $r_0=\mathfrak{f}_{a}/2.$
Let $\ep>0,$ let ${\cal F}\subset A$ be a finite subset, and let $b_0\in A_+\setminus\{0\}$ be given.
\Wlog, we may assume that there is $1/16>\eta>0$ such
that
$$
f_{\eta}(a)x=xf_{\eta}(a)=x\rforal x\in {\cal F}.
$$
By hypothesis, there exist  a sequence
of algebras $D_n\in {{{\cal C}_0'}}$ {{(or $D_n\in {{\cal C}_0^{0}}'$),}}
%{\cal C}_0^{0'}$ (or in ${\cal C}_0',$ in ${\cal C}_\omega'$)
and two sequences of \cpc s $\phi_n: A\to A_n$ and  $\psi_n: A\to D_n$, with 
$D_n\perp \mathrm{Im}\phi_n$,
%$D_n\ \phi_n(A)=0$ 
such that
\beq\label{=tad-2}
\lim_{n\to\infty}\|\phi_n(xy)-\phi_n(x)\phi_n(y)\|=0\andeqn\\\label{=tad-3}
\lim_{n\to\infty}\|\psi_n(xy)-\psi_n(x)\psi_n(y)\|=0\rforal x, \, y\in A,\\\label{=tad-4}
\lim_{n\to\infty}\|x-(\phi_n(x)+ \psi_n(x))\|=0\rforal x\in A,\\
\label{=tad-5}
%\psi_n(x)\in D_n\rforal x\in M_N(A),\\\label{=tad-6}
%\|\psi_n\|=1,\\\label{=tad-7}
\phi_n(a)\lesssim b_0,\andeqn\\\label{=tad-6}
\tau(f_{1/4}(\psi_n(a)))\ge \mathfrak{f}_a \rforal \tau\in {\mathrm{T}}(D_n)
%\rforal x\in {\cal F}
\eneq
%and $\psi_n({\bar a})$ is a strictly positive element of $M_k(D_n),$ $n=1,2,...,$
%where ${\bar a}=\diag(\overbrace{a,a,...,a}^k).$
Put $D_n'=\overline{f_{\eta/2}(\psi_n(a))D_nf_{\eta/2}(\psi_n(a))},$ $n=1,2,....$
By \eqref{=tad-6} and \ref{Pcom4C}, $f_{1/4}(\psi_n(a))$ is full in $D_n.$
Therefore $f_{\eta/2}(\psi_n(a))$ is also full in $D_n.$  This implies that $D_n'\in {\cal C}_0'$
{{or $D_n'\in {{\cal C}_0^0}'.$}}
Define $\psi_{n,0}: A\to D_n'$ by
$$
\psi_{n,0}(x)=(f_{\eta/2}(\psi_n(a)))^{1/2}\psi_n(x)(f_{\eta/2}(\psi_n(a)))^{1/2}\rforal x\in A.
$$
It follows that $\psi_{n,0}(a)$ is full in $D_n'.$
Note that
$$
f_{1/4}(\psi_{n,0}(a))=f_{1/4}(\psi_n(a)).
$$
Therefore,
%\vspace{-0.05in} 
$$
\tau(f_{1/4}(\psi_{n,0}(a)))\ge \mathfrak{f}_a\rforal \tau\in {\mathrm{T}}(D_n').
$$
Choosing large $n,$  replacing $D_n$ by $D_n',$ $\psi$ by $\psi_{n,0},$ and using \eqref{=tad-3} and \eqref{=tad-4},
we see that in Definition \ref{DNtr1div}, we {{may}}  add the condition
that $\psi(a)$ is a strictly positive element of  $D.$

{{By \ref{180915sec2}, $\phi_n(a)$ is strictly positive in $\overline{\phi_n(A)A\phi_n(A)}.$ Therefore, one can replace 
\eqref{DNtrdiv-3}
by the condition that $c\lesssim b$ for  any other  strictly positive element of $\overline{\phi_n(A)A\phi_n(A)}.$}}

\iffalse
Define $\phi_{n,0}: A\to A$ by $\phi_{n,0}(x)=(f_{\eta/2}(\phi_n(a)))^{1/2}\phi_n(x)(f_{\eta/2}(\phi_n(a)))^{1/2}$
for all $x\in A.$ Then $\phi_{n,0}(A)\perp D_n.$ Moreover, 
$(f_{\eta/2}(\phi_n(a)))$ is a strictly positive element of $\overline{\phi_{n,0}(A)A\phi_{n,0}(A)}$
and $f_{\eta/2}(\phi_n(a))\lesssim \phi_n(a)\lesssim b.$ 
Thus, by choosing large $n,$ replacing $\phi_n$ by $\phi_{n,0},$ we may assume 
that $c\lesssim b$ for some strictly positive element of $\overline{\phi_n(A)A\phi_n(A)}.$ 
\fi
To get  the inequality $\|\psi(x)\|\ge (1-\ep)\|x\|$ for all $x\in {\cal F},$ we note that, by \eqref{=tad-4} and 
% we repeat the above argument. As shown above, we can add the requirement
%that $\psi_n(a)$ be strictly  positive in  $D_n$ along with \eqref{=tad-2} to \eqref{=tad-6}.
%The condition 
\eqref{=tad-6} 
%also implies that
\beq\label{=tad-8}
\lim_{n\to\infty}\|\psi_n\|\ge \mathfrak{f}_a.
\eneq
Then, by \eqref{=tad-3} and  \eqref{=tad-8}, since $A$ is simple,
\beq\label{=tad-9}
\lim_{n\to\infty}\|\psi_n(x)\|=\|x\|\rforal x\in A.
\eneq
This implies that, {{choosing $\psi=\psi_n$}} with sufficiently large $n,$ we may always assume
that $\|\psi(x){{\|}}\ge (1-\ep)\|x\|$ for all $x\in {\cal F}.$
\end{proof}

%\begin{prop}\label{DpreTADone}
%Let $A$ be a $\sigma$-unital \CA\,  which is pre-TAD (pre-TA$D_1$) and $a_1\in A_+\setminus \{0\}.$
%Then, \eqref{Dtrdiv-4+} can be strengthened to be
%\beq\label{Dtrdiv-4++}
%f_{1/4}(\psi(a)) \andeqn f_{1/4}(\psi(a_1)) \,\,\,{\rm are\,\,\, full\,\,\, in}\,\,\, D.
%\eneq
%\end{prop}

\begin{thm}\label{UnifomfullTAD}
Let $A$ be a $\sigma$-unital simple \CA\, in ${\cal D}$ (or  in ${\cal D}_0$). Then the following holds.
Fix  a strictly positive element $a\in A$ 
with $\|a\|=1$ and let $1>\mathfrak{f}_a>0$ be a positive number associated
with $a$ {{as}} in Definition \ref{DNtr1div}. There is a map $T: A_+\setminus \{0\}\to \N\times \R$
($c\mapsto (N(c), M(c))\tforal c\in A_+\setminus \{0\}$) with the following property:
For any  finite subset ${\cal F}_0\subset A_+\setminus \{0\},$ 
%   There exists
%   $M>0$ and an integer $N\ge 1,$
 any $\ep>0,$  any
finite subset ${\cal F}\subset A,$ and any $b\in A_+\setminus \{0\},$   there are ${\cal F}$-$\ep$-multiplicative \cpc s $\phi: A\to A$ and  $\psi: A\to D$  for some
\SCA\, $D\subset A$ with $D\perp \phi(A)$ such that
\beq\label{Dtad-1}
&&\|x-(\phi+ \psi)(x)\|<\ep\tforal x\in {\cal F}\cup \{a\},\\\label{2Dtrdiv-2}
&& D\in {{\cal C}_0^0}'\,({or}\,\,\,{\cal C}_0'),\\\label{Dtad-3}
&&\phi(a)\lesssim b,\\\label{2Dtrdiv-4}
&& \|\psi(x)\|\ge (1-\ep)\|x\|\tforal x\in {\cal F},\tand\\
%\\\label{Dtad-4+}
\eneq
$\psi(a)$ is strictly positive in $D.$
Moreover,  $\psi$ may be chosen to be  $T$-${\cal F}_0\cup\{f_{1/4}(a)\}$-full in $\overline{DAD}.$
% for any $\eta>0,$  any $x\in D$ and any $c\in {\cal F}_0\cup \{f_{1/4}(a)\},$
%for each $c\in {\cal F}_0\cup\{a\},$
%there are $x(c,1), x(c,2),...,x(c,N), y(c,1), y(c,2),...,y(c,N)\in D$ with
%$\|x(c,i)\|,\,\, \|y(c,i)\|\le M$ for all $c\in {\cal F}_0,$ $i=1,2,...,N,$ such that
%$$
%\|\sum_{i=1}^N x(c,i)\psi(c)y(c,i)-x\|<\eta.
%$$

Furthermore, we may ensure that
$$
t\circ f_{1/4}(\psi(a))\ge \mathfrak{f}_a\tand t\circ f_{1/4}(\psi(c))\ge
{{{\mathfrak{f}_a}}\over{4\inf\{M(c)^2\cdot N(c): c\in {\cal F}_0\cup\{f_{1/4}(a)\}\}}}
$$
for all $c\in {\cal F}_0$ and for all $t\in {\mathrm{T}}(D).$
%%where $\mathfrak{f}_a$ is  {{stipulated}}
%given 
%%by the definition of ${\cal D}$ (or ${\cal D}_0$) associated with $a.$

\end{thm}

\begin{proof}
Since $A$ is simple, and $f_{1/32}(a)\in {\rm Ped}(A),$ for any $b\in A_+\setminus \{0\},$ there exist $N_0(b)\in \N,$
$M_0(b)>0$  and $x_1(b), x_2(b),...,x_{N_0(b)}(b)\in A$ such that $\|x_i(b)\|\le M_0(b),$ and
\beq\label{UnifomfullTAD-n1}
\sum_{i=1}^{N_0(b)} x_i(b)^* bx_i(b)=f_{1/32}(a).
\eneq
%Let $\mathfrak{f}_a>0$ be  as in Definition \ref{DNtr1div}.

%Let $n_0\ge 1$ be 
{{Choose}} an integer  $n_0$ such that $n_0\mathfrak{f}_a\ge 4.$

Set $N(b)=n_0N_0(b)$ and $M(b)=2M_0(b)$ for all $b\in A_+\setminus \{0\}.$
Let $T: A_+\setminus \{0\}\to \N\times \R_+\setminus \{0\}$ be defined
by $T(b)=(N(b), M(b))$ for  $b\in A_+\setminus \{0\}.$

Choose $\dt_0>0$ and a finite subset ${\cal G}_0\subset A$ such that
\beq\label{UnifomfullTAD-n2}
\|\sum_{i=1}^{N_0}\psi( x_i(b))^* \psi(b)\psi(x_i(b))-f_{1/32}(\psi(a))\|<1/2^{10}\rforal b\in {\cal F}_0,
\eneq
whenever $\psi$ is a ${\cal G}_0$-$\dt_0$-multiplicative \cpc\, from $A$ into a \CA .

Let $\ep>0$ and a finite subset ${\cal F}\subset A$ be given.
Set $\dt=\min\{\ep/4, \dt_0/2\}$ and ${\cal G}={\cal F}\cup {\cal G}_0\cup \{a, f_{1/4}(a)\}.$
Let $n\ge 1$ be an integer and let $b_0\in A_+\setminus \{0\}.$

By the assumption and by \ref{PD0=tad},  
%one has the following:
 there are ${\cal G}$-$\dt$-multiplicative \cpc s $\phi: A\to A$ and  $\psi: A\to D$  for some
\SCA\, $D\subset A$ with $\phi(A)\perp D$ such that {{$D\in {\cal C}_0$ (or $D\in {{\cal C}_0^0}'$),}}
$\psi(a)$ is strictly positive in $D,$ and
\beq\label{UnifomfullTAD-n3}
&&\|x-(\phi+ \psi)(x)\|<\ep\rforal x\in {\cal G},\\\label{UnifomfullTAD-n4}
&& D\in {{\cal C}_0^0}'\,({\rm or}\,\,\,{\cal C}_0'),\\\label{UnifomfullTAD-n5}
&&\phi(a)\lesssim b_0,\\\label{UnifomfullTAD-n6}
&& \|\psi(x)\|\ge (1-\ep)\|x\|\rforal x\in {\cal F},\andeqn\\
%\\\label{Dtad-4+}
%\eneq
%$\psi(a)$ is strictly positive in $D,$ and
%\beq
\label{UnifomfullTAD-n7}
&&\tau(f_{1/4}(\psi(a)))\ge \mathfrak{f}_a\rforal \tau\in {\mathrm{T}}(D).
\eneq
At this point, we can apply (the proof of)  {{Theorem}} \ref{Tqcfull} and  Remark \ref{Rqcfull} to conclude
that $\psi$ is $T$-${\cal F}_0\cup \{f_{1/4}(a_0)\}$-full. 
The last part of the {{conclusion}}
% {\red{theorem}} 
then follows.
%The lemma then follows.

 %By the choice of ${\cal G}$ and $\dt,$ \eqref{UnifomfullTAD-n2} holds
 %for this $\psi.$  It follows from \ref{Lrorm} that  there is $z_b\in A$ with $\|z_b\|\le 1$ such that
 %\beq\label{UTAD-n8}
 %z_b^*\sum_{i=1}^{N_0}\psi( x_i(b))^* \psi(b)\psi(x_i(b))z_b=(f_{1/32}(\psi(a))-1/2^9)_+\ge (3/4) f_{1/4}(\psi(a)).
 %\eneq
%Since $D$ has comparison property described in
%\ref{Pcom4C}, by \ref{UnifomfullTAD-n7}, for any $c\in D$ with $0\le c\le 1,$
%$$
%%\la c\ra \le n_0\la f_{1/4} (\psi(a)) \ra
%$$
%For each $0\le c\le 1$ in $D,$ there is $1>\eta>0$ such that
%$$
%f_\eta(c)cf_\eta(c)-.
%$$
\end{proof}

\begin{cor}\label{ChighrankD}
In Definition \ref{DNtr1div},  for any integer $k\ge 1,$ one may assume that every irreducible representation
of $D$ has dimension at least $k.$
\end{cor}

\begin{proof}
{{Let $T$ be as in the statement of \ref{UnifomfullTAD}. 
Fix an integer $k\ge 1.$
This corollary  can be easily seen by taking ${\cal F}_0$ containing $k$ mutually orthogonal non-zero
positive elements  $e_1, e_2,...,e_k$  with $\|e_i\|=1$ in \ref{UnifomfullTAD} as follows.

 When ${\cal F}_0$ is chosen. 
  Set
  $
  \sigma_0={\mathfrak{f}_a\over{4\inf\{M(c)^2\cdot N(c): c\in {\cal F}_0\cup\{f_{1/4}(a)\}\}}}.
  $
  There exists $\eta_0>0$  such that,
  if $0<b_1,\, b_2\le 1$ are  in any \CA\, with $\|b_1-b_2\|<\eta_0,$ then 
  \beq
  \|f_{1/4}(b_1)-f_{1/4}(b_2)\|<\sigma_0/2.
  \eneq
 By 10.1.12 of \cite{Loringbk}, there exists $\dt_0>0$ satisfying the following property: if 
 $0\le h_i\le 1$ and $\|h_ih_j\|<\dt_0$ ($i\not=j: 1\le i, j\le k$) are in a \CA, then there are mutually orthogonal 
 $h'_1, h_2',...,h_k'$ in that \CA\, such that $\|h_i-h_i'\|<\eta_0,$ $i=1,2,...,k.$
   
  Choose any finite subset ${\cal F}$ containing ${\cal F}_0$ and $\dt>0$ with 
  $\dt<\dt_0.$
  We apply \ref{UnifomfullTAD}. Then 
  \beq
  t(f_{1/4}(\psi(e_i))>\sigma_0\rforal t\in {\rm T}(D),\,\,\, i=1,2,...,k.
  \eneq
  By the choice of $\dt_0$ and applying 10.1.12 of \cite{Loringbk},  there are mutually orthogonal non-zero elements $d_1, d_2,...,d_k\in D$ such that
  \beq
  \|d_i-\psi(e_i)\|<\eta_0,\,\,\, i=1,2,....
  \eneq
  It follows that $\|f_{1/4}(d_i)-f_{1/4}(\psi(e_i)\|<\sigma_0/2,$ $i=1,2,...,k.$
 We then  estimate that
  \beq
  t\circ f_{1/4}(d_i)>t\circ f_{1/4}(\psi(e_i))-\sigma_0/2\ge  \sigma_0/2\rforal t\in {\rm T}(D),\,\,\,i=1,2,...,k. 
  \eneq
  Thus,
  %  since $k$ is given,
%by taking sufficiently small $\ep,$ 
%we may assume that $D$ contains
%$k$ mutually orthogonal non-zero elements which are full. 
%This forces 
$\pi(D)$  admits
$k$ mutually orthogonal non-zero elements in  each irreducible representation
$\pi$ which implies $\pi(D)$ has dimension at least $k.$}}
\end{proof}

Note that, if $D$ is in ${{\cal C}_0^0}'$ or in ${\cal C}_0',$ then $M_k(D)$ is in ${{\cal C}_0^0}'$ or in ${\cal C}_0'$
for every integer $k\ge 1.$ 
Therefore, 
the following proposition follows immediately from the definition.

\begin{prop}\label{PtadMk}
Let $A$ be a $\sigma$-unital simple \CA\, in the class ${\cal D}$ (or in ${\cal D}_0$). Then 
$M_k(A)$ is in  the class ${\cal D}$ (or in ${\cal D}_0$) for every 
integer $k\ge 1.$
\end{prop}

\begin{prop}\label{Phered}

%Let ${\cal S}$ denote the class ${\cal C}_0^{0'},$ or ${\cal C}_0'.$
Let $A$ be a 
%non-unital  
%non-zero 
separable simple \CA\, and let $B\subset A$ be a hereditary \SCA.
Then, 
%(1) If $A$ is TA${\cal S},$ so is $B;$
%(2) 
if  $A$ is in ${\cal D}$ (or in ${\cal D}_0$), so also is $B.$
Moreover, {{if $A\not=\{0\},$}}  {{then}} ${\mathrm{T}}(A)\not=\O.$
\end{prop}

\begin{proof}
%Let $A\in {\cal S},$
Let ${\cal S}$ denote ${\cal C}_0'$ or ${{\cal C}_0^0}'.$
{{We may assume neither $A$ nor $B$ is  zero.}}
%be subcompact.
% and $b\in A_+\setminus \{0\}.$
%We may view $A$ is a hereditary \SCA\, of  $\sigma$-unital simple $A_1$ and there is $e_1\in A_1$ such
%that $e_1x=xe_1=x$ for all $x\in A.$
Let $b\in A_+$ with $\|b\|=1$ and $B=\overline{b{{A}}b}.$
Let $e\in A_+$ be a strictly positive element with $\|e\|=1$ and let $\mathfrak{f}_e$ be as given by \ref{DNtr1div}, as  $A$
is in ${\cal D}$ or {{in}} ${\cal D}_0.$ 
%We assume that $\|b\|=1$ and $\|e\|=1.$ 
Fix $b_0\in B_+\setminus \{0\}.$

%Since $A\in {\cal D}_{00}$ (or ${\cal D}_0$ or ${\cal D}$),
%In both cases (1) and (2),  by \ref{DpreTADone},
%by \ref{UnifomfullTAD},
%or a fixed $a\in A_+\setminus \{0\}$ with $\|a\|=1.$
%there exists $\mathfrak{f}_a>0$ satisfies the following:
%Fix $N\ge 1,$ t
%T
By \ref{UnifomfullTAD}, 
there exists
a sequence
of \SCA s $D_n$ of $A$ in ${\cal S}$ 
%(or in ${\cal D}_0$)
%{\cal C}_0^{0'}$
and two sequences of \cpc s $\phi_n: A\to A$ and $\psi_n: A\to D_n$ with $\phi_n(A)\perp D_n$
such that
\beq\label{Phered-1}
\lim_{n\to\infty}\|\phi_n(xy)-\phi_n(x)\phi_n(y)\|=0\andeqn\\
\lim_{n\to\infty}\|\psi_n(xy)-\psi_n(x)\psi_n(y)\|=0\rforal x, \, y\in A,\\\label{Phered-2}
\lim_{n\to\infty}\|x-(\phi_n+\psi_n)(x)\|=0\rforal x\in A,\\
%\psi_n(x)\in D_n\rforal x\in A,\\
\phi_n(e)\lesssim b_0,\\\label{Phered-2+}
\lim_{n\to\infty}\|\psi_n(x)\|=\|x\|\rforal x\in  A,
\eneq
$f_{1/4}(\psi_n(b))$ is full in $D_n,$  and $\psi_n(e)$ is a strictly positive element of $D_n,$ $n=1,2,....$
Moreover,  we {{may also assume that}}
\beq\label{Phered-1n}
t\circ f_{1/4}(\psi_n(e))\ge \mathfrak{f}_e,\,\,\, t\circ f_{1/4}(\psi_n(b))\ge r_0
\eneq
for all $t\in {\mathrm{T}}(D_n)$ and $n,$   where  $r_0$ is as  previously defined
(as ${{{\mathfrak{f}_e}}\over{4\inf\{M(c)^2\cdot N(c): c=\{b, f_{1/4}(e)\}\}}}$).

%In particular,  in case (2), $f_{1/4}(\psi_n(b))$ is full in $D_n,$ $n=1,2,....$

%As in the proof of \ref{T1D0}, using
By \eqref{Phered-2+},
$$
\lim_{n\to\infty}\|\psi_n|_{B}\|=1.
$$
%Let $\psi_n'=\psi_n|_{B},$ $n=1,2,....$

We also have 
%that
\beq\label{Phered-20+1}
&&\lim_{j\to\infty}\|b-f_{1/2j}(b)^{1/2}bf_{1/2j}(b)^{1/2}\|=0,\,\,
%<1/64Mm^2
%\andeqn
{\rm whence}\\\label{Phered-20+2}
&&\lim_{j\to\infty}\|x-f_{1/2j}(b)^{1/2}xf_{1/2j}(b)^{1/2}\|=0
%<\ep/64Mm^2
\rforal  x\in B.
\eneq

Put $L_n(x)=\phi_n(x)+\psi_n(x)$ for all $x\in A.$
By \eqref{Phered-2},   applying \ref{LRL},
for any $j\ge 2,$  we obtain  $n(j)\ge j$ and  a partial isometry
$v_{j}\in A^{**}$ such that
\beq\label{Phere-23}
&&\hspace{-0.8in}v_{j}v_{j}^*f_{1/2j}(L_{n(j)}(b))=f_{1/2j}(L_{n(j)}(b))v_jv_{j}^*=f_{1/2j}(L_{n(j)}(b)),\\
&&\hspace{-0.8in}v_{j}^*cv_{j}\in B\rforal c\in \overline{f_{1/2j}(L_{n(j)}(b))Af_{1/2j}(L_{n(j)}(b))},
\andeqn\\\label{Phere-23+}
&&\hspace{-0.8in}\lim_{j\to\infty}(\sup\{\|v_{j}^*cv_{j}-c\|: 0\le c\le 1\andeqn c\in
\overline{f_{1/2j}(L_{n(j)}(b))Af_{1/2j}(L_{n(j)}(b))}\})=0.
\eneq
Note that $f_{1/2j}(\psi_{n(j)}(b))\le f_{1/2j}(L_{n(j)}(b)),$ $j=1,2,....$
It follows that
$$
v_{j}^*cv_{j}\in B
$$
for all $c\in \overline{f_{1/2j}(\psi_{n(j)}(b)Af_{1/2j}(\psi_{n(j)}(b))}.$
%By \eqref{Phere-20-2},
%\eqref{T1D0-21},
%{T1D0-15+},
Since $f_{1/4}(\psi_{n(j)}(b))$ is full in $D_{n(j)},$
$f_{1/2j}(\psi_{n(j)}(b))$ is full in $D_{n(j)}$ for all $j\ge 2.$
Consider {{the}} hereditary \SCA\, of $D_{n(j)}$ 
$$
E_{n(j)}'=\overline{f_{1/2j}(\psi_n(b))D_{n(j)}f_{1/2j}(\psi_n(b))},
\,\,\,j=2,3,....
$$
Then $E_{n(j)}'\in {\cal S},$ $j=2,3,....$
Put
$$
E_{j}=v_{j}^*E_{n(j)}'v_{j},\,\,\,j=3,4,....
$$
Then $E_{j}\in {\cal S}$ and $E_{j}\subset B,$ $j=3, 4,....$

Define $\Phi_{j}: B\to B$ by $\Phi_j(a)=v^*_j\phi_{n(j)}(a)v_j$ for all $a\in B,$ 
and $\Psi_j: B\to E_j$ by $\Psi_j(x)=v_j^*f_{1/2j}(\psi_{n(j)}(b))\psi_{n(j)}(x)f_{1/2j}(\psi_{n(j)}(b))v_j,$ $j=3, 4,....$
%Note that $\Psi_j$ maps $A_1$ into $E_{j}$ with
%$\|\Psi_j\|=1,$ $j=1,2,....$
For $j>4,$
\beq\label{Phere-n10}
&&\hspace{-0.4in}f_{1/4}(f_{1/2j}(\psi_{n(j)}(b))\psi_{n(j)}(b)f_{1/2j}(\psi_{n(j)}(b)))=f_{1/4}(\psi_{n(j)}(b))\\\label{Phere-n11}
\hspace{0.4in} &&=
 f_{1/2j}(\psi(b))f_{1/4}(\psi_{n(j)}(b))f_{1/2j}(\psi_{n(j)}(b)).
\eneq
%Since $f_{1/4}(\psi_{n(j)}(b))$ is full in $D_{n(j)}$ for all $j\ge 2,$
It follows that
$f_{1/4}(\Psi_j(b))$ is full in $E_j,$ $j=4,5,....$
We have
\beq\label{Phere-n13}
\lim_{j\to\infty}\|\Phi_{j}(xy)-\Phi_{j}(x)\Phi_j(y)\|&=&0\rforal x, y\in B\andeqn\\\label{Phere-n13+}
\lim_{j\to\infty}\|\Psi_j(xy)-\Psi_j(x)\Psi_j(y)\|&=&0\rforal x,y\in B.
\eneq
Moreover,   applying \eqref{Phered-2}, \eqref{Phere-23+}, and \eqref{Phered-20+2}, we have
\beq\label{Phere-n14}
\lim_{j\to\infty}\|x-(\Phi_j+\Psi_j)(x)\|&=&0\rforal x\in B\andeqn\\
\lim_{n\to\infty}\|\Psi_n(x)\|&=&\|x\|\rforal x\in  B.
\eneq
We also have
%that
$$
\Phi_j(b)\lesssim b_0.
$$
%These further implies that $B$ is  in TA${\cal S}.$
Moreover,   by \eqref{Phered-1n} and \eqref{Phere-n10},
% for all large $j,$
%{T1D0-20-2},
%{T1D0-20} and \eqref{T1D0-21},
\beq\label{1222+}
%t\circ \Psi_j(w^*aw)\ge \mathfrak{f}_a\andeqn
t\circ f_{1/4}(\Psi_j(b))\ge r_0/2\rforal t\in {\mathrm{T}}(E_{n(j)}).
\eneq

The  first part of the proposition follows
% when  one chooses  
on choosing a sufficiently large $j.$

To see that, {{if $A$ is nonzero,}} ${\mathrm{T}}(A)$ is non-empty,  in the preceding argument, take $B=A$ and  choose $t_j\in {\mathrm{T}}(E_{n(j)})$ {{for all  $j$  large enough that $E_{n(j)}$ is nonzero.}}
Let $t$ be a weak* limit of $(t_j\circ \Psi_j).$ Then, by \eqref{1222+}, $t$ is a non-zero linear functional on $A.$
Moreover,  since $t_j\in {\mathrm{T}}(E_{n(j)}),$ by \eqref{Phere-n13+}, $t$ is a trace. This implies  ${\mathrm{T}}(A)\not=\O.$

%Since $A$ is simple,
%we may assume
%that  $e_1\in M_m.$

%as in the proof of \ref{localunit}, for some $n>4,$
%there are $x_1, x_2,...,x_m\in M_K(A)$ (for some integer $m, K\ge 1$) such that
%$$
%\sum_{i=1}^k x_i^*bx_i=f_{1/n}(e_1).
%$$

\end{proof}

%\begin{prop}\label{PtadMk}
%Let $A\in {\cal D}$ (or ${\cal D}_{0}$) with $A={\mathrm{Ped}(A)}.$  Then, for every integer $k\ge 1,$ ${\mathrm{M}}_k(A)\in {\cal D}$ (or
%${\cal D}_0,$).
%\end{prop}

%%We now return to Section 6 and Definition \ref{Dbuild1}.

\begin{lem}\label{LherC}
{{Let $B=A(F_1, F_2, \phi_0, \phi_1).$ Suppose 
that ${{g:=(h,a)}}\in B_+$ {{is}} such that $h_j:=h|_{[0,1]_j}$ has range projection  
$P_j$ satisfying  the following conditions:\\
There is a partition  $0=t_j^0< t_j^1<t_j^2<\cdots<t_j^{n_j}=1$  such that\\
(1) on each open interval $(t_j^l, t_j^{l+1})$, $P_j(t)$ is continuous
and   ${\rm rank}({{P_j(t))}}={{r_{j,l}}}$ is a constant,  \\ 
(2) for each $t_j^l$, $P_j((t_j^l)^+)=\lim\limits_{t\to (t_j^l)^+} P_j(t)$ (if $t_j^l<1$) and $P_j((t_j^l)^-)=\lim\limits_{t\to (t_j^l)^-} P(t)$  (if $t_j^l>0$) exist, \\
(3) $P_j(t_j^l)\leq  P_j((t_j^l)^+)$ and $P_j(t_j^l)\leq  P_j((t_j^l)^-),$  {{and}}\\
(4) $\pi^j(\phi_0(p))=P_j(t_j^0)=P_j(0)=P_j(0^+)$ and $\pi^j(\phi_1(p))=P_j(t_j^{n_j})=P_j(1)=P_j(1^-)$, where $p$ is the %spectral projection of  $a\in F_1$ corresponding to the set $\{\ld\in Sp(a): \ld>0\}$. 
range projection of $a\in F_1.$

{{Then}} ${{\overline{gBg}}}\in {\cal C}.$ }}
%$\overline{hBh}\in {\cal C}.$ }}
\end{lem}

\begin{proof}
For each closed interval $[t_j^l, t_j^{l+1}]$, since  {{the limits}}
\hspace{-0.1in}$$
P_j((t_j^l)^+)=\lim\limits_{t\to (t_j^l)^+} P_j(t)\andeqn P_j((t_j^{l+1})^-)=\lim\limits_{t\to (t_j^{l+1})^-} {{P_j(t)}}
$$
 exist, we can extend $P_j|_{(t_j^l, t_j^{l+1})}$ to the closed interval $[t_j^l, t_j^{l+1}]$, and denote this projection by $P_j^l$. Then we can identify $P_j^lC([t_j^l, t_j^{l+1}], {\rm M}_{r_j})P_j^l$ {{with}} $C([0,1], {{\rm M}}_{r_{j,l}})$
by {{identifying}} $t_j^l$ with $0$ and $t_j^{l+1}$ with $1$, where $r_{j,l}={\rm rank} (P_j^l)$. 
%%Denote that 
{{Set}} $E_2^{j,l}:= {\rm M}_{r_{j,l}}$ and set
%%{{Set}}  
$E_1^{j,l}:=P_j(t_j^l)M_{r_j} P_j(t_j^l)\cong {{{\rm M}}}_{R_{j,l}}.$ 

Since $P_j(t_j^l)\le 
%is sub projection of
 P_j((t_j^l)^+)$, 
 %which is identified with the projection at the point $0$ of the algebra $C([0,1], E_2^{j,l})$, so we can 
{{we may  identify $E_1^{j,l}$ with a unital hereditary \SCA\, of}}
 %a corner sub algebra of 
 $E_2^{j,l}$. {{Denote this identification}} 
 %  denote this map 
 by $\psi_0^{j,l}: E_1^{j,l} \to E_2^{j,l}$. 
 
 Similarly since $P_j(t_j^l)\le 
 %$ is sub projection of $
 P_j((t_j^{l})^-)$, 
 %which is identified with the projection at the point $1$ of the algebra $C([0,1], E_2^{j,l-1})$, so we can identify $E_1^{j,l}$ with a corner sub algebra of $E_2^{j,l-1}$, and denote this map 
{{ we obtain a \hm\,}} $\psi_1^{j,l}: E_1^{j,l} \to E_2^{j,l-1}$
{{which identifies $E_1^{j,l}$ with a unital hereditary \SCA\, of $E_2^{j,l-1}.$}}

 Set $E_1:=pF_1p \oplus \bigoplus_{j=1}^k(\bigoplus_{l=1}^{n_j-1} E_1^{{{j,l}}})$ (note we do  not include $E_1^{{j,l}}$ for $l=0$ and $l=n_j.$ {{ Instead, we include}}
% for which we will use 
 $pF_1p$) {{and let}} $E_2=\bigoplus_{j=1}^k(\bigoplus_{l=0}^{n_j-1} E_2^{{j,l}})$.
Let $\psi_0, \psi_1: E_1\to E_2$ be defined by $\psi_0|_{pF_1p}=\phi_0|_{pF_1p}: pF_1p \to \bigoplus_{j=1}^k E_2^{j,0}$,  $\psi_1|_{pF_1p}=\phi_1|_{pF_1p}: pF_1p \to \bigoplus_{j=1}^k E_2^{j,n_j-1}$, $\psi_0|_{E_1^{j,l}} =\psi_0^{j,l}: E_1^{j,l} \to E_2^{j,l}$ and $\psi_1|_{E_1^{j,l}} =\psi_1^{j,l}: E_1^{j,l} \to E_2^{j,l-1}$. 
{{We then  check}}  $A'={\overline{gBg}}\cong A(E_1, E_2, \psi_0, \psi_1)\in {\cal C}$. Namely, each element $(f,a)=\big((f_1, f_2, \cdots, f_k), a \big) \in {{{\overline{gBg}}}}$ corresponds to an element 
{{$(F, b)\in \{C([0,1], E_2)\oplus E_1: F(0)=\psi_0(b), F(1)=\psi_1(b))\}=
%{{or $
A(E_1, E_2, \psi_0, \psi_1)$,}}
where 
\beq\nonumber
&&\hspace{-0.2in}{{F={{\big(}}f_1^0, f_1^1,\cdots,  f_1^{n_1-1}, f_2^0, f_2^1,\cdots,  f_2^{n_2-1}, \cdots,  f_k^0, f_k^1,\cdots,  f_k^{n_k-1})\andeqn}}\\\nonumber
&&\hspace{-0.2in}{{b={{\big(}}a, f_1(t_1^1), f_1(t_1^2), \cdots, f_1(t_1^{n_1-1}), f_2(t_2^1), f_2(t_2^2), \cdots, f_2(t_2^{n_2-1}), \cdots f_k(t_k^1), f_k(t_k^2), \cdots, f_k(t_k^{n_k-1})\big)}}
\eneq
{{and where $f_j^l(t)\in E_2^{j,l}$ is defined by}}
\hspace{-0.1in}$$f_j^l(t)= f_j((t_j^{l+1}-t_j^l)t+t_j^l)\rforal t\in [0,1], j\in\{1,2,\cdots, k\}, l\in \{0,1,\cdots, n_j-1\}.$$
%$$\big((f_1^0, f_1^1,\cdots,  f_1^{n_1-1}, f_2^0, f_2^1,\cdots,  f_2^{n_2-1}, \cdots,  f_k^0, f_k^1,\cdots,  f_k^{n_k-1}),~~~~~~~~~~~~~~~~~~~~~~~~~~~~~~~~~~~~~~~~~~~ ~~~~$$
%$$
%~~~~~~~~~~(a, f_1(t_1^1), f_1(t_1^2), \cdots, f_1(t_1^{n_1-1}), f_2(t_2^1), f_2(t_2^2), \cdots, f_2(t_2^{n_2-1}), \cdots f_k(t_k^1), f_k(t_k^2), \cdots, f_k(t_k^{n_k-1}))\big)
%$$
%$$~~~~~~~~~~~~~~~~~~~~~~~~~~~~~~~~~~~~~~~~~~~~~~~~~~~~\in C([0,1], \bigoplus_{j=1}^k(\bigoplus_{l=0}^{n_j-1} E_2^{{j,l}}))\oplus pF_1p \bigoplus_{j=1}^k(\bigoplus_{l=1}^{n_j-1} E_1^{{{j,l}}}),$$
%where $f_j^l(t)= f_j((t_j^{l+1}-t_j^l)t+t_j^l)$ for all $t\in [0,1], j\in\{1,2,\cdots, k\}, l\in \{0,1,\cdots, n_j-1\}$.}} 
%{\red{I feel that we should say a few words more here.}}{\blue{Now it should be OK}}

\end{proof}

\begin{NN}\label{Ldert1} Let $A=A(F_1, F_2,\phi_0, \phi_1)\in {\cal C}$
be as \ref{Dbuild1}.  Let $h=(f,a)\in A_+$  with $\|h\|=1.$
{{Recall that we may write $C([0,1], F_2)=\bigoplus_{j=1}^k C([0,1]_j, M_{r_j}{{)}},$ where $[0,1]_j$ denotes
 the $j$-th interval.}}
%(where  $F_1=M_{R_1 }(\C)\oplus M_{R_2 }(\C)\oplus \cdots \oplus M_{R_l}(\C),$  
% $F_2=M_{r_1}(\C)\oplus M_{r_2 }(\C)\oplus \cdots \oplus M_{r_k }(\C)$ ) and $A={\overline{hBh}}$. 
For each fixed $j$,  consider $f_j= f|_{[0,1]_j}$.  
By a simple application  {{of}} Weyl's theorem, 
{{one can}} write the eigenvalues of $f_j(t)$ as {{continuous}} functions of $t,$ 
%(we omit $j$ in the expression---that is by $\ld_i$ we mean $\ld_{j,i}$)
$$\{0\leq \ld_{1,j}(t)\leq \ld_{2,j}(t)\leq\cdots \ld_{i,j}(t)\leq\cdots \leq \ld_{r_j,j}(t)\leq 1\}.$$
%{\blue{It is well known  that  all $\ld_{i,j}$ are {\red{continuous.}}
%(to see it,  for example by Weyl's theorem that $\max\{|\ld_{i,j}(t)-\ld_{i,j}(t')| \}\leq  \|f_j(t)-f_j(t')\|$).}} 
Let $e_1, e_2,...,e_{r_j}$ {{be}} mutually orthogonal rank one projections  and 
put $f'_j=\sum_{i=1}^{r_j}\lambda_{i,j}e_i.$ Then, on each $[0,1]_j,$ 
$f_j$ and $f_j'$ have exactly the same eigenvalues at each point $t\in [0,1]_j.$ 
Let $p\in F_1$ {{denote}} the range projection of $a\in (F_1)_+.$
By using 
a unitary in $C([0,1]_j, {\rm M}_{r_j}),$ it is easy 
 to construct a set of mutually orthogonal  rank one projections 
 $p_1, p_2, \cdots p_i,\cdots, p_{r_j}\in C([0,1], {\rm M}_{r_j})$ such that $g_j(t)=\sum_{i =1}^{r_j}\ld_i (t) p_i$
 satisfies $g_j(0)=f_j(0)$ and $g_j(1)=f_j(1)$. In particular $\sum_{\{i, \ld_i(0)>0\}} p_i{{(0)}}=\pi^j(\phi_0(p))\in {{\rm M}_{r_j}}$
 and $\sum_{\{i, \ld_i(1)>0\}} p_i{{(1)}}=\pi^j(\phi_1(p))\in {{\rm M}_{r_j}}$, where $\pi^j: F_2\to F_2^j={\rm M}_{r_j}$
 is the canonical quotient map to the $j$th summand. 
 Then, with $g|_{[0,1]_j}=g_j,$  $(g, a)\in A_+$. By a result of Thomsen, (see Theorem 1.2 of \cite{Thoms2}) (or \cite{Rl}),
 % {\blue{ See Theorem 1.2 of [Thomsen, Homomorphisms between finite direct sums of circle algebras, Linear and Multilinear Algebra 1992, Vol.32, pp33-50]}}  {\red{This needs a reference}} 
 for each $j$ there {{is a}}  sequence of unitaries $u_n^j\in C([0,1], {\rm M}_{r_j})$ with $u_n^j(0)=u_n^j(1)=\one_{r_j}$ (Note that as $g(0)=f(0)$ and $g(1)=f(1)$, we can choose $u_n^j(0)=u_n^j(1)=1$)  such that $g_j=\lim\limits_{n\to \infty} u_n^jf_j(u_n^j)^*$. Since $u_n^j(0)=u_n^j(1)=\one_{r_j}$, we can put $u_n^j\in C([0,1{{]}}, {\rm M}_{r_j})$ together to define unitary $u_n\in {\widetilde A}$ and get $(g, a)=\lim\limits_{n\to \infty} u_n(f,a)u_n^*.$
 In other words, $(g,a)\sim_{a.u} (f,a)$ in $A.$ Note this, in particular,  implies 
 that $\la (f,a)\ra =\la (g,a)\ra.$ 
 % }} Hence $(g,a)\sim_{a.u} (h,a)$. So by above assertion, WLOG, we can  assume that $h$ itself is of the form of $(g,a)$. 
%That is we assume that each $f_j(t)=\sum_{i =1}^{r_j}\ld_i (t) p_i$. We reserve $g$ for notation for further modification. 
  \end{NN}

\begin{lem}\label{L<thomsen}
{{Let $c=(g,a)\in A(F_1, F_2, \phi_0, \phi_1)_+$ with $\|(g,a)\|=1$ (see \ref{Dbuild1}).
Suppose {{that}}
$$g_j:=g|_{[0,1]_j}=\sum_{i=1}^{r_j}\lambda_{i,j}(t) p_{i,j}(t),$$
where $\lambda_{i,j}\in C([0,1])_+,$ and 
$p_{i,j}\in C([0,1], {\rm M}_{r_j})$  are mutually orthogonal rank one projections.
Then, 
for any $\ep>0,$ there exists $0\le h\le g$ such that $\|h-g\|<\ep,$ 
$(h, a)\in A(F_1, F_2, \phi_0, \phi_1),$ and 
$h_j:=h|_{[0,1]_j}$ satisfies the conditions described in \ref{LherC}.}}
\end{lem}

\begin{proof}
Fix $\ep_1>0$ and $j.$ 
Let $g_j=g|_{[0,1]_j}.$ Let $G_{i,j}=\{t\in [0,1]: \lambda_{i,j}(t)=0\}$. Since all $G_{i,j}$ are closed sets, there is $\dt_0>0$ such that if $0\notin G_{i,j}$ (or $1\notin G_{i,j}$, respectively), then $\dist(0, G_{i,j})>2 \dt_0$ (or $\dist(1, G_{i,j})>2 \dt_0,$ respectively). 
%It is known (see ?, for example) that there are non-negative continuous functions 
%$\lambda_{1,j}, \lambda_{2,j},...,\lambda_{r_j,j}\in C([0,1])$ and mutually orthogonal 
%rank one projections $p_{1, j},p_{2, j},...,p_{r_j, j}\in C([0,1], M_{r_j})$ such that
%\beq
%\|\sum_{i=1}^{r_j} \lambda_{i,j}p_{i,j}-g_j\|<\ep/2\\
%\pi^j(\phi_0(a))=g_j(0)=\sum_{i=1}^{r_j} \lambda_{i,j}p_{i,j}(0)\andeqn\\
%\pi^j(\phi_1(a))=g_j(1)=\sum_{i=1}^{r_j} \lambda_{i,j}p_{i,j}(1).
%\eneq
Fix $\dt>0.$ {\ {such that $\dt<\dt_0.$}}
For each $i,$ there is {{a}} closed set $S_{i,j}$ which is a union of finitely many  closed interval 
containing the set {\ {$G_{i,j}$}}
such that 
\beq
{\rm dist}(s, G_{i,j})<\dt/4\rforal s\in S_{i,j}.
\eneq
{\ {Hence,  ${\rm dist}(0, S_{i,j})>\dt$ (and ${\rm dist}(1, S_{i,j})>\dt$) }}
%$0$ or $1$ are not in $S_{i,j},$ 
if $G_{i,j}$ does not contain them. 
Choose $f_{i,j}\in C([0,1])_+$ such that $f_{i,j}|_{S_{i,j}}=0,$ $1\ge {\ {f}}_{i,j}(t)>0,$ 
if $t\not\in S_{i,j}$ and $f_{i,j}(t)=1$ if ${\rm dist}(t, S_{i,j})>\dt/2.$ 
Put $\lambda_{i,j}'=f_{i,j}\lambda_{i,j}.$  Then $0\le \lambda_{i,j}'\le \lambda_{i,j}.$
Define 
$h_j=\sum_{i=1}^{r_j} \lambda_{i,j}'p_{i,j}.$  Then $h_j\le g_j.$
We can choose $\dt$ sufficiently small to begin with so that
\beq
\|h_j-\sum_{i=1}^{r_j} {{\lambda_{i,j}}}p_{i,j}\|<\ep.
\eneq
Put $h\in C([0,1], F_2)$ such that $h|_{[0,1]_j}=h_j,$ $j=1,2,...,k.$ 
Therefore 
\beq
\|h-g\|<\ep.
\eneq
{{From the construction, we have  $h_j(0)=g_j(0)$ and $h_j(1)=g_j(1)$ (note that if $0\notin G_{i,j}$ (or $1\notin G_{i,j}$), then $f_{i,j}(0)=1$ (or $f_{i,j}(1)=1$)).}}
It follows that $h(0)=g(0)$ and $h(1)=g(1).$ 
Therefore $(h,a)\in A(F_1, F_2, \phi_0, \phi_1).$ Moreover, 
$(h,a)\le (g,a).$

Let $q_{i,j}(t)=p_{i,j}(t)$ if $\lambda_{i,j}'(t)\not=0$ and $q_{i,j}(t)=0$ if $\lambda_{i,j}'(t)=0.$ 
For each $i,$ 
there is a partition $0=t_{i,j}^{(0)}<t_{i,j}^{(1)}<\cdots <t_{i,j}^{(l_j)}=1$ 
such that 
$q_{i,j}$ is continuous on $(t_{i,j}^{(l)}, t_{i,j}^{(l+1)})$.  Namely, on each interval $(t_{i,j}^{(l)}, t_{i,j}^{(l+1)})$, $q_{i,j}(t)$ either constant zero projection or rank one projection $p_{i,j}(t)$ and therefore both $\lim_{s\to t_{i,j}^{(l)}+}q_{i,j}(s)$ and $\lim_{s\to t_{i,j}^{(l+1)}-}q_{i,j}(s)$ exist. Furthermore, if $q_{i,j}(t)$ is zero on the open interval $(t_{i,j}^{(l)}, t_{i,j}^{(l+1)})$, then $q_{i,j}(t)$ is also zero on the boundary (since $\lambda_{i,j}'(t)$ is continuos). Hence we have 
$$q_{i,j}((t_{i,j}^{(l)})^+):=\lim_{s\to t_{i,j}^{(l)}+}q_{i,j}(s)\geq q_{i,j}(t_{i,j}^{(l)})~~~\mbox{and}~~~~q_{i,j}((t_{i,j}^{(l+1)})^-):=\lim_{s\to t_{i,j}^{(l+1)}-}q_{i,j}(s)\geq q_{i,j}(t_{i,j}^{(l+1)}).$$
%and $\lim_{s\to t_{i,j}^{(l)}+}q_{i,j}(s)=q_{i,j}(t_{i,j}^{(l+l)})$ and 
%$\lim_{s\to t_{i,j}^{(l)}-}q_{i,j}(s)=q_{i,j}(t_{i,j}^{(l)}).$
Define $P_j(t)=\sum_{i=1}^{r_j}q_{i,j}(t).$ 
Then $P_j$ satisfies the conditions described in \ref{LherC}.
\end{proof}

We would like to include the following known fact.

\begin{cor}\label{CherCC}
{{Let $A\in {\cal C}$ (or ${\cal C}_0,$ or ${\cal C}_0^{0}$), let $a\in A_+$  be a full element, and 
set $B=\overline{aAa}.$ 
Then, for any $\ep>0,$ there is $0\le b\le a$ such that $b$ is full in $A,$ 
$\|a-b\|<\ep,$ 
%such that 
$\overline{bBb}=\overline{bAb}\in {\cal C}$ (or ${\cal C}_0,$ or ${\cal C}_0^0$).}}
\end{cor}

\begin{proof}
{{Let us first assume that $a$ satisfies the condition on $c$ in  \ref{L<thomsen}.}} 
{{Then, by \ref{L<thomsen} and by \ref{LherC}, there is  $0\le b\le a$ such that, for any
$\ep>0,$ $\|a-b\|<\ep$ and 
$\overline{bAb}\in {{ {\cal C}}}.$ 
{{Write $a=(f,d)\in A(F_1, F_2, \phi_0, \phi_1)_+,$ where $f\in C([0,1], F_2)$ and $d\in F_1.$}}
%{\red{(${\cal C}_0,$ or  ${\cal C}_0^0$).}} 
Note that, if $a$ is full,  $\|f(t)\|\not=0$ for all $t\in [0,1]_j,$ $j=1,2,...,k,$ and $d$ is full in $F_1.$
Since $f(t)$ is also continuous on each $[0,1]_j,$  $\inf\{\|f(t)\|: t\in [0,1]_j\}>0.$
Therefore, ${{ b=(g,d_1)\in A(F_1, F_2, \phi_0, \phi_1)}}$ can be chosen so that $\|{{g(t)}}\|\not=0$ for all $t\in [0,1]_j,$ $j=1,2,...,k,$ and 
also $d_1$ is full in $F_1.$
In other words, $b$ can also be chosen to be full. 
Thus $\tau(b)>0$ for all $\tau\in T(A).$ If $0\not \in \overline{{\rm T}(A)}^{\rm w},$ 
Then $\inf\{\tau(b): \tau\in T(A)\}>0.$ This implies that $0\not\in \overline{{\rm T}(\overline{bAb})}^{\rm w}.$
%If $K_i(A)=\{0\},$ 
Then $B:=\overline{bAb}$ is stably isomorphic to $A$ by  Brown's theorem (\cite{Br1}).
It follows that $K_i(B)\cong K_i(A)$ and $K_0(B)_+=K_0(A)_+.$
%So $K_i(\overline{bBb})=0.$  
%$i=0,1.$ 
{{T}}hus, if $A$ is in ${\cal C}_0$ (or is in ${\cal C}_0^0$), so also is $B.$}}

In general, by \ref{Ldert1}, $a$ is approximately unitarily equivalent to $a'\in A$ which 
satisfies the condition for $c$ in  \ref{L<thomsen}.  Therefore there is an isomorphism 
$\phi: \overline{a'Aa'}\to \overline{aAa}.$ 
Let $b'\le a'$ be as given by the first part of the proof. Choose $b=\phi(b').$ 
{{The conclusion then holds for $b.$}}
% follows.

\end{proof}

\begin{rem}\label{Rsemiproj}

Let $C\in {\cal C}_0'$ (or ${{\cal C}_0^0}'$), let   $e\in C_+$ {{be}} such 
that $\tau(f_{1/2}(e){)}>\mathfrak{f}>0$ for all $\tau\in {\rm T}(C),$ let ${\cal F}$ be a finite subset
in the unit ball of $C,$ 
and let $\ep>0.$  Put $\ep_0=\min\{\mathfrak{f}/4, \ep/4\}.$ 
Choose $\eta>0$ such that 
\beq
\|f_{1/2}(a')-f_{1/2}(b')\|<\ep_0
\eneq
if $0\le a', b'\le 1$ and $\|a'-b'\|<\eta.$ We may assume that $\eta<\ep_0.$
Let $e_C\in C$ be a strictly positive element  such that $\|e_C\|=1$ and 
\beq
\|e_Cfe_C-f\|<\eta/4\rforal f\in {\cal F}\cup \{e, f_{1/2}(e)\}.
\eneq 
By  \ref{CherCC}, there exists $b\in C_+$ {{with}}  $b\le e_C$ and $\|b-e_C\|<\eta/4$ such 
that $B:=\overline{bCb}$ in  ${\cal C}_0$ (or in ${\cal C}_0^{0}$).
Define $\psi: C\to B$ by 
$\psi(c)=bcb$ for all $c\in C.$ 
Then, {{for all $f\in {\cal F}\cup \{e, f_{1/2}(e)\},$}}
\beq
\|\psi(f)-f\|<\eta<\ep,\andeqn \tau(f_{1/2}(\psi(e)))>\mathfrak{f}/2\rforal \tau\in {\rm T}(C).
\eneq
Consequently, {{as $B$ is a hereditary \SCA\, of $C,$}}
\beq
\tau(f_{1/2}(\psi(e)))>\mathfrak{f}/2\rforal \tau\in {\rm T}(B).
\eneq

It follows that, in the definition of ${\cal D}$ and ${\cal D}_0,$ we  may 
%always 
assume that
$D\in {\cal C}_0$ (or $D\in {\cal C}_0^0$).

%Let $A\in {\cal C}_0$ (or ${\cal C}_0^{0}$)  such that it has an ideal $I$ such that $A/I\cong \C.$
%Then, $I$ itself is in ${\cal C}_0

\end{rem}

%%%%%

\begin{cor}\label{inductived0}
{{Let $A$ be a simple \CA\, which is an inductive limit of \CA s in ${\cal C}_0'$ (or in ${{\cal C}_0^0}'$). 
Then $A$ can be also written as an inductive limit of \CA s in ${\cal C}_0$ (or in ${\cal C}_0^0$).}}
\end{cor}

\begin{proof}
{{Let $C\in {\cal C}_0'$ (or $C\in {{\cal C}_0^0}'$). 
Then, by \ref{CherCC}, $C=\overline{\bigcup_{k=1}^{\infty} C_k},$ where 
each $C_k\in {\cal C}_0$ (or $C_k\in {\cal C}_0^0$),  $C_k\subset C_{k+1}$
and $C_k$ is a hereditary \SCA\, of $C.$ $k=1,2,....$}}

{{Suppose that $A=\lim_{n\to\infty}(A_n, \phi_n),$ where $A_n\in {\cal C}_0'$ (or 
$A\in {{\cal C}_0^0}'$) and $\phi_n: A_n\to A_{n+1}$ is a \hm, $n=1,2,...,$
If $m>n,$ put $\phi_{n,m}=\phi_m\circ \phi_{m-1}\circ\cdots \circ \phi_n: A_n\to A_m$ 
and $\phi_{n, \infty}: A_n\to A$  the \hm\, induced by the inductive system. 
Choose a dense sequence $\{x_n\}$ in the unit ball of $A$ such that
$\{x_1,x_2,...,x_n\}\subset \phi_{n, \infty}(A_n),$ $n=1,2,....$
Write $A_n=\overline{\bigcup_{k=1}^{\infty} C_{n,k}},$ where 
$C_{n,k}\in {\cal C}_0$ (or $C_{n,k}\in {\cal C}_0^0$), $C_{n, k}\subset C_{n,k+1},$ $k=1,2,....$ 
\Wlog, we may assume that $x_j\in \phi_{j,\infty}(C_{j,j}),$ $j=1,2,....$ 
Let $y_{j,i}\in C_{j,j}$ such that $\phi_{j, \infty}(y_{j,i})=x_i,$ $i=1,2,..., j$ and $j=1,2,....$}}

{{Let $\{z_{n,k}\}$ be a dense sequence {{in}} $A_n.$  We may assume that $y_{j,i}\in \{z_{j,k}\}$
for $i=1,2,...,j,$ and $j=1,2,....$ 
We may also assume $\phi_n(z_{n,k})\subset \{z_{n+1,k}\},$ $n=1,2,....$
Put ${\cal F}_n=\{y_{n,i}: 1\le i\le n\}\cup\{z_{n,i}: 1\le i\le n\},$ $n=1,2,....$
By the semi-projectivity of \CA s in ${\cal C},$ one easily produces 
a sequence of \hm s  $\psi_n: C_{n,k_n}\to C_{n+1, k_{n+1}}$ for some $k_n\ge n$
such that $\|\psi_n(a)-\phi_n(a)\|<1/2^n$
for all $a\in {\cal F}_n,$ $n=1,2,....$}}

{{Let $B=\lim_{n\to\infty}(C_{n, k_n}, \psi_n).$ 
Then $B$ is an inductive limit of \CA s in ${\cal C}_0$ (or in ${\cal C}_0^0$). 
Let $C=\overline{\bigcup_{n=1}^{\infty}\phi_{n, \infty}(C_{n,k_n})}.$
Then $C$ is a \SCA\, of $A.$  Since $\{x_n\}\subset C,$ $C=A.$
Let $\imath_n: C_{n, k_n}\to C_{n,k_n}$ be the identity map. 
Then,
\beq
\|\phi_n\circ \imath_n(a)-\imath_{n+1}\circ \psi_n(a)\|<1/2^n\rforal a\in {\cal F}_n.
\eneq
By the Elliott approximate intertwining argument, there is an isomorphism $j: B\to C$ 
which is induced by $\{\imath_n\}.$ It follows that $A$ is an inductive limit of 
\CA s in ${\cal C}_0$ (or in ${\cal C}_0^0$).}}
\end{proof}

%%%%%%

\section{
%Properties of
Traces and comparison {{for}} \CA s in the class ${\cal D}$}

\begin{prop}\label{PD0qc}
%Every $\sigma$-unital \CA\, in ${\cal D}$ is quasi-compact.
Let $A$ be a non-zero separable simple \CA\, in ${\cal D}.$
Then ${\mathrm{QT}(A)}={\mathrm{T}(A)}\not=\O.$ Moreover, $0\not\in \overline{{\mathrm{T}(A)}}^\mathrm{w}.$
\end{prop}

\begin{proof}
Let $a_0\in A$ be a strictly positive element of $A$ with $\|a_0\|=1.$
Let $\mathfrak{f}_{a_0}>0$ be as in the definition \ref{DNtr1div}.
%{Dtr1div}.
Fix any $b_0\in A_+\setminus \{0\}.$
Choose a sequence of positive elements
$(b_n)_{n\ge 1}$ which has the following property:
$b_{n+1}\lesssim b_{n,1},$
where $b_{n,1}, b_{n,2},...,b_{n,n}$ are mutually orthogonal positive elements
in $\overline{b_nAb_n}$ such that $b_nb_{n,i}=b_{n,i}b_n=b_{n,i},$
$i=1,2,...,n,$ and $\la b_{n,i}\ra =\la b_{n,1}\ra,$ $i=1,2,...,n.$
%%Note that
%%\beq
%%\lim_{n\to\infty}\sup\{\tau(b_n): \tau\in QT(A)\}=0.
%%\eneq

One obtains  {{(from  Theorem \ref{UnifomfullTAD})}} two sequences of \SCA s $A_{0,n},$ $D_n$ of $A,$ where
$D_n\in {\cal C}_0',$ {{and}}  two sequences
of \cpc s $\phi_{0,n}: A\to A_{0,n}$
%$\|\phi_{i,n}\|=1$ ($i=0,1$)
and
$\phi_{1,n}: A\to D_n$  with $A_{0,n}\perp D_n$ with  the following {{properties:}}
\beq\label{TD0qc-1}
\lim_{n\to\infty}\|\phi_{i,n}(ab)-\phi_{i,n}(a)\phi_{i,n}(b)\|=0\rforal a,\, b\in A,\\\label{TD0qc-1+}
\lim_{n\to\infty}\|a-(\phi_{0,n}+\phi_{1,n})(a)\|=0\rforal a\in A,\\\label{TD0qc-1++}
c_n\lesssim b_n,\\\label{TD0qc-4}
\lim_{n\to\infty}\|\phi_{1,n}(x)\|=\|x\|\rforal x\in A,\\\label{TD0qc-4+}
\tau(f_{1/4}(\phi_{1,n}(a_0)))\ge \mathfrak{f}_{a_0}\rforal \tau\in {\mathrm{T}}(D_n),
\eneq
 and $\phi_{1,n}(a_0)$ is a strictly positive element of $D_n,$
 where $c_n$ is a strictly positive element of $A_{0,n}.$
 %By \ref{PtadMk}, we may assume
 %that $D_n\in {\cal C}_0.$
 Since quasitraces are norm continuous (Corollary II 2.5 of \cite{BH}),
 by \eqref{TD0qc-1+}, 
 \beq\label{TD0qc-4+3}
\lim_{n\to\infty}(\sup\{|\tau(a)-\tau((\phi_{0,n}+\phi_{1,n})(a))|:\tau\in {\rm{QT}}(A)\})=0\rforal a\in A.
\eneq
Since $\phi_{0,n}(a)\phi_{1,n}(a)=\phi_{1,n}(a)\phi_{0,n}(a)=0,$ for any $\tau\in {{\rm{QT}}}(A),$
\beq
\tau((\phi_{0,n}+\phi_{1,n})(a))=\tau(\phi_{0,n}(a))+\tau(\phi_{1,n}(a))\rforal a\in A.
\eneq
Note that, by \eqref{TD0qc-1++}, 
\beq\label{TD0qc-4++}
\lim_{n\to\infty}(\sup\{\tau(\phi_{0,n}(a)): \tau\in {{\rm{QT}}}(A)\})=0\rforal a\in A.
\eneq
%Since $\phi_{0,n}(a)\phi_{1,n}(a)=\phi_{1,n}(a)\phi_{0,n}(a)=0,$  by \eqref{TD0qc-4++}, 
%we have 
%this implies that 
%It follows that
%\beq\label{TD0qc-10}
%\lim_{n\to\infty}(\sup\{|\tau(a)-\tau\circ \phi_{1,n}(a)|: \tau\in {\mathrm{QT}(A)}\})=0\rforal a\in A_+.
%\eneq
Therefore
\beq\label{TD0qc-10}
\lim_{n\to\infty}(\sup\{|\tau(a)-\tau\circ \phi_{1,n}(a)|: \tau\in {\mathrm{QT}(A)}\})=0\rforal a\in A.
\eneq
{{Since $D_n$ is exact, $\tau|_{D_n}$ extends to a trace, say $t_n.$ 
Then, for any $a, b\in  A,$ 
\beq
\tau\circ \phi_{1,n}(a+b)=t_n\circ \phi_{1,n}(a+b)=t_n\circ \phi_{1,n}(a)+t_n\circ \phi_{1,n}(b)\\
=\tau\circ \phi_{1,n}(a)+\tau\circ \phi_{1,n}(b)\rforal a, b
\in A.
\eneq
It follows from \eqref{TD0qc-10}
that, for every $\tau\in {\mathrm{QT}(A)},$ 
\beq\nonumber
\tau(a+b)=\tau(a)+\tau(b)\rforal a, b\in A.
\eneq
Thus $\tau$ extends to a  trace on $A.$ This proves ${\rm QT}(A)={\rm T}(A).$}}
% and by \eqref{TD0qc-1+},
%\beq
%&&\lim_{n\to\infty}(\sup|\tau((a-(\phi_{0,n}+\phi_{1,n})(a))_+)|:\tau\in QT(A)\})=0\andeqn\\
%&&\lim_{n\to\infty}(\sup|\tau((a-(\phi_{0,n}+\phi_{1,n})(a))_-)|:\tau\in QT(A)\})=0\rforal a\in A_+.
%\eneq
%Therefore 
%\beq
%\lim_{n\to\infty}(\sup|\tau(a)-\tau((\phi_{0,n}+\phi_{1,n})(a))|:\tau\in QT(A)\})=0\rforal a\in A_+.
%\eneq
%Then 

{{Choose $s_n\in {\rm T}(D_n),$ $n=1,2,...$  Consider a positive linear functional 
$F_n: A\to \C$ defined by $F_n(a)=t_n\circ \phi_{1,n},$ $n=1,2,...$
%For each $n,$ let ${\tilde t}_n$ be a state of $A$ which extends
%$t_n.$ 
Let $F_0$ be a weak*-limit of $(F_n)_{n\ge 1}.$  Note that, by \eqref{TD0qc-4+},
\beq
F_n(8a_0)=s_n(8\phi_{1,n}(a_0))\ge s_n(f_{1/4}(\phi_{1,n}(a_0))\ge \mathfrak{f}_{a_0}\rforal n.
\eneq
It follows that $F_0\not=0.$}} By \eqref{TD0qc-1}, since $s_n$ is a trace on $D_n,$ 
$F_0$ is a nonzero trace on $A.$ It follows that $T(A)\not=\O.$

Now let $\tau_k\in T(A)$ such that, for some positive linear functional $\tau,$
 $\lim_{n\to\infty}\tau_k(a)=\tau(a)$ for all $a\in A.$
Then, for each $k,$ by \eqref{TD0qc-10}, 
$\lim_{n\to\infty}\|\tau_k|_{D_n}\|=1.$
Consider the restriction 
$t_{k,n}=(\|\tau_k|_{D_n}\|^{-1})\tau|_{D_n}$ {{for large $n.$}}  Then $t_{k,n}\in {\mathrm{T}}(D_n)$ for all $k.$
It follows from \eqref{TD0qc-4+} that
\beq
t_{k,n}( f_{1/4}(\phi_{1,n}(a_0)))\ge \mathfrak{f}_{a_0},\,\,\, n=1,2,....
\eneq
By \eqref{TD0qc-10} and \eqref{TD0qc-1},
\beq
\tau_k(f_{1/4}(a_0))=\lim_{n\to\infty}\tau_{k,n}(f_{1/4}(\phi_{1,n}(a_0)))\ge \mathfrak{f}_{a_0},\,\,\, k=1,2,....
\eneq
Therefore $\tau\not=0.$
This implies that $0\not\in \overline{{\mathrm{T}(A)}}^\mathrm{w}.$

\iffalse
{{as $D_n$ is exact.}}
%Given any $\tau\in {\mathrm{QT}(A)},$
%define $t_n(a)=\tau\circ \phi_{1,n}(a)$ for all $a\in A.$
Thus $t_n$ {{extends to}}   a state of $A.$ 
{{We continue to use $t_n$ for the extension.}}
Let $t_0$ be a {{fixed}}  weak* limit of $(t_n)_{n\ge 1}$ in $A^*.$
Then $t_0\not=0,$ since by
\eqref{TD0qc-4+} and \eqref{TD0qc-1},
\vspace{-0.1in} \beq\label{TD0qc-11}
t_0(f_{1/4}(a_0))\ge \mathfrak{f}_{a_0}.
\eneq
%By \eqref{TD0qc-10}, 
%\beq
%\tau(a+b)=\tau(a)+\tau(b)\rforal a, \, b\in A\input{180408eglntr1.tex}
%%%%\%%%%%%%_+.
%\eneq

It follows from \eqref{TD0qc-1} that $t_0$ is a trace. 
Note that
\beq
\tau\circ \phi_{1,n}(a+b)=t_n\circ \phi_{1,n}(a+b)=t_n\circ \phi_{1,n}(a)+t_n\circ \phi_{1,n}(b)\\
=\tau\circ \phi_{1,n}(a)+\tau\circ \phi_{1,n}(b)\rforal a, b
\in A.
\eneq
It follows from \eqref{TD0qc-10} that

It follows from \eqref{TD0qc-10}
that, for every $\tau\in {\mathrm{QT}(A)},$ 
\beq\nonumber
\tau(a+b)=\tau(a)+\tau(b)\rforal a, b\in A.
\eneq
%%It follows 
{{This shows}} that $\tau$ extends to a trace. Therefore ${{\rm{QT}}}(A)={\rm{T}}(A).$

 The fact that $\mathrm{T}(A)\not=\O$
follows from  the last part of \ref{Phered}.
% together with \ref{PD0=tad}.
By   \eqref{TD0qc-11},   
\fi

\end{proof}

\begin{rem}\label{Rfa0}
Let $A\in {\cal D}$ and let $a\in A_+$ be a strictly positive element with $\|a\|=1.$
In view of \ref{PD0qc}, 
$$
r_0:=\inf\{\tau(f_{1/4}(a)): \tau\in {\mathrm{T}(A)}\}>0.
$$
The  {{proof above}} shows that we may choose
$
\mathfrak{f}_{a}=r_0/2.
$
In fact in the case that $A={\mathrm{Ped}(A)},$ one may choose
$\mathfrak{f}_{a}$  arbitrarily close to 
$$
\lambda_s(A)=\inf\{\tau(a): \tau\in \overline{{\mathrm{T}}(A)}^\mathrm{w}\}.
$$
In the case that $A$ has continuous scale, we may  choose
the strictly positive element  in such  a way that $r_0$ is arbitrarily close
to 1.
\end{rem}

\begin{prop}\label{Pprojless}
Every  \CA\, in ${\cal D}$ is stably projectionless.
\end{prop}

\begin{proof}
Let $A\in {\cal D}.$
Since, by \ref{PtadMk} and \ref{Phered},  $A\in {\cal D}$ 
if and 
only if ${\mathrm{M}}_n(A)\in {\cal D}$ for each $n,$ we only need to show that
$A$ itself has no nonzero projections.
Let $p\in A$ be a nonzero projection.
 %and let
 {{By \ref{PD0qc},}} 
$$
r:=\inf\{\tau(p): \tau\in \overline{{\mathrm{T}(A)}}^\mathrm{w}\}>0.
$$
Choose $r/4>\ep>0.$
Then, by \ref{DNtr1div}, 
\vspace{-0.14in}\beq\label{Pprojless-1}
\|p-(x_1+ x_2)\|<\ep/2,
\eneq
where $x_1\in (A_0)_+$ and $x_2\in  D_+,$ where $A_0=\overline{bAb}$ for some
$b\in A_+$ with
$\mathrm{d}_\tau(b)<r/4$ for all $\tau\in \overline{{\mathrm{T}(A)}}^\mathrm{w},$
%$x_2\in D_+,$  
$D\in {\cal C}_0',$ and $A_0\perp D.$
%A  standard argument shows that t
If $\ep$ is chosen to be a small enough, there are projections
$p_1\in A_0$ and $p_2\in D$ such that
%\beq\label{Pprojless-2}
$\|p-(p_1+p_2)\|<\ep.$
%\eneq
Since $D$ is projectionless, $p_2=0.$
This implies that
$\tau(p)<\tau(p_1)+\ep<r/2$ for all $\tau\in {\mathrm{T}(A)},$
in contradiction with \eqref{Pprojless-1}.

\end{proof}

%\section{Regularity for \CA s in ${\cal D}_{00}$ and in ${\cal D}_{10}$}

\begin{thm}\label{Comparison}
Let $A\in {\cal D}.$ 
%with $0\not\in 
%Suppose that $0\not\in 
Suppose that $A={\rm Ped}(A).$ 
%(or suppose that $A$ is separable). 
%be a
%$\sigma$-unital
%a separable simple \CA.
%which is quasi-compact.
%Suppose that
Let  $a, b\in (A\otimes {\cal K})_+$ be such that
% $0\le {\red{a,\,b}}.$ If 
$\mathrm{d}_\tau(a)\le \mathrm{d}_{\tau}(b)$ for all $\tau\in \overline{{\mathrm{T}(A)}}^\mathrm{w}.$
Then
$$
a\lesssim b.
$$
\end{thm}

\begin{proof}
Recall that, by \ref{compactrace}, $0\not\in \overline{{\mathrm{T}(A)}}^\mathrm{w}.$
Let us first prove the case that $a, b\in A_+$ and $d_\tau(a)<d_\tau(b)$ for all $\tau\in \overline{{\mathrm{T}(A)}}^\mathrm{w}.$

Fix a strictly positive element $a_0\in A$ with $0\le a_0\le 1.$
%If $A$ does not have (SP) property, then,  by \ref{B1sp},
%$A=\overline{\cup_{n=1}^{\infty} A_n}.$  By \ref{Csemiproj},  each $A_n$ has
%the strict comparison for positive elements. A standard argument shows that $A$ also has
%the strict comparison for positive elements (see {{Ex.\ref{chp4Ecomparison}}}).
%\ref{chp4Ecomaprison})
%For the rest of the proof, we assume that $A$ has (SP).
%%%%Let $a,\,b\in A_+$ be two non-zero elements such that
%%%%%%\beq\label{scomp-1}
%%%%%%\mathrm{d}_\tau(a)<\mathrm{d}_\tau(b)\tforal \tau\in \overline{{\mathrm{T}(A)}}^\mathrm{w}.
%%%%%\eneq
%%%%%%%(Recall that, by \ref{compactrace}, if $A={\rm Ped}(A),$ then $0\not\in \overline{{\mathrm{T}(A)}}^\mathrm{w},$
%if $A$ is separable,   by \ref{PD0qc}, $0\not\not\in \overline{{\mathrm{T}(A)}}^\mathrm{w},$ 
%%%so {{this}} hypothesis  is not vacuous.)
%makes sense.)
%For convenience, w
We may assume that $\|a\|{=} \|b\|= 1.$
Let $1/2>\ep>0.$   
%Put $c=f_{\ep/16}(a).$
%where, for $1>t>0,$
%$f_t\in C([0,1])_+$ with $\|f_t\|=1,$ $f_t(s)=0$ if $s\in [0,t/2]$ and $f_t(s)=1$ if
%$s\in [t, 1].$
%If $c$ is Cuntz equivalent to $a,$ then
%zero is an isolated point in ${\rm sp}(a).$ So, $a$ is Cuntz equivalent to a projection.
%Since 
By  \ref{Pprojless}, $A$ is projectionless, and so 
 zero is not an isolated point {{of}} ${\rm sp}(a).$ Therefore,
 there is a non-zero element $c'\in \overline{aAa}_+$ such that
$c'c=cc'=0,$ where $c=f_{\ep/64}(a).$   It follows that 
% (note again $0\not\in \overline{{\mathrm{T}(A)}}^\mathrm{w}$)
%Therefore
%\beq\label{scomp-3}
$$r_0:=\inf\{\mathrm{d}_\tau(b)-\mathrm{d}_{\tau}(c): \tau\in \overline{{\mathrm{T}(A)}}^\mathrm{w}\}>0.$$
%\eneq
%So \red{ei}ther way, (\ref{scomp-2}) holds.
{{(Note that ${\rm d}_\tau(b)-{\rm d}_\tau(c)\ge {\rm d}_\tau(c')$ and use $0\not\in \overline{{\mathrm{T}(A)}}^\mathrm{w}.$)}}

Set  $c_1=f_{\ep/16}(a),$  {{so that $cc_1=c_1.$}}  Then a compactness argument (cf. Lemma 5.4 of \cite{Lnloc}) shows that
%It follows from \ref{Dini} that t
{t}here is $1>\dt_1>0$ such that
%\beq\label{scomp-3+1}
$$\tau(f_{\dt_1}(b))>\tau(c) \,(\ge \mathrm{d}_{\tau}(c_1))\tforal \tau\in \overline{{\mathrm{T}(A)}}^\mathrm{w}.$$
%\eneq
Put $b_1=f_{\dt_1}(b).$
Then
\beq\label{scomp-3+2}
&& \,\,\,\,\,\,\hspace{-0.4in}r:=\inf\{\tau(b_1)-\mathrm{d}_{\tau}(c_1): \tau\in {\mathrm{T}(A)}\}
\ge \inf\{\tau(b_1)-\tau(c): \tau\in {\mathrm{T}(A)}\}>0.
\eneq
Note that $\|b\|=1.$
%Since $A$ is simple and is assumed to have (SP),
%there is a  non-zero projection $ e\in A$ such that $b_1e=e$ and
{{Choosing}} a smaller $\dt_1,$ we may assume that there exist non-zero elements
 $e, e'\in f_{2\dt_1}(b)Af_{2\dt_1}(b)$ with
$0\le e\le e'\le 1$ and $e'e=ee'=e$ such that
%\beq\label{scomp-4}
$$\tau(e')<r/8\tforal \tau\in \overline{{\mathrm{T}(A)}}^\mathrm{w}.$$
%\eneq
Set $r_1=\inf\{\tau(e): \tau\in {\mathrm{T}(A)}\}.$  Note that, as above, since $A$ is simple and 
$0\not\in \overline{{\mathrm{T}(A)}}^\mathrm{w},$  $r_1>0.$
Set $b_2=(1-e')b_1(1-e').$
Thus {{(cf. above)}}, there is $0<\dt_2<\dt_1/2<1/2$ such that
\beq\label{scomp-3+3}
7r/8<\inf\{\tau(f_{\dt_2}(b_2))-{\rm d}_\tau(c_1):\tau\in {\mathrm{T}(A)}\}<r-r_1.
\eneq
Since $f_{\dt_2}(b_2)f_{3/4}(b_2)=f_{3/4}(b_2)$ and since $\overline{f_{3/4}(b_2)Af_{3/4}(b_2)}$ is non-zero,
there is  $e_1\in A_+$ with $\|e_1\|=1$ 
%nonzero elements $e_1, e_2\in A$  
with 
%{{and}} 
%$0\le e_1, e_2\le 1$ such that $e_2e_1=e_1,$ 
$e_1f_{\dt_2}(b_2)=e_1$ {{and}}
$d_\tau(e_1)<r/18$ for all $\tau\in {\mathrm{T}(A)}.$  Choose $\eta<1/4$ and set $e_2=f_{\eta/4}(e_1)$ and 
$e_3=f_{\eta}(e_1).$ Note that $f_{\dt_2}(b_2)e_2=e_2.$ 
%Let
%$$
%r_2=\inf \{\tau(e_1): \tau\in \overline{{\mathrm{T}(A)}}^\mathrm{w}\}>0.
%$$
Let $\sigma_0=\inf\{\tau(e_2): \tau\in {\rm{T}}(A)\}>0.$

{{By}} \ref{Lfullep}, there are $x_1, x_2,...,x_m\in A$ such that 
\beq\label{scomp-5}
%$$
\sum_{i=1}^m x_i^*e_3x_i=f_{1/16}(a_0).
%$$
%%
\eneq
Choose a nonzero element $e_0\in \overline{eAe}_+$ such that $d_\tau(e_0)<\sigma_0/16$
for all $\tau\in {\rm T}(A).$

%Let $\sigma=\min\{\ep^2 r_2/2^{17}, \dt_1/8, r_1\cdot r_2/2^8\}.$
Let $\mathfrak{f}_{a_0}>0$ be as 
%given 
in  Definition \ref{DNtr1div}.
%part of the definition for
%$A\in {\cal D}.$
Set
$$\sigma=\mathfrak{f}_{a_0}\cdot \min\{\ep^2/2^{17}(m+1), \dt_1/8, r_1/2^7(m+1),\sigma_0/16\}.$$
By \eqref{scomp-3+3},
\beq
\tau(f_{\ep^2/2^{12}}(c_1)) +\tau(e_2)<{{\tau(f_{\ep^2/2^{12}}(c_1)) +r/18}}<\tau(f_{\dt_2}(b_2))\rforal \tau\in {\mathrm{T}}(A).
\eneq
%By \eqref{scomp-3+2}  and
Then, by 7.5 of  \cite{CP},
% (see 3.8 of \cite{HL}), 
there are $z_1, z_2,...,z_K\in A$  and $b'\in A_+$ such that
\beq\label{scomp-6-1}
&&\|f_{\ep^2/2^{12}}(c_1)-\sum_{j=1}^Kz_j^*z_j\|<\sigma/4\andeqn
\\\nonumber
&&
\|f_{\dt_2}(b_2)-(b'+e_2+\sum_{j=1}^Kz_jz_j^*)\|<\sigma/4.
\eneq
%Let
% $$
% {\cal F}=\{a, b, c, e, c_1, b_1, b', e_1\}
 %\cup\{x_i, x_i^*:1\le i\le m\}
 %\cup\{z_j, z_j^*: 1\le j\le K\}.
 %$$
 Since $A\in {\cal D},$
%Since $A\in {\cal C}_1,$
%for any $\eta>0,$ 
there exist \SCA s\,   $A_0$, $D\subset A,$ with   $D\in {\cal C}_0'$
and  $A_0\perp D,$
%and
%with the form $C=PM_r(C(X))P,$ where $r\ge 1$ is an integer, $X$ is a compact subset of a finite CW complex with
%dimension $d(X)$ and $P\in M_r(C(X))$ is a projection
such that
\beq\label{scomp-6}
&&\|f_{\ep^2/2^{14}}(c_1)- (f_{\ep^2/2^{14}}(c_{0,2})+f_{\ep^2/2^{14}}(c_2))\|<\sigma,\\\label{scomp-6+1}
&&\|f_{\dt_1/2}(b_2)-(f_{\dt_1/2}( b_{0,3})+f_{\dt_1/2}(b_3))\|<\sigma,\\\label{scomp-6+2}
&&\|f_{\dt_2/4}(b_2)-(f_{\dt_2/4}(b_{0,3})+f_{\dt_2/4}(b_3))\|<\sigma,\\\label{scomp-6+3}
&&\|e_i-(e_{0,i}+e_{1,i})\|<\sigma, i=1,2,  \andeqn e_{2,1}e_{1,1}=e_{1,1}
\\\label{scomp-6+4}
%\|px-xp\|&<&\eta\tforal x\in {\cal F},\\
%{\rm dist}(pxp, C)&<&\eta\tforal x\in {\cal F}\\
%\label{scomp-6+2}
%{d(X)+1\over{\rm rank}P(\xi)}&<&1/256(m+1)\tforal \xi\in X\andeqn\\
&&c_0 \lesssim e_0,\\
&& \|a_0-(a_{0,0}+ a_{1,0})\|<\sigma, \andeqn\\\label{scomp-6+5}
&& \tau(f_{1/4}(a_{1,0}))\ge \mathfrak{f}_{a_0}\rforal \tau\in \overline{{\mathrm{T}}(D)}^\mathrm{w},
\eneq
where $c_0\in A_0$ is a strictly positive element of $A_0,$
and where  $a_{0,0}, b_{0,3}, c_{0,2}, e_{0,1}\in {{(A_0)_+}}$ and $a_{1,0}, b_3, c_2, e_{1,1}\in {{D_+}}.$  
{{By}}
% \eqref{scomp-5} and 
\eqref{scomp-6-1},  we may also obtain $z_j', z_j' , x_i'\in D$ {{and}}  $ b''\in D_+,$ 
 $i=1,2,...,m$ and $j=1,2,...,K,$
  such that
\beq\label{scomp-9}
&&\|\sum_{i=1}^m (x_i')^*e_{1,1}x_i'-f_{1/16}(a_{1,0})\|<\sigma,\\\label{scomp-9-}
\vspace{-0.8in} &&\|f_{\ep^2/2^{14}}(c_2)-\sum_{j=1}^K (z_j')^*z'_j\|<\sigma \andeqn\\\label{scomp-9+}
&&\|f_{\dt_2}(b_3)-(\sum_{j=1}^K z_j'(z_j')^*+e_{2,1}+b'')\|<\sigma.
\eneq
%where $0\le a_{0,0},  b_{0,2}, b_{0,3}, c_{0,0}, e_{0,1} \le 1$ are in $A_0,$
%$0\le a_{1,0}, b_{2}, b_{3}, c_2, e_{1,1}\le 1$ are in $D,$
%$z_j', z_j' , x_i',  b''\in D,$ ($i=1,2,...,m$ and $j=1,2,...,K$)

%By choosing sufficiently small $\eta,$ we obtain $b_3, c_2, b''\in C_+,$ a projection $q_1\in C,$
%$y_1, y_2,...,y_m\in C,$ $ z_1', z_2',...,{z_K'}\in C$ such that
%\beq\label{scomp-7}
%&&\|pcp-c_2\|<\sigma,\,\,\, \|f_{\ep^2/2^{14}}(pcp)-f_{\ep^2/2^{14}}(c_2)\|<\sigma,\\\label{scomp-7+3}
%&&\|pb_2p-b_3\|<\sigma,\,\,\,\|f_{\dt_2}(pb_2p)-f_{\dt_1}(b_3)\|<\sigma,\\\nonumber
%&&\|f_{\dt_2/4}(pb_2p)-f_{\dt_2/4}(b_3)\|<\sigma,
%\\\label{scomp-8}
%&&\|pe_1p-q_1\|<\sigma,\,\,\,\|\sum_{i=1}^m y_i^*q_1y_i-p\|<\sigma
%\eneq
%and,  (using (\ref{scomp-6-1})),  such that
%\beq\label{scomp-9}
%\|f_{\ep^2/2^{14}}(c_2)-\sum_{j=1}^K (z_j')^*z'_j\|<\sigma \andeqn\\\label{scomp-9+}
%\|f_{\dt_2}(b_3)-(\sum_{j=1}^K z_j'(z_j')^*+q_1+b'')\|<\sigma.
%\eneq
%]\%%{{By  \eqref{scomp-6+3}  and \eqref{scomp-6+4},
%\%%\beq\label{scomp-9-}
%%\%\tau(e_{1,1})>\tau(e_1)-\tau(e_0)-\sigma>\sigma_0-\sigma_0/16-\sigma\ge \sigma\rforal \tau\in {\rm T}(A).
%\%\eneq
%}}
%%%%%%%%%
%%\iffalse
{{Note that, by \eqref{scomp-9} and by Lemma \ref{Lrorm},}}
\beq
{{\la f_{1/4}(a_{1,0})\ra\le \la \sum_{i=1}^m (x_i')^*e_{1,1}x_i'\ra \le  m\la e_{1,1}\ra.}}
\eneq
%%\fi
 {{Then, by \eqref{scomp-6+5},}}
\beq\label{scomp-n9}
t(e_{2,1})\ge {\rm d}_t(e_{1,1})\ge \mathfrak{f}_{a_0}/(m+1)>2\sigma \tforal  t\in  {\mathrm{T}}(D).
\eneq
%\fi
%%%%%%%%%%%%%%%%%%%%%%
%Note that, by (\ref{scomp-8}),
% and (\ref{scomp-6+2}),
%\beq\label{scomp-10}
%$${\rm rank}(q_1(\xi))\ge (1/m)\cdot  {\rm rank} (p(\xi))$$
%{\rm rank P}(\xi)/m\ge 256(d(X)+1)\tforal \xi\in X.
%\eneq
%for all irreducible representations $\xi$ of $C.$
%Therefore
%\beq\label{scomp-10+1}
%$$t(q_1)\ge 1/m\tforal t\in T(C).$$
%\eneq
%It follows that
%\beq\label{scomp-12}
%t(q_1)-2\sigma> 9(d(X)+1)/m\tforal t\in T(C).
%\eneq

Therefore, by \eqref{scomp-9}, \eqref{scomp-9-} and \eqref{scomp-9+},
% \eqref{scomp-n9} and \eqref{scomp-9+},
\vspace{-0.14in} {$$\begin{aligned}{\rm d}_t(f_{\ep^2/2^{13}}(c_2))
  &\le  t(f_{\ep^2/2^{14}}(c_2))
  \le \sigma+\sum_{j=1}^Kt((z_j')^*z_j') \\
&= \sigma+\sum_{j=1}^Kt(z_j'(z_j')^*)
\le  t(e_{2,1})-\sigma+\sum_{j=1}^Kt(z_j'(z_j')^*)\\
&\le  t(f_{\dt_1}(b_3))
\le {\rm d}_t(f_{\dt_1/2}(b_3))
\end{aligned}$$}
%\beq\label{scomp-11}
%\hspace{0.3in}d_t(f_{\ep^2/2^{13}}(c_2))
%  +9d(X)/m\\
  %&\le & t(f_{\ep^2/2^{14}}(c_2))\le \sigma+\sum_{j=1}^Kt((z_j')^*z_j') \\
  %+9d(X)/m\\
%&=& \sigma
%+9 d(X)/m
%+\sum_{j=1}^Kt(z_j'(z_j')^*)
%\le  (t(q_1)-\sigma)+\sum_{j=1}^Kt(z_j'(z_j')^*)\\
%&\le & t(f_{\dt_1}(b_3))\le
%d_t(f_{\dt_1/2}(b_3))
%\eneq
for all $t\in {\mathrm{T}}(D).$
It follows  by Proposition \ref{Pcom4C}
%\ref{Csemiproj}
that
%\beq\label{scomp-13}
$$f_{\ep^2/2^{13}}(c_2)\lesssim f_{\dt_1/2}(b_3).$$
%\eneq
By (\ref{scomp-6+2}) and Lemma 2.2 of  \cite{RorUHF2} {{(note $\sigma<\dt_1/8$ and $b_{0,3}\perp b_3$)}},
%\cite{RrUHF},
%\beq\label{scomp-14}
\vspace{-0.1in} $$f_{\dt_1/2}(b_3)\lesssim f_{\dt_2/4}(b_2)\le b_2.$$
%\eneq
%By (\ref{scomp-6+2}),
%\beq\label{scomp-15}
%f_{\ep^2/2^{11}}(pcp)\lesssim f_{\ep^2/2^{12}}(c_2)\le pb_2p.
%\eneq
It {{then}} follows (also {{by}} Lemma 2.2 of  \cite{RorUHF2}) that
%Lemma 5.3 of \cite{Lin-men3} and
%(\ref{scomp-15})  that
\vspace{-0.12in} \beq\nonumber
%label{scomp-16}
\hspace{0.2in}f_{\ep/2}(c) &\lesssim &  f_{\ep^2/2^{11}}(c_{0,2}+c_2)\lesssim c_0\oplus f_{\ep^2/2^{11}}(c_2)
\\\nonumber
&\lesssim & e+b_2  \lesssim b_1\lesssim b.
\eneq
We also have
\beq\nonumber
%\label{scomp-17}
f_{\ep}(a)\lesssim f_{\ep/2}(f_{\ep/16}(a))\lesssim f_{\ep/2}(c)\lesssim b.
\eneq
Since this holds for all $1>\ep>0,$ by 2.4 of  \cite{RorUHF2}, we conclude that
%\vspace{-0.1in} $$
$a\lesssim b.$
%$$

If we only have $d_\tau(a)\le d_\tau(b)$ for all $\tau\in \overline{{\mathrm{T}(A)}}^\mathrm{w},$ since $A$ is stably 
projectionless as mentioned above, as shown at the beginning of the proof, for any $\ep>0,$ 
$\tau(f_{\ep/2}(a))<d_\tau(b)$ for all $\tau\in \overline{{\mathrm{T}(A)}}^\mathrm{w}.$ From what has been proved, 
$f_\ep(a)\lesssim b$ for all $\ep>0.$ Therefore $a\lesssim b.$ This shows the case that $a, b\in A_+.$
For $a, b\in M_n(A)_+$ for some $n\ge 1,$ one notes that, by  \ref{PtadMk}, $M_n(A)\in D.$ Therefore this case is easily reduced 
to the case that $n=1.$  In general, if $a, b\in (A\otimes {\cal K})_+,$ then, for any $\ep>0,$  by the last part of \ref{Lalmstr1}, we may assume that  $f_{\ep/16}(a)$ is in  $M_n(A)$ for some $n\ge 1.$ Hence $\tau\mapsto \tau(f_{\ep/16}(a))$ is bounded and continuous. 
Since $\overline{{\mathrm{T}(A)}}^\mathrm{w}$ is compact, one concludes, as shown above, for some small 
$\dt_1>0,$ $\tau(f_{\ep/16}(a))
<\tau(f_{\dt_1}(b))$ for all $\tau\in \overline{{\mathrm{T}(A)}}^\mathrm{w}.$  As mentioned above, 
we may also assume that $f_\dt(b)\in M_n(A)$ (with possibly larger $n$). Thus, we conclude 
that $f_\ep(a)\lesssim f_\dt(b)\lesssim b.$ It follows that $a\lesssim b.$

\end{proof}

%\begin{cor}\label{Ca=ped}
%Let $A\in {\cal D}.$ Then $A=\mathrm{Ped}(A).$
%\end{cor}

%\begin{proof}
%By  \ref{PD0qc}, assuming $A$ is non-zero,  ${\mathrm{T}}(A)\not=\O$ and $0\not\in \overline{{\mathrm{T}}(A)}^\mathrm{w}.$  Choose a nonzero element $e\in \mathrm{Ped}(A)_+.$
%Then $\inf\{\mathrm{d}_\tau(e): \tau\in \overline{{\mathrm{T}}(A)}^\mathrm{w}\}=:d>0.$  Choose a strictly positive element $a\in A_+$ with
%$\|a\|=1.$   Then, there is an integer $n>0,$ such that $n\mathrm{d}_\tau(e)>\mathrm{d}_\tau(a).$
%By \ref{Comparison}, applied to ${\rm M}_n(A)$ which is in ${\cal D}$ by \ref{PtadMk}, $a\lesssim b,$
%where $b=\diag(\overbrace{e,e,...,e)}.$  It follows that there are $c_{1,k}, c_{2,k},...,c_{n,k}\in A$ such 
%that 
%\beq
%\sum_{i=1}^n c_{1,k}^*ec_{1,k}\to a, \,\,\,{\rm as}\,\,\, k\to \infty.
%\eneq

%strict comparison (\ref{Comparison}) above, one shows that $a\in \mathrm{Ped}(A).$ Consequently, $A=\mathrm{Ped}(A).$
%\end{proof}

\begin{df}\label{DM0}
 Let us denote by ${\cal M}_0$  the class of (non-unital) simple \CA s which are inductive limits
of {{sequences of}}  \CA s in ${\cal C}_0^0.$  We stipulate  that the  maps {{in the sequence}}  be injective and
preserve
%maps 
strictly positive elements, i.e., each map {{should send}} strictly positive elements to strictly 
positive 
% to strictly positive 
elements. {{In fact,  a decomposition with such maps can always be chosen.}}

Every algebraically simple \CA\, $A$ in ${\cal M}_0$ is in ${\cal D}_0.$
To see this, write $A=\lim_{n\to\infty} (C_n, \phi_n),$ where each $C_n$ {{is in}}  ${\cal C}_0^0$ and
$\phi_n: C_n\to C_{n+1}$ is a \hm\, which preserves 
% maps strictly positive elements to 
strictly positive elements
and is injective.

Let $a_1\in C_1$ be a strictly positive element with $\|a_1\|=1.$ Then $a_n=\phi_{1, n}(a_1)\in C_n$ is
a strictly positive element of $C_n,$ $n=1,2,....$  Then $a=\phi_{n, \infty}(a_n)$
%Then it is easy to see that $a$ 
is a strictly positive element  with $\|a\|=1.$
For any $n, $  since $0\not\in {\overline{{\rm T}(C_n)}}^{\rm w}$ (see Definition \ref{Dbuild1}),
%see \ref{compactrace} and \ref{Tpedersen}),
\beq\nonumber
r_n:=\inf\{\tau(a_n): \tau\in {\overline{{\rm T}(C_n)}}^{\rm w}\}>0.
%=\inf\{\tau(a_n): \tau\in {\rm T}(C_n)\}>0.
\eneq
%Note s
Since  $\phi_{m,n}$ is a \hm\,  preserving strictly positive elements, 
%which  maps strictly 
%positive elements to strictly 
%positive elements, 
$t\circ \phi_{m,n}\in {\mathrm{T}}(C_m)$ for all $t\in {\rm T}(C_n)$ and for all $n\ge m.$ 
Thus, $r_n\ge r_m$ for all $n\ge m.$ 

Since $A$ is algebraically simple {{and $f_{1/4}(a)\not=0,$}}
there are $x_1, x_2,...,x_k\in A$ such that
\beq\nonumber
\sum_{i=1}^k x_i^*f_{1/4}(a)x_i=a.
\eneq
{{Set}} $M=2k\max\{\|x_i\|:1\le i\le k\}.$
For some $m\ge 1,$ there are $y_1, y_2,...,y_k\in C_m$ such that
\beq\nonumber
\|\sum_{i=1}^ky_i^*\phi_{1, m}(f_{1/4}(a_1))y_i-a_m\|<r_1/2.
\eneq
We may assume that $\|y_i\|\le 2\|x_i\|,$ $i=1,2,...,k.$ 
Since $r_1\le r_m,$ 
this implies that 
\beq\nonumber
\tau(\phi_{1,m}(f_{1/4}(a_1)))\ge (r_m/2)/ 2M\tforal \tau\in {\rm T}(C_m).
\eneq
Put
$$
\mathfrak{f}_{a}=\inf\{\tau(\phi_{1,m}(f_{1/4}(a_1))):\tau\in {\rm T}(C_m)\}.
%\inf\{\tau(f_{1/4}(a)): \tau\in 
%{\mathrm{T}}(A)\}/2.
$$
Note since $t\circ \phi_{{{m,n}}}\in {{\rm T}}(C_n)$ for all $t\in {\rm T}(C_n),$
$$
t(\phi_{m,n}(f_{1/4}(a_1)))\ge \mathfrak{f}_{a}\rforal t\in {\mathrm{T}}(C_n).
$$

%We may also assume 
%
%It is standard to show that, for any fixed $m\ge 1,$ there exists $N(m)\ge 1$ such that, for  any $n\ge N(m),$
%$$
%t(\phi_{n,m}(f_{1/4}(a)))\ge \mathfrak{f}_{a}\rforal t\in {\mathrm{T}}(C_n).
%$$
From this, one concludes that $A\in {\cal D}_0$ {{(with $\phi=0$ and $\psi={\id}_A$ in Definition \ref{DNtr1div}).}}
\end{df}

\begin{df}\label{DW}
Recall {{that}} ${\cal W}$ is an inductive limit of \CA s 
%of the form 
as described in \eqref{ddraz} (see 
{{\cite{Raz},}} \cite{Tsang},
and \cite{Jb}) which has ${\rm K}_0({\cal W})={\rm K}_1({\cal W})=\{0\}$ and
{{has}}  a unique tracial state. 
Moreover, ${\cal W}=\overline{\bigcup_{n=1}^{\infty} C_n},$
where $C_n\subset C_{n+1}$ and each $C_n$ is in 
% is the closure of a union of an increasing sequence of \CA s in 
${\cal R}_{\mathrm{az}}$ (in
fact as in  \eqref{ddraz}) and inclusion preserves the strictly positive elements. 
 In particular, ${\cal W}\in {\cal M}_0$ and ${\cal W}\in {\cal D}_0.$
Furthermore, we may also assume 
that {{(see \ref{Dbuild1}  and \ref{Dboundedscale} for $\lambda_s$)}}
$$
\lambda_s(C_n)=\inf\{\tau(e_{C_n}): \tau\in {\rm T}(C_n)\}\to 1,
$$
where $e_{C_n}$ is a strictly positive element of $C_n$
(for example, $\lim_{i\to\infty}a_i/(a_i+1)=1$ {{as shown}} in \cite{Jb}).
{{By \cite{Raz}, ${\cal W}$ is unique with these properties.}}
\end{df}

%\begin{df}\label{DTRD0}
%Let $A\in {\cal D}_{00}$ (or in ${\cal D}_{10}$).  We say that $A\in {\cal D}_0$ (${\cal D}_1$) if, in addition, $A$ has stable rank one.

%\end{df}

\section{Tracial approximate divisibility}

%\begin{df}\label{Dstrongcuntz}
%Let $A$  be a \CA\, and let $a,\, b\in A_+.$ We write
%$a\lessapprox b.$ If, for each $1>\ep>0,$ there is a unitary $u\in U({\widetilde A})$ (or in $A$ if $A$ is unital) such that
%$$
%u^*f_{\ep}(a)u\in \overline{bAb}.
%$$

%
%Suppose that $a\lessapprox b.$
%For each $n,$ there are unitaries $u_n\in U({\widetilde A})$
%such that $u_n^*f_{1/n}(a)u_n\in \overline{bAb}.$
%There are  integers $m(n)\ge 1$ such that
%$$
%\|z_n(u_n^*f_{1/n}(a)^{1/2}af_{1/n}(a)^{1/2}u_n)^{1/2}-(u_n^*f_{1/n}(a)au_n)^{1/2}\|<1/n
%$$
%where $z_n=f_{1/m(n)}(b),$ $n=1,2,....$
%There are $g_n\in C_0((0,\infty))$ with $\|g_n\|=1$ such that
%$$
%\|g_n(b)bg_n(b)-z_n^2\|<1/2^n.
%$$
%It follows that
%$$
%\lim_{n\to\infty}\|(u_n^*f_{1/n}(a)^{1/2}af_{1/n}(a)^{1/2}u_n)^{1/2}g_nbg_n(u_n^*f_{1/n}(a)^{1/2}af_{1/n}(a)^{1/2}u_n)^{1/2}-a\|=0.
%$$
%Thus $a\lesssim b.$
%\end{df}

\begin{df}\label{Dappdiv}
Let $A$ be a {{(non-unital)}} $\sigma$-unital simple \CA.
Let us  say that $A$ is  (non-unital)  tracially approximately divisible
if the following property holds:

For any $\ep>0,$ any finite subset ${\cal F}\subset A,$ any
$b\in A_+\setminus \{0\},$ and any
integer $n\ge 1,$ there are $\sigma$-unital \SCA s
$A_0, {{A_1\otimes e_{1,1}\subset {\rm M}_n(A_1)\subset}}A$ such that $A_0\perp {\rm M}_n(A_1),$
$$
{\rm dist}(x, B_d)<\ep \rforal x\in {\cal F},
$$
where $B_d=A_0+A_1\otimes 1_n,$
%\subset B\subset A,$ $B=A_0\oplus  {\mathrm{M}}_n(A_1),$ and 
%\vspace{-0.1in} \beq\label{Dappdiv-1}
%&&B=A_0\oplus  {\mathrm{M}}_n(A_1),\\
%\overbrace{A_1\oplus A_1\oplus\cdots \oplus A_1}^n,\\
%&&B_d=\{(x_0, \diag(\overbrace{x_1,x_1,...,x_1}^n)): x_0\in A_0, x_1\in A_1\}
%\eneq
and $a_0\lesssim  b,$ where $a_0$ is a strictly positive element of $A_0.$
\end{df}

{{In the unital case, this definition is equivalent to 
%Note that, unital version of Definition \ref{Dappdiv} first appeared in 
5.3 of \cite{LinTAI}. (Note, as seen in the proof  of 5.4 of \cite{LinTAI}, 
the unit of the  finite dimensional \SCA\, is required to be $1-q$ there.)}}

\begin{lem}\label{Lmatrixpullb}
Let $D$ be a (non-unital) separable simple \CA\, which  can be written as
$D=\lim_{k\to\infty}(D_k, \phi_{{k}}),$ where each $D_k\in {{\cal C}_0^{0}}'.$
Let $K\ge 1$ be an integer,
let $\ep>0,$ and let ${\cal F} \subset D_{n}$  {{be a finite subset}}  for some $n\ge 1.$  
%{\red{T}}
{{T}}here exist
an integer $m\ge n,$ a  \SCA\, $D_m'={\mathrm{M}}_K(D_m'')\subset D_m,$ 
where $D_m''$ is a hereditary \SCA\, of $D_m,$ and a finite subset ${\cal F}_1\subset D_m''$ 
{{such that}}
%%satisfy the following:
\beq\label{Lmatripb-1}
{{\rm{dist}}}(\phi_{n, m}(f), {\cal F}_1\otimes 1_K)<\ep\rforal f\in {\cal F}.
\eneq
%where $d(x)=\diag(\overbrace{x,x,...,x}^K).$
If each $D_k$ {{is just assumed to belong to}} ${\cal C}_0',$ then
there {{exist}} an integer $m\ge n,$  {{a \SCA\,}} $D_m'={\mathrm{M}}_K(D_m'')\subset D_m,$ 
%%{\red{
{{where $D_m''$ is a hereditary \SCA\, of $D_m,$}} and a finite subset ${\cal F}_1\subset D_m''$
%%satisfy the following:
{{such that}}
\beq\label{Lnz1702-1}
\|\phi_{n,m}(f)-(r(f)+g_f\otimes 1_K)\|<\ep\rforal f\in {\cal F},
\eneq
where $r(f)\in \overline{eD_me}$  and $g_f\in {\cal F}_1 $ for all $f\in {\cal F},$    $e\in (D_m)_+,$  {{and}} $e\lesssim e_d,$
where $e_d$ is a strictly positive element of $D_m''.$

\end{lem}

\begin{proof}
{{We may assume that ${\cal F}$ is in the unit ball of $D.$}}
Consider first the case $D_{{k}}\in {{\cal C}_0^{0}}'$ for all ${{k}}.$
{{By \ref{inductived0}, \wilog, we may assume that $D_k\in {\cal C}_0^0.$}}
One notes that ${\rm K}_0(D)={\rm K}_1(D)=\{0\}.$
One also notes that $D\otimes Q$ is  an inductive limit of \CA s in ${{\cal C}_0^{0}}.$
{{Moreover,  $D$ and $D\otimes Q$ have the same (lower semicontinuous) traces and the same tracial states. 
It follows from 6.2.4 of \cite{Rl} (see Theorem 1.2 of  \cite{Jb}, also \cite{Raz} and \cite{Tsang})}}
that 
%It follows from \cite{Raz} (see also \cite{Tsang} and \cite{Rl}) that 
$D\cong D\otimes Q.$ 
Fix $n$ such that ${\cal F}\subset D_n.$
\Wlog, we may assume that ${\cal F}\subset D_n^{\bf 1}.$
%It follows from \cite{aTz} that $D\cong D\otimes {\cal Z}.$
%By \ref{TZ0stable},
%$D\cong D\otimes {\cal Z}^o.$  We will write $D=D\otimes {\cal Z}^o.$
%As proved \ref{Ttapprdvi}, t
{{Note also that $Q$ is self-absorbing. 
Hence the map $a\mapsto a\otimes 1_Q$ (for all $a\in D$) is approximatly unitarily equivalent 
to the identity map.}}
Therefore, there exists  a \SCA\, $C$ of $D$ with $C \cong D$
%an inner automorphism $\af$ on $D\otimes {\cal Z}^o$
such that
%\vspace{-0.12in}
\beq\label{Lmatripb-2}
%\af(a)=\sum_{i=1}^Ka\otimes a_i\andeqn\\
%{\rm dist}(a, \af(a))
\|\phi_{n, \infty}(a)-{{c(a)\otimes 1_K}}
%\af({\blue{a}})
\|<\ep/4\rforal a\in {\cal F},
\eneq
%Therefore $\af({\cal F})\subset C\otimes M_K\subset D$
%and
%where
%\vspace{-0.12in} $$
%\af(a)=c(a)\otimes 1_K=\diag(\overbrace{c(a),c(a),...,c(a)}^K)\in C\otimes {\mathrm{M}}_K
%$$
for all $a\in {\cal F}$ and for some $c(a)\in C\otimes e_{11}\subset C\otimes M_K\subset D.$

{{For each $a\in {\cal F},$ there exists  $n_1\ge n$ 
%$\af(a)'\in D_n$
and $c(a)'\in D_{n_1}$
such that 
%$\phi_{n, \infty}(\af(a)')=\af(a)$ and 
$\|\phi_{n_1, \infty}(c(a)')-c(a)\otimes e_{11}\|<\ep/16K^2.$
\Wlog, we may assume that $\|c(a)\|, \|c(a)'\|\le 1.$
To simplify notation, \wilog, let us assume that there is $c_0\in C_+$ with
$\|c_0\|=1$ such that
$c_0c(a)=c(a)c_0=c(a)$ for all $a\in{\cal F}.$
Consider {{the}} \SCA\, $B=C_0\otimes {\mathrm{M}}_K,$ where $C_0$ is the \SCA\, of $C$ generated by
$c_0.$  {{Since $C$ is stably projectionless, ${\rm sp}(c_0)=[0,1].$}} 
Then $C_0\otimes {\mathrm{M}}_K{{\cong}} 
%is a quotient of $
C_0((0,1])\otimes {\mathrm{M}}_K.$
%(or $C([0,1])\otimes {\mathrm{M}}_K$).  
%Let $B=C_0\otimes {\mathrm{M}}_K.$  
%%Note that we may identify $c_0$ with $c_0\otimes e_{11}.$
%(or $C([0,1])\otimes {\mathrm{M}}_K$)
%as just mentioned and let $q: B\to C_0\otimes {\mathrm{M}}_K$ be  the standard quotient map.
{{Fix a finite subset ${\cal G}\subset B$ 
which contains $\{c_0\otimes e_{ij}: 1\le i,j\le 1\},$ and $0<\dt<\ep/8.$}} Since 
%$B$ 
$B$ is semiprojective (see, for example, \cite{ELP1}), there is a \hm\,
$H: B\to D_{m_1}$ for some $m_1\ge n_1\ge n$ such that
\beq\label{18semiproj-1}
\|\phi_{m_1,\infty}\circ H(g)-g\|<\dt/K^2\rforal g\in {\cal G}.
\eneq
  {{Set $c_{00}=H(c_0\otimes e_{11}).$}}
 %% Let ${\cal F}_1=\phi_{n, m_1}({\cal F})$ and let ${\cal F}_2'=\{\af(a)', c(a)': a\in \phi_{n, \infty}({\cal F})\}.$}}

 Fix $m>m_1.$ Define $D_m''=
 \overline{\phi_{m_1, m}(c_{00})D_{m}\phi_{m_1, m}(c_{00})}.$
 Then {{the}} \SCA\, $D_m'$ generated by  $D_m''$ and $\phi_{m_1, m}(H(B))$
 is isomorphic {{to}} $D_m''\otimes M_K\subset D_ m.$ }} 
 {{Define 
 \beq\nonumber
%&& x(a)=\phi_{m_1,m}(c_{00}\otimes e_{11})\phi_{n,m}(c(a)')\phi_{m_1,m}(c_{00}\otimes e_{11})
%\andeqn\\
&&x_{1,j}(a)=(c_{00}\otimes e_{j1})\phi_{n_1,m_1}(c(a)')(c_{00}\otimes e_{1j})\in H(B),
\andeqn\\
&& y_{1,j}(a)=\phi_{m_1,m}(c_{00}\otimes e_{j1})\phi_{n_1,m}(c(a)')\phi_{m_1,m}(c_{00}\otimes e_{1,j})
,\,\,j=1,2,...,K,
\eneq
for all  $a\in {\cal F}.$
Note $y_{1,j}=\phi_{m_1, m}(x_{1,j}(a)).$ Moreover,
 one may write 
$
\sum_{j=1}^N x_{1,j}(a)=\sum_{j=1}^Nx_{11}(a)\otimes e_{ii}=x_{11}(a)\otimes 1_K.
$}}
{{By \eqref{18semiproj-1},  for $a\in {\cal F},$
 \beq\nonumber
 %\phi_{m,\infty}(\phi_{m_1,m}((c_{00}\otimes 1_K)(\af(a)')(c_{00}\otimes 1_K))=
\hspace{-0.2in}\phi_{m_1, \infty}(\sum_{j=1}^K x_{1,j}(a))
%\phi_{m, \infty}(\sum_{j=1}^Ky_{1,j}(a))
&\approx_{\dt}&\sum_{j=1}^K(c_0\otimes e_{j1})
%(\af(a)\otimes e_{11})
\phi_{n_1, \infty}(c(a)')(c_0\otimes e_{1,j})\\\nonumber
&\approx_{\ep/16}&\sum_{j=1}^K (c_0\otimes e_{j1})(c(a)\otimes e_{11})(c_0\otimes e_{1j})\\\label{102-n}
&=&\sum_{j=1}^K(c_0c(a)c_0)\otimes e_{jj}=\sum_{j=1}^Kc(a)\otimes e_{jj}=c(a)\otimes 1_K.
\eneq}}
%{\blue{ 
%Thus, for sufficiently large $m>m_1,$ 
%\beq
%\|\phi_{m_1, m}(x_{1,1}(a)\otimes 1_K)-\phi_{n, m}(\af(a)')\|<\ep/4\rforal a\in {\cal F}.
%\eneq
{{Define ${\cal F}_1=\{y_{1,1}(a): a\in {\cal F}\}\subset D_m''\otimes M_K\subset D_m.$ 
%Moreover, 
Then, by \eqref{Lmatripb-2} and \eqref{102-n} above,  \wilog,  choosing a larger $m$ if necessary, 
we may assume that
\beq
\|\phi_{m_1, m}(\phi_{n, m_1}(a))-\phi_{m_1,m}(x_{11}(a)\otimes 1_K)\|<\ep/2\tforal a\in {\cal F}.
\eneq
It follows that
\beq
{\rm dist}(\phi_{n,m}(a), {\cal F}_1\otimes 1_K)<\ep\rforal a\in {\cal F}.
\eneq
This proves the first part of the statement.}}

%Note that one may write $\sum_{j=1}^k x_{1,j}(a)=x_{1,1}(a)\otimes 1_K.$
%be the pre-image of $c_0$ under $q.$
%The lemma follows if we choose a sufficiently large $m\ge m_1$ and
%$D_{m}''=\overline{\phi_{m_1, m}(c_{00})D_{m}\phi_{m_1, m}(c_{00})}$
%as well as $D_m'=D_m''\otimes {\mathrm{M}}_K.$

In the case  $D_n\in {\cal C}_0',$ by \cite{aTz},
$D\otimes {\cal Z}\cong D.$
In ${\cal Z}$ (see the proof of Lemma 2.1 of \cite{Rlz}, and also
Lemma  4.2 of \cite{Rrzstable}),  there are $e_1, e_2,...,e_K, {{d}} \in {\cal Z}_+$
such that $\sum_{j=1}^K e_j+d=1_{\cal Z},$ $e_1, e_2,...,e_K$ are mutually orthogonal, 
 $d\lesssim e_1,$ {{and}}
there exist $w_1, w_2,...,w_K\in {\cal Z}$ such that $e_j=w_jw_j^*$ and
$e_{j+1}=w_j^*w_j.$   Moreover, as in  the proof of Lemma 4.2
of \cite{Rrzstable}, {{since ${\cal Z}$ has stable rank one}}, there is a unitary $v\in {\cal Z}$ such that
$v^*dv\le e_1.$
\Wlog,  identifying $D$ with $D\otimes {\cal Z},$ we may assume
that  $\phi_{n, \infty}(x)=y\otimes 1$ for some $y=y(x)\in D$ 
%and 
for every element $x\in {\cal F}.$
Let $d'=c_0\otimes d, $ $v'=c_0^{1/2}\otimes v,$ $e_j'=c_0\otimes e_j, $ $w_j'=c_0^{1/2}\otimes w_j,$
$j=1,2,...,K.$ Note that $d'+\sum_{j=1}^K e_j'=c_0.$
With sufficiently large $m$ and with a standard perturbation,
we may assume that $d', v', e_j', w_j'\in \phi_{m, \infty}(D_m),$ $j=1,2,...,m,$ and
$\phi_{n,\infty}(x)$ commutes with $d', e_j'$ and $w_j'$ for all $x\in {\cal F}.$
With possibly even larger $m,$  \wilog, there are $d'', v'', e_j'', w_j''\in  D_m$
such that
$d''+\sum_{j=1}^K=c_0',$ $d''=v''(v'')^*,$ $(v'')^*v''\le e_1'',$
$e_1'', e_2'',..., e_K''$ are mutually orthogonal, $(w_j'')(w_j'')^*=e_j'$ and
$e_{j+1}''=(w_j'')^*(w_j''),$ where $c_0'\in (D_m)_+$ {{is}} such
that $c_0'\phi_{n,m}(x)=\phi_{n,m}(x)c_0'=\phi_{n,m}(x)$ for all $x\in {\cal F}$
and
\beq
\|[\phi_{n,m}(x), y]\|<\ep/16K^2\rforal x\in {\cal F}
\eneq
and $y\in \{d''^{1/2}, d'',  v'', e_j'', w_j'', j=1,2,...,K\}.$
Define  $D_m''=\overline{e_j''D_me_j''},$ $r(f)=(d'')^{1/2}\phi_{n,m}(d'')^{1/2}$ and
${\cal F}_1=\{e_1''\phi_{n,m}(x)e_1'': x\in {\cal F}\}$ and identify
$d(e_1''\phi_{n,m}(x)e_1'')$ with
$$
\sum_{j=1}^Ke_j''\phi_{n,m}(x)e_j''\in {\mathrm{M}}_K(D_m'')\,\,\,\,\rforal x\in {\cal F}.
$$
{{The conclusion of the}} lemma follows.
\end{proof}

\begin{thm}\label{TDefsimple}
Let $A\in {\cal D}$ (or ${{A\in}}{\cal D}_0$)  be a separable \CA\, with $A={\rm Ped}(A).$ 
Then  the following {{statement}} holds:
Let  $a_0$ be a strictly positive element of $A$ with $\|a_0\|=1.$
There exists $1>d_A>0$  {{satisfing}} the following condition:
%with  $C={\rm Per}(C)$ 
For 
any  $\ep>0,$ any finite subset  ${\cal F}\subset A,$ and  any  $b_0\in A_+\setminus \{0\},$ 
there {{exist}} a separable simple \CA\, $D=\lim_{n\to\infty}(D_n, \psi_n),$
where 
%each 
$D_n\in {\cal C}_0$ (or 
%each 
$D_n\in {\cal C}_0^0$) and 
  an ${\cal F}$-$\ep$-multiplicative 
%\SCA\, $D\in {\cal C}_0$ (or ${\cal C}_0^0$), and 
\cpc\,  
%$\phi_0: A\to A,$  
${{\phi}}: A\to D_1$ such that, 
for any $n>1,$  there exist  a \cpc\,  $\Phi_n: A\to A$ and an embedding $j_n: D_n\to A$  with
$\Phi_n(A)\perp j_n\circ (\psi_{1, n}{{\circ}}\phi(A))$ such that
%$\phi_i$ is ${\cal F}$-$\ep$-multiplicative ($i=0,1$), 
\beq
&&\|x-(\Phi_n+j_n\circ \psi_{1, n}\circ \phi)(x)\|<\ep\tforal x\in {\cal F},\\\label{10318913-1}
&&c_n\lesssim  b_0,\tand\\
&&\tau(f_{1/4}(\psi_{1, n}\circ \phi(a_0)))>d_A\tforal \tau\in {\rm T}(D_n),
\eneq
where $c_n$ is a strictly positive element of $\overline{\Phi_n(A)A\Phi_n(A)}.$
Moreover, if ${\rm K}_0(A)=\{0\},$ we may assume that 
$(\psi_{1, n}|_{D_1})_{*0}=0.$
\end{thm}

\begin{proof}
%\%{\blue{As in (the proof of) \eqref{PD0=tad} , it suffices to prove  the above statement with  \eqref{10318913-1} is replaced by 
%\%$\Phi_n(a_0)\lesssim b_0.$}} 
Let $1>\mathfrak{f}_{a_0}>0$ be {{as}} in
Definition  \ref{DNtr1div}.
%\ref{Dtr1div}.
Fix an integer $k_0\ge 1$
such that $(\mathfrak{f}_{a_0})^2>2^{-k_0}.$

{{Replacing}} $a_0$ by $g(a_0)$ for some $g\in C_0((0,1])$ with
$0\le g\le 1,$ we may assume that
\beq\label{TDIV-n1}
\tau(a_0)>\mathfrak{f}_{a_0}\rforal \tau\in {\mathrm{T}(A)}
\eneq
(see \ref{Rfa0}).
Fix any $b_0\in A_+\setminus \{0\}.$
Choose a sequence of {{nonzero}} positive elements
${{(b_n)_{n\ge 1}}}$ 
 {{in $A$}} with 
%which has
 the following property: $b_{1}\lesssim b_0$ {{and}} 
$b_{n+1}\lesssim b_{n,1},$ 
where $b_{n,1}, b_{n,2},....,b_{n,2^{n+k_{0}+5}}$ are mutually orthogonal positive elements
in $\overline{b_nAb_n}$ such that $b_nb_{n,i}=b_{n,i}b_n=b_{n,i},$
$i=0,1,2,...,n,$ and $\la b_{n,i}\ra =\la b_{n,1}\ra,$ $i=1,2,...,2^{n+k_0+3}.$

It should be noted that
\beq\label{TDadiv-0}
\sum_{j=m}^{\infty} \sup\{\tau(b_j): \tau\in \overline{{\mathrm{T}(A)}}^\mathrm{w}\}<(\mathfrak{f}_{a_0})^2/2^{m+5}\rforal m\ge 1.
\eneq

One obtains (see {{also the end of}} \ref{Rsemiproj}) two sequences of \SCA s $A_{0,n}$ and  $D_n$ of $A,$ with $A_{0,n}\perp D_n$
and $D_n\in {\cal C}_0$ (or  ${{D_n\in {\cal C}_0^{0}}}$), {{and}}
 two sequences
of \cpc s $\phi_n^{(0)}: A\to A_{0,n}$  and
$\phi_n^{(1)}: A\to D_n$ with $\|\phi_n^{(i)}\|=1$ ($i=0,1$) satisfying  the following conditions:
\beq\label{TDappdiv-1}
&&\lim_{n\to\infty}\|\phi_n^{(i)}(ab)-\phi_n^{(i)}(a)\phi_n^{(i)}(b)\|=0\rforal a,\, b\in A, \,\,i=0,1,\\\label{TDappdiv-1+}
&&\lim_{n\to\infty}\|a-(\phi_n^{(0)}+ \phi_n^{(1)})(a)\|=0\rforal a\in A,\\\label{TDappdiv-1+2}
&&c_n\lesssim b_n,\\\label{TDappdiv-2}
&&\tau(f_{1/4}(\phi_n^{(1)}(a_0)))\ge \mathfrak{f}_{a_0}\rforal \tau\in {\mathrm{T}}(D_n),
\eneq
 and $\phi_n^{(1)}(a_0)$ is a strictly positive element {{of}} $D_n,$
 where $c_n$ is a strictly positive element of $A_{0,n}$ with $\|c_n\|=1.$ 
 %By \ref{Rsemiproj}, one may assume
% that $D_n\in {\cal C}_0^{(0')}.$
As in the proof of  \ref{PD0qc}, {{it follows that}}
\beq\label{TDappdiv-6}
\lim_{n\to\infty}(\sup\{|\tau(a)-\tau\circ \phi_n^{(1)}(a)|: \tau\in {\mathrm{T}(A)}\})=0\rforal a\in A.
\eneq
Consider the sequence $(a_n)$ defined inductively by
$a_1:=\phi_1^{(1)}(a_0),$   $a_2:=\phi_2^{(1)}(a_1),..., a_n:=\phi_n^{(1)}(a_{n-1}),$
$n=1,2,....$

For fixed $n,$  by \eqref{TDappdiv-1+2}, \eqref{TDadiv-0},  \eqref{TDappdiv-1}, and \eqref{TDappdiv-1+},
\beq\label{TDap-7n-1}
\lim_{k\to\infty}{{(}}\sup\{\tau(f_{1/4}(\phi_k^{(1)}(\phi_{n}^{(0)}(b))){{)}}: \tau\in {\rm T}(D_k)\}{{)}}\le (\mathfrak{f}_{a_0})^2/2^{n}
\eneq
for any $0\le b\le 1.$  It follows that, for fixed $n,$  and any {{fixed}} $m\ge n,$ 
\beq\label{TDap-7n-2}
\lim_{k\to\infty}(\sup\{|\tau(f_{1/4}(\phi_{m+k}^{(1)}(a_m-a_n)){{)}}|: \tau\in {\rm T}(D_{m+k})\})\le (\mathfrak{f}_{a_0})^2/2^{n-1}.
\eneq
\Wlog, passing to a subsequence if necessary,  by \eqref{TDappdiv-1+},
we may assume that, for all $m > n,$
\beq\label{TDapdiv-7-}
&&\|a_n-(\phi_m^{(0)}(a_n)+ \phi_m^{(1)}(a_n))\|<{(\mathfrak{f}_{a_0})^2\over{2^{(n+4)^2}}},\\\label{TDapdiv-7}
&&\|f_{1/4}(a_n)-(f_{1/4}(\phi_m^{(0)}(a_n))+ f_{1/4}(\phi_m^{(1)}(a_n)))\|<{(\mathfrak{f}_{a_0})^2\over{2^{(n+4)^2}}},\andeqn\\
\label{TDap-7n+1}
&&\|\phi_m^{(1)}\circ \cdots\circ \phi_{n+1}^{(1)}(f_{1/4}(a_n))-f_{1/4}(a_{n+m})\|<{(\mathfrak{f}_{a_0})^2\over{2^{(n+4)^2}}},\,\,\, n=1,2,...,\\\nonumber
%\eneq
&&\hspace{-0.96in}{\rm{
and, \,\,by\,\, \eqref{TDap-7n-2},\,
% for\,\, any
{{whenever}}}}\,\, m+k\ge n+1,\\
%\beq
\label{TDap-7n-3}
&&\sup\{|\tau(f_{1/4}(\phi_{m+k}^{(1)}(a_{n+1}-a_n))|: \tau\in {\mathrm{T}}(D_{m+k})\}\le {(\mathfrak{f}_{a_0})^2\over{2^{(n+4)^2}}}.
% (\mathfrak{f}_{a_0})^2/2^{n-1}
\eneq

Claim (1):  
%For any $n\ge 1,$
\beq\label{TDapdiv-7+1}
\liminf_{n\to\infty}(\inf\{\tau(f_{1/4}(\phi_m^{(1)}(a_n))): \tau\in {\mathrm{T}}(D_m) \andeqn m>n\})\ge {(\mathfrak{f}_{a_0})^2\over{8}}.
%\liminf_{n\to\infty}(\{\tau(f_{1/4}(\phi_m^{(1)}(a_n))): \tau\in {\mathrm{T}}(D_m): \andeqn m>n\})\ge {(\mathfrak{f}_{a_0})^2\over{8}}.
\eneq

Claim (2):  If we first take a subsequence $(N(k))$ and as above define
$a_1:=\phi^1_{(N(1))}(a_0),$ $a_2:=\phi_{N(2)}^{(1)}(a_1),..., a_n:=\phi_{N(n)}^{(1)}(a_{n-1}),$
$n=0,1,...,$  then  Claim (1) still holds, when
$m$ is replaced by $N(m).$

Let us first explain that Claim (2) follows from Claim (1) since we may  first pass to another subsequence
in the above construction and then apply Claim (1).

We now prove  Claim (1).

Assume  Claim (1) is false.
Then
there exists $\eta_0>0$ such that $ {(\mathfrak{f}_{a_0})^2\over{8}}-\eta_0>0$ and
\beq\label{TDapdiv-7+2}
\liminf_{n\to\infty}(\inf \{\tau(f_{1/4}(\phi_m^{(1)}(a_n))): \tau\in {\mathrm{T}}(D_m)\andeqn m>n\})\le {(\mathfrak{f}_{a_0})^2\over{8}}-\eta_0.
\eneq
By \eqref{TDap-7n-3},
%\eqref {TDap-7n-1}
%\eqref{TDappdiv-2})
there is $n_0\ge 1$ such that, for all $m\ge n\ge n_0$ and $k\ge 1,$ 
%(by \eqref{TDapdiv-7} and \eqref{TDappdiv-2}), 
%\beq\nonumber
%\label{TDappdiv-7+3}
%&&\tau(f_{1/4}(a_n))<\tau(f_{1/4}(\phi_{m+k}^{(1)}(a_n)))+\eta/4, 
%\,\,\,\tau(f_{1/4}(\phi_{m+k}^{(1)}(a_n)))<\tau(f_{1/4}(a_n))+\eta/8
%\andeqn\\\nonumber
%&&\tau(f_{1/4}(\phi_{m+k}^{(1)}(a_m)))\le \tau(f_{1/4}(a_m))+\eta/4\rforal \tau\in {\mathrm{T}}(D_{m+k}).
%\eneq
%Then 
\beq\label{TDappdiv-7+3}
&&\hspace{-0.2in}\tau(f_{1/4}(\phi_{m+k}^{(1)}(a_m)))\le \tau(f_{1/4}(\phi_{m+k}^{(1)}(a_n)))+\eta_0/2\rforal \tau\in {\mathrm{T}}(D_{m+k}).
\eneq
Hence there exists a subsequence  $(n_k)$
 which has the following property:
{{if}} $k'\ge k,$ then 
\beq\label{TDapdiv-8}
t_{n_{k'}}(f_{1/4}(\phi_{n_{k'}}^{(1)}(a_{n_k})))\le (\mathfrak{f}_{a_0})^2/8-\eta_0/2.
\eneq

Consider the positive linear functional $\tau_k$ defined by
$\tau_k(a)=t_k(\phi_{n_k}^{(1)}(a))$
for all $a\in A,$ $k=1,2,....$
%Without loss of generality, by passing to a diagonal selection process, we may assume
%that
%$$
%\{1(m)\}\subset \{2(m)\}\subset \cdots \{k(m)\}\subset ....
%$$
%Choosing a diagonal subsequence
%$$
%1(1)<2(2)<\cdots < k(k)<\cdots.
%$$
%By
%\beq\label{TDapdiv-9}
%t_{m(m)}\phi_{m(m)}^{(1)}(a_{n_k}))<\mathfrak{f}_{a_0}/4.
%\eneq
%for all $m\ge k.$
Let $\tau$ be a weak* limit of ${{(\tau_k)_{k\ge 1}}}.$ It follows \eqref{TDappdiv-2} that
$1\ge \|\tau\|\ge \mathfrak{f}_{a_0}.$  Moreover, by \eqref{TDappdiv-1}, $\tau$  is a trace.
 On the other hand, by \eqref{TDapdiv-8} and \eqref{TDapdiv-7},
\beq\label{TDapdiv-9}
\tau(f_{1/4}(a_{n_k}))<(\mathfrak{f}_{a_0})^2/8\rforal k.
\eneq

It follows from  \eqref{TDapdiv-7}, \eqref{TDadiv-0}  and \eqref{TDappdiv-1+2}  that, if $m>n\ge 1,$
\beq\label{TDapdiv-10}
\tau(f_{1/4}(\phi_m^{(1)}(a_n)))\ge \tau(f_{1/4}(a_n))-(\mathfrak{f}_{a_0})^2/2^{m+5}-(\mathfrak{f}_{a_0})^2/2^{(n+4)^2}.
\eneq
%for all $t\in {\mathrm{T}(A)}.$
Therefore, also using \eqref{TDIV-n1},  for all $k,$
\beq\label{TDapdiv-11}
\tau(f_{1/4}(a_{n_k}))\ge \tau(f_{1/4}(a_0))-(\sum_{j=1}^{n+k}((\mathfrak{f}_{a_0})^2/2^{j+5}-(\mathfrak{f}_{a_0})^2/2^{(j+1)^2}))>(\mathfrak{f}_{a_0})^2/4.
\eneq
This contradicts  with \eqref{TDapdiv-9}  and so   Claim (1) is proved.

%Define
%$\psi_n': D_n\to D_{n+1}$ by
%$\psi_n'(d)=\phi_{n+1}^{(1)}(d)$ for all $d\in D_n,$ $n=1,2,....$

%Note, by \ref{Ca=ped}, $A={\rm Ped}(A).$ 
%Let us consider first the case that $A={\rm Ped}(A).$ 
Since $A={\rm Ped}(A),$  by \ref{Tqcfull},
we obtain a map 
$T:  A_+\setminus \{0\}\to \N\times \R_+\setminus \{0\}$ which has the properties described by  \ref{Tqcfull}.
%%Denote 
{{Set}} $T(a)=(N(a), M(a))$  and $\lambda(a)=(N(a)+1)(M(a)+1)$ for each $a\in A_+\setminus \{0\}.$ 

%Since $A$ is simple, there is, for each $n,$ a map
%$T_n=(N_n, M_n): A_+\setminus \{0\}\to \N\times \R_+\setminus \{0\}$ such that, for any $a\in A_+\setminus\{0\},$
%%there are $x(a)_{1,n}, x(a)_{2,n},...,x(a)_{m_n(a),n}\in A$ with
%$m_n(a)\le N_n(a)$ and $\|x(a)_j\|\le M_n(a)$ such that
%\beq\label{TDapdiv-12}
%\sum_{j=1}^{m_n(a)}x(a)_{j,n}^*ax(a)_{j,n}=f_{1/16}(\phi_{n}^{(1)}(a_1)), n=1,2,....
%\eneq
%%(with $\phi_0^{(1)}(a_0)=a_0$).
%Here, we view $D_n$ as \SCA\,  of $A.$

Now fix a finite subset ${\cal F}\subset A_+\setminus \{0\}$ and $1/16>\ep>0.$
We may assume that $\|a\|\le 1$ for all $a\in {\cal F},$
 $a_0
 %, f_\eta(a_0)
 \in {\cal F}.$ 
 %and 
%\beq\label{DLsimple-n-1}
%\|f_\eta(a_0)af_{\eta}a_0-a\|<\ep/128\rforal a\in {\cal F}.
%\eneq
%Put ${\cal F}'=\{f_\eta(a_0)af_\eta(a_0): a\in {\cal F}\}\cup{\cal F}.$
Let $e^D_{{n}}$ be a strictly positive element of $D_n$ with $0\le {{e^D_n}}\le 1.$
Let ${{(}}{\cal F}_{k,n}'{{)}}$ be an increasing sequence of finite subsets of $(D_k)_+$ such that
the union of these subsets is dense in $(D_k)_+.$
We {{may}} assume that $\phi_k^{(1)}({\cal F})\subset {\cal F}_{k,1}'.$ By \eqref{TDappdiv-1+} 
{{and \eqref{TDappdiv-2},}}
{{\wilog,  by choosing larger $k,$ if necessary,}} we may also assume that  the 
elements $\phi_k^{(1)}({\cal F})$ and therefore those in ${\cal F}_{k,n}'$ are nonzero.
Choose $1/4>\eta_0>0$ such that
\beq\label{Ddvi-nnnnn-100}
\|f_{1/4}(a')-f_{1/4}(a'')\|<\min\{\ep, \mathfrak{f}_{a_0}^2\}/64
\eneq
{{whenever}}  $0\le a', a''\le 1$ and $\|a'-a''\|<\eta_0.$
We may assume that 
\beq
\|x-(\phi_1^{(0)}+\phi_1^{(1)})(x)\|<\min\{\eta_0,\ep\}/16^2\rforal x\in {\cal F}.
\eneq
 Set
\vspace{-0.1in}\beq\label{Ddvi-n12}
{\cal F}_{1,n}''=\{(a-\|a\|/2)_+:a\in {\cal F}_{1,n}'\}\andeqn {\cal F}_{1,n}'''={\cal F}_{1,n}'\cup{\cal F}_{1,n}''.
\eneq
Choose $\sigma_1>0$ such that
\beq
&&\hspace{-0.2in}\|f_{\sigma_1}(e^D_1)xf_{\sigma_1}(e^D_1)-x\|<\min\{\ep,\eta_0\}/16^2\rforal x\in \phi_1^{(1)}({\cal F})\andeqn\\\label{Ddvi-n12+}
&&\hspace{-0.2in}\|f_{\sigma_1}(c_1)\phi_1^{(0)}(x)f_{\sigma_1}(c_{1})-\phi_{1}^{(0)}(x)\|<\min\{\mathfrak{f}_{a_0},\eta_0,\ep\}/16^3\rforal x\in {\cal F}.
\eneq
Define $\Phi_1: A\to {{A_{0,1}}}$ by $\Phi_1(x)=f_{\sigma_1}(c_1)\phi_1^{(0)}(x)f_{\sigma_1}(c_1).$
Put $\ep_0=\min\{\mathfrak{f}_{a_0}^2, \ep/16,\eta_0/16, \sigma_1/2\}.$

Let
${\cal F}_{1,n}$ be a finite subset which also contains
$
{\cal F}_{1,n}'''
%\cup\{x(a)_{j,1}: a\in {\cal F}_{1,n}''', 1\le j\le m_n(a)\}
\cup\{a_1, e_1^D, f_{\sigma_1}(e_1^D), f_{\sigma_1/2}(e_1^D)\}.$
%,f_{\sigma}(a_1), \sigma\in \{1/4, 1/8, 1/16\}\}.
% f_{1/8}(a_i),f_{1/4}(a_i), \,i=0,1\}.
%$
%Choose $1>\sigma_1>0$ such that 
%\beq\label{TDap-12n+1}
%\|f_{\sigma_1}(e^d_1)af_{\sigma_1}(e^d_1)-a\|<\min\{(\mathfrak{f}_{a_0})^2/16, \ep/16\}/4\lambda(a)\rforal a\in
% {\cal F}_{1,1}.
%\eneq
By \ref{Tqcfull},
%and \ref{semiproj}, 
since $D_1$ is  assumed to be in  {{the}}  class ${\cal C}_0,$  there exist an integer $n_2\ge 2$ and 
 a \hm\, $\psi_1: D_1\to D_{n_2}$ such that
 \beq\label{TDap-13n+1}
 \sum_{i=1}^{m_{1}(a)}\phi_{n_2}^{(1)}(x(a)_{i,1})^*\phi_{n_2}^{(1)}(a)\phi_{n_2}^{(1)}(x(a)_{i,1})=f_{1/16}(\phi_{n_2}^{(1)}(a_1)),
 \eneq
 where  $x(a)_{i,1}\in A,$ $m_{1}(a)\le N(a)$ and $\|x(a)_{i,1}\|\le M(a),$ $i=1,2,...,m_1(a),$ and 
\beq\label{TDapdiv-13}
&&\|\psi_1(a)-\phi_{n_2}^{(1)}(a)\|<\ep_0/4\cdot 16\lambda(a)
\eneq
for all  $a\in 
%\{f_{\sigma_1}(e^d_1)xf_{\sigma_1}(e^d_1): x\in 
{\cal F}_{1,1}''',$
%where $D_1'=\overline{f_{\sigma_1}(e^d_1)D_1f_{\sigma_1}(e^d_1)}$ 
%and $D_2''=\overline{f_{\sigma_2'}(e^d_{n_2})D_{n_2}f_{\sigma_2'}(e^d_{n_2})}$
%for some $1>\sigma_2'>0.$
It follows that
\beq
&&\hspace{-0.4in}\|\sum_{i=1}^{m_1(a)}\phi_{n_2}^{(1)}(x(a)_{i,1})^*\psi_1(a)\phi_{n_2}^{(1)}(x(a)_{i,1})-f_{1/16}(\phi_{n_2}^{(1)}(a_1))\|<\ep_0/32
%\min\{{(\mathfrak{f}_{a_0})^2\over{32}}, {\ep\over{32}}\}
\eneq
for all  $a\in {\cal F}_{1,1}''.$
%\{f_{\sigma_1}(e^d_1)xf_{\sigma_1}(e^d_1): x\in {\cal F}_{1,1}'''\}.$ 
Therefore, {{applying}} \ref{Lrorm},  one obtains $y(a)_{i,n_2}\in D_{n_2}$ with $\|y(a)_{i,n_2}\|\le \|x(a)_{i,1}\|+(\mathfrak{f}_{a_0})^2/16$ such that
\beq\label{TDapdiv-14}
\sum_{i=1}^{m_2(a)}y(a)_{i, n_2}^*\psi_{1}(a)y(a)_{i,n_2}=f_{1/8}(\phi_{n_2}^{(1)}(a_1))\rforal a\in {\cal  F}_{1,1}'''.
\eneq
By \eqref{TDappdiv-1}, we may assume, choosing $n_2$ large, 
that
\beq\label{TD-nnn-1++}
&&\|x-(\phi_{n_2}^{(0)}+\phi_{n_2}^{(1)})(x)\|<\ep_0/16^2\rforal x\in {\cal F}\cup \{c_1\}\cup {\cal F}_{1,1},\\\label{TD-nnn-1}
&&\|\phi_{n_2}^{(1)}(f_{\sigma_1/2}(c_1))\phi_{n_2}^{(1)}(f_{\sigma_1/2}(e_1^D))\|<\ep_0/16,\\
&&\|f_{\sigma'}(\phi_{n_2}^{(1)}(c_1))-\phi_{n_2}^{(1)}(f_{\sigma'}(c_1))\|<\ep_0/16,\,\, \sigma'\in \{\sigma_1,\sigma_1/2, 
\sigma_1/4\}, \andeqn
\\\label{TD-nnn-1+}
&&\hspace{-0.2in}\|f_{\sigma_1/2}(\phi_{n_2}^{(1)}(c_1))^{1/2}\phi_{n_2}^{(1)}(\Phi_1(x))f_{\sigma_1/2}(\phi_{n_2}^{(1)}(c_1))^{1/2}
-\phi_{n_2}^{(1)}(\Phi_1(x))\|<\ep_0/16
%\\\label{TD-nnn-1++}
%&&\andeqn \|x-(\phi_{n_2}^{(0)}+\phi_{n_2}^{(1)})(x)\|<\ep_0/16^2
\eneq
for all  $x\in {\cal F}.$
Put $\phi_1: A\to D_1$ by $\phi_1(x)=f_{\sigma_1}(e_1^D)\phi_1^{(1)}(x)(f_{\sigma_1}(e_1^D))$
for $x\in A.$
Define $\phi_{n_2}^{(1)'}$ by
$\phi_{n_2}^{(1)'}(x)=f_{\sigma_1/2}(\phi_{n_2}^{(1)}(c_1))^{1/2}\phi_{n_2}^{(1)}(\Phi_1(x))f_{\sigma_1/2}(\phi_{n_2}^{(1)}(c_1))^{1/2}$
for all $x\in A.$
Define $\Phi_2: A\to {{A_{0,n_2}}}$ 
by {{(note, below,  the two sums are orthogonal sums)}}
$$
\Phi_2(x)=
(1-\psi_1(f_{\sigma_1/2}(e_1^D)))(\phi_{n_2}^{(0)}((\Phi_1+ \phi_1^{(1)})(x))+ \phi_{n_2}^{(1)'}(x))
(1-\psi_1(f_{\sigma_1/2}(e_1^D)))
$$
for all $x\in A.$ Note that $\Phi_2$ is a  \cpc\, and $\Phi_2(A)\perp j_2(\psi_1\circ \phi_1(A))$
(recall that $j_n: D_n\to A$ is the embedding). 
Also {{(note the sum is orthogonal)}}
\beq
\psi_1(f_{\sigma_1/2}(e_1^D))(\phi_{n_2}^{(0)}(\Phi_1+ \phi_1^{(1)})(x))=0\rforal x\in A.
\eneq
By \eqref{TDapdiv-13} and \eqref{TD-nnn-1},
\beq
\phi_{n_2}^{(1)'}(x)\psi_1(f_{\sigma_1/2}(e_1^D))\approx_{\ep_0/16}
\phi_{n_2}^{(1)'}(x)\phi_{n_2}^{(1)}f_{\sigma_1/2}(e_1^D)\approx_{\ep_0/16} 0.
\eneq
Thus, on ${\cal F},$ by \eqref{TD-nnn-1+} and  \eqref{Ddvi-n12+},
\beq\nonumber
&&\hspace{-0.7in}\Phi_2(x)\approx_{\ep_0/8} \phi_{n_2}^{(0)}((\Phi_1+ \phi_1^{(1)})(x))+ \phi_{n_2}^{(1)'}(x)
\approx_{\ep_0/16} \phi_{n_2}^{(0)}((\Phi_1+ \phi_1^{(1)})(x))+ \phi_{n_2}^{(1)}(\Phi_1(x))\\
&&_{\ep/16^2}\approx \phi_{n_2}^{(0)}((\phi_1^{(0)}+ \phi_1^{(1)})(x))+ \phi_{n_2}^{(1)}(\phi_1^{(0)}(x)).
\eneq
We then verify
%  (recall $j_2: D_2\to A$ is the embedding)
\beq\label{TD-nnn-10}
\|x-(\Phi_2(x)+j_2(\psi_1\circ\phi_1(x)))\|<\ep/16\rforal x\in {\cal F}.
\eneq
Note that, by \eqref{TD-nnn-1++},
\beq
\|(\phi_{n_2}^{(0)}(c_1)+ \phi_{n_2}^{(1)}(c_1))-c_1\|<\ep_0/16.
\eneq
By \ref{Lrorm},   {{since $\ep_0<\sigma_1/2$ and $\phi_{n_2}^{(0)}(c_1)\perp \phi_{n_2}^{(1)}(c_1),$}}
\beq
f_{\sigma_1/2}(\phi_{n_2}^{(1)}(c_1)){{\lesssim f_{\sigma_1/2}((\phi_{n_2}^{(0)}(c_1)+ \phi_{n_2}^{(1)}(c_1)))}}\lesssim c_1.
%\phi_{n_2}^{(1)}(f_{\sigma_1}(c_1))\lesssim 
\eneq
{{Note, by the definition of $\Phi_2,$   for any $x\in A_+,$ 
\beq\nonumber
\la \Phi_2(x) \ra &\le & \la (\phi_{n_2}^{(0)}((\Phi_1+ \phi_1^{(1)})(x))+\phi_{n_2}^{(1)'}(x))\ra \\\nonumber
&\le&  \la c_{n_2}\ra+ \la f_{\sigma_1/2}(\phi_{n_2}^{(1)}(c_1))\ra \le \la b_{n_2}\ra+\la  c_1\ra \le \la b_2\ra+\la  b_{1,2}\ra \le\la  b_{1,1}+b_{1,2}\ra.
\eneq}}
%Let $c_2'$ be a strictly positive element of $\overline{\Phi_2(A)A\Phi_2(A)}.$ 
%Then 
%{\blue{(where $c_1\oplus c_1$ is an orthogonal sum in $M_2(A)$)}}
%\beq
%c_2'\lesssim c_1\oplus c_1\lesssim b_{1,1}+b_{1,2}.
%\eneq
It follows that, {{if $c_2'$ is a strictly positive element of $\overline{\Phi_2(A)A\Phi_2(A)}$ (see \ref{180915sec2}),}}
 {{then}}
\beq\label{180914-1}
c_2'\lesssim b_{1,1}+b_{1,2}.
\eneq
%where $\psi_1''(a)=\psi_1(f_{\sigma_1}(e^d_1)af_{\sigma_1}(e^d_1))$ for all $a\in D_1.$
To simplify notation,  passing to a subsequence, if necessary,   \wilog,
we may assume that $n_2=2.$

Put 
%$C_1=D_1,$ $C_2=\overline{\psi_1(D_1)D_2\psi_1(D_1)}$
%and 
$e_2'=\psi_1(e_1^D).$ 
Set
\vspace{-0.1in}\beq\label{Ddvi-nn12+}
{\cal F}_{2,n}''=\{(a-\|a\|/2)_+:a\in {\cal F}_{2,n}'\}\andeqn {\cal F}_{2,n}'''={\cal F}_{2,n}'\cup {\cal F}_{2,n}''.
\eneq
Choose $\sigma_2>0$ such that
\beq
&&\hspace{-0.4in}\|f_{\sigma_2}(e'_2)xf_{\sigma_1}(e'_2)-x\|<\ep/16^3\rforal x\in {\cal F}_{2,2}'\cup\psi_1({\cal F}_{1,2})\cup \phi_2^{(1)}({\cal F})\andeqn\\\label{Ddvi-nn12++}
&&\hspace{-0.4in}\|f_{\sigma_2}(c_2')\Phi_2(x)f_{\sigma_2}(c_2')-\Phi_2(x)\|<\ep/16^3\rforal x\in {\cal F}.
\eneq
{{Recall}} that ${\mathrm{K}}_0(D_1)$ is finitely generated {{(see \ref{Runitz}),}}  say by 
$[p_j]-n_j[1],$ where $p_j\in {\mathrm{M}}_m({\widetilde D}_1)$ is a projection 
with $[\pi_d(p_j)]=n_j[1],$ $\pi_d: {\mathrm{M}}_m({\widetilde D}_1)\to {\mathrm{M}}_m$ 
is the quotient map, $j=1,2,...,k_1.$ We may write $p_j=q_j+h_j,$
where $h_j=(h_j^{(i,k)})$ with $h_j^{(i,k)}\in D_1,$ 
where  $q_j=\diag(\overbrace{1,1,...,1}^{n_j},0,...,0).$ 
If ${\rm K}_0(A)=\{0\},$ \wilog, we may assume that there exists 
$v_j\in {\mathrm{M}}_m({\widetilde A})$ such that
\beq
v_j^*v_j=p_j\andeqn v_jv_j^*=q_j,\,\,\, j=1,2,...,k_1,
\eneq
where we identify $q_j$ with the matrix in ${\mathrm{M}}_m(\C\cdot 1_{\widetilde A}).$
Write $v_j=\lambda_j+s_j,$ where $s_j=(s_j^{(i,k)})$ with $s_j^{(i,k)}\in D_1,$
 $\lambda_j\in {\mathrm{M}}_m(\C \cdot 1_{\widetilde A})$
is a partial isometry and $s_j\in {\mathrm{M}}_m(D_1),$ $1\le j\le k_1.$
Put $\ep_1=\min\{\ep_0/(m16)^2, \sigma_2/(m16)^2\}.$
Let
${\cal F}_{2,n}$ be a finite subset which also contains
$
{\cal F}_{2,n}'''
%\cup
%\{x(a)_{j,2}: a\in {\cal F}_{2,n}'', 1\le j\le m_n(a)\}
\cup\{a_2, e_2', f_{\sigma_2}(e_2'),f_{\sigma_2/2}(e_2')\}
%,
%f_{1/16}(a_i),  f_{1/4}(a_i),\, i=0,1,\, 
%f_{\sigma}(a_2):  \sigma\in \{1/4, 1/8, 1/16\}\}
%\phi_2^{(1)}(a_1)), f_{1/8}(\phi_2^{(1)}(a_1))\}
\cup \psi_1({\cal F}_{1,2})\cup \phi_2^{(1)}({\cal F})
$ as well as $h_j^{(i,k)}$ and $s_j^{(i,k)},$
%Choose $1>\sigma_2>0$ such that  $\sigma_2<\min\{\sigma_2'/2, 1/4\}$ and 
%$$
%\|f_{\sigma_2}(e^d_2)af_{\sigma_2}(e^d_2)-a\|<\min\{(\mathfrak{f}_{a_0})^2/16^2, \ep/16^2\}/4\lambda(a)\rforal a\in {\cal F}_{2,2}.
%$$
 Applying \ref{Tqcfull},
 % and  \ref{semiproj}, 
 since $D_2\in {\cal C}_0,$  
 as in the previous step,  {{we obtain}}  a \hm\, $\psi_2: D_2\to D_{n_3}$ such that 
 %with ${\widetilde{{\cal F}_{2,2}}}=\{f_{\sigma_2}(e^d_2)af_{\sigma_2}(e^d_2): a\in {\cal F}_{2,2}\},$
\beq\label{TDapdiv-15}
\|\psi_2(a)-\phi_{n_3}^{(1)}(a)\|<
%\min\{(\mathfrak{f}_{a_0})^2/16^{2}, \ep/16^{2}\}
\ep_1/4\cdot 16^3\lambda(a)\rforal a\in {\cal F}_{2,2}\andeqn\\
\hspace{-0.3in}\|\sum_{i=1}^{m_2(a)}\phi_{n_3}^{(1)}(x(a)_{i,2})^*\psi_2(a)\phi_{n_3}^{(1)}(x(a)_{i,2})-f_{1/16}(\phi_{n_3}^{(1)}(a_2))\|<\ep_1/16^2
%\min\{(\mathfrak{f}_{a_0})^2/16^{2}, \ep/16^{2}\}
\eneq
for all  $a\in 
%\{f_{\sigma_2}(e^d_2)af_{\sigma_2}(e^d_2):a\in 
{\cal F}_{2,2},$
%'''\},$ 
where $m_2(a)\le N(a)$ and $\|x(a)_{i,2}\|\le M(a)$ for all $a\in {\cal F}_{2,2}'''.$ 
%and where $D_2'=\overline{f_{\sigma_2}(e^d_2)D_2f_{\sigma_2}(e^d_2)}$ and
%$D_{n_3}''=\overline{f_{\sigma_3'}(e^d_{n_3})D_{n_3}f_{\sigma_3'}(e^d_{n_3})}$
%for some $1>\sigma_3'>0.$
%We note that 
%\beq\label{TDap-15n+1}
%\sum_{i=1}^{m_2(a)}f_{\sigma_2}(e^d_2)y(a)_{i, n_2}^*\psi_1''(a)y(a)_{i,n_2}f_{\sigma_2}(e^d_2)
%=f_{\sigma_2}(e^d_2)f_{1/8}(\phi_{n_2}^{(1)}(a_1))f_{\sigma_2}(e^d_2)
%\eneq
%for all $ a\in {\cal  F}_{1,1}'''.$ 
By 
%applying
 \ref{Lrorm},  there are $y(a)_{i,n_3}\in D_{n_3}$ with $\|y(a)_{i,n_3}\|\le \|x(a)_{i,2}\|+
(\mathfrak{f}_{a_0})^2/16^2$ such that
\beq\label{TDapdiv-16}
\sum_{i=1}^{m_2(a)}y(a)_{i,n_3}^*\psi_{2}(a)y(a)_{i,n_3}=f_{1/8}(\phi_{n_3}^{(1)}(a_2))\rforal a\in {\cal  F}_{2,2}'''.
\eneq

By \eqref{TDappdiv-1}, we may assume, by choosing large $n_3,$ 
that
\beq\label{TD-nnn-2++}
&& \|x-(\phi_{n_3}^{(0)}+\phi_{n_3}^{(1)})(x)\|<\ep_1/16^3\rforal x\in {\cal F}\cup \{c_2'\}\cup  {\cal F}_{2,2},\\\label{TD-nnn-2}
&&\|\phi_{n_3}^{(1)}(f_{\sigma_2/2}(c_2'))\phi_{n_3}^{(1)}(f_{\sigma_2/2}(e_2'))\|<\ep_1/16^3,\\
&&\|f_{\sigma'}(\phi_{n_2}^{(1)}(c_2'))-\phi_{n_2}^{(1)}(f_{\sigma'}(c_2'))\|<\ep_1/16^3,\,\, \sigma'\in \{\sigma_2,\sigma_2/2, 
\sigma_2/4\},
\\\label{TD-nnn-2+}
&&\hspace{-0.2in}\|f_{\sigma_2/2}(\phi_{n_3}^{(1)}(c_2'))^{1/2}\phi_{n_3}^{(1)}(\Phi_2(x))f_{\sigma_2/2}(\phi_{n_3}^{(1)}(c_2'))^{1/2}
-\phi_{n_3}^{(1)}(\Phi_2(x))\|<\ep_1/16^2
%\\\label{TD-nnn-2++}
%&&\andeqn \|x-(\phi_{n_2}^{(0)}+\phi_{n_2}^{(1)})(x)\|<\ep_0/16^2
\eneq
for all  $x\in {\cal F}.$
In particular (we continue to use $\phi_k^{(i)}$ for $\phi_k^{(i)}\otimes {\rm id}_{M_m}$),
\beq\nonumber
\|h_j-(\phi_{n_3}^{(0)}(h_j)+\phi_{n_3}^{(1)}(h_j))\|<\ep_1/16^3\andeqn
\|s_j-(\phi_{n_3}^{(0)}(s_j)+\phi_{n_3}^{(1)}(s_j))\|<\ep_1/16^3
\eneq
(when ${\mathrm{K}}_0(A)=0$).
By choosing smaller $\ep_1,$ we may assume that  (for $1\le j\le k_1$)
there is {{a}} partial isometry $u_j\in {\mathrm{M}}_m({\widetilde D}_{n_3})$ such that
\beq
u_j^*u_j=\psi_2^\sim (p_j)\andeqn u_ju_j^*=q_j,
\eneq
where $\|(\psi_2^\sim (p_j)+\phi_{n_3}^{(0)}(h_j))-p_j\|<\ep_1,$  
$\|(u_j+\phi_{n_3}^{(0)}(s_j))-v_j\|<\ep_1,$ where 
we identify ${\mathrm{M}}_m(\C\cdot 1_{{\widetilde D}_1})$ with ${\mathrm{M}}_m(\C\cdot 1_{\widetilde A}),$ and 
where $\psi_2^\sim$  is the extension of $\psi_2$ on ${\mathrm{M}}_m({\widetilde D}_1).$ 
In particular, $(\psi_2)_{*0}=0,$ when ${\mathrm{K}}_0(A)=\{0\}.$ 
%Put $\phi_1: A\to D_1$ by $\phi_1(x)=f_{\sigma_1}(e_1^d)\phi_1^{(1)}(x)(f_{\sigma_1}(e_1^d))$
%for $x\in A.$
Define $\phi_{n_3}^{(1)'}$ by
$$
\phi_{n_3}^{(1)'}(x)=
f_{\sigma_2/2}(\phi_{n_3}^{(1)}(c_2'))^{1/2}\phi_{n_3}^{(1)}(\Phi_2(x))f_{\sigma_2/2}(\phi_{n_3}^{(1)}(c_2'))^{1/2}
$$
for all $x\in A.$
Define $\Phi_3: A\to A$ 
by
$$
\Phi_3(x)=
(1-\psi_2(f_{\sigma_2/2}(e_2'))(\phi_{n_3}^{(0)}((\Phi_2+ j_2\circ \psi_1\circ \phi_1)(x))+ \phi_{n_3}^{(1)'}(x))
(1-\psi_2(f_{\sigma_2/2}(e_2'))
$$
for all $x\in A.$ 
Note $\Phi_3$ is a  \cpc\, and\\
 $\Phi_3(A)\perp j_2(\psi_{1,2}\circ \phi_1(A)),$
where $\psi_{1,2}=\psi_2\circ \psi_1.$
Also
\beq
\psi_2(f_{\sigma_2/2}(e_2'))(\phi_2^{(0)}(\Phi_2+ j_2\circ \psi_1\circ \phi_1)(x)=0\rforal x\in A.
\eneq
By \eqref{TDapdiv-15} and \eqref{TD-nnn-2},
\beq
\phi_{n_3}^{(1)'}(x)\psi_2(f_{\sigma_1/2}(e_2'))\approx_{\ep_1/16^2}
\phi_{n_3}^{(1)'}(x)\phi_{n_3}^{(1)}(f_{\sigma_2/2}(e_2'))\approx_{\ep_1/16^2} 0.
\eneq
Thus, on ${\cal F},$ by \eqref{TD-nnn-2+}  and \eqref{TD-nnn-10} (using also the {{orthogonality of the sum}}),
%and by \eqref{TD-nnn-2+}
%{Ddvi-nn12++},
\beq\nonumber
&&\hspace{-0.2in}\Phi_3(x)+\phi_{n_3}^{(1)}(j_2\circ \psi_1(\phi_1(x)))\approx_{\ep_1/2^7} \phi_{n_3}^{(0)}((\Phi_2+ j_2\circ \psi_1\circ \phi_1)(x))+ \phi_{n_3}^{(1)'}(x)\\\nonumber
&&+\phi_{n_3}^{(1)}(j_2\circ \psi_1(\phi_1(x)))
\approx_{\ep_1/2^8} \phi_{n_3}^{(0)}((\Phi_2+ j_2\circ \psi_1\circ \phi_1)(x))+ \phi_{n_3}^{(1)}(\Phi_2+j_2\circ \psi_1)(x))\\\nonumber
&&_{\ep/16}\approx \phi_{n_3}^{(0)}(x)+\phi_{n_3}^{(1)}(x)\approx_{\ep_1/16^3} x.
\eneq
Therefore, since $\ep_1/2^7+\ep_1/2^8+\ep_1/16^3+\ep_1/16^3<\ep/16^2,$  by  \eqref{TDapdiv-15},
\beq
\|x-(\Phi_3(x)+j_3(\psi_{1,2}\circ\phi_1(x)))\|<\ep/16+\ep/16^2
\eneq
for all 
$x\in {\cal F}.$
Note that, by \eqref{TD-nnn-1++},
\beq
\|(\phi_{n_3}^{(0)}(c_2')+ \phi_{n_3}^{(1)}(c_2'))-c_2'\|<\ep_0/16.
\eneq
By \ref{Lrorm}   {{(recall that $\phi_{n_3}^{(0)}(c_2')\perp \phi_{n_3}^{(1)}(c_2')$),}}
\beq
{{f_{\sigma_2/2}(\phi_{n_3}^{(1)}(c_2'))}}\lesssim c_2'\lesssim b_{1,1}{{+}} b_{1,2}.
%\phi_{n_2}^{(1)}(f_{\sigma_1}(c_1))\lesssim 
\eneq
By the  definition of $\Phi_3$ above, for any $x\in A_+,$
\beq
\la \Phi_3(x)\ra \le  \la c_{n_3}\ra+\la f_{\sigma_2/2}(\phi_{n_3}^{(1)}(c_2'))\ra\le \la b_{n_3,1}\ra +
\la c_2'\ra.
\eneq
Let $c_3'$ be a strictly positive element of $\overline{\Phi_3(A)A\Phi_3(A)}.$ 
Then, by \eqref{180914-1}, 
\beq
\la c_3'\ra \le \la b_{n_3,1}\ra +\la b_{1,1}+b_{1,2}\ra.
\eneq

%where $\psi_{n_3}''(a)=\psi_{2}(f_{\sigma_2}(e^d_{n_2})af_{\sigma_2}(e^d_{n_2}))$ for all $a\in D_1.$
To simplify notation, by passing to a subsequence, if necessary,   \wilog,
we may assume that $n_3=3.$

Continuing this process,
one then obtains a sequence
of \hm s $\psi_n: D_n\to D_{n+1}$ such that
%with ${\widetilde{{\cal F}_{n,n}}}=\{f_{\sigma_n}(e^d_n)af_{\sigma_n}(e^d_n): a\in {\cal F}_{n,n}\},$
\beq\label{TDapdiv-17}
&&\hspace{-0.3in}\|\psi_n(a)-\phi_{n}^{(1)}(a)\|<
%\min\{(\mathfrak{f}_{a_0})^2/16^{n}, \ep/16^{n}\}
\ep_0/4\cdot 16^n\lambda(a)
\rforal a\in {\cal F}_{n,n}\andeqn\\\label{TDapdiv-17+}
&&\hspace{-0.3in}\sum_{i=1}^{m_n(a)}
%f_{\sigma_{n+1}}(e_{d(n+1)})y
y(a)_{i, n+1}^*\psi_n(a)y(a)_{i,n+1}
%f_{\sigma_2}(e_{d(n+1)})
=
%f_{\sigma_{n+1}}(e^d_{n+1})
f_{1/8}(\phi_{n+1}^{(1)}(a_n)) \rforal a\in {\cal F}_{n,n},
%f_{\sigma_{n+1}}(e^d_{n+1}),
\eneq
where ${\cal F}_{n,n}$ contains 
${\cal F}_{n,n}'''\cup\{a_n, e_n', f_{\sigma_n}(e_n'),f_{\sigma_n/2}(e_n')\} \cup \psi_n({\cal F}_{n-1,n-1})\cup 
\phi_n^{(1)}({\cal F}),$
where ({{for}} $m\ge n$)
\beq
{\cal F}_{n,m}''=\{(a-\|a\|/2)_+:a\in {\cal F}_{n,m}'\}\andeqn {\cal F}_{n,m}'''={\cal F}_{n,m}'\cup {\cal F}_{n,m}''.
\eneq
%where $z(a)_{i,n+1}=y(a)_{i,n}f_{\sigma_{n+1}}(e^d_{n+1})$ and 
%$\psi_n''(a)=\psi_n(f_{\sigma_n}(e^d_n)af_{\sigma_n}(e^d_n)),$ for all $ a\in {\cal  F}_{n,n}''',$
%and where 
%\beq\|f_{\sigma_n}(e^d_n)af_{\sigma_n}(e^d_n)-a\|<\min\{(\mathfrak{f}_{a_0})^2/16^2, \ep/16^2\}/4\lambda(a)
%\rforal a\in {\cal F}_{n,n},
%\eneq
%and  $D_n'=\overline{f_{\sigma_n}(e_n^d)D_nf_{\sigma_n}(e_n^d)}.$
Moreover, $m_n(a)\le N(a),$  $\|y(a)_{i,n+1}\|\le M(a)+1$ 
%and $\sigma_n<1/2^n$  
for all $n.$
Furthermore, there {{is a}} \cpc\, $\Phi_n: A\to A$ such that
$\Phi_n(A)\perp j_n(\psi_{1,n}\circ \phi_1(A)),$ 
\beq
&&\hspace{-0.2in}\|x-(\Phi_n(x)+j_n(\psi_{1,n}\circ \phi_1(x)))\|<\sum_{k=1}^n \ep/16^k<\ep\rforal x\in {\cal F}\andeqn\\
&&\hspace{-0.2in}c_n'\lesssim b_{1,1}+ b_{1,2}+b_{2,1}+ \cdots + b_{n,1}\lesssim b_0,
\eneq
 $n=1,2,....$ 

%\|\sum_{i=1}^{m_n(a)}\phi_{n}^{(1)}(x(a)_{i,n})^*\psi_n(a)\phi_{n}^{(1)}(x(a)_{i,n})-f_{1/16}(a_{n})\|<\min\{\mathfrak{f}_{a_0}/2^{3+n}, \ep/2^{3+n}\}
%\eneq
%for all  $a\in {\cal F}_{n,n}'''.$
%Moreover, there  are $y(a)_{i,n}\in D_{n}$ with $\|y(a)_{i,n}\|\le M_n(a)+\mathfrak{f}_{a_0}/2^{3+n}$ such that
%\beq\label{TDapdiv-18}
%\sum_{i=1}^{m_n(a)}y(a)_{i,n}^*\psi_n(a)y(a)_{i,n}=f_{1/8}(a_{n})\rforal a\in {\cal  F}_{n,n}'''.
%\eneq

{{Consider the inductive limit \CA\,}} $D=\lim_{n\to\infty} (D_n, \psi_n).$ (Again, one should note that we have {{passed to}}  a subsequence to
simplify notation.)  Note that, if $A\in {\cal D}_0,$ then each $D_n\in {\cal C}_0^{0}.$
{{Note also that, while, by construction,  $D_n\in A,$  for each $n=1,2,...,$ this is not true for $D.$}}
Let us   verify that $D$ is simple.  Fix a non-zero positive element $d_0\in (D)_+$ with $\|d_0\|=1.$
Since each $D_n$ is stably projectionless, so is $D.$
Fix $1/64>\ep_1>0.$ There is $d\in (D)_+$ such that
$d=\psi_{m,\infty}(d')$ for $d'\in (D_m')_+$ with $\|d'\|=1$
and
\beq\label{TDapdiv-19}
\|d-d_0\|<\ep_1/32.
\eneq
It follows from \ref{Lrorm} that there is $z\in D$ such
that
\beq\label{TDapdiv-20}
(d-\ep_1/16)_+=z^*d_0z.
\eneq

By construction, there is $d''\in 
%\{f_{\sigma_{m'}}(e^d_{m'})af_{\sigma_{m'}}(e^d_{m'}): a\in 
{\cal F}_{m',m'}$ for some
$m'\ge m+16$ such that
\beq\label{TDapdiv-21}
\|\psi_{m,m'}((d'-\ep_1/16)_+)-d''\|<\ep_1/64.
\eneq
There is $y\in D_{m'}$ such
that
\beq\label{TDapdiv-22}
(d''-\ep_1/8)_+=y^*\psi_{m,m'}((d'-\ep_1/4)_+)y.
\eneq
Note that $\ep_1/2\le \|d''\|/8.$

By \eqref{TDapdiv-17+}, there are $x_1, x_2,...,x_L\in D_{m'+1}$
such that
\beq\label{TDapdiv-23}
&&\hspace{-0.8in}\sum_{i=1}^L x_i^*\psi_{m', m'+1}((d''-\ep_1/2)_+)x_i=
%f_{\sigma_{m'+1}}(e^d_{m'+1})
f_{1/8}(a_{m'+1}).
%f_{\sigma_{m'+1}}(e^d_{m'+1}).
%f_{1/8}(\phi_{m'+1}(a_{m'}))=f_{1/8}(a_{m'+1}).
\eneq
%Put 
% ${\widetilde A}_m=f_{\sigma_m}(e^d_{m}))f_{1/4}(\phi_m^{(1)}(a_1))f_{\sigma_m}(e^d_{m}),$ $m=2,3,....$ 
 
Claim (3):  The element $a_{00}:=\psi_{m'+1, \infty}(a_{m'+1})$
%({\widetilde A}_{m'+1}))$ 
 is full in $D.$
% for all $m'\ge 1.$
%Put  \linebreak 
% $c_m=f_{\sigma_m}(e_{dm}))f_{1/4}(\phi_m^{(1)}(a_1))f_{\sigma_1}(e_{dm}),$ $m=2,3,....$ 
In fact, for any $m''>m'+1,$  it follows from \eqref{TDapdiv-17}, \eqref{TDap-7n+1}, 
%\eqref{TDapdiv-17},  
and Claim (2)  that
\beq\nonumber
%\label{TDapdiv-24}
&&\hspace{-0.2in}\tau(\psi_{m'+1, m''}(f_{1/8}(a_{m'+1})))\ge \tau(\phi_{m''}^{(1)}\circ \cdots \circ \phi_{m'+2}^{(1)}(f_{1/8}(a_{m'+1})))-\sum_{j=m'+2}^{m''}(\mathfrak{f}_{a_0})^2/16^{j}\\\nonumber
%\tau(f_{1/4}(\phi_{m''}^{(1)}(a_{m'+1})))-(\mathfrak{f}_{a_0})^2/16^{m'}\\
&&\hspace{0.4in}\ge  \tau(\phi_{m''}^{(1)}\circ \cdots \circ \phi_{m'+2}^{(1)}((f_{1/4}(a_{m'+1}))
%-(\mathfrak{f}_{a_0})^2/16^{m'}
-(\mathfrak{f}_{a_0})^2/16^{m'+1}\\\label{LDsimple-n-2}
&&\hspace{0.4in}\ge \tau(f_{1/4}(a_{m''}))-{(\mathfrak{f}_{a_0})^2\over{2^{(m'+1+4)^2}}}
%-{(\mathfrak{f}_{a_0})^2\over{16^{m'}}}
-{(\mathfrak{f}_{a_0})^2\over{16^{m'+1}}}>{(\mathfrak{f}_{a_0})^2\over{16}}
%&=&\tau(f_{1/4}(\phi_{m''}^{(1)}(a_{m'+1}))-\mathfrak{f}_{a_0}/2^{m'+3}> (\mathfrak{f}_{a_0})^2/16
%\tau(\psi_{m'+1, m''}(f_{1/4}(a_{m'+1})))&\ge& \tau(\phi_{m''}^{(1)}(f_{1/4}(a_{m'+1}))-\mathfrak{f}_{a_0}/2^{m'+3}\\
%&=&\tau(f_{1/4}(\phi_{m''}^{(1)}(a_{m'+1}))-\mathfrak{f}_{a_0}/2^{m'+3}> \mathfrak{f}_{a_0}/16
\eneq
for all $\tau\in {\mathrm{T}}(D_{m''}).$
By \ref{Pcom4C}, we conclude that
$\psi_{m'+1, m''}(a_{m'+1})$ is full in $D_{m''}.$
Therefore $a_{00}$ is full in $\psi_{m'', \infty}(D_{m''})$ for all
$m''>m'+1.$
Hence, the closed {{two-sided}} ideal generated by $a_{00}$ contains  
$\bigcup_{m''>m'+1}\psi_{m'', \infty}(D_{m''}).$ This implies that $a_{00}$ is full in $D,$ which
proves Claim (3).

It follows from \eqref{TDapdiv-23} that $\psi_{m', \infty}((d''-\ep_1/2)_+)$ is full
in $D.$  By \eqref{TDapdiv-22}, the element  $\psi_{m, \infty}((d'-\ep_1/4)_+)$ is full in $D.$
Then, by \eqref{TDapdiv-20}, $d_0$ is full in $D.$
Since $d_0$ is arbitrarily chosen in $(D)_+\setminus \{0\}$ with $\|d_0\|=1,$
this implies that $D$ is indeed simple.

By \eqref{Ddvi-nnnnn-100}, and then
as in the estimate of \eqref{LDsimple-n-2},
%\eqref{LDsimple-102-2},
\beq\nonumber
&&\tau(f_{1/4}(\psi_{1,n}(\phi_1(a_0)))>\tau(f_{1/4}(\psi_{1,n}(\phi_1^{(0)}(a_0)))-\mathfrak{f}_{a_0}/64\\\nonumber
&&=\tau(\psi_{1,n}(f_{1/4}(a_1)))-\mathfrak{f}_{a_0}/64>\tau(\phi_n^{(1)}\circ \cdots 
\phi_2^{(1)}(f_{1/4}(a_1)))-\sum_{j=2}^{n}(\mathfrak{f}_{a_0})^2/16^{j}-\mathfrak{f}_{a_0}^2/64\\\nonumber
&&\ge \tau(f_{1/4}(a_{n}))-(\mathfrak{f}_{a_0})^2/2^{(1+4)^2}-(\mathfrak{f}_{a_0})^2/16^2)(16/15)-\mathfrak{f}_{a_0}^2/64\\\nonumber
&&\ge \mathfrak{f}_{a_0}^2/8-\mathfrak{f}_{a_0}^2/32=\mathfrak{f}_{a_0}^2/16=:d_A\,
%&&\tau(f_{1/4}(\psi_{1, n}(\phi_1(a_0)))=\tau(\psi_{1,n}(\phi_1(a_0))\ge \tau(f_{1/8}(\phi_1(a_0)))\ge \mathfrak{f}_a^2/32
\rforal \tau\in {\mathrm{T}}(D_n).
\eneq
We also note, since $(\psi_2)_{*0}=0,$ {{that}} $(\psi_{1, n})_{*0}=0$ for all $n\ge 2.$
\end{proof}

\begin{thm}\label{TDapprdiv}
Let $A$ be a separable \CA\, in ${\cal D}_0.$ 
%or  in ${\cal D}$ with $K_0(A)=\{0\}.$
Then $A$ is tracially approximately divisible in the sense of  \ref{Dappdiv}.
\end{thm}

\begin{proof}
Let $\ep>0,$ ${\cal F}\subset A$ be a finite subset, $b\in A_+\setminus \{0\},$  and let $K\ge 1$ be an integer.
Let $e_A\in A$ be a strictly positive element with $0\le e_A\le 1.$ 
Choose $1/2>\sigma>0$ such that
\beq\label{TDapp-1}
\|f_\sigma(e_A)af_\sigma(e_A)-a\|<\ep/4\rforal a\in {\cal F}.
\eneq
Let ${\cal F}_1=\{ f_{\sigma}(e_A)af_{\sigma}(e_A): a\in {\cal F}\}$ and 
$A'=\overline{f_\sigma(e_A)Af_\sigma(e_A)}.$ Then $A'$ is algebraic{{ally}} simple, and {{so}} $A'={\rm Ped}(A').$
Choose $b_0\in (A')_+\setminus \{0\}$ such that $\la b_0\ra\le  \la b\ra.$ 

We  apply  \ref{TDefsimple} to $A',$  ${\cal F}_1,$ $\ep/4$ and $b_0.$ 
Let $D$ be as in \ref{TDefsimple}.  Put $C_k=\overline{\psi_{1, k}(\phi(A))D_k\psi_{1, k}(\phi(A))},$
$k=1,2,...,$ and $C=\lim_{k\to \infty} (C_k, \psi_k|_{C_{k-1}}).$  Then $C$ is a hereditary \SCA\, of $D$
and $C_k\in {{\cal C}_0^{0}}'.$ 
By \ref{Lmatrixpullb}, there exist $n\ge 1$ 
and \SCA s $D_n'={\rm M}_K(D_n'')\subset C_n$ such that
\beq\label{TDapp-2}
{\rm dist}(\psi_{1, n}\circ \phi(a), {\cal F}_2\otimes 1_K)<\ep/4\rforal a\in {\cal F}_1,
\eneq
where  $D_n''$ is a hereditary \SCA\, of $C_n$ and ${\cal F}_2\subset D_n''$ is a finite subset.

Let $\Phi_n$ be as in \ref{TDefsimple}. Then $\Phi_n(A)\perp j_n(\psi_{1,n}\circ \phi_1(A))$
and $c_n\lesssim b_0,$ where $c_n$ is a strictly positive element of $\overline{\Phi_n(A)A\Phi_n(A)}.$ 
%%\%Let $j_n: D_n\to A$ be as in
{{Recall that, in}}  \ref{TDefsimple}, {{$D_n\subset A.$}}
Define $A_0=\overline{\Phi_n(A)A\Phi_n(A)}.$ Then $A_0\perp j_n(C)$ {{and}}
$B=A_0\oplus M_K(j_n(D_n''))]\subset A.$ 
%and 
%\beq
%B_d=\{(x_0, \diag(\overbrace{x_1,x_1,...,x_1}^K): x_0\in A_0, x_1\in j_n(D_n'')\}.
%\eneq
Then 
\beq
{\rm dist}(\Phi(x)+j_n\circ \psi_n(x),\,  {{A_0+D_m''\otimes 1_K}})<\ep/4\rforal x\in {\cal F}_1.
\eneq
However, as part of the conclusion of  \ref{TDefsimple},
\beq
\|x-(\Phi(x)+j_n\circ \psi_n(x))\|<\ep/4\rforal x\in {\cal F}_1. 
\eneq
By \eqref{TDapp-1},
\beq
{\rm dist}(a,\,  {{A_0+D_m''\otimes 1_K}})<\ep\rforal a\in {\cal F}.
\eneq
Note we also have $c_n\lesssim b_0\lesssim b.$

%The lemma follows.
\end{proof}

\begin{rem}
In fact, one can also prove the conclusion of \ref{TDapprdiv}  by replacing 
{{the condition}} $A\in {\cal D}_0$ by 
$A\in {\cal D}$ and ${\rm K}_0(A)=0,$ and applying \ref{CDdiag}, {{since}} we will show that a
\CA\, in ${\cal D}$ has stable rank one (which will be done in  \ref{TTstr1}).
This will be carried out in \ref{D=D0K0} below. 
\end{rem}

\begin{cor}\label{LappZ}
Let $A$ be a   separable \CA\, in the class ${\cal D}.$ Then $A$ has the following property:
For any $\ep>0,$ any finite subset ${\cal F}\subset A,$ any $a_0\in A_+\setminus \{0\}$
and any integer $n\ge 1,$
there are  mutually orthogonal elements $e_0, e_{00}, e_{01}\in A_+,$ \cpc s $\phi_0: A\to E_0,$ $\phi_1: A\to E_1$,
and $\phi_2: A\to E_2,$  where $E_0, E_1, E_2$ are \SCA s of $A,$
$E_0=\overline{e_0Ae_0},$ $e_{00}\in E_1,$ $e_{01}\in E_2,$
$E_0\perp E_1,$ ${\mathrm{M}}_n(E_2)\subset E_1,$  with $E_2\in {\cal C}_0'$ and $E_2\subset \overline{e_{01}Ae_{01}}$,  such that
\beq\label{DappZ-1}
&&\|x-(\phi_0+ \phi_1)(x)\|<\ep/2\tand\\
&&\|\phi_1(x)-(r(x)+{{\phi_2(a)\otimes 1_n}})
%\diag(\overbrace{\phi_2(x), \phi_2(x), ...,\phi_2(x)}^n))
\|<\ep/2,\\
&&r(x)\in \overline{e_{00}Ae_{00}}\tforal x\in {\cal F},
\eneq
and $e_0+e_{00}\lesssim a_0$ and $e_{00}\lesssim e_{01}.$
\end{cor}
\begin{proof}
The proof is almost the same as that {{of}}
%In the proof of 
\ref{TDapprdiv}. 
%We will replace 
%, replace 
{{One replaces}}
${{\cal C}_0^{0}}'$ by ${\cal C}_0'$ 
%and keep the entire proof
%to the line ends ``... $D$ is indeed simple" shortly before \eqref{TDapdiv-26}.
and instead of applying the first part of \ref{Lmatrixpullb}, 
%in the last few lines
%of the proof, 
%we apply 
{{one applies}} the second part of \ref{Lmatrixpullb}. We omit the repetition. 
\end{proof}

%\begin{cor}\label{CCCC}
%Let $A\in {\cal D}$ (or in ${\cal D}_0$)  be separable  \CA.  Then the following holds:
%Let  $e_A\in A$ be a strictly positive element $\ep>0,$ let ${\cal F}\subset A,$ and let $a\in A_+\setminus \{0\}.$ 
%There exist a separable simple \CA\, $D\subset A$ which is an inductive limit 
%of \CA s in ${\cal C}_0'$ (or in ${\cal  C}_0^{0'}$) such that $D={\rm Per}(D),$ 
%\cpc s $\phi_0: A\to A$ and $\phi_1: A\to D$ with $\phi(A)\perp D$ such that
%\beq
%cc
%\eneq

%\end{cor}

\begin{thm}\label{TCCdvi}
Let $A$  be a  separable  \CA\, in ${\cal D}_{0}.$ 
%%%or in ${\cal D}$ with  $K_0(A)=\{0\}.$ 
%%%Suppose that 
{{Let}} $a\in A_+$ with $\|a\|=1$ {{be a}}
% is {\red{a}} 
%(or in ${\cal D}$)
 strictly positive element.
% $a\in A_+$ with $\|a\|=1.$
Then the  following  statement is true.

  There exists
$1> \mathfrak{f}_a>0$ such that, for any $\ep>0,$  any
finite subset ${\cal F}\subset A$ and any $b\in A_+\setminus \{0\}$ and any integer
$n\ge 1,$  there are ${\cal F}$-$\ep$-multiplicative \cpc s $\phi: A\to A$ and  $\psi: A\to D$  for some
\SCA\, $D\in {{\cal C}_0^0}'$ with ${\mathrm{M}}_n(D)\subset A$ and ${\mathrm{M}}_n(D)\perp \phi(A)$ such that $\|\psi\|=1,$
\beq\label{CCdiv-1}
&&\|x-
%\diag(\phi(x), \overbrace{\psi(x), \psi(x),...,\psi(x)}^n)
(\phi(x)+\psi(x)\otimes 1_n)\|<\ep\tforal x\in {\cal F}\cup \{a\},
%\\\label{CCdiv-2}
%&&D\in {\cal C}_0^{0'}
%({\rm or}\,\, {\cal C}_0'),
\\\label{CCdiv-3}
&&
%\phi(a)
c\lesssim b,
%\label{CCdiv-4}
%
\\\label{CCdiv-4+}
&&t(f_{1/4}(\psi(a)))\ge \mathfrak{f}_a,\quad t\in {\mathrm{T}}(D),\ \tand
%&&f_{1/4}(\psi(a))\,\,\,
%\psi(f_{1/4}(a_1)) \,\,\,
%{\rm is\,\,\, full\,\,\, in}\,\, D \andeqn
%\psi_D={\rm id}_D\andeqn
\eneq
$\psi(a)$ is strictly positive in $D,$ where $c$ 
%Let $A_0=\overline{\phi(a)A\phi(a)}.$ Then
%$A_0\perp {\mathrm{M}}_n(D).$ 
%Moreover, \eqref{CCdiv-3} could be replaced by
%$c\lesssim b$ for some 
is a strictly positive element  of $A_0:=\overline{\phi(a)A\phi(a)}.$
\end{thm}

%Note that there is a \SCA\, $M_n(D)$ in $A,$ where $D\oplus D\oplus \cdots \oplus D$ sitting
%in the diagonal.

\begin{proof}
%By \ref{TDapprdiv}, $A$ has the property of tracial approximate divisibility.
Fix a strictly positive element $a\in A_+$ with $\|a\|=1.$
It follows from \ref{PD0qc} that $0\not\in \overline{{\mathrm{T}(A)}}^\mathrm{w}.$
Let
\beq\label{CCdvi-10}
r_0=\inf\{\tau(f_{1/2}(a)): \tau\in \overline{{\mathrm{T}(A)}}^\mathrm{w}\}>0.
\eneq

Let $\mathfrak{f}_{a_0}=r_0/6.$ Choose an integer $k_0\ge 1$ such that
$r_0/16>1/k_0.$

Let $1>\ep>0$ and ${\cal F}\subset A$ be a finite subset.
Choose $\ep_1=\min\{\ep/16, r_0/128\}.$ Let ${\cal F}_1\supset {\cal F}\cup \{a, f_{1/4}(a)\}$
be a finite subset of $A.$
Let
$b\in A_+\setminus \{0\},$ and any
integer $n\ge 1$  be given.

Choose $b_0',b_1',...,b_{n+2k_0}'\in \overline{bAb}$ such
that ${{b_0',b_1',..., b_{n+2k_0}'}}$ are mutually orthogonal and mutually equivalent in the sense of Cuntz
and there are non-zero {{and mutually orthogonal}} elements $b_0, b_1,...,b_{n+2k_0}\in A_+$ such
that $b_ib_0'=b_i,$ $i=0,1,...,n+2k_0.$

By \ref{TDapprdiv}, $A$ has the property of tracial approximate divisibility.
Therefore
there are \SCA s
$A_0,A_1$ and $M_n(A_1)$ of $A$ such that {{$A_0\perp M_n(A_1),$}}
$$
{\rm dist}(x, {{A_0+A_1\otimes 1_n}})<\ep_1/2 \rforal x\in {\cal F}_1,
$$
%\vspace{-0.1in}$$
%{\rm dist}(x, B_d)<\ep_1/2 \rforal x\in {\cal F}_1,
%%$$
%%where $B_d\subset B\subset A,$ $B=A_0\bigoplus  {\mathrm{M}}_n(A_1)$ and 
%%\vspace{-0.12in}\beq\label{CCdvi-11}
%B=A_0\bigoplus  {\mathrm{M}}_n(A_1),\,\,
%\overbrace{A_1\oplus A_1\oplus\cdots \oplus A_1}^n,\\
%%B_d=\{(x_0, \diag(\overbrace{x_1,x_1,...,x_1}^n))\in B: x_0\in A_0, x_1\in A_1\}
%%\eneq
and $a_0\lesssim  b_0,$ where $a_0$ is a strictly positive element of $A_0.$
Moreover, there are $y_0\in A_0$ and $y_1\in A_1$ such that
\vspace{-0.12in}\beq\label{CCdvi-12}
&&\|a-(y_0+ \diag(\overbrace{y_1,y_1,...,y_1}^n){{)}}\|<\ep/2\andeqn\\
&&\|f_{1/4}(a)-(f_{1/4}(y_0)+ \diag(\overbrace{f_{1/4}(y_1),f_{1/4}(y_1),...,f_{1/4}(y_1)}^n){{)}}\|<\ep_1/2.
\eneq

Note that
\vspace{-0.1in}\beq\label{CCdvi-12+}
\tau(\diag(\overbrace{f_{1/4}(y_1),f_{1/4}(y_1),...,f_{1/4}(y_1)}^n))\ge r_0-1/(n+2k_0)-\ep_1/2>r_0/3
\eneq
for all $\tau\in {\mathrm{T}(A)}.$

Let $A_0'=\overline{a_0Aa_0}$ and $A_1'=\overline{y_{1}Ay_{1}}.$
Note that $0\not\in \overline{{\mathrm{T}}(A_1')}^\mathrm{w}$ {{by}} \ref{Pheretc}.
Moreover, if $\tau\in {\mathrm{T}(A)},$ then $\|\tau|_{A_1}\|\ge r_0/3.$
We also have
\beq\label{CCdvi-12+1}
\tau(f_{1/4}(y_1))\ge r_0/3\rforal \tau\in {\mathrm{T}}(A_1').
\eneq

Note, by \ref{Rfa0}, in Definition  \ref{DNtr1div}
%{Dgtr11},
the constant $\mathfrak{f}_{y_1}$ can be chosen to be
$r_0/6.$

Let ${\cal G}\subset A_1$ be a finite subset such that the following holds
\beq\label{CCdvi-13}
{\rm dist}(f, {{\{}}(x_0+ \diag(\overbrace{x_1,x_1,...,x_1}^n)): x_0\in A_0, x_1\in {\cal G}\})<\ep_1/2\rforal f\in {\cal F}_1
\eneq
and $y_1\in {\cal G}.$

Note that $A_1'$ is a hereditary \SCA\, of $A.$
{{By \ref{Phered}, $A_1'\in {\cal D}_0.$ Thus,}} there exist  two \SCA s  $B_0$ and $D$ of $A_1',$ where
$D\in {{\cal C}_0^0}'$
%(or ${\cal C}_0^{0'}$),
and two
${\cal G}$-$\ep_1$-multiplicative \cpc s
$\phi_0: A_1'\to B_0$ and $\psi_0: A_1'\to D$ such that
\beq\label{CCdvi-13+}
&&\|x-(\phi_0+\psi_0)(x)\|<\ep_1/2\rforal x\in {\cal G},\\
&&\phi_0(c_0)\lesssim b_1,\,\,\,
\|\psi_0\|=1,\,
%\andeqn
\\\label{CCdvi-13++}
&&\tau\circ f_{1/4}(\psi_0(y_1))\ge r_0/6 \rforal \tau\in {\mathrm{T}}(D),
\eneq
and $\psi_0(y_1)$ is a strictly positive element in $D,$ where $c_0$ is a strictly positive element of
$A_1'.$

Set $A_{00}=A_0\oplus \overbrace{A_0'\oplus A_0'\oplus \cdots \oplus A_0'}^n$
and let
$
c={{a_0}}+\diag(\overbrace{c_0,c_0,...,c_0}^n).
$

Choose  a function $g\in C_0((0,1]),$ define
$\phi_{00}: A\to A_{00}$ by
$$
\phi_{00}(x)=g(a_0)xg(a_0) +\diag(\overbrace{\phi_0(x),\phi_0(x),...,\phi_0(x)}^n)\rforal x\in A.
$$
Then, with a choice of $g,$ we have
\vspace{-0.12in}\beq\label{CCdvi-20}
\|x-(\phi_{00}(x) +\diag(\overbrace{\psi_0(x),\psi_0(x),...,\psi_0(x)}^n))\|<\ep\rforal x\in {\cal F}.
\eneq
Moreover,
$$
\la c\ra \le \la  b_0\ra +\la  b_1\ra + \cdots \la b_n\ra \le \la  b\ra.
%c\lesssim b_0\oplus b_1\oplus \cdots \oplus b_n\lesssim b.
$$
Now let $\phi=\phi_{00}.$ Then $\phi(a)=\phi_0(a)\lesssim c\lesssim b.$  
Put  $\psi=\psi_0.$   Note also \eqref{CCdvi-13++} holds.  It follows that  $\phi,$ $\psi,$ and  $D$ meet
the requirements.

\end{proof}

The following corollary follows from  the combination of \ref{UnifomfullTAD}  and \ref{TCCdvi}.

\begin{cor}\label{Cuniformful}
Let $A$ be a separable  algebraically simple \CA\, in ${\cal D}_{0}$ (cf. \ref{114}).
%%%%%%%%%%%%%%%%%%%%or in ${\cal D}$ with $K_0(A)=\{0\}.$
%with continuous scale.
%(or  in ${\cal D}$).
Then the following {{property}} holds.
Fix  a strictly positive element $a\in A$
with $\|a\|=1$ and let $1>\mathfrak{f}_a>0$ be as in \ref{DNtr1div} (see also \ref{Rfa0}). There is a map $T: A_+\setminus \{0\}\to \N\times \R_+\setminus \{0\}$
($a\mapsto (N(a), M(a))\rforal a\in A_+\setminus \{0\}$) satisfying the following condition:
For any  finite subset ${\cal F}_0\subset A_+\setminus \{0\}.$
%   There exists
%   $M>0$ and an integer $N\ge 1,$
 for any $\ep>0,$  any
finite subset ${\cal F}\subset A$ and any $b\in A_+\setminus \{0\}$ and any integer
$n\ge 1,$  there are ${\cal F}$-$\ep$-multiplicative \cpc s $\phi: A\to A$ and  $\psi: A\to D=D\otimes e_{11}$  for some
\SCA\, ${\mathrm{M}}_n(D)\subset A$ with $\phi(A)\perp {\mathrm{M}}_n(D)$ such that
\beq\label{CDtad-1}
&&\|x-({{\phi(x)+
%\diag(\overbrace{\psi(x), 
\psi(x)\otimes 1_n}})
%,...,\psi(x)}^n))
\|<\ep,\quad x\in {\cal F}\cup \{a\},\\\label{CDtrdiv-2}
&& D\in {\cal C}_0^{0},
%\,({\rm or}\,\,\,{\cal C}_0'),
\\\label{CDtad-3}
&&a_{0}\lesssim b,\,\,
%\\\label{CDtrdiv-4}
%&&
 \|\psi\|=1,\ and
%\\\label{Dtad-4+}
\eneq
$\psi(a)$ is strictly positive in $D,$  where $a_{0}\in \overline{\phi(a)A\phi(a)}$ is
a strictly positive element.
Moreover,  $\psi$ is $T$-${\cal F}_0\cup\{f_{1/4}(a)\}$-full in $D.$
%\overline{DAD}.$

Furthermore, we may assume that
\beq
&&{{t(f_{1/4}(\psi(a)))}}\ge \mathfrak{f}_a\andeqn\\\label{106++}
&& {{t(f_{1/4}({{\psi(c)}}))}}\ge {\mathfrak{f}_a
\over{4\inf\{M(c)^{{2}}\cdot N(c): c\in {\cal F}_0\cup\{f_{1/4}(a)\}\}}}
\eneq
for all $c\in {\cal F}_0$ and for all $t\in {\mathrm{T}}(D),$  and, we may also require that
\beq\label{106+}
\phi(a)\lesssim \psi(a).
\eneq

% for any $\eta>0,$  any $x\in D$ and any $c\in {\cal F}_0\cup \{f_{1/4}(a)\},$
%for each $c\in {\cal F}_0\cup\{a\},$
%there are $x(c,1), x(c,2),...,x(c,N), y(c,1), y(c,2),...,y(c,N)\in D$ with
%$\|x(c,i)\|,\,\, \|y(c,i)\|\le M$ for all $c\in {\cal F}_0,$ $i=1,2,...,N,$ such that
%$$
%\|\sum_{i=1}^N x(c,i)\psi(c)y(c,i)-x\|<\eta.
%$$

%In  addition,  we may assume
%that
%\beq\label{DVIrepeat}
%a_0\lesssim f_{1/4}(\psi(a)).
%\eneq
\end{cor}

\begin{proof}  Note that 
the existence of the map $T$ and {{the fact that}} $\psi$ can be required to be $T$-${\cal F}_0\cup\{f_{1/4}(a)\}$-full
in $D,$ and  {{that}} \eqref{106++} holds,  are applications of \ref{Tqcfull}.

To  see the last part of conclusion, i.e.,  \eqref{106+}, let $1/2>\eta>0$  be such
that
$\tau(f_\eta(a))> \mathfrak{f}_a/2$ for all $\tau\in \overline{{\mathrm{T}(A)}}^{\rm w}$
(see  \ref{Rfa0}) and
choose $b\in A_+\setminus \{0\}$ such that $\mathrm{d}_\tau(b)<\mathfrak{f}_a/4(n+1)$ for all $\tau\in {\mathrm{T}(A)}.$
Then, with $\ep<\eta/4,$  \eqref{CDtad-1} implies that
\beq
f_\eta(a)\lesssim  \phi(a)+\diag(\overbrace{\psi(a), \psi(a),...,\psi(a)}^n).
\eneq
It follows that $\mathrm{d}_\tau(\psi(a))>\mathfrak{f}_a/2(n+1)$ for all $\tau\in \overline{{\mathrm{T}(A)}}^{\rm w}$ or
\beq
\mathrm{d}_\tau(\psi(a))>\mathrm{d}_\tau(b)\ge \mathrm{d}_\tau(\phi(a))\rforal \tau\in \overline{{\mathrm{T}(A)}}^{\rm w}.
\eneq
It follows  by {{(1) of}} \ref{Pconscale}  and  \ref{Comparison} that $\phi(a)\lesssim \psi(a).$

\iffalse
$\tau(f_\eta(a))>1-1/4(n+1)$ for all $\tau\in {\mathrm{T}(A)}$ and
choose $b\in A_+\setminus \{0\}$ such that $\mathrm{d}_\tau(b)<1/4(n+1)$ for all $\tau\in {\mathrm{T}(A)}.$
Then, with $\ep<\eta/4,$  \eqref{CDtad-1} implies that
\beq
f_\eta(a)\lesssim  \phi(a)+\diag(\overbrace{\psi(a), \psi(a),...,\psi(a))}^n).
\eneq
It follows that $\mathrm{d}_\tau(\psi(a))>1/2(n+1)$ for all $\tau\in {\mathrm{T}(A)}$ or
\beq
\mathrm{d}_\tau(\psi(a))>\mathrm{d}_\tau(b)\ge \mathrm{d}_\tau(\phi(a))\rforal \tau\in {\mathrm{T}(A)}.
\eneq
It follows  by \ref{Pconscale} (1) and  \ref{Comparison} that $\phi(a)\lesssim \psi(a).$
\fi
\end{proof}

\begin{rem}\label{Rm17div}
It is clear from the proof that,  for $n=1,$  both \ref{TCCdvi} and \ref{Cuniformful} hold
if $A\in {\cal D}$ (with  now $D\in {\cal C}_0$).
\end{rem}

\section{Stable rank one}

The proof of the following result  is very similar to that of  Lemma 2.1 of \cite{Rlz}.

\begin{lem}\label{NTstr1pre}
Let $A$ be a 
%$\sigma$-unital 
separable, simple and   stably projectionless \CA\, 
such that every hereditary \SCA\, $B$  has 
strict comparison for positive elements  as {{formulated in}}  the conclusion of \ref{Comparison},
and satisfies the 
%Suppose that every hereditary \SCA\, $B$ of $A$
%satisfies the 
conclusion of \ref{LappZ} without assuming {{that}} $E_2$ belongs to a specific class 
of \CA s. 
% and is $\sigma$-unital.
%is tracially approximate divisible as defined
%in \ref{Dappdiv}.
Then, for any hereditary \SCA\, $B$ of $A,$
$$
B\subset {\overline{{\rm{GL}}({\widetilde B})}}
$$

\end{lem}

\begin{proof}
%Note that, by the proof of \ref{Ca=ped}, $A={\rm Ped}(A).$
%It follows that every hereditary \SCA\, of $A$ also has bounded scale.
%In particular, $0\not\in  \overline{T{{(}}A)}^\mathrm{w}.$
%Since every hereditary \SCA\, $B$ of $A$ is in ${\cal D}_{1,0},$
%Since every hereditary \SCA\, $B$ of $A$ has the same said properties,
It is clearly sufficient to 
%it suffices to  
consider the case $B=A.$
%show that $A\subset {\overline{GL({\tilde {{A}}})}}.$
%{\blue{do we need to assume $B$ to be full hereditary, otherwise  $0\not\in \overline{T(B)}^\mathrm{w} $ may not be true}}

Fix an element $x\in A$
%{\widetilde A}$
and $\ep>0.$
%To simplify notation, \wilog, we may assume that $x=1+x'$
Let $e\in A$ with $0\le e\le 1$ be a strictly positive element.
Upon replacing $x$ by
$f_\eta(e)xf_\eta(e)$ for some small $1/8>\eta>0,$  we may assume that
$x\in \overline{f_\eta(e)Af_{\eta}(e)}.$ Put
$B_1=\overline{f_\eta(e)Af_{\eta}(e)}.$

By the assumption, we know that $e$ is not a projection.
{{W}}e obtain a positive element ${{b_0}}\in B_1^{\perp}\setminus \{0\}.$
%{\blue{I think we use $b_0$ not $b_0'$ later}}

Note that
$$
B_{{1}}^{\perp}=\{a\in A: ab=ba=0\rforal b\in B_1\}
$$
is a non-zero  hereditary \SCA\, of $A.$
Since we assume that $A$ is infinite dimensional, $\overline{b_0Ab_0}$
contains  non-zero positive elements $b_{0,1},b_{0,1}', b_{0,2}, {{b_{0,2}'}}\in B_1^{\perp}$
such that
$$
b_{0,1}'\lesssim b_{0,2}'\andeqn b_{0,1}b'_{0,1}=b_{0,1}, b_{0,2}b_{0,2}'=b_{0,2}\andeqn
b_{0,1}'b_{0,2}'=0.
$$

Since $A$ has the strict comparison for positive element{{s}} as in the conclusion of \ref{Comparison},
we can 
choose a large  integer $n\ge 2$ which has the following property:
 if $a_1,a_2,...,a_n\in A_+$ are 
% there are 
 $n$ mutually orthogonal and mutually equivalent positive elements,
%$a_1,a_2,...,a_n\in A_+,$
then
$$
a_1+a_2\lesssim b_{0,1}.
%,\,\,\,i=1,2,...,n.
$$

There is $B_1'\subset B_1$
which has the form
$$
B_1'=B_{1,1}+ D\otimes 1_n,
%\overbrace{D\oplus D\oplus \cdots \oplus D}^n,
$$
where $B_{1,1}$ is a hereditary \SCA\, with a strictly positive element $b_{11}\lesssim b_{0,1}$
and
%$$
%{{x\in_{\ep/16}}} B_{1,d}=\{(b, \overbrace{d,d,...,d}^n): b\in B_{1,1}\andeqn d\in D\},
%$$
%{{That is,
there are $x_0\in B_{1,1}$ and $x_1\in D\setminus \{0\}$ such that
\beq\label{NTstr1-pre-10}
\|x-(x_0+
%{\bar x}_1)
x_1\otimes 1_n)\|<\ep/16.
%\andeqn\\
%{\bar x_1}= \diag(\overbrace{x_1,x_1,...,x_1}^n).
\eneq
%where $x_0\in B_1'$ and $x_1\in D\setminus \{0\}.$

Let $d_0\in D$ be a strictly positive element.  By the choice of
$n,$  $d_0\lesssim b_{0,1}.$

Choose $0<\eta_1<1/4$ such that
\beq\label{NTstr1-pre-11}
\|f_{\eta_1}(d_0)x_1f_{\eta_1}(d_0)-x_1\|<\ep/16.
\eneq
Put $x_1'=f_{\eta_1}(d_0)x_1f_{\eta_1}(d_0).$
Note that
\beq\label{NTstr1-pre-11+}
f_{\eta_1/8}(d_0)\lesssim  b{{'}}_{0,2}.
\eneq
There are $w_i\in A$ such that
\vspace{-0.12in}\beq\label{NTstr1-pre-12}
w_iw_i^*&=&\diag(\overbrace{0,0,...,0}^{i-1}, f_{\eta_1/4}(d_0),0,...,0),\quad i=1,2,...,n,\\
 w_i^*w_i&=&\diag((\overbrace{0,0,...,0}^{i}, f_{\eta_1/4}(d_0),0,...,0),\quad i=1,2,...,n-1,\andeqn\\
 w_n^*w_n&\in& \overline{b{{'}}_{0,2}Ab{{'}}_{0,2}}.\,\,\,
\eneq

There is $v\in A$
such that
\beq\label{NTstr1-pre-13}
v^*v=x_0+\diag(x_1',0,...,0)\andeqn vv^*\in \overline{(b_{0,1}'+b_{0,2}')A(b_{0,1}'+b_{0,2}')}.
\eneq

Put
\vspace{-0.12in}\beq
&&x_i''=\diag(\overbrace{0,0,...,0}^{i-1},x_1',0,...,0),\,\,\,i=1,2,...,n,\\
&&y_i''=\diag(\overbrace{0,0,...,0}^{i-1},f_{\eta_1/4}(d_0),0,...,0),\,\,\,i=1,2,...,n,\hspace{-0.02in}{{\andeqn}}\\
&&z_1=v^*,\,\,\, z_2=v, \,\,\,
z_3=\sum_{i=1}^{n-1}w_i^*x_{i}''\andeqn
%\diag(w_1^*x_1', w_2^*x_1',...,w_{n-1}^*x_1',0)\andeqn\\
z_4=\sum_{i=1}^{n-1}y_i''w_i.
%\diag(f_{\eta_1/4}(d_0))w_1, f_{\eta_1/4}(d_0)w_2,...,f_{\eta_1/4}w_{n-1},0).
\eneq
Note  that
\vspace{-0.12in}\beq
z_3z_2=0, z_1z_4=0.
\eneq
Therefore,
\beq
(z_1+z_3)(z_2+z_4)&=&z_1z_2+z_3z_4\\
&=& v^*v+\diag(0, x_1',x_1',...,x_1')\\
&=&x_0+ \diag(\overbrace{x_1',x_1',...,x_1'}^n){{=x_0+x_1'\otimes 1_n}}.
\eneq
On the other hand,
\vspace{-0.12in}\beq\label{NTstr1-pre-14}
z_1^2=v^*v^*=0,\,\,\, z_1z_3=0.
\eneq
We also compute
that
\vspace{-0.12in}\beq\label{NTstr1-pre-15}
z_3^2=\sum_{i,j}w_i^*x_{i}''w_j^*x_j''=\sum_{i=2}^{n-1}w_i^*x_{i}''w_{i-1}^*x_{i-1}''.
\eneq
Inductively, we compute that
\vspace{-0.12in}\beq\label{NTstr1-pre-16}
z_3^n=0.
\eneq
Thus, by \eqref{NTstr1-pre-14},
% and \eqref{NTstr1-pre-16},
\vspace{-0.05in}\beq\label{NTstr1-pre-17}
(z_1+z_3)^k=\sum_{i=1}^k z_3^iz_1^{k-i}\rforal k.
\eneq
Therefore, by \eqref{NTstr1-pre-14},
and \eqref{NTstr1-pre-16}, for $k=n+1,$
\vspace{-0.02in}\beq\label{NTstr1-pre-18}
(z_1+z_3)^{n+1}=0.
\eneq
We also have that  $z_2z_4=0$ and $z_2^2=0.$
A similar computation shows that $z_4^n=0.$
Therefore, as above, $(z_2+z_4)^{n+1}=0.$
One   has the estimate
%that
$$
\|x-(z_1+z_3)(z_2+z_4)\|<\ep/4.
$$
Suppose that $\|z_i\|\le M$ for $i=1,...,4.$
Consider {{the}} elements of ${\widetilde A}$ 
$$
z_5=z_1+z_3+\ep/16(M+1)\andeqn z_6=z_2+z_4+\ep/16(M+1).
$$
Since $(z_1+z_3)$ and $(z_2+z_4)$ are nilpotent,
both $z_5$ and $z_6$ are invertible in ${\widetilde A}.$
We also estimate
that, by \eqref{NTstr1-pre-10},
$$
\|x-z_5z_6\|<\ep.
$$

\end{proof}

\begin{cor}\label{CD0str1}
Let $A\in {\cal D}$ be a separable \CA.
Then $A$ almost has stable rank one (see \ref{Dalst1}).
\end{cor}
%\begin{df}\label{DTAD}
\begin{proof}
This follows from {{\ref{PtadMk}, \ref{Phered},}} \ref{Pprojless}, \ref{LappZ},  and \ref{NTstr1pre}.
\end{proof}

\begin{cor}\label{CCped}
Let $A\in {\cal D}.$ Suppose that $A$ is separable. Then 
%$A$  is  algebraically simple and 
$A={\rm Ped}(A).$
\end{cor}

\begin{proof}
By \ref{PtadMk}, ${\rm M}_n(A)$ is in ${\cal D}.$ 
The corollary then  follows from  the combination of  \ref{PD0qc}, \ref{CD0str1}, \ref{Comparison},
%\ref{PD0qc}, \ref{Comparison}, \ref{CD0str1}, 
 \ref{Lbkqc}, and \ref{Tpedersen}.
\end{proof}

\begin{rem}\label{114}
Note, by \ref{CCped}, the assumption that $A={\rm Ped}(A)$ in 
%\ref{Ltrace}, \ref{Tcomatilde} and 
\ref{TDefsimple} can be removed.
\end{rem}

\begin{thm}\label{TTstr1}
Let $A$ be a separable 
%stably projectionless 
\CA\, {{in ${\cal D}$.}}
 Then $A$ has stable rank one.
\end{thm}

\begin{proof}
Let $x\in {\widetilde A}.$ We must show that $x\in\overline{{\rm{GL}}({\widetilde A})}.$
{{Applying}} 3.2  and 3.5 of \cite{Rr11}, \wilog, we may assume that
there exists a non-zero positive element $e_0'\in {\widetilde A}$  with $\|e_0'\|=1$ such that
$xe_0'=e_0'x=0.$ We may further assume that there exists 
$e_0\in {\widetilde A}_+$  with $\|e_0\|=1$ such that $e_0e_0'=e_0'e_0=e_0.$
Define 
\beq
\sigma=\inf\{\tau(f_{1/4}(e_0)): \tau\in \overline{\mathrm{T}(A)}^{\rm w}\}.
\eneq

Multiplying by a scalar multiple of the identity, \wilog, we may assume that
$x=1+a,$ where $a\in A.$
%\Wlog, we may assume that $\|x\|\le 1.$ Let $1>1/4>\ep>0.$ 

{{Let}} $0<\ep_0<\ep$ {{be given}} and {{set}} $\ep_1=\min\{\ep_0/(\|x\|+1), \sigma\}.$
Since $A\in {\cal D}, $ there exist   a hereditary \SCA\, $B_0\subset A$ 
%in the class $\mathcal D$ 
and a
\SCA\,
$D\subset A$  with $D\in {\cal C}_0$ such that
\beq\label{TTstr1-5}
\|a-(x_0+x_1)\|<\ep_1/64,\,\,\,
\|e_0-(e_{0,0}+e_{0,1})\|<\ep_1/64,\andeqn\\\label{TTstr1-5+}
\|f_{\dt'}(e_0)-(f_{\dt'}(e_{00})+f_{\dt'}(e_{0,1}))\|<\ep_1/64,\,\,\, \dt'\in \{1/2^k: 2\le k\le 6\},
\eneq
where $x_0, e_{0,0}\in B_0$ and $x_1, e_{0,1}\in D,$ $B_0D=DB_0=\{0\},$
\beq\label{TTstr1-6}
d_\tau(b_0)<  \min\{\ep_1/64,  \sigma/64\}\rforal \tau\in  \overline{\mathrm{T}(A)}^{\rm w}.
\eneq
where $b_0$ is a strictly positive element of $B_0.$ 
Let $p_{B_0}$ {{denote}} the open projection associated with $B_0.$ 
Then, for $\dt'\in \{1/2^k: 2\le k\le 6\},$ 
\beq
\hspace{-0.4in}\|(1+x_1)f_{\dt'}(e_{0,1})\|&=&\|(1+x_1)(1-p_{B_0})f_{\dt'}(e_{0,1})\|\\
&=&\|(1+x_1+x_0)(1-p_{B_0})(f_{\dt'}(e_{0,0})+f_{\dt'}(e_{0,1}))\|\\
&< & \|(1+x_1+x_0)(1-p_{B_0})f_{\dt'}(e_0)\|+(\|x\|+\ep_1/64)\ep_1/64\\
&=& \|(1-p_{B_0})(1+x_1+x_0))f_{\dt'}(e_0)\|+\ep_0/64\\
&<& \|(1-p_{B_0})xf_{\dt'}(e_0)\|+\ep_1/64+\ep_0/64\\
&=&\ep_1/64+\ep_0/64<\ep_0/32.
\eneq
Put 
\beq
x_1'=(-2f_{1/64}(e_{0,1})+f_{1/64}(e_{0,1})^2)+(1-f_{1/64}(e_{0,1}))x_1(1-f_{1/64}(e_{0,1})).
\eneq
Then $x_1'\in D.$   By the calculation above,
\beq
&&(1-f_{1/64}(e_{0,1}))(1+x_1)(1-f_{1/64}(e_{0,1}))=1+x_1'\andeqn\\
&&\|(1+x_1')-(1+x_1)\|<3\ep_0/30.
\eneq
Moreover, 
\beq
(1+x_1')f_{1/64}(e_{0,1})=(1-f_{1/64}(e_{0,1}){{)}}(1+x_1)(1-f_{1/64}(e_{0,1}))f_{1/16}(e_{0,1})=0.
\eneq
We also have, by \eqref{TTstr1-5+}, 
\beq
\tau(f_{1/4}(e_{0,1}){{)}}\ge \tau(f_{1/4}(e_0))-\ep_1/64\ge \sigma-\ep_0/64>\sigma/2\rforal \tau\in \overline{\mathrm{T}(A)}^{\rm w}.
\eneq
Therefore, by \eqref{TTstr1-6},
\beq
d_\tau(b_0)<\tau(f_{1/4}(e_{0,1}))\rforal \tau\in \overline{\mathrm{T}(A)}^{\rm w}.
\eneq
By  \eqref{Comparison}, $b_0\lesssim f_{1/8}(e_{0,1}).$  Note 
that $f_{1/16}(e_{0,1})f_{1/8}(e_{0,1})=f_{1/8}(e_{0,1}).$

To simplify notation, choosing sufficiently small $\ep_0$ and changing notations,
we may assume that
\beq\label{TTstr1-5n}
\|a-(x_0+x_1)\|<\ep/16\andeqn
\|e_0-(e_{0,0}+e_{0,1})\|<\ep/16,
%\andeqn\\
%\|f_{\dt'}(e_0)-(f_{\dt'}(e_{00})+f_{\dt'}(e_{0,1}))\|<\ep_0/64,\,\,\, \dt'\in \{1/2^k: 2\le k\le 6\}.
\eneq
where $x_0, e_{0,0}\in B_0$ and $x_1, e_{0,1}\in D,$
{{and also}}
\beq\label{TTstr1-n-1-1}
e_{0,1}(1+x_1)=(1+x_1)e_{0,1}=0
\eneq
 and 
$b_0\lesssim e_{0,1}',$   where $0\le e_{0,1}'
\le 1$ and  $e_{0,1}' f_{\dt}(e_{0,1})=e_{0,1}'$ for some  $0<\dt<1/4.$ 

We may also assume, \wilog,  that there  are $b_{0,1}, b_{0,1}'\in B_0$ 
%, b_{0,1}'' b_{0,1}'''\in B_0$
with $0\le b_{0,1},\, b_{0,1}'\le 1$ 
% b_{0,1}''\le b_{0,1}'''\le 1$
such that
\beq\label{TTstr1-7}
b_{0,1}x_0=x_0b_{0,1}=x_0, \,\,\,f_{1/16}(b_{0,1}')b_{0,1}=b_{0,1}.
%'''b_{0,1}''=b_{0,1}'',\,\,\,b_{0,1}''b_{0,1}'=e_{0,1}'\andeqn b_{0,1}'b_{0,1}=b_{0,1}.
\eneq
  {{Set}} $A_2=\overline{(f_{\dt}(e_{0,1})+b_{0,1}')A(f_{\dt}(e_{0,1})+b_{0,1}')}.$
Note {{that,}}  and by \ref{Phered}, $A_2\in {\cal D}.$ Since $b_{0,1}'\lesssim b_{0}\lesssim e_{0,1}',$  and by \ref{NTstr1pre}, 
$A_2$ almost has stable rank one,
there is  a unitary $u_1'\in {\widetilde A_2}$ (see \ref{LRL}) such that
\beq\label{TTstr1-8}
(u_1')^*b_{0,1}'(u_1')\in \overline{f_{\dt}(e_{0,1})Af_{\dt}(e_{0,1})}.
\eneq

%Let $A_3$ be the hereditary \SCA\, of $A$ generated by $A_0$ and $A_2.$
%Let $q$ be the open projection  in $A^{**}$ corresponding to $A_3,$
%and l
Let $q_0$ {{denote}} the open projection in $A^{**}$ corresponding to $\overline{b_{0,1}'Ab_{0,1}'},$   and
$q$ be the open projection in $A^{**}$  corresponding to the hereditary 
\SCA\, $A_2.$  {{Then}} $q_0\le q.$ 
Note that
\beq
&x_0q_0=x_0q_0=x_0\andeqn\\\label{TTstr1-9}
&q_0u_1'q_0=(u_1')(u_1')^*(q_0u_1')q_0=(u_1')((u_1')^*q_0u_1')q_0=0.
% f_{\dt/2}(e_{0,1})\le e_{0,1}''\le q.
\eneq
Note also that
\beq\label{TTstr1-10}
&&\|x-(1+x_0+x_1)\|=\|a-(x_0+x_1)\|<\ep/16.
%\andeqn\\
%&&x-(p+x_0+x_1)=a-(x_0+x_1)\in B_0\cap A.
\eneq
%In particular, $1+x_0+x_1\in {\widetilde A}.$
Put $z=1+x_0+x_1. $   Then $z\in {\widetilde A}.$ Put
%Then we also view that $z\in {\tilde B}_0.$
\beq\label{TTstr1-11}
&&z_0=zq_0=(1+x_0)q_0=q_0(1+x_0)\andeqn\\
 &&z_1=z(1-q_0)=(1-q_0)z=(1-q_0)+x_1.
\eneq
Keep in mind
that
%\vspace{-0.1in}\beq\label{TTstr1-11+}
$z_0+z_1=z.$
%\eneq

Now write $u_1'=\lambda 1_{{\widetilde A}_2}+y$ for some $y\in A_2$ and for some
scalar $\lambda\in \C$ with $|\lambda|=1.$
{{Set}} $u_1=\lambda q+y.$
Multiplying by ${\bar \lambda}$ and changing notation,
we may assume  that $u_1=q+y.$
Define
$$
u=1+y=u_1+(1-q).
$$
Since $q_0\le q=1_{{\widetilde A}_2},$
we have 
\beq
q_0u=q_0qu=q_0qu_1=q_0qu_1'=q_0u_1'.
\eneq
%\beq\label{TTstr1-12}
%&&\hspace{-0.5in}z_0u=(1+x_0)q_0u_1=(q_0+x_0)b_{0,1}'u_1{{=(q_0+x_0)b_{0,1}'u_1'}}=b_{0,1}'(q_0+x_0)b_{0,1}'u_1'.\\
%&&\hspace{-0.5in}z_0u
%=(1+x_0)q_0u_1=(
%=(q_0+x_0)b_{0,1}'u_1=b_{0,1}'(q_0+x_0)b_{0,1}'u_1.
%\eneq
%Since $b_{0,1}'e_{0,1}'=0,$  by \eqref{TTstr1-8},
Then, by \eqref{TTstr1-9},
\beq\label{TTstr1-13}
(z_0u)(z_0u)=(1+x_0)q_0uq_0(1+x_0)u=(1+x_0)q_0u_1'q_0(1+x_0)=0.
\eneq
In other words, $z_0u$ is a nilpotent in $A^{**}.$

On the other hand,  by \eqref{TTstr1-n-1-1}, 
$$
z_1f_\dt(e_{0,1})=(1-q_0)(1+x_1)f_\dt(e_{0,1})=0.
$$
Therefore 
\beq\label{TTstr-n10}
z_1c=((1-q_0)+x_1)c=cz_1=0\rforal c\in A_2.
\eneq
Thus, as $y\in A_2,$ 
\beq\label{TTstr-n11}
z_1u=z_1(1+y)=z_1+z_1y=z_1.
\eneq

Put $D_1=D+\C\cdot (1-q_0).$ Then $D_1\in {\cal C}.$ 
Let 
\beq	
D_2=\{d'\in D_1: d'f_\dt(e_{0,1})=f_\dt(e_{0,1})d'=0\}.
\eneq
Then, by \eqref{TTstr-n10},  $z_1\in D_2.$
Note that  $D_2$ is a hereditary \SCA\, of $D_1.$  
If $D_2$ is unital, say $e'$ is the unit, then $e'\not=1-q_0.$ 
Then $(1-q_0)-e'$ is also a nonzero projection.
One of them must be in $D.$ 
Since $D$ is stably projectionless, that one has to be zero. Since $(1-q_0)-e'\not=0,$ 
this leads a contradiction. So $D_2$ is not unital.
Since $D_1$ has stable rank one (see, for example, 3.3 of \cite{GLN}), 
so is $D_2.$  
Let $e_{D_2}$ be a strictly positive element of $D_2.$ 
In $A^{**},$ let $p_d=\lim_{n\to\infty} (e_{D_2})^{1/n}$ (converges in $A^{**}$). 
In particular, $p_d\le 1-q_0.$  So $q_0p_d=0.$
Moreover, since $e_{D_2}f_\dt(e_{0,1})=0,$ 
\beq
p_df_\dt(e_{0,1})=0.
\eneq
We also have $p_dB_0=B_0p_d=0.$ 
It follows that 
$
p_dq=0.
$
Therefore, 
\beq
z_0u(1-p_d)=z_0uq(1-p_d)=z_0uq=z_0u.
\eneq
Hence, $z_0u\in (1-p_d)A^{**}(1-p_d).$

%Then $D_1\cong {\tilde D}.$  Since $D$ has stable rank one,
Since $D_2+\C p_d$ has stable rank one, {{there}} is  an invertible element $z_1'\in D_2+\C p_d$
such that
\beq\label{TTstr1-14}
\|z_1-z_1'\|<\ep/16.
\eneq
We may write $z_1'=\lambda_1p_d+y_d,$  where 
$\lambda_1\in \C$ and $y_d\in D_2.$ We may also write $y_d=\lambda_2 (1-q_0)+d_0,$
where $\lambda_2\in \C$ and $d_0\in D.$

{{Set}}  $I=D\cap D_2.$ Then we have  {{the natural short exact sequence}} 
\beq
0\to I \to D_2+\C p_d\to^{\pi} \C\oplus \C\to 0.
\eneq
Then $\|\pi(z_1-z_1')\|<\ep/16.$ 
Thus, $|\lambda_1|<\ep/16$ and 
$|1-\lambda_2|<\ep/16.$ 
Put 
$$
z_1'':=(1/\lambda_2)z_1'=\eta p_d+(1-q_0)+d_0',
$$
where $\eta=\lambda_1/\lambda_2$ and $d_0'=d_0/\lambda_2\in D.$
Then $|\eta|<\ep/8$ and 
\beq\label{TTstr-n12}
\|z_1-z_1''\|<3\ep/16.
\eneq
Moreover, $z_1''$ is invertible in $D_2+\C p_d.$
%We also have 
%Write $z_1'=\lambda_1\cdot (p-q_0)+d$ for some scalar $\lambda_1\in \C$
%with $|\lambda_1|=1$ and $d\in D.$
%By looking the quotient $D_1/D,$ we may also write
%\beq\label{TTst1-15}
%z_1'=(p-q_0)+d+\eta(p-q_0)\andeqn |\eta|<\ep/32.
%\eneq
\Wlog, we may {{require}} that $\eta\not=0$ (since elements near $z_1'$ are invertible).
We may also write 
\beq
z_1''&=&z_1+(z_1''-z_1)=z_1+ (\eta p_d+(1-q_0)+d_0'-(1-q_0+x_1))\\\label{TTst1-15}
&=&
z_1+\eta p_d+d_0'-x_1.
\eneq
Therefore,
\beq\label{TTstr1-15-1}
\eta p_d+d_0'-x_1=z_1''-z_1.
\eneq
%Now view $z_0u\in (1-p_d)A^{**}(1-p_d).$
Since $z_0u$ is a nilpotent, $z_0u+\eta(1-p_d)$ is invertible in $(1-p_d)A^{**}(1-p_d).$
Let $\zeta_1$ {{denote}} the inverse of $z_0u+\eta(1-p_d)$ in $(1-p_d)A^{**}(1-p_d)$ and
$\zeta_2$ the inverse of $z_1''$ in $D_2+\C\cdot p_d.$
Then
\beq\label{TTstr1-15+}
(z_0u+\eta(1-p_d)\oplus z_1'')(\zeta_1\oplus \zeta_2)=(1-p_d)+p_d=1.
\eneq
It follows that
\beq\label{TTstr1-16}
z_2:=z_0u+\eta(1-p_d)+z_1''\in {\rm{GL}}(A^{**}).
\eneq
However, by \eqref{TTst1-15} and \eqref{TTstr-n11},
\beq\label{TTstr1-17}
z_2 &=&z_0u+\eta(1-p_d) +z_1''\\
&=& z_0u+\eta(1-p_d)+z_1+\eta p_d +(d_0'-x_1)\\
&=&z_0u+z_1 +(d_0'-x_1)+\eta\cdot 1\\
&=& (z_0+z_1)u+(d_0'-x_1)+\eta\cdot 1\\\label{TTstr1-nn20}
&=& zu+(d_0'-x_1)+\eta \cdot 1\in {\widetilde A}.
\eneq
It follows that $z_2\in {\rm{GL}}({\widetilde A}).$
We have (by \eqref{TTstr1-nn20}, \eqref{TTstr1-15-1}, 
%\eqref{TTstr-n11} 
and \eqref{TTstr-n12})
\beq\label{TTstr1-18}
\|zu-z_2\| &=&\|(d_0'-x_1)+\eta \cdot 1\|\\
%(z_0u_1+z_1)-z_2\|\\
&\le & \|d_0'-x_1+\eta p_d\|+\eta\|(1-p_d)\|\\
%\|z_1-(\eta(1-(p-q_0))+z_1')\|\\
&=&\|z_1''-z_1\|+\eta<3\ep/16+\ep/8=5\ep/16.
\eneq
Therefore (using also \eqref{TTstr1-10}),
\beq\label{TTstr1-19}
\|xu-z_2\|<\ep,\,\,\, {\rm or}\,\,\,\|x-z_2u^*\|<\ep.
\eneq
Since $z_2$ is invertible so is $z_2u^*.$ However, $u\in {\widetilde A}. $  One concludes that $z_2u^*$ is in ${\rm{GL}}({\widetilde A}).$

\end{proof}

At this point, we would like to introduce the following definition:

\begin{df}\label{DgTR}
Let $A$ be a simple \CA. {{Suppose that $A$ is stably projectionless.}}
We shall say that $A$ 
%is a stably projectionless simple \CA\ with 
{{has}} generalized tracial rank at most one,
and write ${\rm gTR}(A)\le 1,$  if for any $a\in {\rm Ped}(A)_+,$ $\overline{aAa}\in {\cal D}.$
(This extends the definition of generalized tracial rank at most one in the unital case \cite{GLN}.)
%if $A$ is TA${\cal C}_0'.$

\end{df}

\begin{prop}\label{PDgTR}
A separable  stably projectionless  simple \CA\, $A$ has  generalized tracial rank at most one,
i.e., ${\rm gTR}(A)\le 1,$
if, and  only if,  for some $a\in {\mathrm{Ped}(A)}_+\setminus \{0\},$ $\overline{aAa}\in {\cal D}.$
\end{prop}

\begin{proof}
Let $A$ be a separable stably projectionless  simple \CA. Suppose 
that there is $a\in {\mathrm{Ped}(A)}_+\setminus \{0\},$ $\overline{aAa}\in {\cal D}.$
It suffices to show that, for any $b\in {\mathrm{Ped}(A)}_+\setminus \{0\},$ $\overline{bAb}\in {\cal D}.$

There are $b_1, b_2,...,b_k\in A_+$ and $g_i\in C_0((0, \infty)_+$ such that 
%\beq
$b\le \sum_{i=1}^k g_i(b_i). $
%\eneq
By repeated application of \ref{Lfullep}, and applying 2.3 (a) of \cite{Rr11}, one obtains 
$x_1, x_2,...,x_n\in A$ such that
\beq
\sum_{i=1}^n x_i^*ax_i=b.
\eneq
Let $Z=(x_1^*a^{1/2}, x_2^*a^{1/2},...,x_n^*a^{1/2})$ be {{an}} $n\times n$  matrix in ${\rm M}_n(A)$
with zero {{rows}} except  {{for}} the first {{row}}. 
Then 
\beq
ZZ^*=\diag(b,\overbrace{0,0,...,0}^{n-1})\andeqn  Z^*Z\le {\rm M}_n(\overline{aAa}).
\eneq
Let $Z^*=U(ZZ^*)^{1/2}$ be the polar decomposition of $Z^*$ in ${\rm M}_n(A)^{{**}}.$ 
Then $UZZ^*U^*=ZZ^*\in {\rm M}_n(\overline{aAa}).$ It follows that 
map $x\mapsto UxU^*$ from $\overline{ZZ^*{\rm M}_n(\overline{aAa})ZZ^*}$ to $\overline{bAb}$
is an isomorphism. 
Since $\overline{ZZ^*{\rm M}_n(A)ZZ^*}\cong \overline{bAb},$  
\beq\label{PDgTR-n1}
\overline{bAb}\cong \overline{ZZ^*{\rm M}_n(\overline{aAa})ZZ^*}.
\eneq
It follows from  \ref{PtadMk} that ${\rm M}_n(\overline{aAa})\in {\cal D}.$  Then, by \ref{Phered}, 
$\overline{ZZ^*{\rm M}_n(\overline{aAa})ZZ^*}\in {\cal D}.$  By \eqref{PDgTR-n1}, $\overline{bAb}\in {\cal D},$
as desired.

\end{proof}

\begin{prop}\label{PWAff}
Let $A\in {\cal D}.$ Suppose that $A={\rm Ped}(A)$ (see \ref{CCped}).  Then the map ${\rm Cu}(A)\to {\rm LAff}_{0+}({\overline{\rm{T}(A)}^{\rm w}})$
%$W(A)\to V(A)\sqcup {\rm LAff}_{b,+}(\overline{{\mathrm{T}(A)}}^\mathrm{w})$
is
an 
%order
 isomorphism {{of}} ordered semigroups.
% surjective and injective.
\end{prop}

\begin{proof}
Let us  first show that the map is surjective.
This follows  the same lines of the  proof of 5.3 of \cite{BPT} as shown in 10.5 of \cite{GLN}
using  (4) of \ref{str=1}. (One can also use the proof 6.2.1 of \cite{Rl} by
applying  \ref{LappZ} as (D) in that proof.)

Let us provide the details.
First,  $\overline{{\mathrm{T}(A)}}^\mathrm{w}\subset {\tilde{T}}(A)$ is compact  (see \ref{DTtilde})
and (by \ref{compactrace}), it does not contain zero, and, as $A$ is simple,  for every $a\in A_+\setminus \{0\},$
$\tau(a)>0$ for all $\tau\in \overline{{\mathrm{T}(A)}}^\mathrm{w}.$

The proof follows the same lines as Theorem 5.3 of \cite{BPT}. Note that the injectivity of the map follows from \ref{Comparison}.
 it suffices to show that the map $a\mapsto \mathrm{d}_\tau(a)$ is surjective from $W(A)_+$ onto ${\rm LAff}_{b,0+}({\overline{{\mathrm{T}(A)}}^\mathrm{w}}).$
 %${\rm LAff}_{b,+}|_{\overline{{\mathrm{T}(A)}}^\mathrm{w}}.$
Let $f\in {\rm LAff}_{b,0+}(\overline{{\mathrm{T}(A)}}^\mathrm{w})$  with $f(\tau)>0$ for all  $\tau\in \overline{{\mathrm{T}(A)}}^\mathrm{w}.$
We may assume that $f(\tau)\le 1$ for all $\tau\in {{\overline{{\mathrm{T}(A)}}^\mathrm{w}}}.$
%{\mathrm{T}(A)}.$
 As in the proof of 5.3 of \cite{BPT}, it suffices to find a sequence of $a_i\in {\mathrm{M}}_2(A)_+$ such that
 $a_i\lesssim a_{i+1},$
 $\la a_n\ra\not=\la a_{n+1}\ra$ (in $W(A)$) and
 $$
 \lim_{n\to\infty}\mathrm{d}_{\tau}(a_n)=f(\tau)\rforal \tau\in {{\overline{{\mathrm{T}(A)}}^\mathrm{w}.}}
 % {\mathrm{T}(A)}.}
 $$
Since $f\in {\rm LAff}_{b,0+}({\overline{{\mathrm{T}(A)}}^\mathrm{w}}),$ {{by definition (see \ref{DAq}),}} we can  find 
{{an increasing}} sequence $f_n\in
\Aff_0({\mathrm{T}}_1(A))$
%{\mathrm{T}(A)})_{++}$ 
such that, for all $\tau\in {\overline{{\mathrm{T}(A)}}}^{\mathrm{w}},$ 
\beq\label{Aff=W-L-1}
{{0<}}f_n(\tau)< f_{n+1}(\tau),\,\,\,n=1,2,...{{,\andeqn}}
\lim_{n\to\infty} f_n(\tau)=f(\tau).
%\tforal \tau\in {{\overline{{\mathrm{T}(A)}}^\mathrm{w}.}}
% {\mathrm{T}(A)}.
 \eneq
% Choose a decreasing sequence of positive numbers $(\ep_n)$ such that $\lim_{n\to\infty}\ep_n=0.$
Since $f_{n+1}-f_n$ is continuous and strictly positive on the
compact set $\overline{{\mathrm{T}(A)}}^\mathrm{w},$  there is $\ep_n>0$ such that
$(f_n-f_{n+1})(\tau)>\ep_n$ for all $\tau\in \overline{{\mathrm{T}(A)}}^{\mathrm{w}},$ $n=1,2,....$
We may choose $\ep_n$ so that $\ep_n\searrow 0$ as $n\to\infty.$

Since $A$ is an infinite dimensional simple \CA\, and since $A\in {\cal D}$ (see also \ref{LappZ}),  for each $n,$ there is a \SCA\,
$C_n$ of $A$ with $C_n\in {\cal C}$ and an element $b_n\in (C_n)_+$
such that \beq\label{Aff=W-L-2}
{\rm dim}\pi(C_n)\ge  (16/\ep_n)^2\,\,\,\text{for \,each \,irreducible\,representation} \,\,\pi\,\,\,of \,\,C_n,\\
0<\tau(f_n)-\tau(b_n)<\ep_n/4\rforal \tau\in \overline{{\mathrm{T}(A)}}^\mathrm{w}.
\eneq
Applying
%\ref{Affon2},
\ref{str=1},
one obtains an element $a_n\in {\mathrm{M}}_2(C_n)_+$
such that \beq\label{Aff=W-L-3} 0<t(b_n)-{\rm d}_t(a_n)<\ep_n/4\tforal
t\in {\mathrm{T}}(C_n). \eneq It follows that
\beq\label{Aff=W-L-4}
0<\tau(f_n)-\mathrm{d}_\tau(a_n)<\ep_n/2\tforal \tau\in \overline{{\mathrm{T}(A)}}^\mathrm{w}.
\eneq
One then
checks that $\lim_{n\to\infty}\mathrm{d}_\tau(a_n)=f(\tau)$ for all $\tau\in 
{{\overline{{\mathrm{T}(A)}}^\mathrm{w}.}}$
%{\mathrm{T}(A)}.$ 
Moreover, $\mathrm{d}_\tau(a_n)<\mathrm{d}_\tau(a_{n+1})$ for all $\tau\in
{\mathrm{T}(A)},$ $n=1,2,....$ It follows  from \ref{Comparison} that
$a_n\lesssim a_{n+1},$ $[a_n]\not=[a_{n+1}],$ $n=1,2,....$  This
ends the proof of {{surjectivity.}}

It is clear that the map is order preserving.  By \ref{Comparison}, if $a, b\in (A\otimes {\cal K})_+$ and $d_\tau(a)=d_\tau(b)$
for all $\tau\in \overline{{\mathrm{T}(A)}}^\mathrm{w},$ 
then $a\sim b$ in ${\rm{Cu}}(A).$  So the map is  injective. By \ref{Comparison} again, the inverse map is also order preserving.

\end{proof}

\begin{cor}\label{Ccontsc}
Let $A\in {\cal D}.$ Then there exists an element $a\in A_+\setminus \{0\}$ such that
$\overline{aAa}$ has continuous scale.
\end{cor}

\begin{proof}
{{Note that, by \ref{TTstr1}, $A$ has stable rank one.}} 
Pick  a nonzero element $c\in {\rm Ped}(A)_+$  and set $B=\overline{cAc}.$ 
Then $B={\rm Ped}(B).$ So, we may assume that $A={\rm Ped}(A).$ 
{{{{Hence}}, combining with}} \ref{PWAff} {{(see \cite{CEI}),}}
% one easily finds 
there exists an element $a\in A_+\setminus \{0\}$ such that
$\mathrm{d}_\tau(a)$ is continuous on $\overline{{\mathrm{T}(A)}}^\mathrm{w}.$
%Note, by \ref{PD0qc}, ${{\rm{QT}}}(A)={\rm{T}}(A).$
%It then follows that
%$\mathrm{d}_\tau$ is continuous on ${\widetilde{\mathrm{T}}}(A).$
{{A}}pplying  the second part of \ref{Pconscale},
%\ref{Pconscale},
one concludes that $\overline{aAa}$ has continuous scale.
\end{proof}

\begin{prop}\label{D=D0K0}
Let $A\in {\cal D}$ be a separable \CA\, with ${\rm K}_0(A)=\{0\}.$  Then $A$ has the properties described in  \ref{TCCdvi} (and \ref{Cuniformful}) but replacing ${{\cal C}_0^{0}}'$ (and ${\cal C}_0^{0}$)  by ${\cal C}_0.$
\end{prop}

\begin{proof}
%We may assume, \wilog, that $A$ has continuous scale.
It follows from \ref{PWAff} that the map ${\rm Cu}(A)\to {\rm LAff}_{0+}({\overline{{\mathrm{T}(A)}}^{\rm w}})$ is surjective.
Note that $A$ has stable rank one (by \ref{TTstr1}).
Then,  by \ref{CDdiag} and by  Definition \ref{DNtr1div}, $A$ has the tracially approximate divisible property.
The proof of \ref{TCCdvi} applies to $A$ with ${{\cal C}_0^{0}}'$ replaced by ${\cal C}_0.$
One then also obtains the conclusion of \ref{Cuniformful} with ${\cal C}_0^{0}$ replaced by ${\cal C}_0.$
\end{proof}

We would like to summarize some of the facts {{we have established.}} 
%below.
% the following facts.

\begin{prop}\label{Pgtr1}
Let $A$ be a separable simple \CA. {{Suppose that $A$}} is stably projectionless and ${\rm gTR}(A)\le 1.$
Then the following statements hold.

(1) $A$ has stable rank one;

(2) Every quasitrace of $A$ is a trace;

(3)  ${\mathrm{Cu}}(A)={\rm LAff}_+({\widetilde{\rm{T}}}(A));$
%$A$ has strict comparison for positive elements;

(4) If $A={\mathrm{Ped}(A)},$  then $A\in {\cal D};$

(5) If $B\subset A$ is a  hereditary \SCA, then  ${\rm gTR}(B)\le 1;$

(6)  ${\mathrm{M}}_n(A)$ is stably projectionless and ${\rm gTR}({\mathrm{M}}_n(A))\le 1$ for every integer $n\ge 1.$
\end{prop}

\section{The \CA s ${\cal W}$ and the class ${\cal D}_{0}$}

%\begin{df}\label{DW}
%Let $W$ be the non-unital simple \CA\, with $K_i(W)=\{0\}$ ($i=0,1$),
%which is an inductive limit of non-commutative $1$-dimensional compleces.
%and which has only one trace which is a state.
%
%Let ${\cal Z}_0$ be the stably projectionless separable simple \CA\,
%with $K_0={\rm ker}\rho_{{\cal Z}_0}=\Z$ which is an inductive limit of non-commutative $1$-dimensional compleces
%which has only one trace which is also a state.
%\end{df}

%\begin{prop}\label{PuniqtracePA}
%Let $A$ be a non-unital simple \CA\, which has one and only one densely defined semi-finite
%trace which is also a state. Then $A={\mathrm{Ped}(A)}.$
%\end{prop}

\begin{df}\label{DWtrace}
%%%%%%%%%%%%%%%%%%%%%%%%%%
\iffalse
Let $A$ be a non-unital separable \CA, {{and let}}
$\tau\in {\mathrm{T}(A)}.$
Let us  say that $\tau$ is a ${\cal C}_0$-trace if there exists a sequence of
\cpc s $\phi_n: A\to D_n\in {\cal C}_0$ such that
\beq\nonumber
&&\lim_{n\to\infty}\|\phi_n(ab)-\phi_n(a)\phi_n(b)\|=0\rforal a,\,b\in A\andeqn\\
&&\tau(a)=\lim_{n\to\infty}t_n(\phi_n(a))\rforal a\in A,
\eneq
where $t_n\in  {\mathrm{T}}(D_n),$ $n=1,2,....$
\fi
%%%%%%%%%%%%%%%%%%%%%%%%%
\vspace{0.1in}
{{Recall (see \ref{DW}) that ${\cal W}$ is a unital separable simple 
\CA\, with ${\rm K}_i({\cal W})=0,$ $i=0,1,$ which is in both ${\cal M}_0$ and ${\cal D}_0.$}}
%% inductive limit of

Let $A$ be a non-unital separable \CA, {{and let}} $\tau\in {\mathrm{T}(A)}.$
Let us say that $\tau$ is a $\cal W$-trace
if there exists a sequence of
\cpc s $\phi_n:A\to {\cal W}$
such that
\beq\nonumber
&&\lim_{n\to\infty}\|\phi_n(ab)-\phi_n(a)\phi_n(b)\|=0\rforal a,\,b\in A\andeqn\\
&&\tau(a)=\lim_{n\to\infty}\tau_{\cal W}(\phi_n(a))\rforal a\in A,
\eneq
where $\tau_{\cal W}$ is the unique tracial state on ${\cal W}.$

\end{df}

\begin{thm}\label{Ttracekerrho}
Let $A$ be a separable simple \CA\, with $A={\mathrm{Ped}(A)}.$
If every tracial state $\tau\in {\mathrm{T}(A)}$ is a ${\cal W}$-trace, then
${\mathrm{K}}_0(A)={\rm ker}\rho_A$ {{(see \ref{Dkerrho} for the definition of $\rho_A$).}}
%(for any choice of $e\in {\mathrm{Ped}(A)}_+\setminus \{0\}$).
\end{thm}

\begin{proof}
%Fix $\tau\in {\mathrm{T}(A)}.$
Suppose that there are two projections $p, q\in {\mathrm{M}}_k({\widetilde A})$ such that
$x=[p]-[q]\in {\mathrm{K}}_0(A)$ and $\tau(p)\not=\tau(q)$ for some $\tau\in  {\mathrm{T}(A)}.$
%Note, here we also assume  
{{Recall that}}  $[p]-[q]\in {\rm{K}}_0(A)$ {{means}} that $\pi(p)$ and $\pi(q)$ have the same rank in ${\mathrm{M}}_k(\C),$
where $\pi: {\mathrm{M}}_k({\widetilde A})\to {\mathrm{M}}_k(\C)$ is the quotient map.

{{Set}} $d=|\tau(p)-\tau(q)|.$
Denote still by $\tau$ the {{canonical}} extension of $\tau$ to ${\widetilde A}$ 
%as well as
and also to
%on 
${\mathrm{M}}_n({\widetilde A}).$
If $\tau$ were a ${\cal W}$-trace, {{then}} there  would be a sequence $(\phi_n)$
of \cpc s from ${\mathrm{M}}_k(A)$ into ${\mathrm{M}}_k({\cal W})$  such that
\beq\nonumber
&&\lim_{n\to\infty}\|\phi_n(a)\phi_n(b)-\phi_n(ab)\|=0\rforal a, b\in {\mathrm{M}}_k(A)\andeqn\\
&&\tau(a)=\lim_{n\to\infty}\tau_{\cal W}\circ\phi_n(a)\rforal a\in {\mathrm{M}}_k(A).
\eneq
%where $t_0$ is the unique tracial state on ${\cal W}.$
{{Denote by}}  ${\tilde \phi}_n: M_k({\widetilde{A}})\to {\mathrm{M}}_k({{\widetilde{{\cal W}}}})$ {{the canonical unital}}
extension of  the \cpc\,  $\phi_n.$
%such that ${\tilde \phi_n}(1)=1_{\tilde {\cal W}}.$
%It follows that
{{then}}
$$
\lim_{n\to\infty}\|{\tilde \phi}_n(a){\tilde \phi}(b)-{\tilde \phi}(ab)\|=0\rforal a, b\in {\mathrm{M}}_k({\widetilde A}).
$$
Let $\tau_{\cal W}$ also denote the {{canonical}} extension of $\tau_{\cal W}$ on ${\mathrm{M}}_k({{{\widetilde {\cal W}}}}).$
Then we also have
$$
\tau(a)=\lim_{n\to\infty} t\circ {\tilde \phi}_n(a)\rforal a\in {\mathrm{M}}_k({\widetilde A}).
$$
%To simplify notation, \wilog, 
{{Passing to a subsequence,}}  we may assume that
\beq\label{Ttkr-1}
|\tau_{\cal W}\circ{\tilde \phi}_n(p)-\tau_{\cal W}\circ {\tilde \phi}_n(q)|\ge d/2\rforal n.
\eneq
%There is an integer $n_0\ge 1$ such that
There are projections $p_n,\, q_n\in {\mathrm{M}}_k({{{\widetilde{\cal W}}}})$ such that
\beq\label{Ttkr-2}
%&&\|p_n-{\tilde \phi}_n(p)\|<1/2\andeqn \|q_n-{\tilde \phi}_n(q)\|<1/2\rforal n\ge n_0,\\
&&\lim_{n\to\infty}\|{\tilde \phi}_n(p)-p_n\|=0\andeqn
%\\\label{Ttkr-3}
%&&
\lim_{n\to\infty}\|{\tilde \phi}_n(q)-q_n\|=0.
\eneq	

Since $\pi(p)$ and $\pi(q)$ have the same rank, there {{exists}} $v\in {\mathrm{M}}_k({\widetilde A})$ such that
$\pi(v^*v)={{\pi(p)}}$ and $\pi(vv^*)={{\pi(q)}}.$
{{Denote by}}  $\pi_w: {\mathrm{M}}_k({\widetilde {\cal W}})\to {\mathrm{M}}_k$ the quotient map.  Then
%for all $n\ge n_0,$
\beq\label{Ttkr-4}
\lim_{n\to\infty}\|\pi_w\circ \phi_n(v^*v)-\pi_w(p_n)\|=0\andeqn
\lim_{n\to\infty}\|\pi_w\circ \phi_n(vv^*)-\pi_w(q_n)\|=0.
\eneq
It follows that $\pi_w(p_n)$ and $\pi_w(q_n)$ are equivalent projections in ${\mathrm{M}}_k$ for all
large $n.$  Since ${\rm K}_0({\cal W})=0,$ it follows
that $[p_n]-[q_n]=0$ in  ${\rm K}_0({\cal W}),$ which means
that $p_n$ and $q_n$ are equivalent in ${\mathrm{M}}_k({\widetilde {\cal W}})$ since ${\widetilde {\cal W}}$ has stable rank one.
In particular,
$$
\tau_{\cal W}(p_n)=\tau_{\cal W}(q_n)
$$
for all sufficiently large $n,$  in contradiction  with  \eqref{Ttkr-1} and \eqref{Ttkr-2}.
% and \eqref{Ttkr-3}
%hold.

\end{proof}

\begin{prop}\label{Pwtracest}
Let $A$ be a separable simple \CA\,  with a ${\cal W}$-trace $\tau\in {\mathrm{T}(A)}.$
Let $0\le a_0\le 1$ be a strictly positive element of $A.$ 
%(and $A\not=0$).
Then there exists a sequence of \cpc s  $\phi_n: A\to {\cal W}$  such that $\phi_n(a_0)$ is a 
strictly positive element,  and 
\beq\nonumber
&&\lim_{n\to\infty}\|\phi_n(a)\phi_n(b)-\phi_n(ab)\|=0\tforal a,\, b\in A\andeqn\\
&&\tau(a)=\lim_{n\to\infty} \tau_{\cal W}\circ \phi_n(a)\tforal a\in A.
\eneq
%where $\tau_{\cal W}$ is the unique tracial state of ${\cal W}.$
\end{prop}

\begin{proof}
We may assume
that
\beq\label{123--1}
\tau(a_0^{1/n})> 1-1/2n,\,\,\, n=1,2,....
\eneq
Since $\tau$ is a ${\cal W}$-trace, there exists a sequence  of \cpc s
$\psi_n: A\to {\cal W}$ such that
\beq\nonumber
&&\lim_{n\to\infty}\|\psi_n(a)\psi_n(b)-\psi_n(ab)\|=0\rforal a,\, b\in A\andeqn\\\label{123-n1}
&&\tau(a)=\lim_{n\to\infty} \tau_{\cal W}\circ \psi_n(a)\rforal a\in A.
\eneq

Put $b_n=\psi_n(a_0^{1/n}).$  
%\Wlog, 
By {{\eqref{123--1} and
\eqref{123-n1},}}  passing to a subsequence of $(\psi_n),$ we may assume
\beq\label{Pwtst-1}
\tau_{\cal W}(b_n)\ge 1-1/n,\,\,\, n=1,2,....
\eneq

Consider {{the}} the hereditary \SCA\, $B_n=\overline{\psi_n(A){\cal W}\psi_n(A)}.$
{{Since $a_0^{1/n}$ is a strictly positive element, by \ref{180915sec2}, $b_n$ is a strictly positive element of $B_n.$}}
%  Choose a strictly positive element
%$0\le e\le 1$ of ${\cal W}$ {\blue{and a strictly positive element $e_n$ of  $B_n.$
%One can choose $e_n$ so that $b_n\le e_n\le 1$.}} 
{{Recall that ${\rm{Cu}}^{\sim}(B_n)={\rm{Cu}}^\sim({\cal W})=\R\cup\{+\infty\}.$
 Let $r_n=\la b_n\ra\in \R_+\subset \R\cup\{+\infty\},$
$n=1,2,....$  Consider the map $r\to r_n \cdot r$ (for $r\in \R$) and $+\infty\to +\infty.$
By Theorem 1.0.1 of}} \cite{Rl} (see also 6.2.4 of \cite{Rl}, Theorem 1.2 of \cite{Jb}, or  Theorem 1.1 of \cite{Raz}, and 
\ref{inductived0}), there is a \hm\, 
%Recall that ${\rm{\bf{Cu}}^\sim
%By \cite{Rl}, there is a \hm\,
$h_n: B_n\to {\cal W}$ such that
% $\la h_n(b_n)\ra =\la e\ra$ in ${\mathrm{Cu}}({\cal W})$ and
$h_n(b_n)$ is strictly positive.  Since $B_n$ has a unique trace,
there is $\af_n>0$ such that
\beq\label{Pwtst-2}
\af_n\tau_{\cal W}(b)=\tau_{\cal W}\circ h_n(b)\rforal b\in B_n,\,\,\, n=1,2,...
\eneq
Since $h_n(b_n)$ is strictly positive,
\beq\label{Pwtst-3}
\lim_{k\to\infty}\tau_{\cal W}\circ h_n(b_n^{1/k})=1.
\eneq
Since $B_n\subset {\cal W},$ and by {{\eqref{Pwtst-2},}} $\af_n\tau_{\cal W}|_{B_n}=\tau_{\cal W}\circ h_n,$ 
% \eqref{Pwtst-2} and 
\eqref{Pwtst-3} {{implies}} that
%\beq\label{Pstst-4}
$\af_n\ge 1.$
%\eneq

On the other hand, by \eqref{Pwtst-1}, together with \eqref{Pwtst-2},  we have that
$$
1-1/n \le \tau_{\cal W}(b_n)= {\tau_{\cal W}\circ h_n(b_n)\over{\af_n}}\le  
%{\tau_{\cal W}\circ h_n(e_n)\over{\af_n}}\le 
1/\af_n,\,\,\, n=1,2,....
$$
Therefore,
$$
{{1\le}} \af_n\le {1\over{1-1/n}},\,\,\, n=1,2,...,
$$
from which  it follows that $\lim_{n\to\infty} \af_n=1.$
%{\blue{By Brown's stable isomorphism theorem \cite{Br1}, one may view ${\cal W}$ as a hereditary \CA\, 
Set $\phi_n=h_n\circ \psi_n.$ 
%One verifies that 
{{Since $\phi_n(a_0^{1/n})=b_n$ is strictly positive (in $B_n$), by \ref{180915sec2} as above, the sequence}} $(\phi_n)$ meets the requirements.
\end{proof}

The following two statements will be  established in \cite{egln00}.

\begin{thm}\label{TMW}
Let $A$ be a separable simple \CA\, with finite nuclear dimension and  with $A={\mathrm{Ped}(A)}$ such that
${\mathrm{T}(A)}\not={\O},$ ${\rm K}_0(A)={\rm ker}\rho_A,$ and
every tracial state is a ${\cal W}$-trace.
Suppose also that  every hereditary \SCA\, of $A$ with continuous scale  is tracially approximately divisible. Then $A\in {\cal D}_0.$
%In particular, $A\otimes U\in {\cal D}_{0}$ for any UHF-algebra $U.$
\end{thm}

\begin{thm}\label{TWtrace}
Let $A$ be a  separable simple  \CA\, with finite nuclear dimension  {{and}} with $A={\mathrm{Ped}(A)}.$
Suppose that ${\mathrm{T}(A)}\not={\O}.$
Then
%$A\otimes Z,
$A\otimes {\cal W}\in {\cal D}_{0}.$
%where $Z$ is as described in \ref{Pwtrace}.
\end{thm}

%\begin{proof}
%It follows \ref{Pwtrace} that every tracial state of $A\otimes {\cal W}$ is a ${\cal W}$-trace.  It follows from \ref{Ttracekerrho}
%that $K_0(A\otimes {\cal W})=\{0\}.$ Moreover, $A\otimes {\cal W}\otimes Q=A\otimes {\cal W}.$

%\rho_{A\otimes {\cal W}}.$
%Then \ref{TMW} applies.
%\end{proof}

\providecommand{\href}[2]{#2}

%&&&&&&
%\section{Regularity of  \CA s in ${\cal D}_{0}$}

%\begin{lem}\label{str1-L1}
%Let $A$ be a unital \CA\, of stable rank one.
%Suppose that $x\in A$ is not invertible. Then, for any $\ep>0,$
%there exists $y\in A,$  a unitary $u\in U(A)$  and a non-zero $a\in A_+$ such that
%$$
%\|x-y\|<\ep\andeqn a(yu)=(yu)a=0.
%$$
%\end{lem}

%&&&&&


\begin{thebibliography}{10}

{\small
% \bibitem{AAP} C.  A. Akemann, J.  Anderson and G. K. Pedersen, {\em Excising states of $C^*$-algebras},
% Canad. J. Math. {\bf 38} (1986), 1239--1260.

% \bibitem{Arv} W. Arveson, {\em Notes on extensions of \CA s},
%Duke Math. J. {\bf 44} (1977),  329--355


\bibitem{Btrace} B. E.  Blackadar, {\em Traces on simple AF \CA s},  J. Funct. Anal.  {\bf 38} (1980),  156--168.

\bibitem{BH} B. E. Blackadar and D. E. Handelman, {\em Dimension functions and traces on \CA s},  J. Funct. Anal. 
{\bf 45} (1982),  297--340.


 \bibitem{Br1}  L. G. Brown, {\em Stable isomorphism of hereditary subalgebras of \CA s},
  Pacific J. Math.  {\bf 71} (1977),  335--348.

\bibitem{BP} {L. G. Brown and G. K. Pedersen, {\em On the geometry of the unit ball of a \CA},
J. Reine  Angew. Math.  {\bf 469} (1995), 113--147.}

\bibitem{BPT} N. Brown,  F. Perera and A. Toms, {\em The Cuntz semigroup, the Elliott conjecture, and dimension functions on $C^*$-algebras},  J. Reine Angew. Math.  {\bf 621}  (2008), 191--211.

\bibitem{BT} N. P.   Brown and A. S. Toms,  {\em Three applications of the Cuntz semigroup},  Int. Math. Res. Not. IMRN 2007, no. 19, Art. ID rnm068, 14 pp.


\bibitem{CEI} K.  Coward, G. A. Elliott and C.  Ivanescu, {\em The Cuntz semigroup as an invariant for \CA s} J. Reine Angew. Math.  {\bf 623}  (2008), 161--193.

\bibitem{Cu1} {{J. Cuntz, {\em The structure of multiplication and addition in simple \CA s},  Math. Scand. {\bf 40} (1977),  215--233.}}

\bibitem{CP}  J. Cuntz and G. K.  Pedersen, {\em Equivalence and traces on \CA s},
 J. Funct. Anal.  {\bf 33} (1979),  135--164.

\bibitem{ED} M.   D\u{a}d\u{a}rlat and S.  Eilers, {\em  On the classification of nuclear \CA s}, Proc. London Math. Soc.
{\bf 85} (2002),  168--210.

\bibitem{DL}
M.~D{\u a}d{\u a}rlat and T.~Loring, \emph{A universal multicoefficient theorem
  for the $\textrm{Kasparov}$ groups}, Duke Math. J. \textbf{84} (1996),
  355--377.


\bibitem{ELP1} S.  Eilers, T. Loring and G. K. Pedersen, {\em Stability of anticommutation relations: an application of noncommutative CW complexes}, J. Reine Angew. Math. {\bf 499} (1998), 101--143.

\bibitem{ELP2} S. Eilers, T. Loring and G. K. Pedersen, {\em  Fragility of subhomogeneous C*-algebras with one-dimensional spectrum},  Bull. London Math. Soc.  {\bf 31} (1999), 337--344.

\bibitem{Ellicm} G. A. Elliott, {\em The classification problem
for amenable $C^*$-algebras}. Proceedings of the International Conference of Mathematics,
%Vol 2
Z\"urich, 1994, 922--932, Birkh{\"a}user, Basel, 1995.

\bibitem{point-line}
G. A. Elliott, {\em An invariant for simple $\mbox{C}$*-algebras}, Canadian Mathematical Society.
  1945--1995 \textbf{Vol 3}, (1996), 61--90, Canadian Math. Soc., Ottawa, ON, 1996.

\bibitem{EGLN} G. A.   Elliott, G. Gong, H. Lin and Z. Niu, {\em  On the classification of simple amenable \CA s with finite decomposition rank, II},  arXiv:1507.03437.

\bibitem{egln00} G. A.   Elliott, G. Gong, H. Lin and Z. Niu, {\em The classification of simple separable KK-contractible C*-algebras with finite nuclear dimension}, arXiv:1712.09463


\bibitem{EN-TAS} G. A.   Elliott and Z. Niu, {\em On tracial approximation}, J. Funct. Anal. {\bf 254} (2008), 396--440.


\bibitem{ERS11} G. A.  Elliott, L.  Robert  and L. Santiago, {\em The cone of lower semicontinuous traces on a
\CA},  Amer. J. Math. {\bf 133}  (2011),  969--1005.

\bibitem{GL1} G. Gong and H.  Lin, {\em Almost multiplicative morphisms and K-theory},
 Internat. J. Math.  {\bf 11} (2000),  983--1000.

\bibitem{GL2} G. Gong and H.  Lin, {\em  On classification of non-unital simple amenable \CA s, II},
in preparation.


\bibitem{GLN} G.~Gong, H.~Lin and Z.~Niu, {\em Classification of finite simple amenable ${\mathcal Z}$-stable $C^*$-algebras},
preprint, arXiv:1501.00135.

 \bibitem{HL} J. Hua and H. Lin, {\em Rotation algebras and Exel trace formula},  {Canad. J. Math. {\bf  67} (2015), 404--423. }

\bibitem{Jb} B.  Jacelon,  {\em A simple, monotracial, stably projectionless \CA}, J. Lond. Math. Soc.  {\bf 87} (2013),  365--383.

 \bibitem{KP} E.  Kirchberg and N. C.  Phillips, {\em Embedding of exact \CA s in the Cuntz algebra ${\cal O}_2$},
  J. Reine Angew. Math.  {\bf 525}  (2000), 17--53.

 \bibitem{KW} E.  Kirchberg and W.  Winter,  {\em Covering dimension and quasidiagonality},
  Internat. J. Math.  {\bf 15} (2004),  63--85.



\bibitem{Lncs1} H.   Lin, {\em Simple \CA s  with continuous scales and simple corona algebras},
 Proc. Amer. Math. Soc.  {\bf 112} (1991), 871--880.

\bibitem{Lnsuniq} H. Lin, {\em Stable approximate unitary equivalence of homomorphisms}, J. Operator Theory
{\bf 47} (2002), 343--378.

 \bibitem{Lnbk} H. Lin, {\em  An introduction to the classification of amenable \CA s},  World Scientific Publishing Co., Inc., River Edge, NJ, 2001. xii+320 pp. ISBN: 981-02-4680-3.

 \bibitem{Lncrell} H.   Lin, {\em  Traces and simple $C^*$-algebras with tracial topological rank zero},
  J. Reine Angew. Math. {\bf 568} (2004), 99--137.

\bibitem{Lncs2} H.  Lin, {\em Simple corona \CA s}, Proc. Amer. Math. Soc. {\bf 132}  (2004), 3215--3224.

\bibitem{LinTAI}
H.~Lin, \emph{Simple nuclear {C*}-algebras of tracial topological rank one},
  J. Funct. Anal. \textbf{251} (2007),  601--679.

\bibitem{Lncuntz} H.~Lin, \emph{Cuntz semigroups of C*-algebras of stable rank one and projective Hilbert modules},
prepirnt, 2010,  arXiv:1001.4558.
% \bibitem{Lnauct} H. Lin, {\em An approximate universal coefficient theorem}, Trans. Amer. Math. Soc. {\bf 357} (2005), 3375--3405.

%\bibitem{Linajm} H. Lin, {\em Asymptotically unitary equivalence and asymptotically inner automorphisms},  Amer. J. Math. {\bf 131} (2009), 1589--1677.

%\bibitem{LnHomtp}
%H.~Lin, \emph{Approximate homotopy of homomorphisms from {$\textrm{C}(X)$}
%  into a simple {C*}-algebra}, Mem. Amer. Math. Soc. \textbf{205} (2010),
%  no.~963, vi+131.


%\bibitem{Ln-hmtp} H.  Lin, {\em Homotopy of unitaries in simple \CA s with tracial rank one},  J. Funct. Anal. {\bf 258} (2010),  1822--1882.

%\bibitem{Lnclasn}
%H.~Lin, \emph{Asymptotic unitary equivalence and classification of simple
%  amenable {C*}-algebras}, Invent. Math. \textbf{183} (2011),  385--450.


\bibitem{Lnloc}  H.  Lin, {\em Locally AH algebras},  Mem. Amer. Math. Soc. {\bf  235} (2015), no. 1107.
% vi+109 pp.
%ISBN: 978-1-4704-1466-5; 978-1-4704-2225-7.

\bibitem{Lin-AU11}
H.~Lin, \emph{Homomorphisms from {AH}-algebras},  J. Topol.  Anal.
{\bf 9},  (2017),  67--125. \,\,
%arXiv: 1102.4631v1.

\bibitem{Lncbms} H. Lin, {\em  From the Basic Homotopy Lemma to the classification of \CA s},
CBMS Regional Conference Series in Mathematics, {\bf 124},   Amer.  Math.  Soc., Providence, RI, 2017. vi+240 pp. ISBN: 978-1-4704-3490-8.


\bibitem{Loringbk}  T. A. Loring, {\rm Lifting solutions to perturbing problems in \CA s}, Fields Institute Monographs, 8. American Mathematical Society, Providence, RI, 1997. x+165 pp. ISBN: 0-8218-0602-5. 
%(Reviewer: Hiroshi Takai) 46L85 (46J99 46L05 47A62)

\bibitem{Pbook}  G. K. Pedersen, {\em $C^*$-Algebras and Their Automorphism Groups},
 London Mathematical Society Monographs, 14. Academic Press, Inc.,
 %[Harcourt Brace Jovanovich, Publishers],
 London-New York, 1979. ix+416 pp. ISBN: 0-12-549450-5.

 \bibitem{Pedjot87} G. K. Pedersen, {\em Unitary extensions and polar decompositions in a \CA},  J. Operator Theory {\bf 17} (1987),  357--364.

\bibitem{Pclass} N. C.  Phillips, {\em A classification theorem for nuclear purely infinite simple \CA s}, Doc. Math. {\bf 5} (2000), 49--114.

\bibitem{Raz} S. Razak {\em On the classification of simple stably projectionless $C^*$-algebras},
Canad. J. Math.  {\bf 54} (2002),  138--224.


\bibitem{Rff1} M. A.  Rieffel, {\em  Dimension and stable rank in the K-theory of \CA s}. 
Proc. London Math. Soc. {\bf 46} (1983),  301--333.

 \bibitem{Rl} L.  Robert, {\em Classification of inductive limits of 1-dimensional NCCW complexes},  Adv. Math.  {\bf 231} (2012),  2802--2836.

 \bibitem{Rlz} L.   Robert, {\em Remarks on ${\cal Z}$-stable projectionless $C^*$-algebras},  Glasgow Math. J. {\bf 58} (2016), 273--277.

\bibitem{RorUHF2}  M.  R\o rdam, {\em On the structure of simple $C^*$-algebras tensored with a UHF-algebra},  J. Funct. Anal. {\bf 100} (1991),  1--17.


\bibitem{Rr11} M.   R\o rdam, {\em On the structure of simple \CA s tensored with a UHF-algebra. II},  J. Funct. Anal. {\bf 107} (1992),  255--269.


 \bibitem{Ror-KL-I} M.~R{\o}rdam, {\em Classification of certain infinite simple {C*}-algebras},  J. Funct. Anal.  {\bf 131}  (1995),   415--458.


\bibitem{Rrzstable} M.   R\o rdam, {\em The stable and the real rank of  ${\cal Z}$-absorbing \CA s},  Internat. J. Math.
{\bf 15} (2004), 1065--1084.

 %Aaron Tikuisis

%Regularity for stably projectionless, simple C?C?-algebras

\bibitem{RW} M.  R{\o}rdam and W.  Winter, {\em The Jiang-Su algebra revisited}, J. Reine Angew. Math.
{\bf 642} (2010), 129--155.

\bibitem{Thoms2}  K. Thomsen, {\em Homomorphisms between finite direct sums of circle algebras},
 Linear  Multilinear Algebra  {\bf 32} (1992), 33--50.


\bibitem{aTz} A.  Tikuisis,  {\em  Nuclear dimension, ${\cal Z}$-stability, and algebraic simplicity for stably projectionless
\CA s}, Math. Ann.  {\bf 358} (2014),  729--778.

\bibitem{TWW}
A~Tikuisis, S.~White and W.~Winter,
\emph{Quasidiagonality of nuclear {C*}-algebras}, Ann. of Math. {\bf 185} (2017), 229--284.


%J. Funct. Anal., 263 (5) (2012), pp. 1382?1407

\bibitem{Tsang} K.-W.  Tsang, {\em On the positive tracial cones of simple stably projectionless \CA s},
 J. Funct. Anal.  {\bf 227} (2005),  188--199.

\bibitem{Wann05} W.   Winter, {\em On topologically finite-dimensional simple \CA s},  Math. Ann. {\bf  332} (2005), 843--878.
\bibitem{Wnz} W.   Winter, {\em Nuclear dimension and  ${\cal Z}$-stability of pure  $C^*$-algebras},
 Invent. Math.  {\bf 187} (2012),  259--342.

\bibitem{Winter-Z}
W.~Winter, \emph{Localizing the $\textrm{Elliott}$ conjecture at strongly
  self-absorbing $\textrm{C*}$-algebras}, J. Reine Angew. Math. {\bf 692} (2014), 193--231.


\bibitem{Wccross} W.   Winter, {\em Classifying crossed product \CA s},  Amer. J. Math. {\bf 138} (2016), 793--820.

\bibitem{WZ} W.  Winter and J. Zacharias, {\em The nuclear dimension of \CA s}, Adv. Math.  {\bf 224} (2010),  461--498.
}
\end{thebibliography}
\end{document}